\documentclass[11pt,letterpaper]{amsart}

\usepackage[T1]{fontenc}
\usepackage[utf8]{inputenc}
\usepackage{multirow}
\usepackage{graphicx}
\usepackage{bbm}
\usepackage{bbding}
\usepackage{color}
\usepackage{nicefrac}
\usepackage{babel}
\usepackage{mathrsfs}
\usepackage{amsbsy}
\usepackage{amstext}
\usepackage{amsthm}
\usepackage{amssymb}
\usepackage{stmaryrd}
\usepackage{xy}
\xyoption{arrow}
\xyoption{matrix}

\usepackage{amsfonts}
\usepackage[unicode=true,pdfusetitle,
 bookmarks=true,bookmarksnumbered=false,bookmarksopen=false,
 breaklinks=false,pdfborder={0 0 0},pdfborderstyle={},backref=false,colorlinks=true]
 {hyperref}
\hypersetup{ linkcolor=blue,  citecolor=blue, urlcolor=blue}
\PassOptionsToPackage{normalem}{ulem}
\usepackage{ulem}
2

\providecommand{\tabularnewline}{\\}

\numberwithin{table}{section}

\makeatother

\providecommand{\tabularnewline}{\\}
\usepackage{tikz}
\usetikzlibrary{calc}
\newlength{\lyxlabelwidth}      
	{\settowidth{\lyxlabelwidth}{#2}
		\begin{description}[font=\normalfont,style=sameline,
			leftmargin=\lyxlabelwidth,#1]}
	{\end{description}}

\providecommand{\tabularnewline}{\\}

\numberwithin{equation}{section}
\numberwithin{figure}{section}
\theoremstyle{plain}
\newtheorem{thm}{\protect\theoremname}[section]
\theoremstyle{definition}
\newtheorem{defn}[thm]{\protect\definitionname}
\theoremstyle{plain}
\newtheorem{cor}[thm]{\protect\corollaryname}
\theoremstyle{plain}
\newtheorem{prop}[thm]{\protect\propositionname}
\theoremstyle{plain}
\newtheorem{lem}[thm]{\protect\lemmaname}
\theoremstyle{definition}
\newtheorem{xca}[thm]{\protect\exercisename}
\theoremstyle{remark}
\newtheorem{rem}[thm]{\protect\remarkname}
\theoremstyle{definition}
\newtheorem*{defn*}{\protect\definitionname}
\theoremstyle{plain}
\newtheorem*{thm*}{\protect\theoremname}
\theoremstyle{plain}
\newtheorem*{prop*}{\protect\propositionname}
\theoremstyle{plain}
\newtheorem*{question*}{\protect\questionname}

\usepackage{setspace}
\usepackage{etoolbox}

\usepackage{bbm}
\usepackage{caption}
\usepackage[capbesideposition={left,center}]{floatrow}
\usepackage{verbatim,enumitem}
\usepackage{babel}
\usepackage{xcolor}
\usepackage{array}
\usepackage{amstext}
\usepackage{amssymb}
\usepackage{graphicx}

\usepackage[abbrev,alphabetic]{amsrefs}

\setlist[enumerate]{leftmargin=*,widest=0}

\makeatletter

\newcommand{\lyxmathsym}[1]{\ifmmode\begingroup\def\b@ld{bold}
  \text{\ifx\math@version\b@ld\bfseries\fi#1}\endgroup\else#1\fi}

\providecommand{\tabularnewline}{\\}

\numberwithin{equation}{section}
\numberwithin{figure}{section}
\theoremstyle{plain}
\theoremstyle{definition}
\theoremstyle{remark}
\theoremstyle{plain}
\theoremstyle{definition}
\newtheorem{example}[thm]{\protect\examplename}
\theoremstyle{plain}
\theoremstyle{plain}
\theoremstyle{plain}
\newtheorem{conjecture}[thm]{\protect\conjecturename}
\theoremstyle{remark}
\newtheorem*{acknowledgement*}{\protect\acknowledgementname}

\newcommand\blfootnote[1]{%
  \begingroup
  \renewcommand\thefootnote{}\footnote{#1}%
  \addtocounter{footnote}{-1}%
  \endgroup
}

\newcommand{\bmx}{\left( \begin{matrix}}
\newcommand{\emx}{\end{matrix} \right)}
\newcommand{\bsmx}{\left( \begin{smallmatrix}}
\newcommand{\esmx}{\end{smallmatrix} \right)}

\providecommand{\keywords}[1]{\textbf{\textit{Keywords: }} #1}

\providecommand{\msc}[1]{\textbf{\textit{2020 Mathematics Subject Classification: }} #1}

\DeclareMathOperator{\Aut}{Aut}
\DeclareMathOperator{\rank}{rank}

\DeclareMathOperator{\Span}{Span}
\DeclareMathOperator{\dist}{dist}
\DeclareMathOperator{\diag}{diag}

\DeclareMathOperator{\stab}{Stab}
\DeclareMathOperator{\Hom}{Hom}

\DeclareMathOperator{\Spec}{Spec}
\DeclareMathOperator{\ord}{ord}

\DeclareMathOperator{\vol}{vol}

\DeclareMathOperator{\tr}{trace}
\DeclareMathOperator{\col}{col}

\definecolor{red}{rgb}{1,0,0}

\definecolor{orange}{rgb}{1,0.5,0}

\definecolor{purple}{rgb}{.5,.2,.9}

\definecolor{blue}{rgb}{.1,.1,.8}

\definecolor{green}{rgb}{.4,.6,.4}

\definecolor{brown}{rgb}{.6,.3,.3}

\newcommand{\p}{\mathfrak{p}}

\newcommand{\e}{\varepsilon}

\newcommand{\Q}{\mathbb{Q}}
\newcommand{\R}{\mathbb{R}}
\newcommand{\C}{\mathbb{C}}
\newcommand{\Z}{\mathbb{Z}}
\newcommand{\A}{\mathbb{A}}
\newcommand{\N}{\mathbb{N}}

\newcommand{\F}{\mathbb{F}}
\newcommand{\E}{\mathbb{E}}

\newcommand{\B}{\mathcal{B}}
\newcommand{\T}{\mathcal{T}}
\newcommand{\mO}{\mathcal{O}}
\newcommand{\mP}{\mathcal{P}}

\newcommand{\mL}{\mathcal{L}}

\newcommand{\mA}{\mathcal{A}}
\newcommand{\mS}{\mathcal{S}}

\newcommand{\mK}{\mathcal{K}}
\newcommand{\mX}{\mathcal{X}}
\newcommand{\mE}{\mathcal{E}}

\newcommand{\GG}{\mathscr{G}}
\newcommand{\EE}{\mathscr{E}}
\newcommand{\MM}{\mathscr{M}}
\newcommand{\CC}{\mathscr{C}}
\newcommand{\FF}{\mathscr{F}}

\newcommand{\tG}{\tilde{G}}
\newcommand{\tB}{\tilde{B}}
\newcommand{\tK}{\tilde{K}}

\newcommand{\wtG}{\widetilde{G}}
\newcommand{\wtB}{\widetilde{B}}
\newcommand{\wtK}{\widetilde{K}}
\newcommand{\one}{\mathbbm1}

\DeclareMathOperator{\b1}{\mathbf1}

\newcommand{\Mod}[1]{\,\left(\textup{mod}\;#1\right)}

\theoremstyle{remark}
\newtheorem*{rems*}{Remarks}

\allowdisplaybreaks[2]

\makeatother

\providecommand{\corollaryname}{Corollary}
\providecommand{\definitionname}{Definition}
\providecommand{\examplename}{Example}
\providecommand{\lemmaname}{Lemma}
\providecommand{\exercisename}{Exercise}
\providecommand{\propositionname}{Proposition}
\providecommand{\remarkname}{Remark}
\providecommand{\theoremname}{Theorem}
\providecommand{\conjecturename}{Conjecture}
\providecommand{\questionname}{Question}

\makeatother

\providecommand{\corollaryname}{Corollary}
\providecommand{\definitionname}{Definition}
\providecommand{\lemmaname}{Lemma}
\providecommand{\propositionname}{Proposition}
\providecommand{\remarkname}{Remark}
\providecommand{\theoremname}{Theorem}
\providecommand{\acknowledgementname}{Acknowledgement}

\makeatletter
\def\l@subsection{\@tocline0{2pt}{2pc}{}{}}
\makeatother

\begin{document}

\title{Ramanujan bigraphs}
\author{Shai Evra, Brooke Feigon, Kathrin Maurischat, Ori Parzanchevski}

\address{Shai Evra: Einstein Institute of Mathematics, The Hebrew University of Jerusalem, Israel}
\email{shai.evra@mail.huji.ac.il}
\address{Brooke Feigon: Department of Mathematics,
The City College of New York,
CUNY,
New York, NY 10031} \email{bfeigon@ccny.cuny.edu }
\address{Kathrin Maurischat: Lehrstuhl  f\"ur Agebra und Darstellungstheorie,
RWTH Aachen University, 
Pontdriesch 12-16, D-52056 Aachen} \email{kathrin.maurischat@art.rwth-aachen.de }
\address{Ori Parzanchevski: Einstein Institute of Mathematics, The Hebrew University of Jerusalem, Israel}
\email{ori.parzan@mail.huji.ac.il}

\maketitle

\begin{abstract}
In their seminal paper, Lubotzky, Phillips and Sarnak (LPS) defined the notion of regular Ramanujan graphs and gave an explicit construction of infinite families of $(p+1)$-regular Ramanujan Cayley graphs, for infinitely many primes $p$.
In this paper we extend the work of LPS and its successors to bigraphs (biregular bipartite graphs), in several aspects: we investigate the combinatorial properties of various generalizations of the notion of Ramanujan graphs, define a notion of Cayley bigraphs, and give explicit constructions of infinite families of $(p^3+1,p+1)$-regular Ramanujan Cayley bigraphs, for infinitely many primes $p$.

Both the LPS graphs and our ones are arithmetic, arising as quotients of Bruhat-Tits trees by congruence subgroups of arithmetic lattices in a $p$-adic group, $PGL_2(\Q_p)$ for LPS and $PU_3(\Q_p)$ for us. 
In both cases the Ramanujan property relates to the  Ramanujan Conjecture (RC), on the respective groups. 
But while for $PGL_2$ the RC holds unconditionally, this is not so in the case of $PU_3$. 
We find explicit cases where the RC does and does not hold, and use this to construct arithmetic non-Ramanujan $(p^3+1,p+1)$-Cayley bigraphs as well, and prove that nevertheless they satisfy the Sarnak-Xue density hypothesis.

On the combinatorial side, we present a pseudorandomness characterization of Ramanujan bigraphs, and a more general notion of biexpanders. 
We also show that the graphs we construct exhibit the cutoff phenomenon with bounded window size for the mixing time of non-backtracking random walks, either as a consequence of the Ramanujan property, or of the Sarnak-Xue density hypothesis.
Finally, we present some other applications which follow from our work: golden and super golden gates for $PU_3$, Ramanujan and non-Ramanujan complexes of type $\widetilde{A}_2$, optimal strong approximation for $p$-arithmetic subgroups of $PSU_3$ and vanishing of the first Betti numbers of Picard modular surfaces.
\end{abstract}

\blfootnote{
\keywords{ \textbf{\textit{Keywords:}} Ramanujan graphs, bigraphs, Cayley bigraphs, non-backtracking spectrum, pseudorandomness, graph sparsification, cutoff, simply-transitive lattices, Ramanujan conjecture, automorphic representations of $U(3)$, A-packet, Iwahori-spherical, Sarnak-Xue density hypothesis, Golden Gates, Picard modular surfaces}

\msc{Primary: 11F70 Number Theory (math.NT); Secondary: 05C48 
Combinatorics (math.CO); Group Theory (math.GR)}
}

\tableofcontents

\newpage
\section{Introduction} \label{sec:intro}

A connected $(k\!+\!1)$-regular graph $X$ was defined in \cite{LPS88} to be a \emph{Ramanujan graph} if every eigenvalue $\lambda$ of its adjacency matrix satisfies either
\begin{equation}\label{eq:reg-ram}
\lambda = \pm(k+1), \quad \text{or} \quad |\lambda| \leq 2\sqrt{k}.    
\end{equation}
The main result of \cite{LPS88} is an explicit construction of infinite families of regular Ramanujan graphs, which are furthermore Cayley graphs, now called the LPS graphs. 

A natural question that arises is how to define, and how to construct non-regular Ramanujan graphs. In this paper we commence a systematic study of this problem, building a complete theory for the case of biregular bipartite graphs (\textit{bigraphs}, for short).

To see where Definition \eqref{eq:reg-ram} comes from, we recall that $2\sqrt{k}$ is the spectral radius of the infinite $(k\!+\!1)$-regular tree. The eigenvalues $\pm(k\!+\!1)$ are usually called \textit{trivial}, thus $X$ is Ramanujan when its nontrivial eigenvalues are bounded by the spectral radius of its covering tree. 

However, an equivalent definition is that $X$ is Ramanujan if its nontrivial eigenvalues all \textit{belong} to the adjacency spectrum of the covering tree, which was shown to be $[-2\sqrt{k},2\sqrt{k}]$ by Kesten \cite{Kesten1959}. Already in the case of bigraphs, the natural generalizations of these definitions do not agree: The spectrum of the $(K\!+\!1,k\!+\!1)$-biregular tree, for $k<K$, is
\begin{equation*}
\left[-\sqrt{K}-\sqrt{k},-\sqrt{K}+\sqrt{k}\right]\cup\big\{0\big\}\cup\left[\sqrt{K}-\sqrt{k},\sqrt{K}+\sqrt{k}\right],\label{eq:spec-bireg-tree}
\end{equation*}
so being bounded by the spectral radius $\sqrt{K}+\sqrt{k}$ is not the same as being in the spectrum. We call these two definitions \textit{Weakly-Ramanujan} and \textit{adj-Ramanujan} respectively.\footnote{We remark that \cite{marcus2013interlacing} proves the existence of weakly-Ramanujan bigraphs of every degrees, but their method cannot guarantee adj-Ramanujanness.}

In fact, we introduce an even stronger definition, for which we reserve the term \textit{Ramanujan graph}. For regular graphs it is still equivalent to the standard definition, but in general it is stronger, and has the benefit of controlling the spectrum of other operators on the graph, notably the \textit{non-backtracking} operator. We argue that this is the "right" definition, as we show that it is spectrally optimal, and that it yields stronger combinatorial and probabilistic expansion properties, such as Diaconis' total-variation cutoff phenomenon.

In this paper we construct bigraphs which are Ramanujan in the strongest sense, and which have an explicit Cayley-like description à la LPS. 

The LPS graphs arise as quotients of the regular tree by congruence arithmetic lattices in $PGL_2(\Q_p)$ (see also \cite{margulis1988explicit}), and it is clear that using other $p$-adic Lie groups of rank one, one can obtain biregular graphs. The story becomes complicated since the Ramanujan conjecture, which underlies the graph Ramanujan property, and is true for $PGL_2(\Q_p)$ by the work of Deligne, is simply false in general. In this work we find specific congruence arithmetic lattices in the group $PU_3(\Q_p)$ for which the Ramanujan conjecture does hold, and use them to construct infinite families of Ramanujan bigraphs. On the other hand, we present other arithmetic lattices which violate the Ramanujan conjecture, and devise from these lattices infinite families of non-Ramanujan graphs, which are nonetheless adj-Ramanujan.

All this illustrates the delicate connection between the automorphic representations of $PU_3$ corresponding to congruence lattices, and the combinatorial properties of their quotient graphs. Under this relation, the graphs are Ramanujan when the \textit{Naive Ramanujan Conjecture} holds for these representations of $PU_3$. On the other hand, for the adj-Ramanujan property it suffices that the \textit{Generalized Ramanujan Conjecture} holds.
Note that in the LPS case the NRC holds for all congruence lattices in $PGL_2$, while in our case of $PU_3$, the NRC is not true in general, and we have to work harder to find specific lattices for which the NRC holds.

Another challenge is to give an explicit description of the graphs one obtains as quotients of the tree. The presentation of the LPS graphs as Cayley graphs arises from a special feature of the lattices they use, which is a simply-transitive action on the associated tree. This machinery does not serve us as Cayley graphs are always regular. We introduce a new notion of Cayley bigraphs, which is strongly explicit, and find special lattices that allow us to present our Ramanujan and non-Ramanujan graphs as Cayley bigraphs -- see Example \ref{exa:X2q} for such a family.

\subsection{Non-regular Ramanujan graphs} \label{subsec:sub:ram-def}

Ramanujan graphs were motivated by the search for \textit{expanders}, graphs with nontrivial eigenvalues of small magnitude, which have many applications in mathematics and computer science (see \cite{HLW06,Lub12}). The Alon-Boppana theorem \cite{LPS88,Nil91} states that (regular) Ramanujan graphs are spectrally optimal: it says that for any $\varepsilon>0$, there is no infinite family of $(k\!+\!1)$-regular graphs with all nontrivial eigenvalues bounded by $2\sqrt{k}-\varepsilon$. 

Turning to the general case, let $X$ be a finite (undirected) graph, $\T_{X}$ its universal covering tree, $\rho(A_{\T_X})$ the spectral radius of $A_{\T_X}$, 
and $\mathfrak{pf}_{X}$ the Perron-Frobenius eigenvalue of $A_{X}$. We call $\Spec_0(A_{X}) = \Spec(A_{X})\backslash\left\{ \pm\mathfrak{pf}_{X}\right\}$ the \emph{nontrivial spectrum} of $X$. 

A generalized Alon-Boppana Theorem due to \cite{greenberg1995spectrum} asserts that no infinite family of graphs with a common covering tree $\T$ can have all of its nontrivial eigenvalues bounded by $\rho(A_\T)-\varepsilon$. This leads to the first definition of Ramanujan graphs:

\begin{defn*}[1] The graph $X$ is \textbf{\emph{Weakly Ramanujan}} if every nontrivial eigenvalue of $A_X$ is bounded (in absolute value) by $\rho\left(A_{\T_{X}}\right)$. 
\end{defn*}
This definition goes back to \cite{greenberg1995spectrum}, and is the one used in \cite{marcus2013interlacing}, where the authors prove the existence of infinite families of weakly Ramanujan bigraphs of any degrees. It turns out however, that one can do better, insisting that the nontrivial eigenvalues actually belong to the spectrum of the covering tree. 

\begin{defn*}[2] The graph $X$ is \textbf{\emph{adj-Ramanujan}} if every nontrivial eigenvalue of $A_{X}$
belongs to $\Spec(A_{\T_{X}})$.
\end{defn*}

This definition is motivated by two fundamental perspectives:

{\uline{Extremal~behavior}}: If $X_i$ is a sequence of finite graphs covered by a common tree $\T$ and $\sup\left\{ \mathrm{girth}(X_i)\right\} =\infty$, then every $\lambda\in\mathrm{Spec}\left(A_{\T}\right)$ is a limit point of $\bigcup_i\mathrm{Spec}(A_{X_i})$. This is an extension of the Alon-Boppana theorem, realized independently by various researchers (going back at least to \cite{McKay1981expectedeigenvaluedistribution}, \cite[Thm.\ 2]{grigorchuk1999asymptotic}). It follows from convergence of moments of the spectral measure of $A_{X_i}$, and in fact holds more generally, for a sequence of quotients of any infinite graph whose injectivity radii are unbounded.

{\uline{Random\ behavior}\ \cite{Bordenave2019Eigenvaluesrandomlifts}}: For any $\varepsilon>0$, the non-trivial spectrum of a random cover of $X$ is contained in an $\varepsilon$-neighborhood of $\Spec(A_{\T_{X}})$ asymptotically almost surely (as the cover size grows to $\infty$).\smallskip{}

The notion of Ramanujan graphs that we propose is stronger than the two above, and originates in the study of the generalization of Ramanujan graphs to simplicial complexes of general dimension.

\begin{defn*}[3] Let $X$ be a finite graph or simplicial complex with universal cover $\mathcal{U}$. Denoting $G=\mathrm{Aut}\left(\mathcal{U}\right)$, we say that an operator $T$ on (all, or some) cells of $\mathcal{U}$ is \emph{geometric} if it commutes with $G$, and an eigenvalue of $T$ acting on $X$ is considered \emph{trivial} if its eigenfunction, lifted
to $\mathcal{U}$, is constant on each orbit of the derived group
$G'=\left[G,G\right]$. 
We call $X$ a \textbf{\emph{Ramanujan\ (graph/complex)}} if the nontrivial spectrum of \textit{every} geometric operator $T$ acting on $X$ is contained in $\Spec\left(T_{\mathcal{U}}\right)$.
\end{defn*}

Definition (3), for which we reserve the term Ramanujan, is somewhat harder to digest than the previous ones. It may also be harder to verify as it asks for “Ramanujanness” of every geometric operator, but for this reason it is also the most useful, as the combinatorial results in Section \ref{sec:Combinatorics} demonstrate. Nevertheless, for regular graphs it can be shown to be equivalent to the the standard definition \eqref{eq:reg-ram} which takes into account only the adjacency operator, and we will show that for the graphs studied in this paper it is equivalent to a statement regarding the adjacency spectrum, which takes into account eigenvalue multiplicity, and not only magnitude (see Theorem \ref{thm:ram-local-crit}(2) and Corollary \ref{cor:ram-by-theta}(1)). 

While definition (3) arose from the study of higher-dimensional simplicial complexes\footnote{See \cite{ballantine2000ramanujan,cartwright2003ramanujan,li2004ramanujan,Lubotzky2005a,sarveniazi2007explicit,first2016ramanujan,kamber2016lp,kang2016riemann,Lubetzky2017RandomWalks}; the definition we propose here agrees with \cite{Evra2018RamanujancomplexesGolden, Chapman2019CutoffRamanujancomplexes} for quotients of Bruhat-Tits buildings. From the number-theoretic perspective, it is equivalent to all Iwahori-spherical local factors of the associated infinite-dimensional automorphic representations being tempered.}, already for graphs it enables the study of the ``non-backtracking'' (NB) operator $B=B_{X}$, which acts on functions on the \emph{directed} edges in $X$ by
\begin{equation}
(Bf)(v\!\rightarrow\!u)=\sum_{v\ne w\sim u}f(u\!\rightarrow\!w).\label{eq:NB-op}
\end{equation}
This operator indeed commutes with $\mathrm{Aut}\left(\T_{X}\right)$,
and it is highly useful in graph theory. 
Denoting $\Spec_0(B_{X})=\Spec B_{X}\backslash\{\pm\mathfrak{pf}_{B_X}\}$ (where $\mathfrak{pf}_{B_X}$ is the Perron-Frobenius eigenvalue of $B_{X}$), we are naturally led to consider two more ``Ramanujanness'' definitions for a graph $X$:

\textbf{(4) \emph{NB-Ramanujan}}: every $\mu\in\Spec_0(B_{X})$ satisfies $\mu\in\Spec B_{\T_{X}}$.\smallskip{}

\textbf{(5) \emph{Riemann Hypothesis}}: every $\mu\in\Spec_0(B_{X})$ satisfies $\left|\mu\right| \leq \rho\left(B_{\T_{X}}\right)$. This could also be called ``weakly NB-Ramanujan''. The chosen name reflects the fact that $\rho\left(B_{\T_{X}}\right)=\sqrt{\mathfrak{pf}_{B_{X}}}$ \cite{Angel2015nonbacktrackingspectrum}, which implies that $\left|\mu\right|\leq\rho\left(B_{\T_{X}}\right)$ is precisely the Riemann Hypothesis for the Ihara zeta function of the graph $X$ -- see Section \ref{subsec:Geodesic-prime-number} for more about this perspective.

Observing our five Ramanujan definitions, it is clear that $(5)\Leftarrow(4)\Leftarrow(3)\Rightarrow(2)\Rightarrow(1)$ tautologically for any graph. It turns out that for regular graphs the five definitions are actually equivalent, so there is little meaning in asking which is the right one. In fact, there are papers which consider (regular) Ramanujan graphs, which actually exploit Ramanujanness of $B$, such as \cite{lubetzky2016cutoff}, and this is not guaranteed by Ramanujanness of $A$ in general.

In this paper we study the case of biregular bipartite graphs, which we call \emph{bigraphs} for short, and where the various definitions no longer agree. We suggest a notion of \textit{biexpanders} (Definition \ref{def:biexp}), which is a weakened version of the (full) Ramanujan property for bigraphs, and which seems to capture bigraph pseudorandomness better than the classical definition of expanders, which relates to the adjacency spectrum alone.

The trivial adjacency eigenvalues of a finite $(K\!+\!1,k\!+\!1)$-regular bigraph $X$ are $\pm\mathfrak{pf}_{X}=\pm\sqrt{(K\!+\!1)(k\!+\!1)}$, and its covering tree is the biregular tree $\T_{K\!+\!1,k\!+\!1}$. 
Throughout the paper we shall assume $K>k$, as the regular case $K=k$ is well understood. 
The spectrum of $\T_{K\!+\!1,k\!+\!1}$ was shown in \cite{Godsil1988Walkgeneratingfunctions} to be
\begin{equation}
\Spec(A_{\T_{K\!+\!1,k\!+\!1}})=\left[-\sqrt{K}-\sqrt{k},-\sqrt{K}+\sqrt{k}\right]\cup\big\{0\big\}\cup\left[\sqrt{K}-\sqrt{k},\sqrt{K}+\sqrt{k}\right].\label{eq:spec-bireg-tree}
\end{equation}
This means that that $X$ is weakly/adj-Ramanujan if and only if every nontrivial eigenvalue $\lambda\in\Spec_0(A_{X})$ satisfies
\begin{align}
\left|\lambda\right| & \leq\sqrt{K}+\sqrt{k}\qquad & \text{(weakly Ramanujan bigraph),}\nonumber \\
\lambda=0\quad\text{or}\quad & \left|\lambda^2-K-k\right|\leq2\sqrt{Kk} & \text{(adj-Ramanujan bigraph).}\label{eq:aram-def}
\end{align}

One can show by examples that $(1)\nRightarrow(2)\nRightarrow(3)$ for bigraphs, but it turns out that $(3),(4),(5)$ are equivalent (see Corollary \ref{cor:ram-by-theta}, Theorem \ref{thm:ram-local-crit} and Remark \ref{rem:Ram-NBRam}). 
The trivial eigenvalues of $B_X$ are $\pm\mathfrak{pf}_{B_{X}} = \pm\sqrt{Kk}$, and $X$ is NB-Ramanujan (equivalently, Ramanujan) if and only if every $\mu\in\Spec_0(B_{X})$ satisfies
\begin{align}
\mu\in\Spec(B_{\T_{K\!+\!1,k\!+\!1}})=\left\{ z\in\C\,\middle|\,\left|z\right|=\sqrt[4]{Kk}\right\} \cup\left\{ \pm i\sqrt{k},\pm1\right\} &  & \text{(Ramanujan bigraph)}.\label{eq:ram-bg-def}
\end{align}

\subsection{Ramanujan bigraphs and applications}

The main results of this paper are a study of the extremal combinatorial properties of Ramanujan bigraphs, some of which turn out to be different from their regular counterparts, and giving an explicit construction of such graphs. 
The question of giving an efficient description for biregular graphs which is obtained as a quotient of an infinite tree by an infinite group is already interesting: our goal is to give a bigraph analogue of the famous LPS (Lubotzky-Philips-Sarnak) graphs \cite{LPS88}.
These are regular Ramanujan graphs which have a neat description as Cayley graphs of finite groups, but by definition, Cayley graphs are always regular. 
In order to describe explicitly our Ramanujan bigraphs, we introduce in Section \ref{sec:bicayley} a new construction which we call a \emph{Cayley bigraph}.
Let us give an example to convey the flavor of our construction:

\begin{example}
\label{exa:X2q}
Let $q$ be a prime with $q\equiv1\Mod3$. 
Fix $\rho\in\F_q$ with $\rho^2=-3$ (one always exists). 
Consider the following set of $9$ elements of order $3$ in the group $PSL_3(\F_q)$:
\[
S_q=\left\{ \tfrac{1}{2}\!\left(\begin{smallmatrix}\vphantom{-}2\\
 & -1 & -a\\
 & 3/a & -1
\end{smallmatrix}\right),\tfrac{1}{2}\!\left(\begin{smallmatrix}-1 &  & -a\\
 & \vphantom{-}2\\
3/a &  & -1
\end{smallmatrix}\right),\tfrac{1}{2}\!\left(\begin{smallmatrix}-1 & -a\\
3/a & -1\\
 &  & \vphantom{-}2
\end{smallmatrix}\right)\,\middle|\,{a\in\F_q\atop a^{3}=\rho^{3}}\right\} \subseteq PSL_{3}\left(\F_q\right).
\]
The bigraph $X_{\EE}^{2,q}$ is defined as follows: its left side is $L=PSL_3(\F_q)$, its right side is
\[
R=\bigcup_{s\in S_q}\nicefrac{PSL_3(\F_q)}{\left\langle s\right\rangle }=\left\{ \left\{ \ell,\ell s,\ell s^2\right\} \,\middle|\,\ell\in L,s\in S_q\right\} 
\]
(each right vertex is a coset of $\left\langle s\right\rangle=\{I,s,s^2\}$), and the edges are defined by membership, i.e.\ $\ell\sim r$ if and only if $\ell\in r$ (for $\ell\in L$ and $r\in R$).

The graphs $X_{\EE}^{2,q}$ are studied in Section \ref{subsec:eisenstein-lattice-explicit} (the set $S_q$ is the reduction of $\{A^{+}_{a,b}\,|\,0\le a,b \le 2\}$ from \eqref{eq:Eisnetin-stars} modulo $q$). 
They are $(9,3)$-biregular, and form a special case of $(p^3\!+\!1,p\!+\!1$)-regular Cayley bigraphs we denote $X_{\EE}^{p,q}$, described in Theorem \ref{thm:main-Eis} below. The non-backtracking spectrum of  some $X_{\EE}^{p,q}$ graphs is shown in Figure \ref{fig:Eis-graphs}. They are all adj-Ramanujan, but \textit{not} Ramanujan: the eigenvalues $\pm ip^{3/2}$ appear in their non-backtracking spectrum, but not in that of their covering tree.

One can obtain proper Ramanujan bigraphs from $X_{\EE}^{p,q}$ as well: When $q\equiv1\Mod3$, the group $PSL_3\left(\F_q\right)$ naturally acts on $X_{\EE}^{p,q}$, and the quotient of $X_{\EE}^{p,q}$ by the subgroup $\left(\begin{smallmatrix}* & * & *\\
0 & * & *\\
0 & * & *
\end{smallmatrix}\right)\leq PSL_3(\F_q)$ is Ramanujan.
This quotient can also be described as a \textit{Schreier bigraph} (see Definition \ref{def:biSch}), associated with the action of $PSL_3(\F_q)$ on the projective plane $\mathbb{P}^2(\F_q)$. 
\end{example}

The graphs $X_{\EE}^{p,q}$ arise from an arithmetic lattice $\Lambda^p_{\EE}$ described in (\ref{eq:Eis-lat}), which we named ``Eisenstein''. 
In Section \ref{sec:lattices} we construct three more lattices, giving rise to more families $X_{\GG}^{p,q},X_{\MM}^{p,q},X_{\CC}^{p,q}$. We give now a simplified version of our main Theorem (\ref{thm:main-Eis}) for the Eisenstein graphs $X_\EE^{p,q}$. 
The results for the other lattices from Section \ref{sec:lattices} are similar, with the exception that the graphs $X_{\MM}^{p,q}$, $X_{\CC}^{p,q}$ are fully Ramanujan, and so is $X_{\GG}^{p,q}$, assuming a conjecture regarding levels in $A$-packets (Conjecture \ref{conj:level-packet}). Figure \ref{fig:Gau-graphs} shows the non-backtracking spectrum of some $X_{\GG}^{p,q}, X_{\MM}^{p,q}$ graphs, which is indeed contained in the spectrum of the tree.

\begin{thm*}[\ref{thm:main-Eis}, simplified]
Let $p,q$ be primes with $p\equiv2\Mod3$, $q\notin \{ 3,p\} $, and $\omega=\tfrac{-1+\sqrt{-3}}2$.
Let
\begin{equation}
S_p:=\Bigg\{ g\in M_3\left(\Z\left[\omega\right]\right)\,\Bigg|\,{g^{*}g=p^2I,\;g\text{ is not scalar},\atop g\equiv\left(\begin{smallmatrix}1 & * & *\\
* & 1 & *\\
* & * & 1
\end{smallmatrix}\right)\Mod3}\Bigg\},\label{eq:set-Sp}
\end{equation}
and let $S_p=\bigsqcup_iS_p^i$ be the partition induced by the equivalence relation 
\[
g\sim h\text{ if and only if } g^{*}h\in pM_3(\Z[\omega]).
\]
Denote 
\[
\mathbf{G}_q:=\begin{cases}
PSL_3\left(\F_q\right) & q\equiv1\Mod3\\
PSU_3\left(\F_q\right) & q\equiv2\Mod3,
\end{cases}\quad\text{and}\quad S_{p,q}^i:=S_p^i\Mod{q}\overset{{\scriptscriptstyle (\star)}}{\subseteq}\mathbf{G}_q
\]
where $(\star)$ implies mapping $\omega$ to a root of $x^2+x+1$ in $\F_q$ or in $\F_{q^2}$ according to $q\Mod3$.
The Cayley bigraphs
\[
X_{\EE}^{p,q}=CayB\left(\mathbf{G}_q,\left\{ S_{p,q}^i\right\} _i\right)
\]
(see Definition \ref{def:biCay}) satisfy:

\begin{enumerate}
\item $X_{\EE}^{p,q}$ is an adj-Ramanujan $(p^3\!+\!1,p\!+\!1)$-regular bigraph, with left side of size $|\mathbf{G}_q|\approx q^8/3$.
\item $X_{\EE}^{p,q}$ is non-Ramanujan.
\item The family $\left\{ X_{\EE}^{p,q}\right\} _q$ satisfies the Sarnak-Xue density hypothesis (see details in Theorem \ref{thm:main-Eis}).
\item The group $\mathbf{G}_q$ acts on the set
\[
Y_q:=\begin{cases}
\mathbb{P}^2(\F_q) & q\equiv1\Mod3\\
\left\{ v\in\mathbb{P}^2(\F_q[\omega])\,\middle|\,v^{*}\!\cdot\!v=0\right\}  & q\equiv2\Mod3
\end{cases},
\]
and the Schreier bigraphs $Y_{\EE}^{p,q}=SchB(Y^{q},\left\{ S_{p,q}^i\right\} _i)$ are (fully) Ramanujan.

\item The girth of $X_{\EE}^{p,q}$ is larger than $2\log_pq$.

\item \label{enu:X_E-Bounded-cutoff}The family $\left\{ X_{\EE}^{p,q}\right\} _q$ exhibits bounded cutoff: the non-backtracking random walk on $X_{\EE}^{p,q}$ goes from $(1-\varepsilon)$-mixing to $\varepsilon$-mixing (in total-variation) in a number of steps which does not depend on $q$.

\item (Diameter) For small enough $\varepsilon>0$, and $\ell\geq\tfrac12\log_p|E_{X_{\EE}^{p,q}}|+2\log_p\left(\tfrac1{\varepsilon}\right)+3$, for any $e\in E_{X_{\EE}^{p,q}}$ we have
\[
\left|\left\{ e'\in E_{X_{\EE}^{p,q}}\,\middle|\,{\text{there is a non-backtracking path}\atop \text{of length \ensuremath{\ell} from \ensuremath{e} to \ensuremath{e'}}}\right\} \right|\geq\left(1-\varepsilon\right)|E_{X_{\EE}^{p,q}}|.
\]
Furthermore, for any two directed edges $e_1,e_2$ in $X_{\EE}^{p,q}$ there is a non-backtracking path from $e_1$ to $e_2$ of length at most $\log_p|E_{X_{\EE}^{p,q}}|+10$.
\end{enumerate}
\end{thm*}

The cutoff result which appears in the theorem is a good example of the extremal expansion quality of Ramanujan graphs: we show in Section \ref{subsec:Cutoff} that this property holds both for Ramanujan bigraphs, and for adj-Ramanujan bigraphs which satisfy the Sarnak-Xue density hypothesis. Other combinatorial expansion results are obtained in Sections \ref{subsec:Clash-counting} and \ref{subsec:sparse}: a pseudorandomness result which we call "Clash counting", which can be viewed a bigraph substitute for the Expander Mixing Lemma for regular graphs, and a sparsification result which relates Ramanujan bigraphs to finite projective geometries. Another classical perspective on Ramanujan graphs is given by the Riemann Hypothesis for the Ihara zeta functions, which counts prime cycles in the graph: this is explored in Section \ref{subsec:Geodesic-prime-number}.

In addition to bigraphs, the lattices constructed in this paper also give new examples of Golden Gates (optimal topological generators) for the compact group $PU(3)$, and explicit constructions of Ramanujan complexes of type $\widetilde{A}_2$. 
These applications were the focus of the paper \cite{Evra2018RamanujancomplexesGolden}, and we comment on them briefly in Sections \ref{subsec:complexes} and \ref{subsec:Golden-gates}. 
A new feature we encounter here is the failure of the Ramanujan conjecture for the Eisenstein lattices, from which we obtain the first explicit examples of \emph{non-Ramanujan} (Cayley) complexes of type $\widetilde{A}_2$, namely, the Cayley complexes $X_{\EE}^{p,q}$ when $p\equiv1\Mod3$. 
Several examples are shown in Figure \ref{fig:A2-complexes}, with the ``endoscopic'' (non-Ramanujan) spectrum marked in red. 
Sections \ref{subsec:optimal-SA} and \ref{subsec:Picard} explore two other applications of our work, to optimal strong approximation and to the Betti numbers of certain Picard surfaces.

It should be noted that all of the Ramanujan bigraphs which we construct in this paper are $(p^3+1,p+1)$-regular for some prime $p$. 
A natural, and likely challenging question is: 
\begin{question*}
For which values of $K,k$ do there exist infinite families of Ramanujan $(K\!+\!1,k\!+\!1)$-bigraphs?
\end{question*}
This question is open for \textbf{every} pair which is not of the form $K=k^3$ with $k$ being a prime-power\footnote{In this paper we restrict to prime $k$. We work with imaginary quadratic extensions of $\Q$, and by replacing these with more general CM-fields one can make $p$ a prime-power, at the cost of losing the explicit Cayley bigraph description.}. It is also open for the weaker notion of adj-Ramanujan! Only for weakly-Ramanujan we know that such families exist for all $K,k$, by the pioneering work of \cite{marcus2013interlacing}.

\begin{figure}
\begin{centering}
\begin{minipage}[c][1\totalheight][t]{0.4\columnwidth}%
\begin{center}
\includegraphics[scale=0.4]{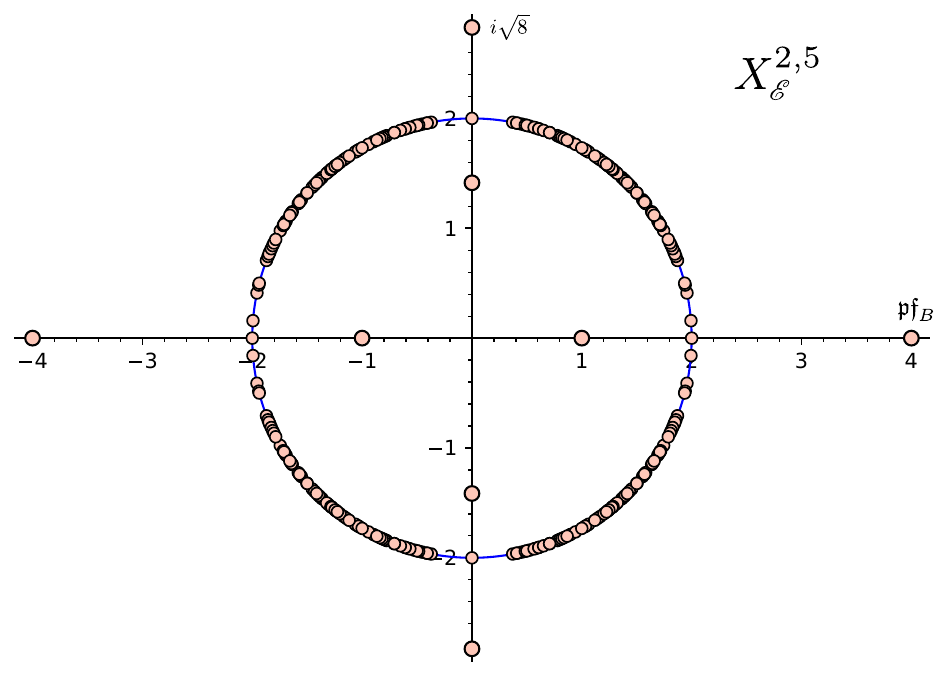}\hfill{}
\par\end{center}%
\end{minipage}%
\begin{minipage}[c][1\totalheight][t]{0.59\columnwidth}%
\begin{center}
\hfill{}\includegraphics[scale=0.5]{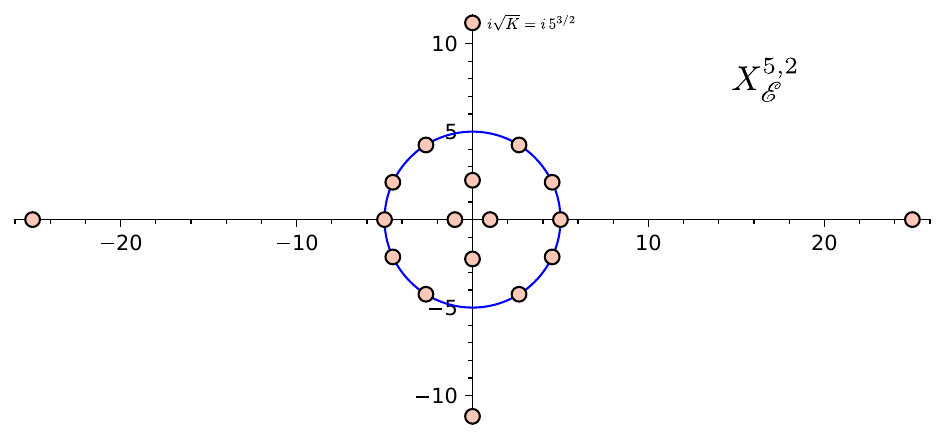}
\par\end{center}
\begin{center}
\hfill{}\includegraphics[scale=0.5]{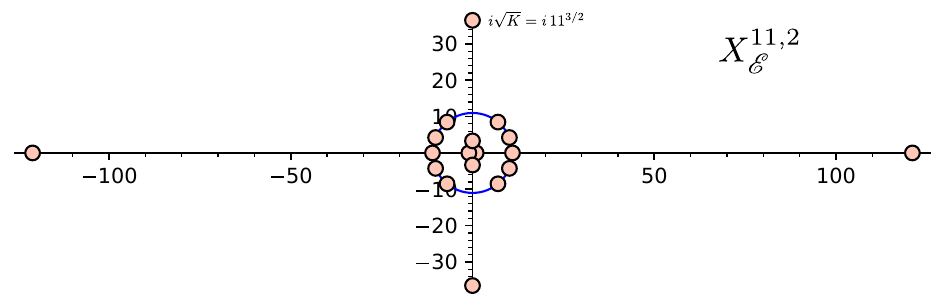}
\par\end{center}%
\end{minipage}
\par\end{centering}
\caption{\label{fig:Eis-graphs}The non-backtracking spectrum of some of the $(p^3\!+\!1,p\!+\!1)$-regular bigraphs $X_{\EE}^{p,q}$. These are Cayley bigraphs of $PSL_3(q)$ for $q\equiv1\left(3\right)$, and
of $PSU_3(q)$ for $q\equiv2\left(3\right)$. They are adj-Ramanujan,
but not Ramanujan: the eigenvalues $\pm i\sqrt{K}=\pm ip^{3/2}$ are neither trivial nor in the spectrum of the covering tree $\T_{p^3+1,p+1}$.}
\end{figure}

\begin{figure}
\begin{centering}
\includegraphics[scale=0.6]{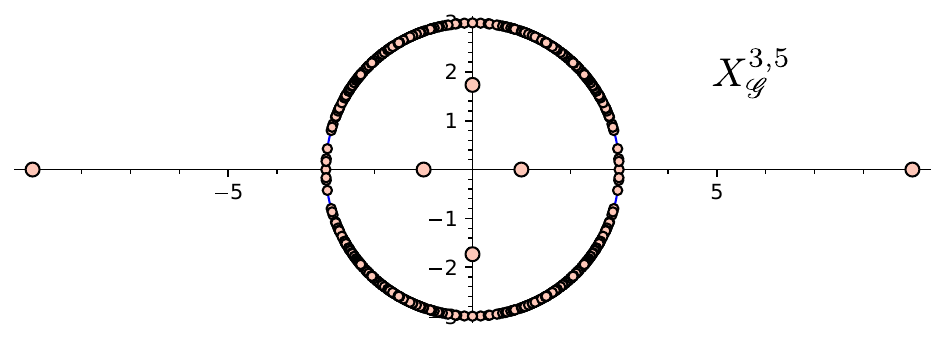}
\par\end{centering}
\begin{centering}
\includegraphics[scale=0.7]{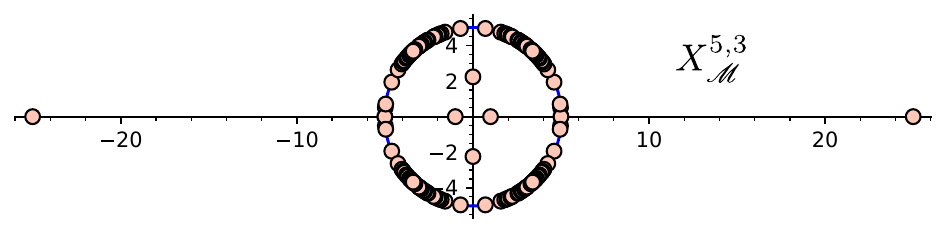}
\par\end{centering}
\begin{centering}
\includegraphics[scale=0.8]{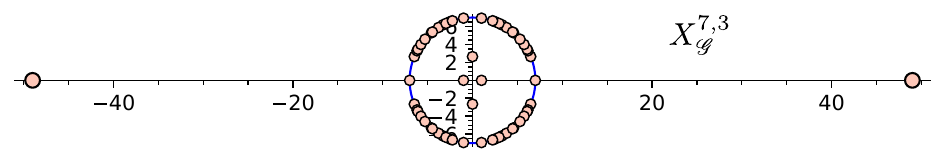}
\par\end{centering}
\begin{centering}
\includegraphics[scale=0.8]{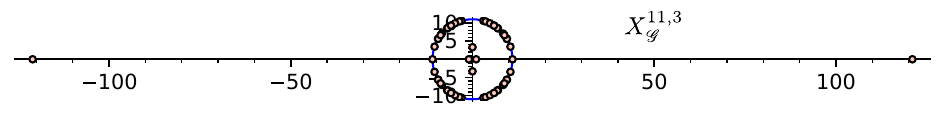}
\par\end{centering}
\caption{\label{fig:Gau-graphs} The non-backtracking spectrum of some Ramanujan bigraphs $X_{\GG}^{p,q}$ and $X_{\MM}^{p,q}$. These are $(p^3\!+\!1,p\!+\!1)$-regular Cayley bigraphs of either $PSL_3(q)$ or $PSU_3(q)$, according to $q$ (see Section \ref{subsec:Mum-CMSZ} for details).}
\end{figure}


\subsection{Number theory}

Let us briefly describe the number theory which lies behind our construction, and the main challenges which arise in trying to generalize LPS to the biregular case. The LPS graphs are finite Cayley graphs obtained as quotients of the $(p\!+\!1)$-regular Bruhat-Tits tree associated with $SL_2(\Q_p)$, by congruence subgroups of a special arithmetic lattice, which acts simply-transitively on the tree. 
Several $p$-adic algebraic groups have biregular Bruhat-Tits
trees (see \cite{Tits1979Reductivegroupsover,Carbone2001classificationrank1} for a complete list), and we focus on the case of $U_3(\Q_p)$, whose tree is $(p^3\!+\!1,p\!+\!1)$-biregular. 
By \cite{Tits1966Classificationalgebraicsemisimple}, all co-compact arithmetic lattices in $U_3$ come from either:

(I)\label{type-I-unitary} Classical matrix unitary groups (isometry groups of three-dimensional Hermitian spaces).

(II) Unitary groups of division algebras with an involution of the
second kind.

The graphs we study are the quotients of the tree by congruence subgroups of such lattices. 
In order for them to be Ramanujan, the infinite-dimensional automorphic representations associated with these subgroups should have tempered local factors. 
This holds for arithmetic lattices in $SL_2(\Q_p)$ by the work of Deligne (the Ramanujan-Petersson conjecture) and the Jacquet-Langlands correspondence, but in $U_3$ the story becomes more complicated. 
By the works of Harris-Taylor \cite[Thm.\ C]{harris2001geometry} and Rogawski \cite[Thm.\ 14.6.3]{Rogawski1990Automorphicrepresentationsunitary} (see also \cite[Prop.\ 1.4]{Clozel2002Automorphicformsand}), this is true for lattices of type (II); this is used in \cite{Ballantine2011Ramanujanbigraphsassociated,ballantine2015explicit} to present Ramanujan bigraphs as quotients of trees by congruence subgroups of lattices of type (II). 
This however does not give an explicit description in the spirit of LPS. 
The explicit Cayley graph description in LPS, and Cayley bigraph description in this paper, require that the lattice we begin with act simply-transitively on the Bruhat-Tits tree\footnote{By transitivity on a biregular tree we always mean transitivity on one side of the tree.}, which is not the case in \cite{Ballantine2011Ramanujanbigraphsassociated,ballantine2015explicit}.

In order to give an explicit description of Ramanujan bigraphs, in Section \ref{sec:lattices} we construct several lattices of type (I) which act simply-transitively on the tree of $U_3(\Q_p)$. 
But switching to type (I) has its cost -- the Ramanujan property does not necessarily hold anymore! This is the failure of the \emph{Naive Ramanujan Conjecture}, first observed in \cite{Howe1979} (see also the surveys \cite{Sarnak2005NotesgeneralizedRamanujan, winnie2020ramanujan}). Its refined version, called the \textit{Generalized Ramanujan Conjecture} only asserts that generic cuspidal automorphic representations have tempered local factors. This was proved for cohomological representations of $U(1,n)$ by Shin \cite{Shin2011Galoisrepresentationsarising}. Combining this with Rogawski's work \cite{Rogawski1990Automorphicrepresentationsunitary} we deduce the GRC for definite $U(3)$, but we are left with the considerable task of understanding the non-generic representations which may (and sometimes do) arise for our lattices.

To relate the Ramanujan conjecture to Ramanujan bigraphs, we study in Section \ref{sec:local-rep} the representation theory of $U_3(\Q_p)$, and give representation-theoretic conditions for a quotient of its Bruhat-Tits tree to be Ramanujan and adj-Ramanujan. 
In Section \ref{section:automorphic} we specialize to quotients by congruence subgroups of arithmetic lattices, and examine the possible failure of the Ramanujan property for the automorphic representations associated with them. 
The implications of our results to bigraphs are the following:

\begin{enumerate}
\item Congruence subgroups of type (I) lattices (and thus all congruence lattices) give adj-Ramanujan bigraphs.
\item Some of the lattices constructed in Section \ref{sec:lattices} give (fully) Ramanujan bigraphs.
\item The principal congruence subgroups of the Eisenstein lattice $\Lambda^p_{\EE}$ (see Section \ref{sec:lattices}) give \emph{non-Ramanujan} bigraphs.
\item The principal congruence subgroups of type (I) lattices give bigraphs which satisfy the Sarnak-Xue density hypothesis.
\end{enumerate}

For the proof of (1) we need Rogawski's work \cite{Rogawski1990Automorphicrepresentationsunitary} on $U_3$ (both the classification of the automorphic spectrum, and various instances of Langlands functoriality), as well as the resolution of the Ramanujan-Petersson conjecture for cohomological self-dual representations of $GL_n$ by Shin \cite{Shin2011Galoisrepresentationsarising}, which itself uses the Fundamental Lemma \cite{Ngo2010Lelemmefondamental}, and previous works by Harris-Taylor, Clozel, Kottwitz and ultimately Deligne (see the survey \cite{shin2020construction}). 
For (2), we establish several criteria under which one obtains Ramanujan graphs, and not only adj-Ramanujan, and we show that some of the lattices from Section \ref{sec:lattices} satisfy one of them. 
For (3), we construct explicit automorphic representations of $U_2\times U_1$ whose endoscopic lifts to $U_3$ appear in every principal congruence quotient of the lattice $\Lambda^p_{\EE}$; for this we need to compute $\varepsilon$-factors of Hecke characters, verify depth preservation of the theta correspondence, and non-vanishing of $p$-adic periods by orbital integrals over $U_3$.
Finally, for (4) we adjust to the definite case the analysis of endoscopic character relations carried out by Marshall for indefinite unitary groups in \cite{Marshall2014Endoscopycohomologygrowth}.

Section \ref{sec:applications} ties everything together to construct explicit Ramanujan bigraphs, and some other applications. 
The results of Sections \ref{sec:local-rep} and \ref{section:automorphic} are used to show that the congruence subgroups of the lattices constructed in Section \ref{sec:lattices} give rise to bigraphs which are adj-Ramanujan and satisfy the Sarnak-Xue density hypothesis, and that some of them are (fully) Ramanujan while others are not.
The results of Sections \ref{sec:bicayley} and \ref{sec:lattices} are used, on the other hand, to give an explicit Cayley or Schreier description of the quotients of the biregular tree by these subgroups. 
Finally, the results of Sections \ref{sec:Spectral-analysis} and \ref{sec:Combinatorics} are used to study the combinatorial properties of these graphs. The structure of the paper, with the internal dependencies between sections is depicted in Figure \ref{fig:flowchart}.

\begin{figure}[h]
    \centering
    \includegraphics[width=0.95\linewidth]{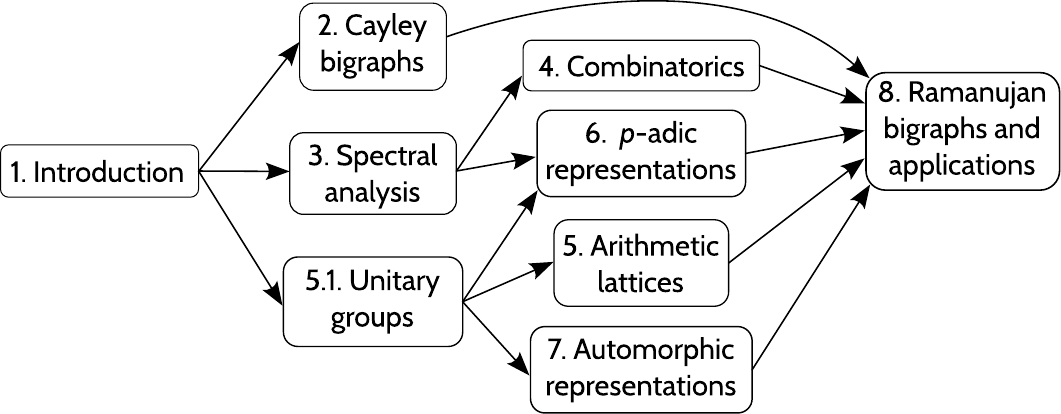}
    \caption{A schematic view of section dependency. While Section \ref{sec:applications} makes use of all previous ones, it can also be read beforehand, relying on the previous results as black boxes.} 
    \label{fig:flowchart}
\end{figure}

\subsection{Summary of results} \label{subsec:Summary-of-results}

In this section we give a quick summary of all the main results of the paper. Throughout the paper $X=\left(L\sqcup R,E\right)$ denotes a connected, undirected, $(K\!+\!1,k\!+\!1)$-biregular bipartite graph (namely, $\deg|_{L}\equiv K+1$ and $\deg|_{R}\equiv k+1$), with $K>k$. 
Here $L$ and $R$ denote the left and right vertices, respectively, and $E$ denotes the directed edges in $X$ (each edge appears with both directions). 
We denote $\left|L\right|=n$, so that $\left|R\right|=\frac{K+1}{k+1}n$ and $\left|E\right|=2n(K+1)$.
We now briefly describe the main results of each section of the paper.

\subsubsection*{\uline{\mbox{$\mathsection2$} Cayley Bigraphs}}

Let $G$ be a group, and $S^1,\ldots,S^{K+1}\subseteq G$ be disjoint subsets of size $k$, such that $S=\bigsqcup_iS^i$ is symmetric ($\left\{ s^{-1}\,\middle|\,s\in S\right\} =S$), $1\notin S$, and if $s,t\in S^i$ with $s\neq t$ then $s^{-1}$ and $s^{-1}t$ belong to the same $S^{j}$. 

\begin{defn*}[\ref{def:biCay}]
The \emph{Cayley bigraph} $CayB\left(G,\left\{ S^i\right\} \right)$ is a $(K\!+\!1,k\!+\!1)$-bigraph defined by $L=G$, $R=\nicefrac{G\times\left\{ 1, \dots, K+1\right\} }{\sim}$ where 
\[
(g,i)\sim(h,j)\quad\iff\quad\begin{array}{l}
\mbox{either }\quad(g,i)=(h,j),\\
\mbox{or }\quad g^{-1}h\in S^i\ \mbox{and}\ h^{-1}g\in S^{j},
\end{array}
\]
and $E=\left\{ \left\{ g,[g,i]\right\} \,\middle|\,g\in G,i\in\left\{ 1,\dots, K+1\right\}\right\} $, where $[g,i]$ is the equivalence class of $\left(g,i\right)\in R$.
\end{defn*}


When $G$ acts on a set $X$, one can similarly define a Schreier bigraph $SchB\left(X,\left\{ S^i\right\} \right)$ (see Section \ref{sec:bicayley} for the details). 
These constructions allow us to describe efficiently quotients of groups which act simply-transitively on one side of a biregular tree:

\begin{thm*}[\ref{thm:biCay-tree}]
Let $\Lambda$ be a group which acts simply-transitively on the left side of $\T=\T_{K\!+\!1,k\!+\!1}$, and $v_0\in L_{\T}$.
If $v_1,\ldots,v_{K+1}$ are the neighbors of $v_0$ and $S^i=\left\{ 1\neq g\in\Lambda\,\middle|\,gv_0\sim v_i\right\} $, then $\T\cong CayB\left(\Lambda,\left\{ S^i\right\} \right)$.
Furthermore, if $G=N\backslash\Lambda$ for some $N\trianglelefteq\Lambda$, then the quotient graph $X=N\backslash\T$ is isometric to $CayB\left(G,\left\{ S^i\right\} \right)$ (where $S^i$ denotes the image of $S^i$ in $G$). 
If $N\leq\Lambda$ is not normal, $X$ is isomorphic to the Schreier bigraph $SchB\left(N\backslash\Lambda,\left\{ S^i\right\} \right)$.
\end{thm*}

Example \ref{exa:X2q} above is obtained from the lattice $\Lambda^p_{\EE}$ defined in (\ref{eq:Eis-lat}), with $p=2$. 
This lattice acts simply-transitively on the left side of $\T_{9,3}$, and $X_{\EE}^{2,q}$ is the quotient of the tree $\T_{9,3}$ by $N=\Lambda^2_{\EE}\left(q\right)=\left\{ g\in\Lambda^2_{\EE}\,\middle|\,g\equiv I\Mod q\right\} $.
The Ramanujan bigraph described in Example \ref{exa:X2q} is the Schreier bigraph of $PSL_3(\F_q)$ acting on $\mathbb{P}^2(\F_q)$,
with respect to the same generators (the corresponding $N$ is $\Lambda^2_{\EE}\left[q\right]:=\left\{ g\in\Lambda^2_{\EE}\,\middle|\,g\equiv\left(\begin{smallmatrix}* & * & *\\
0 & * & *\\
0 & * & *
\end{smallmatrix}\right)\Mod q\right\} $).

\subsubsection*{\uline{\mbox{$\mathsection3$} Spectral Analysis}}

Let $A=A_{X}$ be the adjacency operator of $X$, which acts
on functions on $V=L\sqcup R$ by $Af(v)=\sum_{v\sim u}f(u)$. 
The spectrum of $A$ is symmetric around zero, and it is easy to see that $\ker A=\ker\left(A|_{R}\right)\oplus\ker\left(A|_{L}\right)$.
Since $A$ maps $L^2(R)$ to $L^2(L)$ and $\left|R\right|>\left|L\right|$, $A|_{R}$ must have a nontrivial kernel, but there is no reason for $A|_{L}$ to have one.
We call 
\[
\mE=\mE_X\overset{{\scriptscriptstyle def}}{=}\dim\ker\left(A|_{L}\right)
\]
the \emph{excessiveness }of $X$, and denote by 
\[
\mathfrak{pf}_{X}=\lambda_1>\lambda_2\geq\lambda_3\geq\ldots\geq\lambda_{n-\mE}
\]
the positive eigenvalues of $A$.
We parametrize both the adjacency and non-backtracking spectrum by the following union of three line segments in $\C$:
\[
\Theta_{K,k}=\left[-i\log\sqrt{Kk},0\right]\cup\left[0,\pi\right]\cup\left[\pi,\pi+i\log\sqrt{K/k}\right].
\]
With every $\lambda\in\left[0,\mathfrak{pf}\right]$ we associate
\[
\vartheta=\vartheta_{\lambda}=\arccos\left(\frac{\lambda^2-K-k}{2\sqrt{Kk}}\right)\in \Theta_{K,k},
\]
and observe that $\lambda\in\left[\sqrt{K}+\sqrt{k},\sqrt{K}-\sqrt{k}\right]$ if and only if $\vartheta\in\left[0,\pi\right]$, so that 
\[
X\text{ is adj-Ramanujan } \qquad \Leftrightarrow\qquad\vartheta_{\lambda_2},\ldots,\vartheta_{\lambda_{n-\mE}}\in\left[0,\pi\right]=\Theta_{K,k}\cap\R.
\]
Let $B=B_{X}$ be the non-backtracking operator acting on $L^2(E)$
by (\ref{eq:NB-op}). For $\lambda\in\left[0,\mathfrak{pf}\right]$,
we denote 
\[
\mu_{\lambda}^{\pm}=\sqrt{e^{\pm i\vartheta_{\lambda}}\sqrt{Kk}}.
\]
Given an adjacency eigenfunction $Af=\lambda f$, we construct in (\ref{eq:F-and-G-eigfun}) functions $F^{\pm},\widetilde{F}^{\pm}\in L^2(E)$ which satisfy $BF^{\pm}=\mu^{\pm}F^{\pm}$ and $B\widetilde{F}^{\pm}=-\mu^{\pm}\widetilde{F}^{\pm}$
(Proposition \ref{prop:F-B-ef}). 


As the operator $B$ is not normal, these eigenfunctions are not orthogonal, which makes spectral analysis difficult. 
The main result of this section is an orthonormal basis for $L^2(E)$, in which $B$ decomposes as a  block-diagonal matrix with blocks of size at most $4\times4$, and where each nontrivial block is itself block-anti-diagonal. 
This will be used in the combinatorial applications in Section \ref{sec:Combinatorics}.

\begin{thm*}[\ref{thm:B-decomp}]
Let $\mathfrak{pf}=\lambda_1>\lambda_2\geq\ldots\geq\lambda_{n-\mE_X}$ be the positive eigenvalues of $A_{X}$. 
Denoting $\mu_i^{\pm}=\mu_{\lambda_i}^{\pm}$, the operator $B_{X}$ is unitarily equivalent to a block-diagonal matrix composed of:
\begin{enumerate}
\item The block $\left(\begin{smallmatrix} & K\\ k \end{smallmatrix}\right)$.
\item \begin{enumerate}
    \item For each $2\leq j\leq n-\mE_X$ with $\vartheta_{\lambda_{j}}\in\left[0,\pi\right]$, the block
    \begin{equation}
    \left(\begin{smallmatrix}0 & 0 & \mu_{j}^{+}\sqrt[4]{\frac{k}{K}} & k-1\\ 0 & 0 & 0 & -\mu_{j}^{-}\sqrt[4]{\frac{k}{K}}\\ \mu^{+}\sqrt[4]{\frac{K}{k}} & K-1 & 0 & 0\\ 0 & -\mu_{j}^{-}\sqrt[4]{\frac{K}{k}} & 0 & 0 \end{smallmatrix}\right).\label{eq:B_W_Ram_case}
    \end{equation}
    \item For each $2\leq j\leq n-\mE_X$ with $\vartheta_{\lambda_{j}}\notin\left[0,\pi\right]$, a block of the form $\left(\begin{smallmatrix}0 & 0 & \mu_{j}^{+}\alpha_{j} & \gamma_{j}\\ 0 & 0 & 0 & \mu_{j}^{-}\beta_{j}\\ \mu_{j}^{+}/\alpha_{j} & \delta_{j} & 0 & 0\\ 0 & \mu_{j}^{-}/\beta_{j} & 0 & 0 \end{smallmatrix}\right)$, where $\alpha_{j},\beta_{j}\in\R_{>0}$ and $\gamma_{j},\delta_{j}\in\C$ are all bounded by $K$. 
    \end{enumerate}
\item $\mE_X$ times the block $\left(\begin{smallmatrix} & i\\ iK \end{smallmatrix}\right)$.
\item $\tfrac{K-k}{k+1}n+\mE_X$ times the block $\left(\begin{smallmatrix} & ik\\ i \end{smallmatrix}\right)$.
\item $\chi\left(X\right)$ times the diagonal block $\left(\begin{smallmatrix}1\\ & -1 \end{smallmatrix}\right)$, where $\chi\left(X\right)=\frac{|E|}2-|V|+1=\frac{Kk-1}{k+1}\cdot n + 1$.
\end{enumerate}
\end{thm*}

In particular, this shows that
\begin{equation}
\Spec B_X=\begin{cases}
\left\{ \pm\sqrt{Kk}\right\} \cup\left\{ \pm\mu_2^{\pm},\ldots,\pm\mu_{n}^{\pm}\right\} \cup\left\{ \pm i\sqrt{k},\pm1\right\}  & \mE_X=0\\
\left\{ \pm\sqrt{Kk},\pm i\sqrt{K}\right\} \cup\left\{ \pm\mu_2^{\pm},\ldots,\pm\mu_{n-\mE_X}^{\pm}\right\} \cup\left\{ \pm i\sqrt{k},\pm1\right\}  & \mE_X>0,
\end{cases}\label{eq:SpecB}
\end{equation}
which was already shown in \cite{Kempton2016NonBacktrackingRandom,brito2018spectral} by a detailed analysis of the Ihara-Bass-Hashimoto formula. 
The spectrum of $B$ on the tree is 
\[
\Spec(B_{\T_{K\!+\!1,k\!+\!1}})=\left\{ z\in\C\,\middle|\,\left|z\right|=\sqrt[4]{Kk}\right\} \cup\left\{ \pm i\sqrt{k},\pm1\right\} ,
\]
and upon verifying that $\lambda\in[\sqrt{K}+\sqrt{k},\sqrt{K}-\sqrt{k}]\ \Leftrightarrow |\mu_{\lambda}^{\pm}|=\sqrt[4]{Kk}$, we obtain the following Ramanujan criteria:
\begin{equation} \label{eq:NB-adj-ram}
{X\text{ is}\atop \text{NB-Ramanujan}} \quad \Leftrightarrow \quad {X\text{ satisfies the}\atop \text{Riemann Hypothesis}} \quad \Leftrightarrow \quad {X\text{ is adj-Ramanujan}\atop \text{and }\mE_X=0.}
\end{equation}

\subsubsection*{\uline{\mbox{$\mathsection4$} Combinatorics}}

In this section we explore some combinatorial properties of Ramanujan bigraphs. 
The importance of the strong definition of Ramanujan becomes clear here, as for most of the proofs the adj-Ramanujan property does not suffice. 

Our first result is in the spirit of the Expander Mixing Lemma: this classic theorem states that a graph with concentrated adjacency spectrum behaves pseudorandomly, in the sense that the number of edges between any two sets of vertices $S,T$ is roughly the expected number, in a random graph of the same edge density. 
For Ramanujan bigraphs with $K\gg k$, the Expander Mixing Lemma is not
very useful, since it only uses the fact that the spectrum is contained
in the rather large interval $[-\sqrt{K}-\sqrt{k},\sqrt{K}+\sqrt{k}]$ (comparing to $\mathfrak{pf}=\sqrt{(K+1)(k+1)}$), and not the concentration in the two small strips $\pm[\sqrt{K}-\sqrt{k},\sqrt{K}+\sqrt{k}]$. 
Looking for a pseudorandomness result which makes use of this, we think of the graph as mapping every vertex of $L$ to $K+1$ elements of $R$, and study the number of \emph{clashes} arising from two subsets $S,T\subseteq L$, namely, pairs of edges which leave $S$ and $T$ respectively, and have the same endpoint. 
Denoting by $\left|Cl\left(S,T\right)\right|$ the number of clashes, we present a ``Clash Counting Lemma'':

\begin{thm*}[\ref{thm:clash} for Ramanujan bigraphs]
If $X$ is Ramanujan, then for $S,T\subseteq L$
\begin{multline*}
\left|\vphantom{\frac{a}{b}}\left|Cl\left(S,T\right)\right|-\left(\tfrac{Kk+1}{n}\left|S\right|\left|T\right|+(k-1)\left|S\cap T\right|\right)\right|
\\\leq 2\sqrt{Kk}\sqrt{\left|S\right|\big(1-\tfrac{|S|}{n}\big)\left|T\right|\big(1-\tfrac{|T|}{n}\big)}.
\end{multline*}
\end{thm*}

We stress that we need here Ramanujan and not only adj-Ramanujan, even though we study the combinatorics of the adjacency mapping. 
We suggest a notion of biexpander, which is a quantitative version of being (fully) Ramanujan: $X$ is an $\varepsilon$-biexpander if its nontrivial spectrum is contained in $\pm[\sqrt{K}-\varepsilon,\sqrt{K}+\varepsilon]$, and in addition $\mE_X=0$. 
The Clash Counting Lemma (Theorem \ref{thm:clash}) applies to biexpanders in general, and it is shown in \cite{Morovits2023DirectedExpanderGraphs} that the converse holds as well: a bigraph with a pseudorandom clash counting is a biexpander (see Theorem \ref{thm:convclash} for the quantitative details).

A related topic explored in Section \ref{subsec:sparse} is that of sparsification. 
Regular expanders are sparsifiers of the complete regular graph, but it turns out that biexpanders (and in particular Ramanujan bigraphs) are not good sparsifiers of the complete biregular graph, since the excessiveness $\mE$ of the latter is in fact maximal. 
In Proposition \ref{prop:inc-sparse} we show that what biexpanders do sparsify is incidence graphs in finite projective geometries, which are shown in Proposition \ref{prop:incidence} to form a family of zero-biexpanders. In Section \ref{subsec:biexp-example} we give an example of families of strong biexpanders arising from incidence relations between vertices and edges in high-dimensional Ramanujan complexes. These are in fact $O(\sqrt{k})$-biexpanders, which makes them ``almost Ramanujan'' (note that a $(K\!+\!1,k\!+\!1)$-bigraph is Ramanujan iff it is a $\sqrt{k}$-biexpander).

In Section \ref{subsec:Cutoff} we study the total-variation mixing time of the simple random walk on vertices (SRW), and the non-backtracking random walk on directed edges (NBRW). 
Both SRW and NBRW are 2-periodic, so we restrict our attention to the behavior at even times. 
We denote by $N=n(K+1)=\frac{\left|E\right|}2$ the number of edges going from $L$ to $R$. 

\begin{thm*}[Thm.\ \ref{thm:cutoff-NBRW}, Cor.\ \ref{cor:cutoff-SRW}]
A family of $(K\!+\!1,k\!+\!1)$-regular Ramanujan bigraphs exhibits:
\begin{enumerate}
\item Cutoff for SRW at time $\frac{(K+1)(k+1)}{Kk-1}\log_{\sqrt{Kk}}n$,
with a window of size $O\left(\sqrt{\log n}\right)$.
\item Cutoff for NBRW at time $\log_{\sqrt{Kk}}N$ with window size bounded
by $3\log_{\sqrt{Kk}}\log N$.
\end{enumerate}
\end{thm*}
These results are analogous to the ones in \cite{lubetzky2016cutoff}
for regular graphs. Note that $\log_{\sqrt{Kk}}N$ is optimal for
the NBRW, since  two consecutive steps take an edge to at most
$Kk$ edges. We point out again that even for the analysis of SRW
we require Ramanujan, and not only adj-Ramanujan; in \cite{lubetzky2016cutoff}
this issue does not arise, as for regular graphs the two notions coincide.
Our next result is a bigraph analogue of the main result of \cite{Nestoridi2021Boundedcutoffwindow},
which shows that regular Ramanujan graphs with logarithmic girth exhibit
NBRW cutoff with a bounded window size. We make
use of the strategy of \cite{Nestoridi2021Boundedcutoffwindow}, which
is bootstrapping expansion from the time before the girth "kicks in", but our methods are different, and have the advantage
of analyzing NBRW on edges and not only on vertices.

\begin{thm*}[\ref{thm:bounded-cutoff}]
If $\FF$ is a family of $(K\!+\!1,k\!+\!1)$-regular Ramanujan bigraphs, and for some $m\ge 2$ every $X\in\FF$ satisfies $\mathrm{girth}\,X\geq\frac2{m-1}\log_{\sqrt{Kk}}N$, then $\FF$ exhibits bounded NBRW cutoff. 
In fact, if $t_{\varepsilon}$ denotes the total-variation $\varepsilon$-mixing time of NBRW then all large enough $X\in\FF$ satisfy
\[
\log_{\sqrt{Kk}}N-\log_{\sqrt{Kk}}\left(\tfrac1{\varepsilon}\right)<t_{1-\varepsilon}<t_{\varepsilon}<\log_{\sqrt{Kk}}N+2\log_{\sqrt{Kk}}\left(\tfrac{7m^2}{\varepsilon}\right).
\]
\end{thm*}

In Section \ref{sec:applications} we will show that our explicit Ramanujan Cayley bigraphs $X_{(*)}^{p,q}$ indeed have logarithmic girth. 
This is enough to prove bounded cutoff for the graphs $X_{\MM}^{p,q},X_{\CC}^{p,q}$ and possibly $X_{\GG}^{p,q}$, but not for the Eisenstein graphs $X_{\EE}^{p,q}$, which are only adj-Ramanujan. 
To prove that $X_{\EE}^{p,q}$ too exhibits bounded cutoff, as claimed in Theorem \ref{thm:main-Eis}(\ref{enu:X_E-Bounded-cutoff}) above, we establish cutoff results for adj-Ramanujan bigraphs which satisfy a Sarnak-Xue type density hypothesis:

\begin{thm*}[\ref{thm:cutoff-density}]
If $\FF$ is a family of left-transitive adj-Ramanujan $(K\!+\!1,k\!+\!1)$-bigraphs, which satisfy the density hypothesis $\mE_X < N^{\delta}$ with $\delta=\frac2{1+\log_k(K)}$, then $\FF$ exhibits NBRW-cutoff at time $\log_{\sqrt{Kk}}N$, with window size bounded by
\begin{enumerate}
\item $\left(\tfrac2{1-\delta}\right)\log_{\sqrt{Kk}}\log N$ in general, and
\item $\frac1{\delta}\log_{\sqrt{Kk}}\frac{K}{\varepsilon^2}$ if all $X\in\FF$ satisfy $\mathrm{girth}\,X\geq\frac2{m-1}\log_{\sqrt{Kk}}N$ and $\varepsilon\leq m^{\frac{\delta}{\delta-1}}\sqrt{K}$.
\end{enumerate}
\end{thm*}

The theorem assumes that each $X$ is left-transitive, namely that $\mathrm{Aut}\left(X\right)$ acts transitively on its left side.
This is always the case for Cayley bigraphs, and in particular for $X_{\EE}^{p,q}$. 
Since $\log_{k}K=3$ for $X_{\EE}^{p,q}$, we need a density of $\mE_{X}<N_{X}^{1/2}$ for the non-Ramanujan eigenvalues.
In Section \ref{section:automorphic} we show that principal congruence lattices in $U_3$
give rise to bigraphs which satisfy $\mE_{X}\ll_{\varepsilon}N_{X}^{3/8+\varepsilon}$.

We end the section with a short exposition of zeta functions of graphs and their relation to prime cycle counting, obtaining a version of the Prime Number Theorem for Ramanujan and adj-Ramanujan bigraphs (recall from \eqref{eq:NB-adj-ram} that an adj-Ramanujan bigraph is Ramanujan iff $\mE_X=0$):

\begin{thm*}[\ref{thm:PNT}]
If $\pi(m)$ is the number of prime cycles of length $m$ in an adj-Ramanujan $(K\!+\!1,k\!+\!1)$-bigraph $X$ with $N$ edges, then
\[
\left|\pi(2m)-\frac{(Kk)^{m}}{m}-\mE_{X}\frac{\left(-K\right)^{m}}{2m}\right|\leq2N(Kk)^{m/2}.
\]
\end{thm*}

\subsubsection*{\uline{\mbox{$\mathsection5$} Simply-transitive Lattices in Unitary Groups}}

In this section we introduce the unitary group $U_3$, which plays a central role in our explicit construction of Ramanujan bigraphs. 
Subsection \ref{subsec:unitarygroups} defines this group in the language of algebraic group schemes, and describes the Bruhat-Tits tree on which the $p$-adic group $U_3(\Q_p)$ acts. 
The rest of the section is devoted to constructing several arithmetic lattices which act simply-transitively on the hyperspecial vertices of this tree, except for a finite number of ``ramified'' primes $p$ (these are the places at which the associated arithmetic group is ramified, or where a congruence condition is imposed to give a simply transitive action). 
We list these lattices here, while the full and precise statement is given in Theorem \ref{thm-simply-transitive} using the notion of a strong unitary root datum defined in Section \ref{subsec:lattices-main}.

\begin{thm*}[\ref{thm-simply-transitive}]
The following $p$-arithmetic lattices act simply-transitively on the hyperspecial vertices of the Bruhat-Tits building of the corresponding $p$-adic group $PU_3(E,\Phi)(\Q_p)$:
\begin{description}

\item [{($\EE$) Eisenstein}] $p\neq3$,
\begin{equation}
\Lambda_{\EE}^p=\left\{ A\in PU_3(\Q[\sqrt{-3}],I)\left(\Z\left[\tfrac1{p}\right]\right)\,\middle|\,A\equiv\left(\begin{smallmatrix}1 & * & *\\
* & 1 & *\\
* & * & 1
\end{smallmatrix}\right)\Mod3\right\} .\label{eq:Eis-lat}
\end{equation}
This lattice is new, and unlike the three that follow, gives rise both to Ramanujan and non-Ramanujan bigraphs.

\item [{($\GG$) Gauss}] $p\neq2$,
\[
\Lambda_{\GG}^p=\left\{ A\in PU_3(\Q[i],I)\left(\Z\left[\tfrac1{p}\right]\right)\,\middle|\,A\equiv\left(\begin{smallmatrix}1 & * & *\\
* & 1 & *\\
* & * & 1
\end{smallmatrix}\right)\Mod{2+2i}\right\} .
\]
This lattice was constructed in \cite{Evra2018RamanujancomplexesGolden}.
\item [{($\MM$) Mumford}] $p\neq2,7$, $\Phi=\left(\begin{smallmatrix}3 & \overline{\lambda} & \overline{\lambda}\\
\lambda & 3 & \overline{\lambda}\\
\lambda & \lambda & 3
\end{smallmatrix}\right)$ where $\lambda=\tfrac{-1+\sqrt{-7}}2$
\begin{equation}\label{eq:MM_lambda}
\Lambda_{\MM}^p=\left\{ A\in PU_3(\Q[\sqrt{-7}],\Phi)\left(\Z\left[\tfrac1{p}\right]\right)\,\middle|\,A\equiv\left(\begin{smallmatrix}* & * & *\\
0 & * & *\\
0 & 0 & *
\end{smallmatrix}\right)\Mod{\lambda}\right\} .
\end{equation}
This lattice is a variation of  Mumford's lattice \cite{mumford1979algebraic}; we take a congruence condition different from Mumford's -- our condition is chosen in order to exclude endoscopic lifts, thus obtaining Ramanujan bigraphs.

\item [{($\CC$) CMSZ}] $p\neq3,5$, $\Phi=\left(\begin{smallmatrix}10 & -2(\eta+2) & \eta+2\\ -2(\bar{\eta}+2) & 10 & -2(\eta+2)\\ \bar{\eta}+2 & -2(\bar{\eta}+2) & 10 \end{smallmatrix}\right)$, where $\eta=\frac{1-\sqrt{-15}}2$. 
Here $\Lambda_{\CC}^p$ is a sublattice of $PU_3(\Q[\sqrt{-15}],\Phi)\left(\Z\left[1/p]\right]\right)$ defined by a rather complicated congruence condition -- see Theorem \ref{thm-simply-transitive}. 
This lattice is also a variation on a lattice constructed by Cartwright-Mantero-Steger-Zappa in \cite{Cartwright1993Groupsactingsimply} (see also \cite{Kato2006ArithmeticstructureCMSZ}), using a different congruence condition.
\end{description}
\end{thm*}

For primes $p$ which are inert in the quadratic imaginary field, the Bruhat-Tits building is a $\left(p^3+1,p+1\right)$-biregular tree, and each of the above lattices acts simply-transitively on the left (hyperspecial) vertices in the tree. 
For $p$ which are split, we have $PU_3\left(\Q_p\right)\cong PGL_3\left(\Q_p\right)$, and the lattice acts simply-transitively on \emph{all} the vertices of the two-dimensional building of $PGL_3\left(\Q_p\right)$.
This gives new examples of arithmetic $\widetilde{A}_2$-groups in characteristic zero, for infinitely many $p$. 
This expands on the work of Cartwright-Mantero-Steger-Zappa \cite{Cartwright1993Groupsactingsimply}, which classifies arithmetic $\widetilde{A}_2$-groups in $PGL_3\left(\Q_p\right)$ for $p=2$ and $3$ (by Margulis arithmeticity, any $\widetilde{A}_2$-group in $PGL_3\left(\Q_p\right)$ is arithmetic, but giving an explicit arithmetic presentation is non trivial). 
The Eisenstein and Gauss lattices do not appear in \cite{Cartwright1993Groupsactingsimply}, since $2$ and $3$ are ramified or inert for them.

Our methods and proofs are different from those in \cite{Cartwright1993Groupsactingsimply}.
Any group acting simply-transitively on a contractible simplicial complex has a nice presentation coming from the complex \cite{Brown1984Presentationsgroupsacting}, and \cite{Cartwright1993Groupsactingsimply} look for matrix subgroups of $PGL_3\left(\Q_p\right)$ having a presentation coming from the $\widetilde{A}_2$-building. 
In contrast, in Section \ref{subsec:lattices-transitive} we use the Mass formula of \cite{gan2001exact}, which is an explicit form of Prasad's volume formula (which is in itself an extension of Siegel's formula), to show that the groups in cases ($\EE$), ($\GG$), ($\MM$), ($\CC$) above are of class number one with respect to natural choices of adelic open compact subgroups $K$, i.e.
\[
G(\A) = G(\Q)\cdot K.
\]
This implies that the principal $S$-arithmetic lattices $\Gamma^p_{(\ast)}$ in the corresponding  unitary groups act transitively on the hyperspecial vertices of the building. 
To apply the Mass formula, we need a close examination of the parahoric subgroups, which we do in Section \ref{subsec:lattices-parahorics}.

In Section \ref{subsec:lattices-simply} we  show that the congruence conditions on the lattices $\Lambda^p_{(\ast)}$ given in Theorem \ref{thm-simply-transitive} remove the non-trivial stabilizers of hyperspecial vertices, without losing the transitivity property. 
For our Mumford ($\MM$) and CMSZ ($\CC$) lattices, we manage to find congruence conditions which additionally ensure that no non-Ramanujan eigenvalues appear in the quotient of the building/tree by congruence subgroups of the lattice (see Section \ref{section:automorphic}).

The various statements of strong approximation used in this paper are summarized in Section \ref{subsec:lattices-strong-approx}. 
There, we also give the proof of Example \ref{exa:X2q} and include a characterization of hyperspecial maximal compact subgroups.

\subsubsection*{\uline{\mbox{$\mathsection6$} Local Representation Theory}}

This section serves as a bridge between the spectral theory of bigraphs
and the representation theory of $p$-adic unitary groups. 

We begin by classifying the Iwahori-spherical (I.S.) irreducible representations $W$ of $G=U_3(E,\Phi)$ (where $E$ is an inert quadratic extension of a local field of residue order $q$), in terms of their \emph{Satake parameters}. 
The number $z\in\C^{\times}$ is called a Satake parameter for $W$ if $W$ embeds in the normalized parabolic induction of a certain character $\chi_z$ of a Borel subgroup of $G$. One can assume without loss of generality that $\Phi=\scalebox{0.7}{\ensuremath{\left(\begin{smallmatrix} &  & 1\\ & 1\\ 1 \end{smallmatrix}\right)}}$, in which case the upper-triangular matrices form a Borel subgroup, and $\chi_z$ is defined by $\chi_z\left(\left(\begin{smallmatrix}\alpha & * & *\\  & * & *\\ &  & * \end{smallmatrix}\right)\right)=z^{\ord_q\alpha}$. 
Every I.S.~representation has either one or two Satake parameters, and we obtain:

\begin{prop*}[\ref{prop:satake}]
The Satake parametrization identifies the I.S.\ dual of $G$ with
the non-Hausdorff space 
\[
\frac{\C\backslash\left\{ -q^{\pm1},0,q^{\pm2}\right\} }{z\sim1/z}\cup\left\{ -q^{\pm1},q^{\pm2}\right\} ,
\]
and the unitary ones are those with Satake parameter in $S^1\cup\left[-q,-\tfrac1{q}\right]\cup\left[\tfrac1{q^2},q^2\right]$.
\end{prop*}

We denote by $\mathrm{Sat}\left(W\right)$ the Satake parameters of $W$, and say that $W$ is of \textit{A-type} if $\mathrm{Sat}\left(W\right)={-q}$. If $\Lambda$ is a cocompact lattice in $G$ which acts on its Bruhat-Tits tree $\B$ without fixed points, then $X=X_{\Lambda}=\Lambda\backslash\B$ is a finite $(q^3+1,q+1)$-bigraph, whose spectrum is intimately linked to the representation $L^2\left(\Lambda\backslash G\right)$ of $G$. 
This connection is described in detail in Table \ref{tab:rep-spec},
and we obtain the following:

\begin{prop*}[\ref{prop:Bspec_reps}, \ref{prop:Aspec_reps}]
If $L^2\left(\Lambda\backslash G\right)=\bigoplus_iW_i\oplus\widehat{\bigoplus}_iU_i$ is the decomposition of $L^2\left(\Lambda\backslash G\right)$ as
a $G$-rep., where $W_i$ are the I.S.\ components and
the $U_i$ are the rest, then
\begin{align*}
\Spec\left(B_{X}\right) & =\bigcup\nolimits _i\left\{ \pm q\sqrt{z}\,\middle|\,z\in\mathrm{Sat}\left(W_i\right)\right\} ,\\
\Spec\left(A_{X}\right) & =\{0\}^{\#\left\{ i\,\middle|\,-q^{\pm1}\in\mathrm{Sat}(W_i)\right\} }\cup\biguplus\nolimits _i\left\{ \pm\sqrt{q^3+\big(z_i+\tfrac1{z_i}\big)q^2+q}\right\} 
\end{align*}
where the $\biguplus$ union is over all $i$ such that $q^{-2},-q^{\pm1}\notin\mathrm{Sat}\left(W_i\right)$, and $z_i$ is any choice of parameter in $\mathrm{Sat}(W_i)$. 
Furthermore, $\mE_X$ equals the number of $W_i$ of A-type.
\end{prop*}

Denoting by $K$ the maximal compact subgroup $U_3(\mO_E,\Phi)$ of $G$, we obtain some representation theoretic criteria for Ramanujanness, which will be useful in the final Sections:

\begin{thm*}[\ref{thm:ram-local-crit}]
Let $\Lambda\leq G$ be a cocompact lattice, and $X=X_{\Lambda}=\Lambda\backslash\B$.
\begin{enumerate}
\item $X$ is adj-Ramanujan if and only if every $K$-spherical irreducible representation $W\leq L^2\left(\Lambda\backslash G\right)$ is one-dimensional, tempered or of A-type.
\item The following are equivalent:
\begin{enumerate}
    \item $X$ is Ramanujan.
    \item $X$ is NB-Ramanujan.
    \item $X$ satisfies the Riemann Hypothesis.
    \item Every $K$-spherical irreducible representation $W\leq L^2\left(\Lambda\backslash G\right)$ is one-dimensional or tempered.
    \end{enumerate}
\end{enumerate}
\end{thm*}

In Section \ref{subsec:Ramanujan-global-criterion} we specialize to the case that $\Lambda$ is a congruence arithmetic lattice. 
In this case, the graph $X=\Lambda\backslash\B$ can also be identified with a locally symmetric adelic space $G(\Q)\backslash G(\A)/K'$ for an appropriate compact open subgroup $K'\leq G(\A)$. 
The spectral analysis can then be carried out in terms of automorphic representations, whose local factors at $p$ are the representations of $G(\Q_p)$ encountered earlier. 
Theorem \ref{thm:ram-global-crit} gives global Ramanujan criteria in these settings, which will be used in Sections \ref{section:automorphic} and \ref{sec:applications}.

\subsubsection*{\uline{\mbox{$\mathsection7$} Automorphic Representation Theory}}

In Section \ref{section:automorphic} we study the automorphic representations of a definite inner form $G$ of $U_3$ and their invariant vectors under certain compact subgroups $K'=\prod_v K_v' \leq G(\A)$. 
By Section \ref{sec:local-rep}, this will lead to results about the existence of both Ramanujan and non-Ramanujan bigraphs. 
 
An automorphic representation $\pi$ of $G$ is said to be Ramanujan if $\pi_p$ is tempered for all $p$, and it is said to be of A-type if it belongs to a global A-packet (see Definition \ref{defn:A-packet}).
Let $\mA_G(K')$ denote the set of automorphic representations of $G$ of level $K'$. 
We say that $\mA_G(K')$ is {\it Ramanujan} (resp.\ {\it A-Ramanujan}) if for any $\pi \in \mA_G(K')$, either $\pi$ is one-dimensional or $\pi$ is Ramanujan (resp.\ or $\pi$ is of A-type). 

We now outline the main results of this section and then summarize the methods we use to prove them. 
First we prove the following statement.

\begin{thm*} [\ref{thm:A-Ram}]
$\mA_G(K')$ is A-Ramanujan for any $K'$.
\end{thm*}

Next we show that $\mA_G(K')$ is Ramanujan for $K'$ satisfying one of the following two conditions. This result strengthens the Ramanujan-type Theorems of \cite[Thm.~1.4 and 1.5]{Evra2018RamanujancomplexesGolden}, in two manners: the assumptions are relaxed, and the result is stronger, as it guarantees temperedness also at the ramified local factors, thus confirming a conjecture suggested in \cite[Rmk.~5.8]{Evra2018RamanujancomplexesGolden}.

\begin{thm*} [\ref{thm:Ram-gen}]
$\mA_G(K')$ is Ramanujan if at least one of the following holds:
\begin{enumerate}
\item there exists a prime $p$ which ramifies in $E$, such that $K'_p$ contains an Iwahori subgroup, or
\item there exists a compact open $K''$ such that  $K' \leq K''$, $\mA(K'')$ is Ramanujan, and for any prime $p$ for which $K'_p \neq K''_p$, $K'_p$ contains an Iwahori subgroup.
\end{enumerate}
\end{thm*}

Specializing to the Eisenstein case ($\EE$), we prove the following result.

\begin{thm*} [\ref{cor:Ram-Eis}]
Let $G=U_3(\Q[\sqrt{-3}],I)$, $K'_p$ contains an Iwahori subgroup for every prime $p \ne 3$, and
\begin{equation*} 
K'_3 = K_3(C) :=  \left\lbrace g \in G(\Z_3) \;:\; g \equiv \left(\begin{smallmatrix}1 & * & *\\
* & 1 & *\\
* & * & 1
\end{smallmatrix}\right)\mod3 \right\rbrace.
\end{equation*}
Then $\mA_G(K')$ is Ramanujan.
\end{thm*}

Our final positive Ramanujan-type result is conditional on Conjecture \ref{conj:level-packet}, which stipulates that for an inert or ramified prime $p$ and compact subgroup $K_p'$, if the non-tempered representation in the A-packet does not have $K_p'$-invariant vectors, then neither does the supercuspidal representation from the same A-packet.
For full details see Section \ref{section:automorphic}. 
In Proposition \ref{prop:conj} we prove the conjecture holds for inert primes $p>10$ and principal congruence subgroups as well as Iwahori congruence subgroups.

\begin{thm*} [\ref{thm:Ram-Eis-kq}]
Let $G=U_3(\Q[\sqrt{-3}],I)$ and $K'\leq G(\A)$, where $K_3' = K_3(C)$, $K'_q=\{g\in K_q : \det(g)^6\equiv1\pmod{q}\}$ for some prime $q\equiv 1 \pmod{12}$,  and $K'_p = G(\Z_p)$ for any $p\ne 3,q$. 
Assume Conjecture \ref{conj:level-packet} holds for $G$, $K'_3$ and $p=3$.
Then $\mA_G(K')$ is Ramanujan.  
\end{thm*}

In the other direction we show the existence of  non-Ramanujan representations in the Eisenstein case ($\EE$). 
This, joined with Theorem \ref{thm:A-Ram}, gives the first known construction of an infinite family of adj-Ramanujan bigraphs which are non-Ramanujan.


\begin{thm*} [\ref{thm:NonRam-Eis}]
Let $G=U_3(\Q[\sqrt{-3}],I)$. 
Then $\mA_G(K')$ is non-Ramanujan if either
\begin{enumerate}
\item $K'_p = G(\Z_p)$ for any $p\ne 3$ and $K'_3 = \boldsymbol{I}_3(3)$ (see $\mathsection7$ for the definition), or
\item $K'_q = \{g \in G(\Z_q) \,:\, g \equiv I \mod{q}\}$ for some prime $q \geq 5$, $K'_3 = K_3(C)$ and $K'_p = G(\Z_p)$ for any $p\ne 3,q$.
\end{enumerate}
\end{thm*}

Finally, we prove the Sarnak-Xue Density Hypothesis (SXDH) for definite unitary $3\times 3$ groups $G=U_3(E,\Phi)$. 
Let $S$ be a finite set of places of $\Q$ which contains $\{\infty, 2, 3, 5, 7\}$ and the primes in which $G$ ramifies. Fix $K'_\ell \leq G(\Z_\ell)$ a finite index subgroup for any $\ell \in S$, where $K'_\infty = G(\R)$.
For any integer $N  = \prod_i p_i^{e_i}$, such that $p_i \not\in S$, denote $K_{p_i}(p_i^{e_i}) = \{g\in K_{p_i} \,:\, g \equiv I \mod{p_i^{e_i}} \}$, and define $K'(N) = \prod_{v \in S} K'_v \prod_{p_i \ne v \not\in S} K_v \prod_i K_{p_i}(p_i^{e_i}) \leq G(\A)$.
Let $\mA_{G,\b1}$ be the set of $\pi \in \mA_G$ with trivial central character, $\mA^A_{G,\b1}$ be the set of $\pi \in \mA_{G,\b1}$ which are A-type, $V(N) := \bigoplus_{\pi \in \mA_{G,\b1}} \pi^{K'(N)}$ and $V_A(N) := \bigoplus_{\pi \in \mA^A_{G,\b1}} \pi^{K'(N)}$. 
The SXDH for $G$ reads as follows.

\begin{conjecture}[SXDH]
For any $\e >0$ there exists $C_\e > 0$, such that for any $N$ coprime to $S$,
\[
\dim V_A(N) \leq C_\e \cdot \dim V(N)^{\frac12 + \e}.
\]
\end{conjecture}

Adjusting Marshall's endoscopic arguments \cite{Marshall2014Endoscopycohomologygrowth} from the cohomological to the definite settings, we get the following stronger result: 

\begin{thm*} [\ref{thm:SXDH}]
For any $\e >0$ there exists $C_\e > 0$, such that for any $N$ coprime to $S$,
\[
\dim V_A(N)  \leq C_\e \cdot \dim V(N)^{\frac3{8} + \e}.
\]
\end{thm*}

In Section \ref{automorphic:rogawski} we begin our proof of these theorems by outlining how we can use Rogawski's work on $U_3$ to  transfer the automorphic representations we wish to study to automorphic representations on the quasi-split form of $U_3$. 
Rogawski's classification of the automorphic representations on the latter group, along with results towards the Generalized Ramanujan-Petersson Conjecture (GRPC) by \cite{Shin2011Galoisrepresentationsarising} imply that if $\pi$ is non-tempered at $p$ then it must be from an ``A-packet'' (for the full definition of Rogawski's local and global A-packets see Definition \ref{defn:A-packet}). 
To prove representations of a certain level are Ramanujan we show they cannot be  in any A-packets and to prove other representations are non-Ramanujan  we  show the existence of particular non-tempered $K_p'$-invariant representations. 
From the work of Rogawski we know these representations occur in the A-packets associated to certain Hecke characters. 

In Section  \ref{automorphic:CFT} we begin a more detailed analysis of the representations occurring in A-packets by computing their multiplicity in particular cases (a priori we know it is either zero or one).
We use a multiplicity formula of Rogawski for representations that belong to A-packets (see Theorem \ref{thm:Rogawski-A-packets}) that is dependent on the epsilon factor of the character associated to the A-packet.
We build an explicit family of Hecke characters and compute their epsilon-factors via class field theory, $p$-adic integrals and formulas from \cite{Tate1979Numbertheoreticbackground}, \cite{Rogawski1992Analyticexpressionnumber} and \cite{kudla2004tate}.
We also keep track of the conductor of these characters  for computing the level of the representation in the corresponding A-packet in the next subsection.

In Section \ref{automorphic:level}  we compute the levels of given automorphic representations that belong to A-packets via their local factors. For non-tempered representations, $\pi^n$, we describe an explicit relationship between the level of the character that induces the representation that $\pi^n$ is a component of and the  level of $\pi^n$.
We also prove ${\pi^n}$ has a non-vanishing $K'$-invariant vector if an explicit $p$-adic period integral is non-vanishing. 
We then  evaluate this integral for various $K'$ to analyze the level of $\pi^n$ in terms of the level of the associated character. 
Finally, we use depth preservation of the theta correspondence (both using known results by \cite{pan2001splittings}, \cite{pan2002depth} and \cite{Gelbart1991LfunctionsFourier}, as well as extending these results) to study the level of the supercuspidal representations that are in A-packets. 

In Section \ref{automorphic:proofs} we combine the results of Sections \ref{automorphic:rogawski}, \ref{automorphic:CFT} and  \ref{automorphic:level} to prove the main (global) theorems of this section.
We also prove two new local results -- first that at $3$  two particular supercuspidal representations do not have any non-zero   $K_3(C)$-invariant vectors (Prop \ref{cor:supercuspidal}), and second, some special cases of Conjecture \ref{conj:level-packet} (Prop \ref{prop:conj}).
The first result is proved by embedding the supercuspidal representation in a global representation and using our global results. 
The first case of the conjecture is proved by combining Rogawski's trace formula and Ferrari's generalization of the fundamental lemma \cite{Ferrari2007Theoremedelindice}, as used in \cite{Marshall2014Endoscopycohomologygrowth} and \cite{Gerbelli-Gauthier2019Growthcohomologyarithmetic}, with our analysis of the relationship between the level of a character and the level of the non-tempered representation that the character induces. 
The other case of the conjecture is proved using the theta correspondence.

In Section \ref{automorphic:SXDH} we modify Marshall's argument in \cite{Marshall2014Endoscopycohomologygrowth} to the definite case to prove the Sarnak-Xue Density Hypothesis for $U_3$ associated to a definite form. 

\subsubsection*{\uline{\mbox{$\mathsection8$} Ramanujan Bigraphs and Applications}}

This section combines the results of all previous ones to give explicit constructions of Ramanujan bigraphs, and several other applications. 
The main theorem was already stated in Section \ref{sec:intro} for the case of the Eisenstein lattice, so we skip it here. 
Section \ref{subsec:Mum-CMSZ} proves an analogue theorem for the Mumford and CMSZ lattices. 
The main difference is that here our choice of congruence conditions yields the full Ramanujan property for the Cayley bigraphs and not only the Schreier ones. 
The Gauss lattice gives Ramanujan Cayley bigraphs under Conjecture \ref{conj:level-packet}, and some examples of these are presented in Figure \ref{fig:Gau-graphs}.

In Section \ref{subsec:complexes} we explore another application
of our work -- new constructions of Ramanujan complexes as Cayley
complexes of finite groups. 
When $p$ splits in $E$, we have $PU_3(\Q_p)\cong PGL_3(\Q_p)$, whose Bruhat-Tits building $\B_p$ is two-dimensional. 
Ramanujan complexes, defined in \cite{li2004ramanujan,Lubotzky2005a}, are quotients of such buildings whose spectral theory mimic that
of the building, which is their universal cover. 
This application was the focus of \cite{Evra2018RamanujancomplexesGolden}, but here we obtain a new phenomenon: the endoscopic lifts afforded by the Eisenstein
lattice give us the first explicit examples of finite \emph{non-Ramanujan} $\widetilde{A}_2$-complexes. 

The vertices of the Bruhat-Tits building of $PGL_d(\Q_p)$ are colored by $\nicefrac{\Z}{d\Z}$, and the $i$-th \emph{Hecke operator} $A_i$ acts on them by $(A_if)(v)=\sum f(w)$ where the sum is over all neighbors $w$ of $v$ such that $\col(w)-\col(v)=i$. 
This is a geometric operator and we denote its non-trivial spectrum by $\Spec_0$. In the case of $PGL_3(\Q_p)$, the spectrum of $A_1$ already determines whether a quotient $X=\Lambda\backslash\B_p$ is Ramanujan, which occurs when
\[
\Spec_0(A_1|_X)\subseteq\Spec(A_1|_{\B_p})=\left\{ p\left(\alpha+\beta+\overline{\alpha\beta}\right)\,\middle|\,\alpha,\beta\in\C,\ |\alpha|=|\beta|=1\right\} 
\]
(shown in blue in Figure \ref{fig:A2-complexes}). 
Unlike in the bigraph case, $PGL_3(\Q_p)$ admits a continuum of non-tempered representations which appear as local factors of endoscopic lifts. 
The $A_1$-eigenvalues they give rise to form the closed curve:
\[
\mathfrak{E}_p=\left\{ \alpha p^{3/2}+\tfrac{p}{\alpha^2}+\alpha\sqrt{p}\,\middle|\,\alpha\in\C,\ |\alpha|=1\right\} .
\]
Together, the Ramanujan spectrum $\Spec(A_1|_{\B_p})$ and the endoscopic spectrum $\mathfrak{E}_p$ form the Automorphic Hamantash depicted in Figure \ref{fig:A2-complexes}, along with the $A_1$-spectrum of several Cayley complexes $X_{\EE}^{p,q}$ arising from the Eisenstein lattice. 
It is interesting to remark that in positive characteristic, the entire Hamantash is obtained as the nontrivial $A_1$-spectrum of the quotient of the building of $PGL_3(\F_p((t)))$ by the non-uniform arithmetic lattice $PGL_3(\F_p[t])$ \cite{hong2021spectrum}.

\begin{figure}[H]
\begin{centering}
\begin{minipage}[c][1\totalheight][t]{0.49\columnwidth}%
\begin{center}
\hfill{}\includegraphics[scale=0.4]{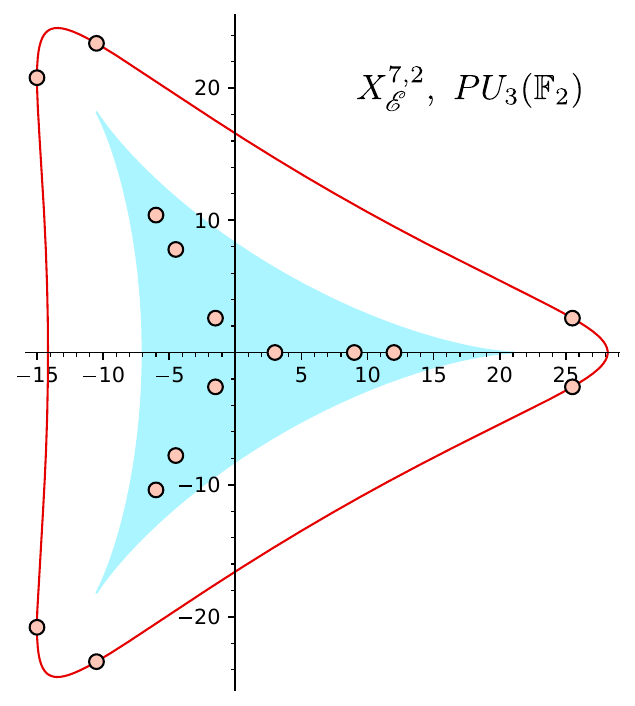}\hfill{}
\par\end{center}%
\end{minipage}%
\begin{minipage}[c][1\totalheight][t]{0.49\columnwidth}%
\begin{center}
\hfill{}\includegraphics[scale=0.4]{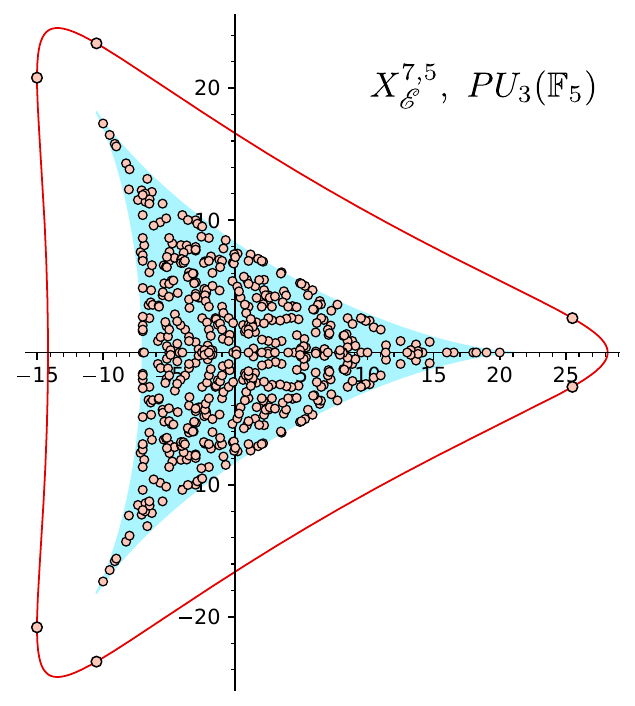}\hfill{}
\par\end{center}%
\end{minipage}
\par\end{centering}
\begin{centering}
\begin{minipage}[c][1\totalheight][t]{0.49\columnwidth}%
\begin{center}
\hfill{}\includegraphics[scale=0.4]{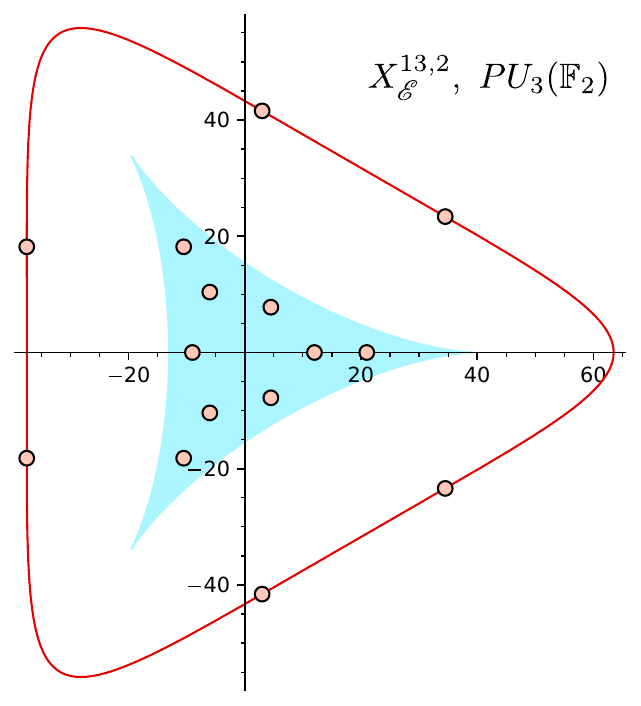}\hfill{}
\par\end{center}%
\end{minipage}%
\begin{minipage}[c][1\totalheight][t]{0.49\columnwidth}%
\begin{center}
\hfill{}\includegraphics[scale=0.4]{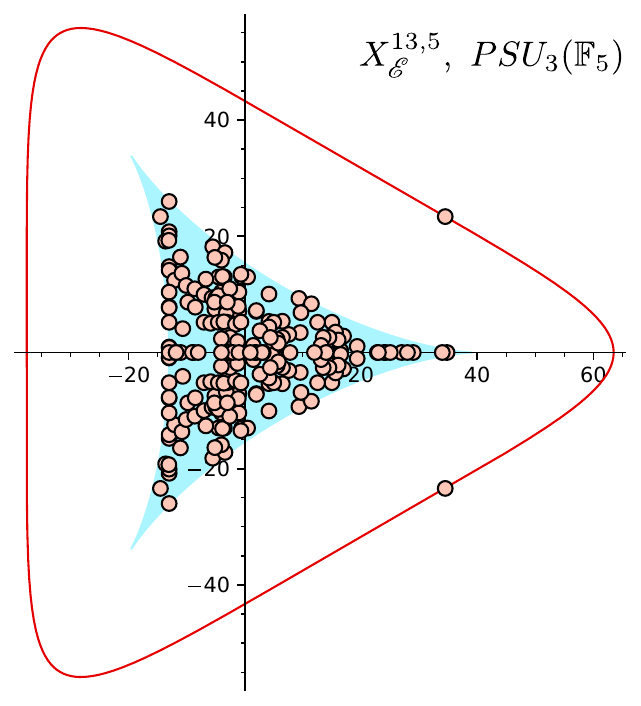}\hfill{}
\par\end{center}%
\end{minipage}
\par\end{centering}
\caption{\label{fig:A2-complexes}
The automorphic Hamantash depicts the spectrum of the Hecke operator $A_1$ on an $\widetilde{A}_2$-complex; it is composed of the endoscopic crust $\mathfrak{E}_p$ (the red curve), and the Ramanujan filling, which is the spectrum of $A_1$ on the building of $PGL_3\left(\Q_p\right)$ (the blue shade). The dots show the nontrivial $A_1$-spectrum of a few $X_{\EE}^{p,q}$ Cayley complexes, with the underlying group $\mathbf{G}_q$ specified.}
\end{figure}

Our next application (Section \ref{subsec:Golden-gates}) is a construction of new golden gate and super golden gate sets for the compact Lie group $PU(3)$.
These notions were introduced in \cite{sarnak2015letter, Parzanchevski2018SuperGoldenGates}, in relation to problems in quantum computation.
A finite set $\Sigma \subset PU(n)$ is called a golden gate set if it  exhibits an almost optimal covering rate as well as an efficient approximation and navigation algorithms (see \cite[Sec.\ 2.2]{Evra2018RamanujancomplexesGolden} for the precise definition).
If the almost optimal covering rate and the algorithms hold only up to an action of a finite subgroup $D \leq PU(n)$, then we say that $\Sigma$ is a golden gate set for $PU(n)/D$.
Finally, a golden gate set whose elements are of finite order is called a super golden gate set.

Super golden gate sets were constructed in \cite{Parzanchevski2018SuperGoldenGates} for $PU(2)$ and in \cite{Evra2018RamanujancomplexesGolden} for $PU(3)$. 
In both cases the constructions are achieved by finding generating sets of congruence $p$-arithmetic subgroups $\langle \Sigma \rangle = \Lambda\leq G(\Z[1/p])$, where $G(\R)\cong PU(n)$, such that $\Lambda$ acts simply-transitively on the edges of Bruhat-Tits trees of $G(\Q_p)$ and also satisfies the Ramanujan property (see Section \ref{subsec:Golden-gates}).
We construct some new super golden gate sets for $PU(3)$, coming from our $p$-arithmetic lattices. 
Here is one such example:
 
\begin{thm*} [\ref{thm:super-GG-Eis}]
Let $G = PU_3(\Q[\omega],I)$, $\omega = \frac{-1+\sqrt{-3}}2$, and denote $\Sigma = \{ A, \sigma, \tau\}$, where
\begin{equation} \label{eq:Eisentein-2-mats}
A := \tfrac12 \bsmx -1 & 0 & \sqrt{-3}\\ 0 & 2 & 0\\ \sqrt{-3} & 0 & -1 \esmx,\; \sigma := \bsmx 1 \\ & \omega\\ &  & \omega^2 \esmx, \;
\tau := \bsmx & 1\\ & & 1\\ 1 \esmx \in G(\Z[1/2]). 
\end{equation}
Then $\langle \Sigma \rangle$ is a congruence subgroup which acts simply transitive on the edges of the Bruhat-Tits tree of $G(\Q_2)$, and $\Sigma$ is a super golden set for $PU(3)/D$, where $D = G(\Z)\cap \mathbf{I}_2$ and $\mathbf{I}_2\leq G(\Z_2)$ is an Iwahori subgroup.
\end{thm*}

Another application of our results is new instances of the optimal strong approximation phenomenon (or optimal lifting) for $p$-arithmetic groups (see \cite{Golubev2022CutoffgraphsSarnakXue}).
This phenomenon was first proved by Sarnak in \cite{sarnak2015letter} for the arithmetic group $SL_2(\Z)$ and its modulo homomorphisms $\mod{q}\,:\, SL_2(\Z) \rightarrow  SL_2(\Z/q\Z)$, $q\in \N$.
By the classical strong approximation property these homomorphisms are surjective. The optimal strong approximation property for $SL_2(\Z)$ is the claim that for any $\varepsilon >0$, if $q$ is large enough then for all but an $\varepsilon$-fraction of the elements $g \in SL_2(\Z/q\Z)$, there is a lift $\gamma \in SL_2(\Z)$, $\gamma \mod{q} \equiv g$, such that $\|\gamma\| \leq |SL_2(\Z/q\Z)|^{\frac12 +\varepsilon}$, where $\|\cdot\|$ is the Euclidean norm $\|g\|^2 = \sum_{i,j} g_{ij}^2$.
This is indeed optimal since by \cite[(11)]{sarnak2015letter}, the number of elements $\gamma \in SL_2(\Z)$ of norm bounded by $T$ is $O(T^2)$. 

Consider the $p$-arithmetic analogue, where $SL_2(\Z)$ is replaced by a $p$-arithmetic lattice $\Lambda \leq PSU_3(\Q_p)$, where $p$ is an inert prime, endowed with the modulo homomorphisms $\mod{q}\,:\,\Lambda \rightarrow \mathbf{G}_q$, $q$ a prime, where $\mathbf{G}_q = PSL_3(\F_q)$ when $q$ split and $\mathbf{G}_q = PSU_3(\F_q)$ when $q$ inert.
By the strong approximation property this map is surjective (see Corollary \ref{cor:SA-classical}).
Consider the level function, $\ell\,:\,\Lambda \rightarrow \N$, $\ell(g) = -2\min_{i,j}\ord_p g_{ij}$. 
The optimal strong approximation property for $\Lambda$ is the claim that for any $\varepsilon >0$, if $q$ is large enough then for all but an $\varepsilon$-fraction of the elements $g \in \mathbf{G}_q$, there is a lift $\gamma \in \Lambda$, $\gamma \mod{q} \equiv g$, of level bounded by $\ell(\gamma) \leq (1+\varepsilon) \log_{p^2}|\mathbf{G}_q|$.
This is indeed optimal since the number of elements $\gamma \in \Lambda$ of level bounded by $r$ is $O(p^{2r})$.

\begin{prop*} [\ref{prop:OSA}]
Let $\Lambda = \langle S_p \rangle \leq PSU_3(\Q_p)$ be the $p$-arithmetic group, where $S_p$ is as in Theorem \ref{thm:main-Eis} or \ref{thm:main-Mum} and $p$ is an inert prime.
Then $\Lambda$ has the optimal strong approximation property.
\end{prop*}

Our last application is a new vanishing of first cohomology result for certain Picard modular surfaces associated to an indefinite unitary group of type (I), which are the real analogues of the arithmetic quotients of Bruhat-Tits trees constructed in this paper.  

Let $G$ be an indefinite projective unitary group scheme over $\Z$, i.e.\ $G(\R)\cong PU(2,1)$, whose symmetric space is the complex unit ball $\mathbb{B}\leq \C^2$. 
Then $G(\Z)$ and all of its finite index subgroups are lattices in $G(\R)$.
Let $\Lambda(q) = \{g\in G(\Z)\,:\, g\mod{q}\equiv I\}$, and call a lattice $\Gamma \leq G(\Z)$ a level $q$ congruence subgroup if $\Lambda(q) \leq \Gamma$. 
Denote its corresponding Picard modular surface (over $\C$) by $X(\Gamma) = \Gamma \backslash \mathbb{B}$ and denote its first Betti number (also known as its irregularity) by $b_1\left( X(\Gamma)\right) := \dim H^1\left( X(\Gamma),\C \right)$.

As mentioned before a unitary group scheme is either of type (I), i.e. $G=PU_3(E,\Phi)$ where $E$ is a quadratic imaginary field and $\Phi$ an indefinite Hermitian form, or of type (II), i.e. $G=PU(D,\sigma)$ where $D$ is a division algebra and $\sigma$ an involution of the second type.
Say that $\Gamma \leq G(\Z)$ and $X(\Gamma)$ are of type (I) or (II) if $G$ is.

In \cite[Thm.15.3.1]{Rogawski1990Automorphicrepresentationsunitary}, Rogawski proved that for Picard modular surfaces of type (II), their first Betti number vanishes.
This vanishing of first Betti number is the archimedean analogue of the Ramanujan property studied in this paper.
Using our analysis of the automorphic spectrum of unitary matrix groups, we are able to generalize this result of Rogawski to certain Picard modular surfaces of type (I).

\begin{thm*} [\ref{thm:Picard}] 
Let $G=PU_3(E,\Phi)$ be an indefinite projective unitary group scheme over $\Z$ of type (I), $q \in \N$ coprime to the discriminant of $E$ and $\Gamma\leq G(\Z)$  a level $q$ congruence subgroup.
Then $b_1\left( X(\Gamma) \right) = 0$.
\end{thm*}

As mentioned above, the vanishing of first cohomology is the real analogue of the Ramanujan property for the congruence subgroup of $p$-arithmetic groups studied in the paper.
We end with the following open question: what is the real analogue of the simple-transitivity property?

\begin{acknowledgement*}
The authors wish to thank Cristina Ballantine who helped initiate this project and made fundamental contributions to it. 
The authors thank Sug Woo Shin for pointing out a strengthening of our results in Section \ref{section:automorphic}.
The authors thank Alexander Lubotzky for his valuable input on an earlier draft of the paper.
This research was supported by the IAS Summer Collaborators Program, and the authors are grateful for the hospitality of the Institute for Advanced Study. 
They are also thankful to the American Institute of Mathematics, where this project was promoted during a workshop on Arithmetic Golden Gates. 
The authors were supported by NSF grant DMS-1638352 and ISF grant 1577/23 (S.E.), Simons Foundation Collaboration Grant 635835 (B.F.), and ISF grant 2990/21 (O.P.).
\end{acknowledgement*}

\section{Cayley Bigraphs} \label{sec:bicayley}

In this section we introduce a group-theoretic construction of bigraphs, which we call \emph{Cayley bigraphs}. 
This will allow us to give explicit constructions of infinite families of Ramanujan bigraphs, in a similar fashion to the regular Ramanujan Cayley graphs in \cite{LPS88}.

Let $S$ be a symmetric subset of a group $G$ (i.e.\ $S^{-1}=S$).
Recall that the Cayley graph $Cay(G,S)$ is an $|S|$-regular graph with vertices $V=G$ and edges $E=\{\{g,gs\}\,|\,g\in G,s\in S\}$.
It is connected if and only if $S$ generates $G$, and loop-free if and only if $1\notin S$.
Note that $G$ acts on this graph by left multiplication, and the action is simply transitive on the vertices.


\begin{defn} \label{def:axioms} 
Let $0\leq k\leq K$, and let $\vec{S}=\{S^1,\ldots,S^{K+1}\}$, where $S^1, \dots, S^{K+1}$ are pairwise disjoint subsets of a group $G$, each of size $k$. 
Let $S=\bigsqcup_{i=1}^{K+1}S^i$ and define $i\colon S\rightarrow\N$ by $s\in S^{i(s)}$. 
The pair $(G,\vec{S})$ satisfies the \emph{bi-Cayley axioms} if 
\begin{enumerate}
\item $S$ is symmetric ($\{s^{-1}|s\in S\}=S$), and $1\notin S$;
\item if $s,t\in S^i$ then $s=t$ or $s^{-1}t \in S^{i(s^{-1})}$.
\end{enumerate}
\end{defn}

\begin{defn} \label{def:eqrel}
If $\left(G,\smash{\vec{S}}\right)$ satisfies the bi-Cayley axioms, define the following relation $\sim$ on $G\times[K\!+\!1]$ (where $\left[n\right]=\{1,\ldots,n\}$): 
\begin{equation} \label{eq:eq-class-def}
(g_1,j_1)\sim(g_2,j_2)\quad\iff\quad\begin{array}{l}
\mbox{either }\quad(g_1,j_1)=(g_2,j_2),\\
\mbox{or }\quad g_1^{-1}g_2\in S^{j_1}\ \mbox{and}\ g_2^{-1}g_1\in S^{j_2}
\end{array}.
\end{equation}
Denote by $\left[g,j\right]$ the equivalence class of $(g,j)\in G\times[K\!+\!1]$, and observe that 
\[
[g,j]=\{(g,j)\}\cup\{(gs,i(s^{-1}))\mid s\in S^{j}\}.
\]
\end{defn}

\begin{lem}
\label{lem:eqrel}The relation $\sim$ is an equivalence relation. 
\end{lem}

\begin{proof}
It is clear that $\sim$ is reflexive and symmetric. 
Suppose $(g_1,j_1)\sim(g_2,j_2)\sim(g_3,j_3)$; we can assume that the three are distinct, for otherwise $(g_1,j_1)\sim(g_3,j_3)$ follows.
It then follows from \eqref{eq:eq-class-def} that $s:=g_2^{-1}g_1\in S^{j_2}$ and $t:=g_2^{-1}g_3\in S^{j_2}$, so by the bi-Cayley axioms $g_1^{-1}g_3=s^{-1}t\in S^{i(g_1^{-1}g_2)}=S^{j_1}$. 
The same argument shows that $g_3^{-1}g_1\in S^{j_3}$, hence $(g_1,j_1)\sim(g_3,j_3)$.
\end{proof}

\begin{defn} \label{def:biCay} 
For $(G,\smash{\vec{S}})$ which satisfies the bi-Cayley axioms, we define the \textit{Cayley bigraph} $CayB\left(G,\smash{\vec{S}}\right)=\left(L\sqcup R,E\right)$ to be the $(K\!+\!1,k\!+\!1)$-bigraph with vertices and edges 
\[
L=G,\quad R=\nicefrac{G\times[K\!+\!1]\,}{\,\sim},\quad E=\left\{ \,\left\{ g,[g,i]\right\} \,\middle|\,g\in G,\;i\in[K\!+\!1]\right\} .
\]
\end{defn}

We note that $G$ acts on the graph by 
\begin{equation}\label{eq:G-act-cay}
\forall g\in G,\;h\in L,\;[h,i]\in R:\quad g.h=gh,\;g.[h,i]=[gh,i],
\end{equation}
and the action is simply transitive on the left side. Let us illustrate
this definition with a special example.

\begin{example} \label{exa:bc-invol}
Let $G$ be a group and $S=\{s_1,\ldots,s_{K+1}\}\subseteq G$
a set of involutions, i.e.\ elements of order two (for example, $G=\mbox{Sym}(n)$ and $S=\{(1,2),(2,3),\ldots,(n\!-\!1,n)\}$ or all transpositions -- see Figure \ref{fig:S3_star}). 
Defining $S^i=\{s_i\}$ for $i\in[K\!+\!1]$, the pair $(G,\smash{\vec{S}})$ satisfies the bi-Cayley axioms, and the Cayley bigraph $CayB\left(G,\smash{\vec{S}}\right)$ coincides with the barycentric subdivision of the Cayley graph $Cay(G,S)$.
Namely, upon placing a new vertex in the middle of each edge of $Cay(G,S)$, the old (resp.\ new) vertices form the left (resp.\ right) vertices of the Cayley bigraph $CayB\left(G,\smash{\vec{S}}\right)$. 
\end{example}

\begin{figure}[h]
\parbox[c]{1\columnwidth}{%
\begin{center}
\includegraphics[scale=2]{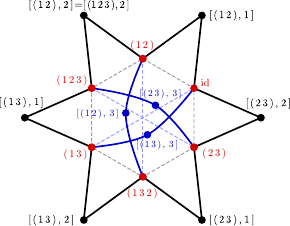}
\par\end{center}%
}
\caption{\label{fig:S3_star}Two Cayley bigraphs arising from involutions in $G=Sym(3)$, as in Example \ref{exa:bc-invol}. In black and red is the $(2,2)$-bigraph obtained from $S=\left\{ \left(1\,2\right),\left(2\,3\right)\right\} $, with the left (resp.\ right) vertices drawn in red (resp.\ black); the dashed edges are those of the \textquotedblleft original\textquotedblright{} Cayley graph $Cay(G,S)$. Adding the blue edges and vertices gives the $\left(3,2\right)$-bigraph obtained from $S=\left\{ \left(1\,2\right),\left(2\,3\right),\left(1\,3\right)\right\}$, and now the right side consists of the black and blue vertices together.}
\end{figure}

The bigraph analogue of a Schreier graph is slightly more involved.
Recall that for a (right) $G$-set $X$ the Schreier graph $Sch(G,X,S)$ has vertices $V=X$ and edges $E=\{\{x,xs\}\,|\,x\in X,s\in S\}$.
Even when $1\notin S$, this graph may have loops as $s\neq1$ may fix some $x\in X$, and multiple edges arise similarly. 
Every $G$-set is a disjoint union of transitive ones, and every transitive $G$-set $X$ is isomorphic to $H\backslash G$ for some $H\leq G$ (unique up to conjugation). 
The graph $Sch(G,H\backslash G,S)$ is the same as the quotient of $Cay(G,S)$ by $H$, and this can be used to define Schreier bigraphs. 
Namely, define the Schreier bigraph $SchB\left(G,H\backslash G,\smash{\vec{S}}\right)$ to be the quotient of $CayB\left(G,\smash{\vec{S}}\right)=\left(L\sqcup R,E\right)$ by the group $H$, acting as in \eqref{eq:G-act-cay}. 
The technical issue which arises is that while $H$ acts freely on the left side $L$ and thus also on $E$, its action on $R$ may have non-trivial stabilizers. 
In this case the quotient is a weighted graph (or \emph{orbigraph}).
The weights ensure that the correspondence between functions on the quotient graph, and $H$-invariant functions on the covering graph, intertwines with geometric operators (e.g.\ the adjacency operator). 
This leads to the following:

\begin{defn} \label{def:biSch}
For $(G,\vec{S})$ satisfying the bi-Cayley axioms and a right $G$-set $X$, the Schreier bigraph $SchB(G,X,\vec{S}) = (L\sqcup R,E)$ is defined by 
\[
L=X,\quad R=\nicefrac{X\times[K+1]\,}{\sim},\quad E=\left\{ \{ x,[x,i]\} \,\middle|\,x\in X,i\in[K+1]\right\} ,
\]
where $(x_1,j_1)\sim(x_2,j_2)$ if either $(x_1,j_1)=(x_2,j_2)$,
or there exists $s\in S^{j_1}$ with $x_1s=x_2$ and $s^{-1}\in S^{j_2}$.
Again $\sim$ is an equivalence relation, with equivalence classes
\[
[x,j]=\{(x,j)\}\cup\{(xs,i(s^{-1}))\mid s\in S^{j}\}.
\]
The \emph{weight }of $r=[x,j]\in R$ is 
\[
w_{r}:=\frac{\left|r\right|}{k+1} =\frac{1+|\stab_G(x)\cap S^j\cap (S^j)^{-1}|}{k+1},
\]
and the adjacency operator of $SchB(G,X,\vec{S})$ is given by
\[
A_{\ell,r}=w_{r}\cdot\left|\left\{ j\,\middle|\,[\ell,j]=r\right\} \right|,\quad A_{r,\ell}=\left|\left\{ j\,\middle|\,[\ell,j]=r\right\} \right|\qquad\forall\ell\in L,r\in R
\]
(the term $\left|\left\{ j\,\middle|\,[\ell,j]=r\right\} \right|$
counts the possibly multiple edges, and the weight $w_{r}$ accounts for the stabilizers).
Note that $A$ is not symmetric, but it is self-adjoint with respect
to the $w_{r}$-weighted inner product. 
\end{defn}


\begin{rem} \label{rem:R-by-neighbors}
In the Cayley case, we could have also labeled the right vertices by their neighbor-sets in the left side.
This would give 
\begin{equation}
R=\left\{ \left\{ gs\,\middle|\,s\in S^i\cup\{1\}\right\} \,\middle|\,g\in G,i\in[K\!+\!1]\right\} ,\label{eq:right-by-neighbors}
\end{equation}
and then edges are given by membership, i.e.\ $\ell\sim r$ if and only if $\ell\in r$ for $\ell\in L,r\in R$ (this was used in Example \ref{exa:X2q}).
In the Schreier case, however, such a presentation is not always possible since two different right vertices might have the same set of neighbors.
\end{rem}

\begin{rem} \label{rem:biCay=00003DbiSch}
If $H\trianglelefteq G$ and the quotient map $\phi\colon G\rightarrow G/H$ is injective on $\{1\}\cup\bigcup_iS^i$, then by construction 
\[
SchB\left(G,H\backslash G,\smash{\vec{S}}\right)\cong H\backslash CayB\left(G,\smash{\vec{S}}\right)\cong CayB\left(H\backslash G,\{\phi(S^i)\}_i\right).
\]
With some abuse of notation, we shall allow ourselves to refer to
this graph as $CayB\left(H\backslash G,\{\phi(S^i)\}_i\right)$
even when $\phi$ is not injective; the only caveat is that one must
compute the map $i\colon S\rightarrow \N$ from Definition
\ref{def:axioms} for the original $S$ and not for its image under
$\phi$.
\end{rem}

Next, we show that when a group acts on a biregular tree in a nice manner, the tree can be identified with a Cayley bigraph of the group.

\begin{thm} \label{thm:biCay-tree}
Let $\Lambda$ be a group acting on the $(K\!+\!1,k\!+\!1)$-biregular tree $\T=\T_{K\!+\!1,k\!+\!1}=\left(L_{\T}\sqcup R_{\T},E_{\T}\right)$, so that the action is simply-transitive on $L_{\T}$. 
Let $v_0\in L_{\T}$ be a fixed vertex and $v_1,\dots,v_{K+1}\in R_{\T}$ its neighbors. 
Denoting 
\[
\forall i\in[K\!+\!1]:\qquad S^i=\left\{ 1\ne s\in\Lambda\,\middle|\,\mbox{dist}(sv_0,v_i)=1\right\} ,
\]
the Cayley bigraph $CayB(\Lambda,\vec{S})=(L\sqcup R,E)$ is isomorphic to $\T$, via 
\[
f(v)=\begin{cases}
gv_0 & v=g\in L\\
gv_i & v=[g,i]\in R
\end{cases}.
\] 
\end{thm}

Combining this theorem with the observations in Remark \ref{rem:biCay=00003DbiSch} (for $G=\Lambda$) yields a useful corollary:

\begin{cor} \label{cor:biCay-quotient}
Let $\left(\Lambda,\smash{\vec{S}}\right)$ be as in Theorem \ref{thm:biCay-tree} and let $\phi\colon\Lambda\rightarrow G_{\phi}$ be a surjective homomorphism. 
Then the quotient of $\T$ by $\ker\Lambda_{\phi}$ is isomorphic to $SchB(\Lambda,G_{\phi},\smash{\vec{S}})$, where $\Lambda$ acts on $G_{\phi}$ via $\phi$. 
Furthermore, when $\phi$ is injective on $\{1\}\cup\bigcup_iS^i$, this is precisely $CayB(G_{\phi},\{\phi(S^i)\}_{i=1}^{K+1})$.
\end{cor}

The proof of Theorem \ref{thm:biCay-tree} requires the following
two Lemmas.

\begin{lem} \label{lem:axioms} 
The pair $\left(\Lambda,\smash{\vec{S}}\right)$ in Theorem \ref{thm:biCay-tree} satisfies the bi-Cayley axioms. 
\end{lem}

\begin{proof}
Since $v_0$ is the unique common neighbor of $v_i,v_{j}$ ($i\neq j$),
$S^i$ and $S^{j}$ are disjoint. 
We note that $S:=\bigsqcup_{i=1}^{K+1}S^i=\{g\in\Lambda\mid\mbox{dist}(gv_0,v_0)=2\}$, from which follows that $S^{-1}=S$ and $1\not\in S$. 
Let $s\in S$, $i=i(s)$, $j=i(s^{-1})$ and $s\ne t\in S^i$. 
Note that $v_i$ is the unique vertex connecting $v_0$ and $sv_0$ and similarly $v_{j}$ is the unique vertex connecting $v_0$ and $s^{-1}v_0$.
By applying $s$ we get that $sv_{j}$ is the unique vertex connecting $sv_0$ and $v_0$, hence $sv_{j}=v_i$. 
Then $\mbox{dist}(s^{-1}tv_0,v_{j})=\mbox{dist}(tv_0,sv_{j})=\mbox{dist}(tv_0,v_i)=1$, i.e. $s^{-1}t\in S^{j}$. 
\end{proof}

\begin{lem} \label{lem:iden} 
Let $\left(\Lambda,\smash{\vec{S}}\right)$ be as
in Theorem \ref{thm:biCay-tree}. Then for any $g_1,g_2\in\Lambda$
and $j_1,j_2\in[K\!+\!1]$, 
\[
g_1v_{j_1}=g_2v_{j_2}\quad\iff\quad(g_1,j_1)\sim(g_2,j_2).
\]
\end{lem}

\begin{proof}
Assume $g_1v_{j_1}=g_2v_{j_2}$. If $g_1=g_2$, then $v_{j_1}=v_{j_2}$ implies $j_1=j_2$ as well. If $g_1\neq g_2$, then 
\[ 
\dist(g_1^{-1}g_2v_0,v_{j_1})=\dist(g_2v_0,g_1v_{j_1})=\dist(g_2v_0,g_2v_{j_2})=\dist(v_0,v_{j_2})=1
\]
implies $g_1^{-1}g_2\in S^{j_1}$. 
Arguing similarly gives $g_2^{-1}g_1\in S^{j_2}$, so that $(g_1,j_1)\sim(g_2,j_2)$. 
Conversely, let $(g_1,j_1)\sim(g_2,j_2)$. 
If $(g_1,j_1)=(g_2,j_2)$, then $g_1v_{j_1}=g_2v_{j_2}$, and otherwise, $g_1^{-1}g_2\in S^{j_1}$ and $g_2^{-1}g_1\in S^{j_2}$.
Therefore $v_{j_1}$ is the unique vertex connecting $v_0$ and $g_1^{-1}g_2v_0$ and similarly $v_{j_2}$ connects $v_0$ and $g_2^{-1}g_1v_0$. 
Applying $g_1^{-1}g_2$ to the latter shows that $g_1^{-1}g_2v_{j_2}$ connects $g_1^{-1}g_2v_0$ and $v_0$, hence $g_1^{-1}g_2v_{j_2}=v_{j_1}$, which proves $g_1v_{j_1}=g_2v_{j_2}$. 
\end{proof}
\begin{proof}[Proof of Theorem \ref{thm:biCay-tree}]
The map $f|_{L}$ is bijective since $\Lambda$ acts simply-transitively on $L_{\T}$, and $f|_{R}$ is well defined and injective by Lemma \ref{lem:iden}. 
It is also surjective, since the $\Lambda$-orbits of $\{v_1,\ldots,v_{K+1}\}$ cover $R_{\T}$. 
Finally, each edge $\{g,[g,i]\}\in E$ is sent to an edge in $\T$, namely $\{gv_0,gv_i\}$, and each edge in $\T$ is of this form for some $g\in\Lambda$ and $i\in[K\!+\!1]$, so that $f$ is a graph isomorphism.
\end{proof}

\subsection{\label{subsec:eisenstein-lattice-explicit}Explicit example}

In Section \ref{sec:lattices} we construct several arithmetic groups acting simply-transitively on (the left side of) biregular trees.
For each such group Corollary \ref{cor:biCay-quotient} gives infinitely many explicit Cayley bigraphs, for example taking congruence quotients.
In Sections \ref{section:automorphic} and \ref{sec:applications} we show that all of these graphs are adj-Ramanujan, and that some are Ramanujan whereas some are not. We end this section by describing a special example of Cayley (resp.\ Schreier) bigraphs, which arise from one of these groups (the Eisentein lattice $\Lambda_{\EE}^p$
with $p=2$).

Let $E=\Q[\sqrt{-3}]$ and $\omega=\frac{-1+\sqrt{-3}}2$ (a primitive
third root of unity), and denote (as in \eqref{eq:Eisentein-2-mats})
\begin{align}
A & := \tfrac12 \bsmx -1 & 0 & \sqrt{-3}\\ 0 & 2 & 0\\ \sqrt{-3} & 0 & -1 \esmx,\; \sigma := \bsmx 1 \\ & \omega\\ &  & \omega^2 \esmx, \;
\tau := \bsmx & 1\\ & & 1\\ 1 \esmx, \text{ and} \\
S^{3a+b+1} & :=\left\{ A_{a,b}^{\pm}=\sigma^{a}\tau^{b}A^{\pm1}\tau^{-b}\sigma^{-a}\right\} \qquad0\leq a,b\leq2.\label{eq:Eisnetin-stars}
\end{align}
Each of $S^1,\ldots,S^{9}$ is a symmetric set of size two, so that $S=\sqcup S^i$ is symmetric and $i\colon S\rightarrow \N$ from Definition \ref{def:axioms} is given by $i(j)=j$. 
Let $\Lambda$ be the group generated by $S$ in $PGL_3(E)$. 
This group (which is denoted by $\Lambda_{\EE}^2$ in Sections 5-8) is a subgroup of $PU_3(\Z[1/2])$, which acts naturally on a certain $(9,3)$-biregular tree $\T$ described in Section \ref{subsec:unitarygroups}. 
It is shown in Section \ref{sec:lattices} that $\Lambda$ acts simply-transitively on the left vertices of $\T$ (Theorem \ref{thm-simply-transitive}) and explicit computation using Proposition \ref{prop:nbs-crit} shows that for a certain left vertex $v_0$ with neighbors $v_1,\ldots,v_9$, each $S^i$ sends $v_0$ to the two other neighbors of $v_i$. 
Thus, $CayB\left(\Lambda,\smash{\vec{S}}\right)$ coincides with $\T$ by Theorem \ref{thm:biCay-tree}. 
A special feature of this example is that each $S^i\cup\{I\}$ is in fact a subgroup of $\Lambda$. 
Thus, in the presentation of $CayB$ suggested in Remark \ref{rem:R-by-neighbors}, each vertex in the right side of $\T$ is represented by a coset of $S^i\cup\{I\}$ in $\Lambda$. 
Corollary \ref{cor:biCay-quotient} gives us examples of finite Cayley bigraphs: Let $q\ne2,3$ be a prime, $\F=\begin{cases}
\F_q & q\equiv1\Mod3\\
\F_{q^2} & q\equiv2\Mod3
\end{cases}$ and $\overline{\omega}$ a primitive third root of unity in $\F^{\times}$.
Let $\phi_q\colon PGL_3(\Z[\omega])\rightarrow PGL_3(\F)$ be the modulo $q$ homomorphism with $\phi_q(\omega)=\overline{\omega}$, let $\mathbf{G}_q=\phi_q(\Lambda)$ and $\Lambda_q=\Lambda\cap\ker\phi_q$, and let 
\[
X_{\EE}^{2,q}=CayB\left(\mathbf{G}_q,\smash{\{\phi_q(S^i)\}_{i=1}^{9}}\right)\cong\Lambda_q\backslash\T,
\]
where the isomorphism is by Corollary \ref{cor:biCay-quotient}, and Lemma \ref{lem:injective} below shows this is an honest Cayley bigraph. Using
again the labeling from Remark \ref{rem:R-by-neighbors}, each right
vertex in $X_{\EE}^{2,q}$ is labeled by a coset of some subgroup $\phi_q(S^i)\cup\{1\}\leq \mathbf{G}_q$, and we obtain the graphs described in Example \ref{exa:X2q}.
Theorem \ref{thm:main-Eis} shows that $\mathbf{G}_q$ are classical
groups over finite fields, and that $X_{\EE}^{2,q}$ are adj-Ramanujan,
but not Ramanujan. Furthermore, the group $\mathbf{G}_q$ acts transitively
on the projective plane over $\F$, or a certain subset of
it, as follows: 
\[
Y_q:=\begin{cases}
\mathbb{P}^2(\F_q) & q\equiv1\Mod3\\
\left\{ v\in\mathbb{P}^2(\F_q[\omega])\,\middle|\,v^{*}\!\cdot\!v=0\right\}  & q\equiv2\Mod3,
\end{cases}
\]
and the corresponding Schreier bigraphs $Y_{\EE}^{2,q}=SchB(\mathbf{G}_q,Y_q,\smash{\{\phi_q(S^i)\}_{i=1}^{9}})$
are Ramanujan (Theorem \ref{thm:main-Eis}).

\begin{lem}
\label{lem:injective} In the notations above, $\phi_q$ (with $q\neq2,3$)
is injective on $S\sqcup\left\{ I\right\} $ .
\end{lem}

\begin{proof}
The diagonal coefficients of any $s\in S$ cannot all agree modulo
$q$ for $q>3$, hence $\phi_q(s)\ne\phi_q(I)$. Similarly, if
$s,s'\in S$ and $\phi_q(s)=\phi_q(s')$, then by examining their
diagonal we see that $s=\sigma^{j}s'\sigma^{-j}$ for some $0\leq j\leq2$.
From the non-zero off-diagonal coordinates we get $\frac{\sqrt{-3}}2=\frac{\sqrt{-3}}2\omega^{2j}$
in $\F$, which implies $j=0$ and $s=s'$.
\end{proof}
\begin{rem}
\label{ex:eisenstein-mod3} For $q=3$ one can still define $X_{\EE}^{2,3}$,
using $\phi_3\colon PGL_3(\Z[\omega])\rightarrow PGL_3(\nicefrac{\F_3[x]}{(x^2)})$
with $\phi_3(\omega)=1-x$, and show that it is non-Ramanujan. This
might be proved similarly to Theorem \ref{thm:main-Eis}, though in
this specific case it is easier to check it with a computer. Moreover,
the subgroup $\mathbf{G}_3=\left\langle \phi_3(S)\right\rangle $
is isomorphic to $(\Z/3\Z)^3$. To see this, first note that $\phi_3(A_{*,*}^{\pm})\in I+xM_3(\F_3)$
which implies that $\mathbf{G}_3$ is Abelian. Second note that
$\phi_3(\sigma)$ commutes with $\phi_3(A)$ since $x(1-x)^{j}\equiv x(1-x)^{j+2}\Mod{3,x^2}$,
so that $\mathbf{G}_3$ is generated by $\phi_3(A_{0,*})$. Finally,
verify that $A^3=I$ and that $\phi_3(A_{0,*}^{\pm})$ are linearly
independent over $\F_3$. (In the language of Section \ref{sec:lattices},
$\mathbf{G}_3$ is a subgroup of the projective unitary group of
the non-étale extension $\nicefrac{\F_3[x]}{(x^2)}$ of
$\F_3$, and the latter can be described by $\left\lbrace A+xB\in M_3(\nicefrac{\F_3[x]}{(x^2)})\;|\;A^{t}A=I,\;AB^{t}=BA^{t}\right\rbrace /\{\pm I\}$).
\end{rem}

\begin{figure}[h]
\parbox[c]{1\columnwidth}{%
\begin{center}
\hspace*{\fill}\hspace*{\fill}\includegraphics[scale=0.2]{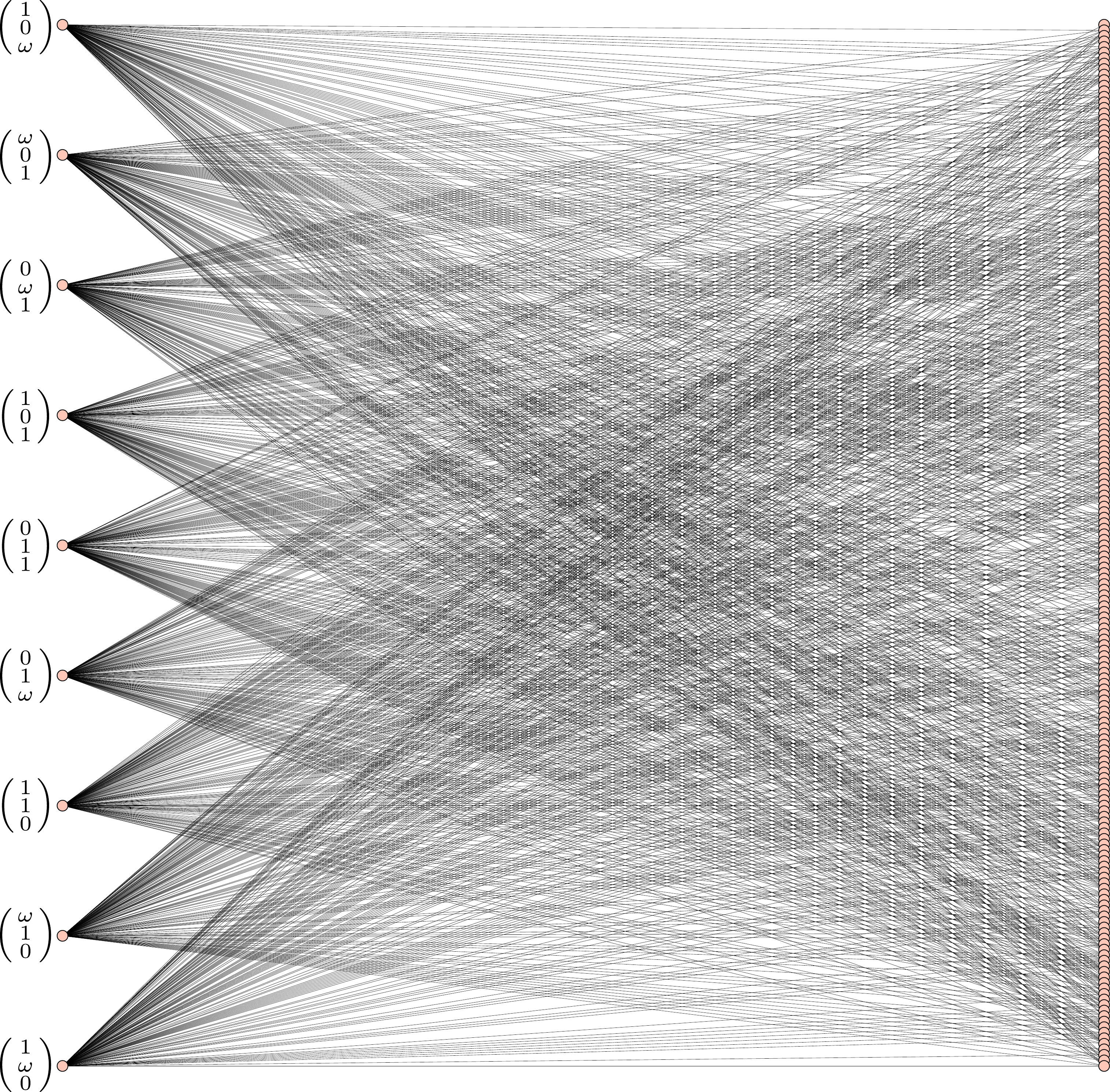}\hspace*{\fill}
\par\end{center}%
}\caption{The (Ramanujan) Schreier bigraph $Y^{5,2}$ arising from the Eisenstein lattice $\Lambda_{\EE}^{5}$ acting on the vectors of norm zero in $\F_{4}^3$.}
\end{figure}

\section{Spectral Analysis} \label{sec:Spectral-analysis}

Let $X=\left(V=L\sqcup R,E\right)$ be a connected undirected $(K\!+\!1,k\!+\!1)$-biregular
bipartite graph with $K>k$ and $\left|L\right|=n$, so that $\left|R\right|=\frac{K+1}{k+1}n$.
Throughout this section we denote by $E$ the \emph{directed }edges
of $X$; each edge appears with both directions, i.e.\ $\left|E\right|=2n(K+1)$.
We denote by $A=A_{X}$ the adjacency operator on $L^2\left(V\right)$.
It is easy to see that $\ker A=\ker\left(A|_{R}\right)\oplus\ker\left(A|_{L}\right)$,
and since $A$ maps $L^2(R)$ to $L^2(L)$ and $\left|R\right|>\left|L\right|$,
$A|_{R}$ must have a nontrivial kernel. There is however no reason
for $A|_{L}$ to have one (and generically it does not - see \cite{brito2018spectral}).
We define the \emph{excessiveness} of $X$ by 
\[
\mE=\mE_{X}\overset{{\scriptscriptstyle def}}{=}\dim\ker\left(A|_{L}\right),
\]
and observe that $\left(A|_{R}\right)=(A|_{L})^{T}$ implies $\rank(A|_R)=n-\mE$, so that
\[
\mathcal{N}_{X}\overset{{\scriptscriptstyle def}}{=}\dim\ker\left(A|_{R}\right)=|R|-n+\mE=\tfrac{K-k}{k+1}n+\mE,
\]
and in addition that $\rank A=\rank A|_{L}+\rank A|_{R}=2\left(n-\mE\right)$.
The spectrum of $A$ is symmetric: if for $f\in L^2(V)$ we denote
$\tilde{f}:=f|_{L}-f|_{R}$, then $Af=\lambda f$ implies $A\tilde{f}=-\lambda\tilde{f}$,
so in total $A$ has $n-\mE$ positive eigenvalues (counting with
multiplicities), which we denote by 
\[
\sqrt{(K+1)(k+1)}=\mathfrak{pf}=\lambda_1>\lambda_2\geq\ldots\geq\lambda_{n-\mE}>0.
\]
We define the following union of three line segment in $\C$:
\[
\Theta_{K,k}=\left[-i\log\sqrt{Kk},0\right]\cup\left[0,\pi\right]\cup\left[\pi,\pi+i\log\sqrt{K/k}\right],
\]
and with every $\lambda\in\left[0,\mathfrak{pf}\right]$ we associate
\[
\vartheta=\vartheta_{\lambda}=\arccos\left(\frac{\lambda^2-K-k}{2\sqrt{Kk}}\right).
\]
We note that $\lambda\mapsto\vartheta_{\lambda}$ maps $\left[0,\mathfrak{pf}\right]$
bijectively onto $\Theta_{K,k}$, and takes $\left[\sqrt{K}+\sqrt{k},\sqrt{K}-\sqrt{k}\right]$
onto $\left[0,\pi\right]$. Thus, writing $\vartheta_{j}=\vartheta_{\lambda_{j}}$,
$X$ is adj-Ramanujan (see \eqref{eq:aram-def}) if and only if $\vartheta_2,\ldots,\vartheta_{n-\mE}\in\left[0,\pi\right]$.
The $\vartheta$-parametrization will be useful in analyzing walks
on the graph in Section \ref{sec:Combinatorics}, and is inspired by a well-known parametrization used for the regular case (see e.g.\ \cite{davidoff2003elementary,Nestoridi2021Boundedcutoffwindow}).

We now turn to the non-backtracking $B=B_{X}$, which acts on $L^2(E)$
by $Bf(v\!\rightarrow\!u)=\sum_{v\ne w\sim u}f(u\!\rightarrow\!w)$.
Its Perron-Frobenius eigenvalue is $\mathfrak{pf}_{B}=\sqrt{Kk}$,
and its spectrum is also symmetric: decomposing $E=\overrightarrow{LR}\sqcup\overleftarrow{LR}$,
where $\overrightarrow{LR}=\left\{ \ell\!\rightarrow\!r\,\middle|\,\begin{smallmatrix}\ell\in L,r\in R\\
\ell\sim r
\end{smallmatrix}\right\} $ and $\overleftarrow{LR}=\left\{ \ell\!\leftarrow\!r\,\middle|\,\begin{smallmatrix}\ell\in L,r\in R\\
\ell\sim r
\end{smallmatrix}\right\} $, we define $\widetilde{F}:=F|_{\overrightarrow{LR}}-F|_{\overleftarrow{LR}}$
and observe that 
\begin{equation}
BF=\mu F\qquad\Rightarrow\qquad B\widetilde{F}=-\mu\widetilde{F}.\label{eq:Btilde}
\end{equation}

We associate with every $\lambda\in\left[0,\mathfrak{pf}\right]$
the two complex numbers
\begin{equation}
\mu^{\pm}=\mu_{\lambda}^{\pm}=\sqrt{\tfrac12\big(\lambda^2-K-k\pm\sqrt{(\lambda^2-K-k)^2-4Kk}\big)}=\sqrt{e^{\pm i\vartheta_{\lambda}}\sqrt{Kk}},\label{eq:mu-def}
\end{equation}
where the external square root is chosen so that $\mathrm{arg}(\mu^\pm)\in [0,\pi)$. We will show in Proposition \ref{prop:F-B-ef}
that if $\lambda\in\Spec A$ then $\pm\mu_{\lambda}^{\pm}\in\Spec B$,
save for some singular cases in which the corresponding eigenfunction
vanishes. However, the eigenvalues alone are not enough for spectral analysis, since
their corresponding eigenfunctions are not orthogonal, as $B$ is
not normal. The main goal of this section is:
\begin{thm}
\label{thm:B-decomp}If $\lambda_1>\lambda_2\geq\ldots\geq\lambda_{n-\mE_{X}}>0$
are the positive eigenvalues of $A$, then $B$ is unitarily equivalent
to a block-diagonal matrix composed of:
\begin{enumerate}
\item The block $\left(\begin{smallmatrix} & K\\
k
\end{smallmatrix}\right)$.
\item 
\begin{enumerate}
\item For each $2\leq j\leq n-\mE_{X}$ with $\lambda_{j}\in\left[\sqrt{K}-\sqrt{k},\sqrt{K}+\sqrt{k}\right]$,
the block
\begin{equation}
\left(\begin{smallmatrix}0 & 0 & \mu_{j}^{+}\sqrt[4]{\frac{k}{K}} & k-1\\
0 & 0 & 0 & -\mu_{j}^{-}\sqrt[4]{\frac{k}{K}}\\
\mu^{+}\sqrt[4]{\frac{K}{k}} & K-1 & 0 & 0\\
0 & -\mu_{j}^{-}\sqrt[4]{\frac{K}{k}} & 0 & 0
\end{smallmatrix}\right).\label{eq:B_W_Ram_case}
\end{equation}
\item For each $2\leq j\leq n-\mE_{X}$ with $\lambda_{j}\notin\left[\sqrt{K}-\sqrt{k},\sqrt{K}+\sqrt{k}\right]$,
a block of the form $\left(\begin{smallmatrix}0 & 0 & \mu_{j}^{+}\alpha_{j} & \gamma_{j}\\
0 & 0 & 0 & \mu_{j}^{-}\beta_{j}\\
\mu_{j}^{+}/\alpha_{j} & \delta_{j} & 0 & 0\\
0 & \mu_{j}^{-}/\beta_{j} & 0 & 0
\end{smallmatrix}\right)$, where $\alpha_{j},\beta_{j}\in\R_{>0}$ and $\gamma_{j},\delta_{j}\in\C$
are all bounded by $K$.\footnote{With some effort one can compute $\alpha_{j},\beta_{j},\gamma_{j},\delta_{j}$
explicitly from the proof, but they are not pleasant and will not
be needed for our combinatorial purposes in Section \ref{sec:Combinatorics}.
We also remark that case 2(b) does not occur in arithmetic congruence
quotients of $U_3$-buildings -- see Sections \ref{sec:local-rep},
\ref{section:automorphic}.} 
\end{enumerate}
\item $\mE_{X}$ times the block $\left(\begin{smallmatrix} & i\\
iK
\end{smallmatrix}\right)$.
\item $\mathcal{N}_{X}$ times the block $\left(\begin{smallmatrix} & ik\\
i
\end{smallmatrix}\right)$.
\item $\chi(X)$
times the diagonal block $\left(\begin{smallmatrix}1\\
 & -1
\end{smallmatrix}\right)$, where $\chi(X)=\frac{|E|}2-|V|+1=\frac{Kk-1}{k+1}\cdot n+1$.
\end{enumerate}
\end{thm}

Note that since $\mathrm{Spec}\left(\begin{smallmatrix}0 & \alpha\\
\beta & 0
\end{smallmatrix}\right)=\left\{ \pm\sqrt{\alpha\beta}\right\} $ and $\mathrm{Spec}\left(\begin{smallmatrix}0 & 0 & \alpha\mu^{+} & *\\
0 & 0 & 0 & \beta\mu^{-}\\
\mu^{+}/\alpha & * & 0 & 0\\
0 & \mu^{-}/\beta & 0 & 0
\end{smallmatrix}\right)=\left\{ \pm\mu^{+},\pm\mu^{-}\right\} $, in particular we obtain the following:
\begin{cor}[\cite{Kempton2016NonBacktrackingRandom,brito2018spectral}]
\label{cor:B-spec}
The spectrum of $B_X$ is
\begin{align*}
&\mathrm{Spec}(B_X)
\\&=\begin{cases}
\left\{ \pm\sqrt{Kk},\pm i\sqrt{k},\pm1\right\} \cup\left\{ \pm\mu_{\lambda}^{\pm}\,\middle|\,0<\lambda\in\mathrm{Spec}(A_X)\right\}  & \mE_{X}=0\\
\left\{ \pm\sqrt{Kk},\pm i\sqrt{K},\pm i\sqrt{k},\pm1\right\} \cup\left\{ \pm\mu_{\lambda}^{\pm}\,\middle|\,0<\lambda\in\mathrm{Spec}(A_X)\right\}  & \mE_{X}>0.
\end{cases}
\end{align*}
\end{cor}
This was shown in the cited papers by a detailed study of the Ihara-Bass formula following \cite{hashimoto1989zeta}. The approach we take is inspired by \cite{kotani2000zeta} and has the advantage of giving the eigenvectors as well, which is important for the combinatorial applications in Section \ref{sec:Combinatorics}. We further obtain:
\begin{cor}
\label{cor:ram-by-theta}
\begin{enumerate}
\item The graph $X$ is NB-Ramanujan if and only if it is adj-Ramanujan and additionally
$\mE_{X}=0$.
\item $X$ is adj-Ramanujan if and only if $\Spec\left(A^2\big|_{L}\right)\subseteq\left\{ (K+1)(k+1)\right\} \cup\left[K+k-2\sqrt{Kk},K+k+2\sqrt{Kk}\right]\cup\left\{ 0\right\} $.
\item $X$ is NB-Ramanujan if and only if $\Spec\left(A^2\big|_{L}\right)\subseteq\left\{ (K+1)(k+1)\right\} \cup\left[K+k-2\sqrt{Kk},K+k+2\sqrt{Kk}\right]$.
\item The NB-Ramanuajan property and the Riemann hypothesis are equivalent
for bigraphs.
\end{enumerate}
\end{cor}

\begin{proof}
(1) This follows from the observation that 
\begin{align*}
\lambda\in[\sqrt{K}+\sqrt{k},\sqrt{K}-\sqrt{k}]\ \Leftrightarrow\ \vartheta_{\lambda}\in[0,\pi]\ \\\Leftrightarrow\ |\mu_{\lambda}^{+}|,|\mu_{\lambda}^{-}|\leq\sqrt[4]{Kk}\ \Leftrightarrow\ |\mu_{\lambda}^{+}|,|\mu_{\lambda}^{-}|=\sqrt[4]{Kk},
\end{align*}
and the fact that $|\pm i\sqrt{k}|<\sqrt[4]{Kk}<|\pm i\sqrt{K}|$;
(2,3) These are immediate from (1); (4) NB-Ramanujan implies the R.H.\ for
any finite graph by \cite{Angel2015nonbacktrackingspectrum} (and
for bigraphs by \cite{hashimoto1989zeta}). For bigraphs the converse
also holds since if $\left|\mu_{\lambda}^{\pm}\right|<\sqrt[4]{Kk}$
then $\left|\mu_{\lambda}^{\mp}\right|>\sqrt[4]{Kk}$. 
\end{proof}
\begin{rem}
In fact, the Riemann Hypothesis, NB-Ramanujan and Ramanujan property
are all equivalent for bigraphs, but for the proof we need the theory
of Hecke algebras -- see Section \ref{sec:local-rep}.
\end{rem}

We now begin our analysis of the non-backtracking spectrum, noting first that $\mu^{+},\mu^{-},-\mu^{+},-\mu^{-}$
are all the roots of 
\begin{equation}
\mu^{4}+\left(K+k-\lambda^2\right)\mu^2+Kk=0.\label{eq:mu-by-lambda}
\end{equation}
Since $\mu^{\pm}\neq0$, and $\mu^{+}\neq-\mu^{-}$ by the choice
of branch in (\ref{eq:mu-def}), equation (\ref{eq:mu-by-lambda})
has only two solutions precisely when
\begin{equation}
\mu^{+}=\mu^{-}\quad\Leftrightarrow\quad\vartheta\in\left\{ 0,\pi\right\} \quad\Leftrightarrow\quad\lambda=\sqrt{K}\pm\sqrt{k}\quad\Leftrightarrow\quad\mu^{\pm}\in\left\{ \sqrt[4]{Kk},i\sqrt[4]{Kk}\right\} \label{eq:singular-case}
\end{equation}
(and Theorem \ref{thm:B-decomp} shows that $B$ is diagonalizable
if and only if this does not occur in the spectrum). As $\lambda$ ranges over $\left[0,\mathfrak{pf}\right]$
(and $\vartheta$ over $\Theta_{K,k}$), $\pm\mu^{\pm}$ ranges over
\begin{equation}
\left\{ \pm\mu_{\lambda}^{\pm}\,\middle|\,\lambda\in\left[0,\mathfrak{pf}\right]\right\} =\pm\left[1,\sqrt{Kk}\right]\cup\left\{ z\in\C\,\middle|\,\left|z\right|=\sqrt[4]{Kk}\right\} \cup\pm\left[i\sqrt{k},i\sqrt{K}\right],\label{eq:mu-range}
\end{equation}
and the points of special importance are:
\begin{center}
\begin{tabular}{|c|c|c|c|}
\hline 
$\lambda$ & $\vartheta$ & $\mu^{+}$ & $\mu^{-}$\tabularnewline
\hline 
\hline 
$\mathfrak{pf}=\sqrt{(K+1)(k+1)}$ & $-i\log\sqrt{Kk}$ & $\mathfrak{pf}_{B}=\sqrt{Kk}$ & 1\tabularnewline
\hline 
$\sqrt{K}+\sqrt{k}$ & $0$ & $\sqrt[4]{Kk}$ & $\sqrt[4]{Kk}$\tabularnewline
\hline 
$\sqrt{K}-\sqrt{k}$ & $\pi$ & $i\sqrt[4]{Kk}$ & $i\sqrt[4]{Kk}$\tabularnewline
\hline 
0 & $\pi+i\log\sqrt{K/k}$ & $i\sqrt{k}$ & $i\sqrt{K}$\tabularnewline
\hline 
\end{tabular}
\par\end{center}
\begin{defn}
\label{def:f**-W}Given $f\colon V\rightarrow\R$, we define
$f_{\ell o},f_{ri},f_{\ell i},f_{ro}\colon E\rightarrow\C$,
by 
\begin{align*}
f_{\ell o}(\ell\!\rightarrow\!r) & =f(\ell) & f_{ri}(\ell\!\rightarrow\!r) & =f(r) & f_{\ell i}(\ell\!\rightarrow\!r) & =0 & f_{ro}(\ell\!\rightarrow\!r) & =0\\
f_{\ell o}(\ell\!\leftarrow\!r) & =0 & f_{ri}(\ell\!\leftarrow\!r) & =0 & f_{\ell i}(\ell\!\leftarrow\!r) & =f(\ell) & f_{ro}(\ell\!\leftarrow\!r) & =f(r)
\end{align*}
(for neighboring vertices $\ell\in L$ and $r\in R$). We write $f_{\star\star}$
to indicate one of the four, and denote 
\[
W_{f}=\mathrm{Span}\left\{ f_{\star\star}\right\} .
\]
We note that $\tilde{f}_{\ell o}=-f_{\ell o}$, $\tilde{f}_{ri}=-f_{ri}$,
$\tilde{f}_{\ell i}=f_{\ell i}$ and $\tilde{f}_{ro}=f_{ro}$, hence
$W_{f}=W_{\tilde{f}}.$

For $f\colon V\rightarrow\R$ satisfying $Af=\lambda f$,
fix $\mu^{\pm}=\mu_{\lambda}^{\pm}$ and define $F^{\pm}=F_{f}^{\pm}$
and $G^{\pm}=G_{f}^{\pm}$ by
\begin{equation}
\begin{alignedat}1F^{\pm} & :=\lambda f_{\ell o}-\left((\mu^{\pm})^2+K\right)f_{ri}-\mu^{\pm}\lambda f_{\ell i}+\left(\mu^{\pm}+K/\mu^{\pm}\right)f_{ro}\\
G^{\pm} & :=\left(\mu^{\pm}+k/\mu^{\pm}\right)f_{\ell o}-\mu^{\pm}\lambda f_{ri}-\left((\mu^{\pm})^2+k\right)f_{\ell i}+\lambda f_{ro}.
\end{alignedat}
\label{eq:F-and-G-eigfun}
\end{equation}
Omitting the $\pm$ choice, this means that for neighboring $\ell\in L$
and $r\in R$ 
\[
\left\{ \negmedspace\negthickspace\begin{array}{l}
F(\ell\!\rightarrow\!r)=\lambda f(\ell)-\left(\mu^2+K\right)f(r)\\
F(\ell\!\leftarrow\!r)=-\mu\lambda f(\ell)+\left(\mu+K/\mu\right)f(r)
\end{array}\right.\qquad\left\{ \negmedspace\negthickspace\begin{array}{l}
G(\ell\!\rightarrow\!r)=\left(\mu+k/\mu\right)f(\ell)-\mu\lambda f(r)\\
G(\ell\!\leftarrow\!r)=-\left(\mu^2+k\right)f(\ell)+\lambda f(r).
\end{array}\right.
\]
In the general case $F^{\pm}$ and $G^{\pm}$ differ by a scalar,
but sometimes one of them vanishes: if $\lambda=0$, then $F^{-}=G^{+}=0$
(since $\mu^{-}=i\sqrt{K}$ and $\mu^{+}=i\sqrt{k}$), but we shall
see that $F^{+},G^{-}$ do not.
\end{defn}

\begin{prop}
\label{prop:F-B-ef}If $Af=\lambda f$, then $BF^{\pm}=\mu^{\pm}F^{\pm}$,
$BG^{\pm}=\mu^{\pm}G^{\pm}$, $B\widetilde{F^{\pm}}=-\mu^{\pm}\widetilde{F^{\pm}}$
and $B\widetilde{G^{\pm}}=-\mu^{\pm}\widetilde{G^{\pm}}$.
\end{prop}

\begin{proof}
For each choice of $\mu=\mu^{\pm}$ with corresponding $F=F^{\pm}$,
we have
\begin{align*}
BF\left(\ell\rightarrow r\right) & =\sum\nolimits _{\ell\neq\ell'\sim r}F\left(\ell'\leftarrow r\right)=\sum\nolimits _{\ell\neq\ell'\sim r}\left(-\mu\lambda f(\ell')+(\mu+K/\mu)f(r)\right)\\
 & =-\mu\lambda\left(\sum\nolimits _{\ell\neq\ell'\sim r}f(\ell')\right)+k(\mu+K/\mu)f(r)
 \\&=-\mu\lambda\left(\left(Af\right)(r)-f(\ell)\right)+k(\mu+K/\mu)f(r)\\
 & =-\mu\lambda\left(\lambda f(r)-f(\ell)\right)+k(\mu+K/\mu)f(r)
 \\&=\mu\left[\lambda f(\ell)-\left(\lambda^2-k-Kk/\mu^2\right)f(r)\right]\\
 & \overset{(*)}{=}\mu\left[\lambda f(\ell)-\left(\mu^2+K\right)f(r)\right]
 \\&=\mu F\left(\ell\rightarrow r\right)
\end{align*}
where $(*)$ makes use of (\ref{eq:mu-by-lambda}). The other direction is somewhat
simpler:
\begin{align*}
BF\left(\ell\leftarrow r\right) & =\sum\nolimits _{\ell\sim r'\neq r}F\left(\ell\rightarrow r'\right)
\\& =K\lambda f(\ell)-(\mu^2+K)\sum\nolimits _{\ell\sim r'\neq r}f(r')\\
 & =K\lambda f(\ell)-(\mu^2+K)(\lambda f(\ell)-f(r))
 \\& =-\mu^2\lambda f(\ell)+(\mu^2+K)f(r)
 \\& =\mu F\left(\ell\leftarrow r\right).
\end{align*}
The verification that $BG^{\pm}=\mu^{\pm}G^{\pm}$ is similar,
and we conclude using (\ref{eq:Btilde}).
\end{proof}
\begin{lem}
\label{lem:inner-product}If $f,f'\colon V\rightarrow\R$
are eigenfunctions of $A$ with non-negative eigenvalues $\lambda,\lambda'$
then:
\begin{enumerate}
\item $\lambda\left\langle f|_{R},f'|_{R}\right\rangle =\lambda'\left\langle f|_{L},f'|_{L}\right\rangle $, and if $\lambda\ne0$ then $\left\Vert f|_{R}\right\Vert =\left\Vert f|_{L}\right\Vert =\tfrac{\left\Vert f\right\Vert }{\sqrt{2}}$.
\item $\left\{ f_{\ell o},f_{ri}\right\} \bot\left\{ f_{\ell i}',f_{ro}'\right\} $,
and 
\begin{align}
\left\langle f_{\ell i},f'_{ro}\right\rangle =\left\langle f_{\ell o},f'_{ri}\right\rangle  & =\lambda\langle f|_{R},f'|_{R}\rangle=\lambda'\left\langle f|_{L},f'|_{L}\right\rangle \nonumber \\
\left\langle f_{\ell i},f'_{\ell i}\right\rangle =\left\langle f_{\ell o},f'_{\ell o}\right\rangle  & =(K+1)\left\langle f|_{L},f'|_{L}\right\rangle \label{eq:lili}\\
\left\langle f_{ri},f'_{ri}\right\rangle =\left\langle f_{ro},f'_{ro}\right\rangle  & =(k+1)\left\langle f|_{R},f'|_{R}\right\rangle .\label{eq:riri}
\end{align}
\item If $\lambda\neq0$ then $f\bot f'$ $\Rightarrow$ $f_{\star\star}\bot f'_{\star\star}$.
\item If $\lambda=0$ and $f|_{L}=0$ or $f|_{R}=0$ then $f\bot f'$ $\Rightarrow$
$f_{\star\star}\bot f'_{\star\star}$.
\item If $\lambda\notin\{0,\mathfrak{pf}\}$ then $\left\{ f_{\star\star}\right\} $
are linearly independent, so that $\dim W_{f}=4$.
\end{enumerate}
\end{lem}

\begin{proof}
\begin{enumerate}
\item Since $A$ is self-adjoint, $\lambda\left\langle f|_{R},f'|_{R}\right\rangle =\lambda\left\langle f,f'|_{R}\right\rangle =\left\langle Af,f'|_{R}\right\rangle =\left\langle f,A\left(f'|_{R}\right)\right\rangle =\left\langle f,\left(Af'\right)|_{L}\right\rangle =\lambda'\left\langle f,f'|_{L}\right\rangle =\lambda'\left\langle f|_{L},f'|_{L}\right\rangle$. When $\lambda\ne0$ this shows in particular that $f|_{R}$ and $f|_L$ have the same norm, hence the second assertion follows by Pythagoras.
\item $\left\{ f_{\ell o},f_{ri}\right\} \bot\left\{ f_{\ell i}',f_{ro}'\right\} $
since they are supported on disjoint sets of edges ($\overrightarrow{LR}$
and $\overleftarrow{LR}$), and the rest is routine computation, e.g.
\begin{align*}
\left\langle f_{\ell i},f'_{ro}\right\rangle  & =\sum_{\ell\sim r}f_{\ell i}(\ell\!\leftarrow\!r)\overline{f'_{ro}(\ell\!\leftarrow\!r)}\\&=\sum_{\ell\sim r}f(\ell)f'(r)=\sum_{r\in R}f'(r)(Af)(r)=\lambda\left\langle f|_{R},f'|_{R}\right\rangle \\
\left\langle f_{\ell i},f'_{\ell i}\right\rangle  & =\sum_{\ell\sim r}f_{\ell i}(\ell\!\leftarrow\!r)\overline{f'_{\ell i}(\ell\!\leftarrow\!r)}=\sum_{\ell\sim r}f(\ell)f'(\ell)=(K+1)\left\langle f|_{L},f'|_{L}\right\rangle .
\end{align*}
\item If $\lambda\neq0$ then $0=\left\langle f,f'\right\rangle =\left\langle f|_{L},f'|_{L}\right\rangle +\left\langle f|_{R},f'|_{R}\right\rangle =\left(\frac{\lambda'}{\lambda}+1\right)\left\langle f|_{R},f'|_{R}\right\rangle $
implies $f|_{R}\bot f'|_{R}$, hence $f|_{L}\bot f'|_{L}$ as well,
so (2) gives $f_{\star\star}\bot f'_{\star\star}$.
\item If $\lambda=0$ and $f|_{L}=0$ ($f|_{R}=0$ is similar) then $f\bot f'$
implies $f|_{R}\bot f'|_{R}$, and we continue as in (3).
\item By (2), the Gram determinant of $\left\{ f_{\star\star}\right\} $
is $\left\Vert f|_{L}\right\Vert ^{4}\left\Vert f|_{R}\right\Vert ^{4}\left(\lambda^2-\left(K+1\right)\left(k+1\right)\right)^2$,
which vanishes if and only if either $f|_{R}=0$ or $f|_{L}=0$ (either of which
implies $\lambda=0$), or $\lambda=\mathfrak{pf}$.\qedhere
\end{enumerate}
\end{proof}
\begin{prop}
\textcolor{red}{}If $f$ is a $\lambda$-eigenfunction of $A$ for
$\lambda\notin\left\{ 0,\sqrt{K}\pm\sqrt{k}\right\} $, then $\Span\left\{ f_{\star\star}\right\} =\Span\left\{ F^{\pm},\widetilde{F^{\pm}}\right\} $.
\end{prop}

\begin{proof}
The matrix $T=\left(\begin{smallmatrix}\lambda & -\left(\mu^{+^2}+K\right) & -\mu^{+}\lambda & \mu^{+}+\frac{K}{\mu^{+}}\\
\lambda & -\left(\mu^{-^2}+K\right) & -\mu^{-}\lambda & \mu^{-}+\frac{K}{\mu^{-}}\\
\lambda & -\left(\mu^{+^2}+K\right) & \mu^{+}\lambda & -\left(\mu^{+}+\frac{K}{\mu^{+}}\right)\\
\lambda & -\left(\mu^{-^2}+K\right) & \mu^{-}\lambda & -\left(\mu^{-}+\frac{K}{\mu^{-}}\right)
\end{smallmatrix}\right)$ transforms $\left\{ f_{\star\star}\right\} $ to $\left\{ F^{\pm},\widetilde{F^{\pm}}\right\} $,
so that $\det(T)=-\frac{2\lambda^2K}{\mu^{+}\mu^{-}}\left(\mu^{+}-\mu^{-}\right)^2\left(\mu^{+}+\mu^{-}\right)^2$
implies that $W_{f}=\Span\left\{ F_i^{\pm},\widetilde{F_i^{\pm}}\right\} $
unless $\lambda=0$ or $\mu^{+}=\mu^{-}$.
\end{proof}
Eigenfunctions of $A$ only ``explain'' some of the eigenfunctions
of $B$, and the rest, which we now describe, are of topological nature (the fact that all
eigenfunctions of $B$ are described by Propositions \ref{prop:F-B-ef}
and \ref{prop:cycles} follows from the proof of Theorem \ref{thm:B-decomp}):
\begin{prop}
\label{prop:cycles}For a closed cycle $\gamma=\left(\ell_1,r_1,\ell_2,r_2,\ldots,\ell_{n}=\ell_1,r_{n}=r_1\right)$
in $X$, define $p_{\gamma},n_{\gamma}\colon E\rightarrow\C$
by
\[
\left\{ \begin{alignedat}1p_{\gamma}\left(\ell_i\rightarrow r_i\right) & =p_{\gamma}\left(\ell_{i+1}\leftarrow r_i\right)=1\\
p_{\gamma}\left(\ell_i\leftarrow r_i\right) & =p_{\gamma}\left(\ell_{i+1}\rightarrow r_i\right)=-1,
\end{alignedat}
\right.
\]
 $p_{\gamma}\left(e\right)=0$ for any
edge $e$ which does not appear in $\gamma$, and $n_\gamma=\widetilde{p_\gamma}$. Then:
\begin{enumerate}
\item $Bp_{\gamma}=p_{\gamma}$ and $Bn_{\gamma}=-n_{\gamma}$.
\item $p_{\gamma}\bot n_{\gamma'}$ for any two cycles $\gamma,\gamma'$.
\item For any $f\colon V\rightarrow\C$, $p_{\gamma}$ and $n_{\gamma}$
are orthogonal to $f_{\star\star}$.
\end{enumerate}
\end{prop}

\begin{proof}
\emph{(1)} We prove that $Bp_{\gamma}|_{\overrightarrow{LR}}=p_{\gamma}|_{\overrightarrow{LR}}$; $\overleftarrow{LR}$ is analogous, and $Bn_{\gamma}=-n_{\gamma}$ follows from \eqref{eq:Btilde}:
\begin{align*}
Bp_{\gamma}(\ell_i\to r_i)&=\sum\nolimits _{\ell_i\neq\ell'\sim r_i}p_{\gamma}\left(\ell'\leftarrow r_i\right)
\\& =p_{\gamma}\left(\ell_{i+1}\leftarrow r_i\right)=1=p_{\gamma}(\ell_i\to r_i),
\end{align*}
 and if $(\ell\to r)\not\in\gamma$ then either $r\ne r_i$ for
any $i$, or $r=r_i$ and $\ell\ne\ell_i,\ell_{i+1}$, hence 
\begin{align*}
Bp_{\gamma}(\ell\to r)&=\sum\nolimits _{\ell\neq\ell'\sim r}p_{\gamma}\left(\ell'\leftarrow r\right)
\\&=\begin{cases}
\begin{array}{c}
0\\
1-1
\end{array} & \begin{array}{c}
r\ne r_i\\
r=r_i,\;\ell\ne\ell_i,\ell_{i+1}
\end{array}=0=p_{\gamma}(\ell\to r).\end{cases}
\end{align*}
\emph{(2)} Clearly if $(\ell\to r)\not\in\gamma\cap\gamma'$ then
\[
p_{\gamma}(\ell\to r)\cdot n_{\gamma'}(\ell\to r)=0=p_{\gamma}(\ell\leftarrow r)\cdot n_{\gamma'}(\ell\leftarrow r)
\]
 and if $(\ell\to r)=\left(\ell_i\rightarrow r_i\right)\in\gamma\cap\gamma'$
then 
\[
p_{\gamma}(\ell_i\rightarrow r_i)\cdot n_{\gamma'}(\ell_i\rightarrow r_i)+p_{\gamma}(\ell_i\rightarrow r_i)\cdot n_{\gamma'}(\ell_i\rightarrow r_i)=1\cdot1+1\cdot(-1)=0,
\]
hence $\langle p_{\gamma},n_{\gamma}\rangle=\sum_{e\in E}p_{\gamma}(e)n_{\gamma}(e)=0$.

\emph{(3)} We treat $f_{\ell i}$, and leave the other cases to the
reader: 
\begin{align*}
\langle p_{\gamma},f_{\ell i}\rangle & =\sum_{e\in E}p_{\gamma}(e)f_{\ell i}(e)=\sum_i\left(p_{\gamma}(\ell_i\leftarrow r_i)f(\ell_i)+p_{\gamma}(\ell_{i+1}\leftarrow r_i)f(\ell_{i+1})\right)
\\&=\sum_i\left(f(\ell_{i+1})-f(\ell_i)\right)=0,
\\
\langle n_{\gamma},f_{\ell i}\rangle & =\langle\widetilde{p_{\gamma}},f_{\ell i}\rangle=\langle p_{\gamma},\widetilde{f_{\ell i}}\rangle=\langle p_{\gamma},f_{\ell i}\rangle=0.\qedhere
\end{align*}
\end{proof}
We can now prove the main Theorem of this section:
\begin{proof}[Proof of Theorem \ref{thm:B-decomp}]
Let us fix an orthonormal basis for $L^2(V)$: 
\[
f_1,\ldots,f_{n-\mE},\tilde{f}_1,\ldots,\tilde{f}_{n-\mE},g_1,\ldots,g_{\mE},h_1,\ldots,h_{\mathcal{N}},
\]
where $\left\{ f_{*}\right\} $ is an orthonormal system with $Af_{j}=\lambda_{j}f_{j}$
(recall $\lambda_i>0$), $\left\{ g_{*}\right\} $ is an orthonormal
basis for $\ker A|_{L}$ and $\left\{ h_{*}\right\} $ for $\ker A|_{R}$.
We denote 
\[
L_{V}^2(E)\overset{{\scriptscriptstyle def}}{=}\Span_{\C}\left\{ \varphi_{\star\star}\,\middle|\,\varphi\in L^2(V)\right\} ,
\]
and observe that $L_{V}^2(E)=\Span\left\{ \varphi_{\star\star}\,\middle|\,\varphi\in\left\{ f_{*},g_{*},h_{*}\right\} \right\} $
since $W_{\tilde{f}_i}=W_{f_i}$. From Lemma \ref{lem:inner-product}(3,4) we further obtain that there is an orthogonal decomposition
$L_{V}^2(E)=\bigoplus_{\varphi\in\left\{ f_{*},g_{*},h_{*}\right\} }W_{\varphi}$,
and our first goal is to construct for each $W_{\varphi}$ an orthonormal
basis in which $B|_{W_{\varphi}}$ has a nice form. We denote $F_{j}^{\pm}=F_{f_{j}}^{\pm}$,
$G_{j}^{\pm}=G_{g_{j}}^{\pm}$ and $H_{j}^{\pm}=F_{h_{j}}^{\pm}$.
\begin{enumerate}
\item For $j=1$, $\lambda_1=\mathfrak{pf}$ and $f_1$ is constant
on each of $R$ and $L$, so that $W_1:=W_{f_1}=\Span\left\{ \mathbf1_{\overrightarrow{LR}},\mathbf1_{\overleftarrow{LR}}\right\} $.
The basis $\mathbf1_{\overrightarrow{LR}},\mathbf1_{\overleftarrow{LR}}$ is orthonormal up to scaling, and in it $B|_{W_1}=\left(\begin{smallmatrix} & K\\
k
\end{smallmatrix}\right)$. The functions $F_1^{+}$ and $\widetilde{F_1^{+}}$ (which are scaling
of $\sqrt{k}\cdot\mathbf1_{\overrightarrow{LR}}\pm\sqrt{K}\cdot\mathbf1_{\overleftarrow{LR}}$)
form a non-orthogonal $B|_{W_1}$-eigenbasis, whereas $F_1^{-}=\widetilde{F_1^{-}}=0$.
\item For $2\leq j\leq n-\mE$ and $f=f_{j}$, $W_{j}:=W_{f_{j}}$ is $4$-dimensional by Lemma \ref{lem:inner-product}(5). We split into two cases:\\
\textbf{(2a)} $\lambda_{j}\in\left[\sqrt{K}-\sqrt{k},\sqrt{K}+\sqrt{k}\right]$, which is equivalent to $|\mu|=\sqrt[4]{Kk}$. It is tempting to start with $F_{j}^{\pm},\widetilde{F_{j}^{\pm}}$
which form a $B$-eigenbasis for $W_{j}$, except for the special
cases $\lambda_{j}=\sqrt{K}\pm\sqrt{k}$ (in which $F_{j}^{+}=F_{j}^{-}$
and $B\big|_{W_{j}}$ is non-diagonalizable). It turns out
however
that it is better to fix $\mu=\mu_{\lambda_{j}}^{+}$ and $F=F_{j}^{+}$,
and define 
\[
\mathscr{B}=\left\{ b_1=F\big|_{\overrightarrow{LR}}\,,\quad b_2=\overleftrightarrow{F}\big|_{\overrightarrow{LR}}\,,\quad b_3=F\big|_{\overleftarrow{LR}}\,,\quad b_{4}=\overleftrightarrow{F}\big|_{\overleftarrow{LR}}\right\} ,
\]
where $\overleftrightarrow{\square}$ inverts edges, namely, $\overleftrightarrow{F}(\ell\rightarrow r)=F\left(\ell\leftarrow r\right)$
and vice-versa (in particular, $\overleftrightarrow{f_{\ell o}}=f_{\ell i}$ and $\overleftrightarrow{f_{ro}}=f_{ri}$). It is immediate from Proposition \ref{prop:F-B-ef}
that $Bb_1=\mu b_3$ and $Bb_3=\mu b_1$; computations of the same spirit as in its proof show that $Bb_2=(K-1)b_3+\tfrac{K}{\mu}b_{4}$ and $Bb_{4}=(k-1)b_1+\tfrac{k}{\mu}b_2$,
so that $\left[B\big|_{W_{j}}\right]_{\mathscr{B}}=\left(\begin{smallmatrix}0 & 0 & \mu & k-1\\
0 & 0 & 0 & k/\mu\\
\mu & K-1 & 0 & 0\\
0 & K/\mu & 0 & 0
\end{smallmatrix}\right)$. So far, we did not use the assumption $|\mu|=\sqrt[4]{Kk}$. Using Lemma \ref{lem:inner-product} and \eqref{eq:mu-by-lambda} we have:
\begin{align*}
\left\Vert b_{1}\right\Vert ^{2} & =\left\Vert F_{j}^{+}\big|_{\overrightarrow{LR}}\right\Vert ^{2}=\left\Vert \lambda f_{\ell o}-(\mu^{2}+K)f_{ri}\right\Vert ^{2}\\
 & =\lambda^{2}\left\Vert f_{\ell o}\right\Vert ^{2}-2\Re\left(\lambda(\mu^{2}+K)\left\langle f_{\ell o},f_{ri}\right\rangle \right)+\left|\mu^{2}+K\right|^{2}\left\Vert f_{ri}\right\Vert ^{2}\\
 & =\lambda^{2}\cdot\tfrac{K+1}{2}-2\Re\left(\lambda(\mu^{2}+K)\tfrac{\lambda}{2}\right)+\left|\mu^{2}+K\right|^{2}\cdot\tfrac{k+1}{2}\\
 & =\tfrac{1}{2}\big(\mu-\tfrac{1}{\mu}\big)\big(\mu+\tfrac{K}{\mu}\big)\left((K+\overline{\mu}^{2}-1)k-\mu^{2}\right)\\
 & = \tfrac{1}{2}(\mu-\tfrac1{\mu})(\mu+\tfrac{k}{\mu})(\mu+\tfrac{K}{\mu})(\tfrac{Kk}{\mu}-\mu),
\end{align*}
where the last equality is the first time we have used $|\mu|=\sqrt[4]{Kk}$. By its definition we have $\left\Vert b_{4}\right\Vert=\left\Vert b_{1}\right\Vert$, and computing similarly gives $b_i\bot b_j$ for $i\neq j$, and $\left\Vert b_{2}\right\Vert^2=\left\Vert b_{3}\right\Vert^2=\sqrt{{K}/{k}}\left\Vert b_{1}\right\Vert^2$. Thus, $\mathscr{B}'=\left\{ \sqrt[4]{K}b_1,\sqrt[4]{k}b_2,\sqrt[4]{k}b_3,\sqrt[4]{K}b_{4}\right\} $
is an orthonormal basis up to scaling, and $\left[B\big|_{W_{j}}\right]_{\mathscr{B}'}$
is the matrix in \eqref{eq:B_W_Ram_case}, as $\mu^{-}=-\overline{\mu}$.\\
\textbf{(2b)} For $\lambda_{j}\notin\left[\sqrt{K}-\sqrt{k},\sqrt{K}+\sqrt{k}\right]$
the eigenvalues $\pm \mu^\pm$ are distinct, and $W_{j}=\Span\left\{ F_{j}^{\pm},\widetilde{F_{j}^{\pm}}\right\} $. From Proposition
\ref{prop:F-B-ef} we see that $B\cdot(F_{j}^{\pm}|_{\overrightarrow{LR}})=\mu^{\pm}F_{j}^{\pm}|_{\overleftarrow{LR}}$,
so that 
\[
\left[B|_{W_{j}}\right]_{\mathscr{B}}=\left(\begin{smallmatrix}0 & 0 & \mu_{j}^{+} & 0\\
0 & 0 & 0 & \mu_{j}^{-}\\
\mu_{j}^{+} & 0 & 0 & 0\\
0 & \mu_{j}^{-} & 0 & 0
\end{smallmatrix}\right),\qquad\text{for}\qquad\mathscr{B}=\left\{ {b_1=F_{j}^{+}|_{\overrightarrow{LR}},b_2=F_{j}^{-}|_{\overrightarrow{LR}},\atop b_3=F_{j}^{+}|_{\overleftarrow{LR}},b_{4}=F_{j}^{-}|_{\overleftarrow{LR}}\phantom{,}}\right\} .
\]
The basis $\mathscr{B}$ is not orthonormal, but $\left\{ b_1,b_2\right\} \bot\left\{ b_3,b_{4}\right\} $
as they are supported on disjoint sets of edges. Gram-Schmidt process then gives a change-of-basis matrix
$P=\left(\begin{smallmatrix}x & * & 0 & 0\\
 & y & 0 & 0\\
 &  & z & *\\
 &  &  & w
\end{smallmatrix}\right)$ which transforms $\mathscr{B}$ to an orthonormal basis $\mathscr{B}'$,
and 
\begin{equation}
\left[B|_{W_{j}}\right]_{\mathscr{B}'}=P\left[B|_{W_{j}}\right]_{\mathscr{B}}P^{-1}=\left(\begin{smallmatrix}0 & 0 & \frac{x}{z}\mu_{j}^{+} & *\\
0 & 0 & 0 & \frac{y}{w}\mu_{j}^{-}\\
\frac{z}{x}\mu_{j}^{+} & * & 0 & 0\\
0 & \frac{w}{y}\mu_{j}^{-} & 0 & 0
\end{smallmatrix}\right)\label{eq:B_gen_B'}
\end{equation}
has the desired form. Finally, since $\mathscr{B}'$ is orthonormal,
every entry of $\left[B|_{W_{j}}\right]_{\mathscr{B}'}$ is bounded
by $\left\Vert B|_{W_{j}}\right\Vert _2\leq\left\Vert B\right\Vert _2=K$,
and in addition $\left|\mu_{j}^{\pm}\right|\geq1$ by (\ref{eq:mu-range}),
so that $|\frac{x}{z}|,|\frac{z}{x}|,|\frac{y}{w}|,|\frac{w}{y}|$
and the two $*$ in (\ref{eq:B_gen_B'}) are bounded by $K$.
\item Let $g=g_{j}\in\ker A|_{L}$. As $\left\Vert g|_{L}\right\Vert =\left\Vert g\right\Vert =1$ and $\mu^{-}=i\sqrt{K}$, taking 
\[
\mathscr{B}=\left\{ b_1=G_{j}^{-}|_{\overrightarrow{LR}}=\tfrac{(K-k)}{\sqrt{K}}ig_{\ell o},b_2=G_{j}^{-}|_{\overleftarrow{LR}}=(K-k)g_{\ell i}\right\} 
\]
we have $\left[B|_{W_{g}}\right]_{\mathscr{B}}=\left(\begin{smallmatrix} & i\sqrt{K}\\
i\sqrt{K}
\end{smallmatrix}\right)$ by Proposition \ref{prop:F-B-ef}, and
\[
\left\Vert b_1\right\Vert =(K-k)\sqrt{1+1/K},\qquad b_1\bot b_2,\qquad\left\Vert b_2\right\Vert =(K-k)\sqrt{K+1}
\]
by Lemma \ref{lem:inner-product}(2). Thus, $\left[B|_{W_{g}}\right]_{\mathscr{B}'}=\left(\begin{smallmatrix} & i\\
iK
\end{smallmatrix}\right)$ for the orthonormal basis $\mathscr{B}'=\{b_1/\left\Vert b_1\right\Vert ,b_2/\left\Vert b_2\right\Vert \}$.
\item Similarly, for $h=h_{j}\in\ker A|_{R}$ Proposition \ref{prop:F-B-ef}
and $\mu^{+}=i\sqrt{k}$ give $\left[B|_{W_{h}}\right]_{\mathscr{B}}=\left(\begin{smallmatrix} & i\sqrt{k}\\
i\sqrt{k}
\end{smallmatrix}\right)$ where 
\[
\mathscr{B}=\left\{ b_1=H_{j}^{+}|_{\overrightarrow{LR}}=\left(K-k\right)g_{ri},b_2=H_{j}^{+}|_{\overleftarrow{LR}}=\tfrac{k-K}{\sqrt{k}}ig_{ro}\right\} ,
\]
and $\left\Vert h|_{R}\right\Vert =1$ together with (\ref{eq:riri})
imply 
\[
\left\Vert b_1\right\Vert =(K-k)\sqrt{k+1},\qquad b_1\bot b_2,\qquad\left\Vert b_2\right\Vert =(K-k)\sqrt{1+1/k},
\]
so that $\left[B|_{W_{h}}\right]_{\mathscr{B}'}=\left(\begin{smallmatrix} & ik\\
i
\end{smallmatrix}\right)$ for the normalized basis $\mathscr{B}'=\{b_1/\left\Vert b_1\right\Vert ,b_2/\left\Vert b_2\right\Vert \}$.
\item Combining (1)-(4) we obtain a block-diagonal form for $B|_{L_{V}^2(E)}$,
and we now also see that its dimension is 
\[
\dim L_{V}^2\left(E\right)=2+4\left(n-\mE-1\right)+2\mE+2\mathcal{N}=2\left(\left|V\right|-1\right).
\]
Next, we choose a maximal spanning tree in $X$, and let $\Gamma$
be the set of closed cycles with one edge outside of it, so that $\left|\Gamma\right|=\frac{|E|}2-|V|+1=\chi\left(X\right)$.
Denoting $P=\Span\left\{ p_{\gamma}\,\middle|\,\gamma\in\Gamma\right\} $
and $N=\Span\left\{ n_{\gamma}\,\middle|\,\gamma\in\Gamma\right\} $,
we have $\dim P=\dim N=\left|\Gamma\right|$. In addition, $P$, $N$
and $L_{V}^2(E)$ are mutually orthogonal by Proposition \ref{prop:cycles},
so that $P\oplus N=L_{V}^2\left(E\right)^{\bot}$ by dimension considerations.
Proposition \ref{prop:cycles} also shows that $B|_p=I$ and $B|_{N}=-I$,
so that any orthonormal bases for $P$ and $N$ give together an orthonormal
basis for $L_{V}^2\left(E\right)^{\bot}$ in which $B|_{L_{V}^2\left(E\right)^{\bot}}=\mathrm{diag}\left(I_{\left|\Gamma\right|},-I_{\left|\Gamma\right|}\right)$. 
\end{enumerate}
\end{proof}
We remark that for non-regular graphs sometimes one is interested
in the Markov operator $\mathcal{M}_{X}=D^{-1}A$ or the Laplace operator
$\mL_{X}=D-A$, where $D$ is the degree operator $\left(Df\right)(v)=\deg\left(v\right)f\left(v\right)$.
For biregular graphs however, the spectrum of $A$ determines that
of $\mathcal{M}_{X}$ and $\mL_{X}$. We finish the section
with a nice exercise which shows in particular that the Laplace spectrum
does distinguish between adj-Ramanujan and Ramanujan bigraphs.
\begin{xca}
\begin{enumerate}
\item The Laplace spectrum of $X$ is 
\begin{align*}
\Spec\mL_{X}=\left\{ 0\right\} &\cup\left\{ \tfrac{K+k+2\pm\sqrt{(K-k)^2+4\lambda_{j}^2}}2\,\middle|\,2\leq j\leq n-\mE\right\} 
\\&\cup\left\{ k+1\right\} ^{\mathcal{N}}\cup\left\{ K+1\right\} ^{\mE}\cup\left\{ K+k\right\} .
\end{align*}
\item $X$ is adj-Ramanujan if and only if
\begin{align*}
\Spec\mL_{X}\subseteq\{0\}&\cup\left[\tfrac{K+k+2-\sqrt{(K-k)^2+4(\sqrt{K}+\sqrt{k})^2}}2,\tfrac{K+k+2-\sqrt{(K-k)^2+4(\sqrt{K}-\sqrt{k})^2}}2\right]\\
&\cup\left\{ k+1\right\} \cup\left\{ K+1\right\} \\&\cup\left[\tfrac{K+k+2+\sqrt{(K-k)^2+4(\sqrt{K}-\sqrt{k})^2}}2,\tfrac{K+k+2+\sqrt{(K-k)^2+4(\sqrt{K}+\sqrt{k})^2}}2\right]
\\&\cup\left\{ K+k\right\} ,
\end{align*}
and it is Ramanujan if and only if in addition $K+1\notin\Spec\mL_{X}$.
\end{enumerate}
\end{xca}

\section{Combinatorics} \label{sec:Combinatorics}

We continue with $X=(L\sqcup R,E)$ being a connected $(K\!+\!1,k\!+\!1)$-bigraph with $|L|=n$,  $\mE=\mE_X = \dim\ker\left(A_X|_{L}\right)$, and $\mathfrak{pf}=\lambda_1>\lambda_2\geq\ldots\geq\lambda_{n-\mE_X}$ the positive eigenvalues of $A_X$.

\subsection{Pseudorandomness and biexpansion} \label{subsec:Clash-counting}

Let $E\left(S,T\right)$ denote the set of edges connecting two
sets of vertices $S$ and $T$ in a graph. The famous \emph{Expander
Mixing Lemma }is a simple yet powerful tool which bounds the deviation
of $\left|E\left(S,T\right)\right|$ from its pseudorandom expectation
(see \cite{HLW06}). In our case it takes this form:
\begin{thm}[EML]
If $X=\left(L\sqcup R,E\right)$ is a $(K\!+\!1,k\!+\!1)$-bigraph
with $\lambda_2\leq\varepsilon$, then 
\begin{equation}
\left|\left|E\left(S,T\right)\right|-\tfrac{k+1}{|L|}\left|S\right|\left|T\right|\right|\leq\varepsilon\sqrt{\left|S\right|\big(1-\tfrac{|S|}{|L|}\big)\left|T\right|\big(1-\tfrac{|T|}{|R|}\big)}\label{eq:EML}
\end{equation}
for any $S\subseteq L$ and $T\subseteq R$.
\end{thm}

\begin{proof}[Sketch of Proof]
Denoting by $\one_{S}$ the characteristic function of $S$, expand
$\left|E\left(S,T\right)\right|=\left\langle A\one_{S},\one_{T}\right\rangle $
in an eigenbasis for $A$, and isolate the contributions of the $\pm\mathfrak{pf}$-eigenvectors.
\end{proof}
For Ramanujan bigraphs we have $\varepsilon=\sqrt{K}\!+\!\sqrt{k}$,
and the EML is useful when $K\approx k$, but becomes disappointing
when $K\gg k$ and $\sqrt{K}+\sqrt{k}$ approaches the trivial eigenvalue
$\mathfrak{pf}=\sqrt{(K+1)(k+1)}$. However, the EML only takes advantage
of the fact that the nontrivial spectrum lies in $[{-\sqrt{K}\!-\!\sqrt{k}},{\sqrt{K}\!+\!\sqrt{k}}]$,
and not of its concentration in two narrower sub-strips (together
with $0$). We suggest here a different way to think of pseudorandomness
in bigraphs, of a similar nature, but which takes into account the
difference between Ramanujan and weakly-/adj-Ramanujan bigraphs. We
first define a notion of biexpander, which is a weakened version of
the (full) Ramanujan property:
\begin{defn}\label{def:biexp}
A $(K\!+\!1,k\!+\!1)$-bigraph $X$ is an
\emph{$\varepsilon$-biexpander} if $\varepsilon<\sqrt{K}$, $\mE_{X}=0$ and 
\[
\Spec\left(A_{X}\right)\subseteq\left\{ 0\right\} \cup\pm\left[\sqrt{K}-\varepsilon,\sqrt{K}+\varepsilon\right]\cup\left\{ \pm\sqrt{(K+1)(k+1)}\right\} .
\]
\end{defn}

Note that by \eqref{eq:aram-def} and \eqref{eq:NB-adj-ram}, $X$ is NB-Ramanujan if and only if it is a $\sqrt{k}$-biexpander. Ramanujan bigraphs are in fact optimal biexpanders, in the sense that there is no infinite family of $\varepsilon$-biexpanding $(K\!+\!1,k\!+\!1)$-bigraphs with $\varepsilon<\sqrt{k}$ (this follows from the generalized Alon-Boppana theorem \cite{greenberg1995spectrum,grigorchuk1999asymptotic}). We remark that \cite{brito2018spectral} shows that random bigraphs are almost-Ramanujan (in the strong sense), namely, for any $\varepsilon>0$, a random bigraph is a $\sqrt{k}+\varepsilon$-biexpander with probability approaching one as the graph size grows to infinity.

Thinking of $A$ as mapping every element of $L$ to $K\negmedspace+\!1$
elements of $R$, we study the number of \emph{clashes} arising from
two subsets $S,T\subseteq L$, namely, pairs of distinct edges $e,e'$
which leave $S$ and $T$ respectively, and end in the same vertex:
\[
Cl\left(S,T\right)=\left\{ \left(e,e'\right)\in E^2\,\middle|\,\begin{matrix}e\in S\times R,\ e'\in T\times R\\
e\neq e',\left|e\cap e'\cap R\right|=1
\end{matrix}\right\} .
\]

\begin{thm}[Clash Counting Lemma]
\label{thm:clash}If $X$ is a $(K\!+\!1,k\!+\!1)$-regular $\varepsilon$-biexpander
with $\left|L\right|=n$, then
\begin{align*}
\left|\vphantom{\frac{a}{b}}\left|Cl\left(S,T\right)\right|-\left(\tfrac{Kk+k+1}{n}\left|S\right|\left|T\right|-\left|S\cap T\right|\right)\right| \\ \leq3\sqrt{K}\varepsilon\sqrt{\left|S\right|\big(1-\tfrac{|S|}{n}\big)\left|T\right|\big(1-\tfrac{|T|}{n}\big)},
\end{align*}
 and
\begin{align*}
\left|\vphantom{\frac{a}{b}}\left|Cl\left(S,T\right)\right|-\left(\tfrac{Kk+k+1-\varepsilon^{2}}{n}\left|S\right|\left|T\right|+(\varepsilon^{2}-1)\left|S\cap T\right|\right)\right| \\ \leq2\sqrt{K}\varepsilon\sqrt{\left|S\right|\big(1-\tfrac{|S|}{n}\big)\left|T\right|\big(1-\tfrac{|T|}{n}\big)}
\end{align*}
for any sets $S,T\subseteq L$.
\end{thm}

\begin{proof}
As each $\left(e,e'\right)\in Cl(S,T)$ constitutes a non-backtracking
path of length two from $S$ to $T$, we have 
\begin{align*}
\left|Cl\left(S,T\right)\right|&=\left\langle \left(A^2-(K+1)\right)\one_{S},\one_{T}\right\rangle 
\\&=\left\langle \left(A^2-K-\varepsilon^2\right)\one_{S},\one_{T}\right\rangle +(\varepsilon^2-1)\left|S\cap T\right|.
\end{align*}
We denote $\mathfrak{P}=\Span\{\one_{L},\one_{R}\}=\Span\{f_1,\tilde{f}_1\}$
and $U=\Span\left\{ f_2,\tilde{f}_2\ldots,f_{n},\tilde{f}_{n}\right\} $,
obtaining an orthogonal $A$-stable decomposition $L^2(V)=\mathfrak{P}\oplus\ker A\oplus U$.
Decomposing accordingly $\one_{S}=\mathbb{P}_{\mathfrak{P}}(\one_{S})+\mathbb{P}_{\ker A}(\one_{S})+\mathbb{P}_{U}(\one_{S})$
we observe first that $\mE_{X}=0$ implies $\mathbb{P}_{\ker A}(\one_{S})=0$.
Since $\mathbb{P}_{\mathfrak{P}}(\one_{S})=\frac{|S|}{n}\one_{L}$
and $\Spec A|_{\mathfrak{P}}=\pm\mathfrak{pf}_{A}$ we have 
\begin{align*}
&\left\langle \left(A^2-K-\varepsilon^2\right)\mathbb{P}_{\mathfrak{P}}(\one_{S}),\one_{T}\right\rangle 
\\=&\left(Kk+k+1-\varepsilon^2\right)\left\langle \tfrac{|S|}{n}\one_{L},\one_{T}\right\rangle =\tfrac{Kk+k+1-\varepsilon^2}{n}\left|S\right|\left|T\right|,
\end{align*}
and since $\Spec A|_{U}\subseteq[\sqrt{K}-\varepsilon,\sqrt{K}+\varepsilon]$
and $\left\Vert \mathbb{P}_{U}\one_{S}\right\Vert =\sqrt{\left|S\right|-\Vert\mathbb{P}_{\mathfrak{P}}\one_{S}\Vert^2}$
we have 
\begin{align*}
&\left|\left|Cl\left(S,T\right)\right|-\left(\tfrac{Kk+k+1-\varepsilon^2}{n}\left|S\right|\left|T\right|+(\varepsilon^2-1)\left|S\cap T\right|\right)\right| 
\\
 =&\left|\left\langle \left(A^2-K-\varepsilon^2\right)\mathbb{P}_{U}(\one_{S}),\one_{T}\right\rangle \right|\\
  \leq&\left\Vert \left(A^2-K-\varepsilon^2\right)\big|_{U}\right\Vert \left\Vert \mathbb{P}_{U}(\one_{S})\right\Vert \left\Vert \mathbb{P}_{U}(\one_{T})\right\Vert \\
  \leq&2\sqrt{K}\varepsilon\sqrt{|S|(1-\tfrac{|S|}{n})}\sqrt{|T|(1-\tfrac{|T|}{n})}.
\end{align*}
This is the second bound in the Theorem. To obtain the first one, repeat the same proof replacing $\varepsilon^2$ with zero throughout. Using $\left\Vert \left(A^2-K\right)\big|_{U}\right\Vert \leq(2\sqrt{K}\varepsilon+\varepsilon^2)$
in the last step, we obtain
\begin{align*}
&\left|\left|Cl\left(S,T\right)\right|-\big(\tfrac{Kk+k+1}{n}\left|S\right|\left|T\right|-\left|S\cap T\right|\big)\right| 
\\\leq&(2\sqrt{K}\varepsilon+\varepsilon^2)\sqrt{\left|S\right|\big(1-\tfrac{|S|}{n}\big)\left|T\right|\big(1-\tfrac{|T|}{n}\big)}.\\
  \leq&3\sqrt{K}\varepsilon\sqrt{\left|S\right|\big(1-\tfrac{|S|}{n}\big)\left|T\right|\big(1-\tfrac{|T|}{n}\big)}.\qedhere
\end{align*}
\end{proof}

\begin{rem}
(1) In the Ramanujan case the Clash Counting Lemma reads
\begin{eqnarray}
\left|\vphantom{\frac{a}{b}}\left|Cl\left(S,T\right)\right|-\left(\tfrac{Kk+1}{n}\left|S\right|\left|T\right|+(k-1)\left|S\cap T\right|\right)\right|
\nonumber\\ 
\leq2\sqrt{Kk}\sqrt{\left|S\right|\big(1-\tfrac{|S|}{n}\big)\left|T\right|\big(1-\tfrac{|T|}{n}\big)}.\label{eq:clash-ram}
\end{eqnarray}
For an adj-Ramanujan graph, we encounter the additional term 
\[
\left\langle \left(A^2-K-\varepsilon^2\right)\mathbb{P}_{\ker A}(\one_{S}),\one_{T}\right\rangle =(K+k)\left\langle \mathbb{P}_{\ker A}(\one_{S}),\mathbb{P}_{\ker A}(\one_{T})\right\rangle ,
\]
which (without additional knowledge on either \textbf{$\one_{S}$}
or $\one_{T}$) would enlarge the error term $2\sqrt{Kk}$ in (\ref{eq:clash-ram})
to $K+k$. 
(2) A converse for the Expander Mixing Lemma was proved in \cite{BiluLinial2006},
showing that satisfying (\ref{eq:EML}) is in fact equivalent to
being an $\varepsilon$-expander, up to a logarithmic factor. It is
shown in \cite{Morovits2023DirectedExpanderGraphs} that the same
is true for clash counting and biexpansion:
\begin{thm}[\cite{Morovits2023DirectedExpanderGraphs}]\label{thm:convclash}
If a $(K\!+\!1,k\!+\!1)$-regular bigraph $X=(L\sqcup R,E)$ satisfies
\[
\left|\left|Cl\left(S,T\right)\right|-\big(\tfrac{Kk+k+1}{|L|}\left|S\right|\left|T\right|-\left|S\cap T\right|\big)\right|\leq\sqrt{K}\alpha\sqrt{|S||T|}
\]
for any $S,T\subseteq L$, then $X$ is a $O\left(\alpha\left(1+\log\frac{Kk}{\alpha}\right)\right)$-biexpander.
\end{thm}
\end{rem}

\subsection{Sparsification}\label{subsec:sparse}

Another perspective on expansion is given by the notion of spectral
sparsification \cite{spielman2011spectral}. Roughly speaking, a graph
$X$ is called a \emph{sparsifier} of a graph $Y$ with the same set
of vertices if $\alpha A_{X}-A_{Y}$ is small in an appropriate sense,
where $\alpha>1$ is some scaling factor\footnote{We focus here on sparsification of the adjacency operator, as it relates
to pseudorandomess. Other works (e.g.\ \cite{spielman2011spectral})
address the associated Laplace or Markov operators, which relate to
cut sizes and the behavior of random walk, respectively.}. For example, if $X=\left(L\sqcup R,E\right)$ is a $k$-regular
bigraph on $2n$ vertices and $Y=\left(L\sqcup R,L\times R\right)$
is the complete bipartite graph on the same vertices, then it is easy
to see that $\left\Vert \frac{n}{k}A_{X}-A_{Y}\right\Vert \leq\varepsilon n$
if and only if $\lambda_2\left(A_{X}\right)\leq\varepsilon k$; thus, sparsification
of the complete graph is equivalent to being an expander. However,
in the $(K\!+\!1,k\!+\!1)$-biregular case a Ramanujan bigraph $X=(L\sqcup R,E)$ is \emph{not
}a good sparsifier for the complete bipartite graph $Y=(L\sqcup R,L\times R)$.
The reason is the kernel of $A$ on the left side: we require that
$A_{X}$ be injective on $L$, and assuming $K>k$ the other eigenvalues
are bounded away from zero, but for $Y$ the opposite is true, as the kernel of $A|_L$ is almost maximal: $\mE\left(Y\right)=|L|-1$. This raises the question: what
do biexpanders and Ramanujan bigraphs sparsify? A possible answer is given by finite geometry:
\begin{prop}\label{prop:incidence}
Let $k$ be a prime power, $d\geq3$, and $\mathscr{P}^{d,k}=\left(\mL\sqcup\mP,E\right)$
the incidence bigraph of lines ($\mL$) and planes ($\mP$)
in $\F_{k}^{d+1}$. Then $\mathscr{P}^{d,k}$ is a $\left(K\!+\!1,k\!+\!1\right)$-regular
zero-biexpander ($\varepsilon=0$) with $n=|\mL|=\frac{k^{d+1}-1}{k-1}$
and $K=\frac{n-k-1}{k}$.
\end{prop}

\begin{proof}
First, $K+1=\frac{n-1}{k}=\frac{k^{d}-1}{k-1}$ is indeed the number
of planes containing a fixed line. By the discussion at the beginning
of Section \ref{sec:Spectral-analysis}, it remains to show that $f\colon\mL\rightarrow\R$
with $\sum_{\ell\in\mL}f\left(\ell\right)=0$ satisfies $A^2f=Kf$.
Since every line is contained in $K+1$ planes, and every pair of
different lines is contained in a unique plane, indeed
\begin{align*}
\left(A^{2}f\right)\left(\ell\right) & =\sum_{\ell\subset p\in\mP}\left(Af\right)\left(p\right)=\sum_{\ell\subset p}\sum_{\ell'\subset p}f\left(\ell'\right)=\sum_{\ell'\in\mL}\left|\left\{ p\in\mP\,\middle|\,\ell,\ell'\in p\right\} \right|f\left(\ell'\right)\\
 & =(K+1)f(\ell)+\sum_{\ell'\neq\ell}f(\ell')=Kf\left(\ell\right).\qedhere
\end{align*}
\end{proof}
Note that by the miracle of geometry, for line-plane graphs the clash
counting problem is deterministic, as is the edge counting problem
between sets in a complete graph. The next claim shows that for fixed
$k$, the operator $A^2\big|_{L}$ on a biexpander sparsifies the
parallel operator on a line-plane graph with the same left side. We
cannot compare the adjacency operators on the entire graphs, since
they have right sides of different sizes.
\begin{prop}\label{prop:inc-sparse}
Let $X=\left(L\sqcup R,E\right)$ be a $(K\!+\!1,k\!+\!1)$-regular
$\varepsilon$-biexpander, where $k$ is a prime power. If $\left|L\right|=\frac{k^{d+1}-1}{k-1}$
for some $d\geq3$, then identifying $L$ with $\mL\subseteq\mathscr{P}^{d,k}$
(in any manner) gives
\[
\left\Vert \tfrac{n}{k(K+1)}A_{X}^2\big|_{L}-A_{\mathscr{P}^{d,k}}^2\big|_{\mL}\right\Vert \leq\tfrac{k+1}{k}+\tfrac{\left|\varepsilon^2-1\right|+2\sqrt{K}\varepsilon}{(K+1)k}n
\]
\end{prop}

\begin{proof}
The operator $T=\tfrac{n}{k(K+1)}A_{X}^2\big|_{L}-A_{\mathscr{P}^{d,k}}^2\big|_{L}$
acts on constant functions by $\frac{n}{k(K+1)}\mathfrak{pf}_{X}^2-\mathfrak{pf}_{\mathscr{P}^{d,k}}^2=\frac{k+1}{k}$,
and if $f$ is a nonconstant $\lambda$-eigenfunction of $A^2\big|_{L}$,
then $T$ acts on $f$ by $\tfrac{n\lambda}{k(K+1)}-\tfrac{n-k-1}{k}=\tfrac{k+1}{k}+\tfrac{\lambda-K-1}{k(K+1)}n$.
Since $X$ is an $\varepsilon$-biexpander we have $\left|\smash{\sqrt{\lambda}-\sqrt{K}}\right|\leq\varepsilon$,
and thus 
\[
\left|\lambda-K-1\right|\leq\left|\lambda-K-\varepsilon^2\right|+\left|\varepsilon^2-1\right|\leq2\sqrt{K}\varepsilon+\left|\varepsilon^2-1\right|.\qedhere
\]
\end{proof}

\subsection{Incidences in Ramanujan complexes}\label{subsec:biexp-example}

Certain incidence relations in Ramanujan complexes give rise to a
rich family of excellent biexpanders. Specifically, they give arbitrarily
large families of $(K\!+\!1,k\!+\!1)$-bigraphs which are ``Ramanujan up to a constant'', i.e.\ $O(\sqrt{k})$-biexpanders. 

Recall that we have defined a Ramanujan complex in Section \ref{subsec:sub:ram-def} as a
simplicial complex on which the nontrivial spectrum of each geometric operator is confined to that of the same operator on its universal cover. Here, we focus on Ramanujan complexes of type $\widetilde{A}_{d}$. These were first defined and constructed in
\cite{li2004ramanujan, Lubotzky2005a}, albeit using a weaker definition, which only considers specific geometric operators on vertices. We give here a concise description of these complexes, and refer to the cited papers and to \cite{Lubotzky2013, first2016ramanujan, Lubetzky2017RandomWalks} for more details.

For $d\geq2$ and a prime $p$, we now define the Bruhat-Tits building $\mathcal{B}=\mathcal{B}_{d,p}$
associated with $G=PGL_{d}(\mathbb{Q}_{p})$, and refer the reader to \cite{Brown1989, Garrett1997} for a detailed presentation. The building $\mathcal{B}$ is a $(d-1)$-dimensional simplicial clique
complex with the following underlying graph: the vertices correspond to $\mathbb{Q}_{p}^{\times}$-homothety
classes of $\mathbb{Z}_{p}$-lattices in $\mathbb{Q}_{p}^{d}$, and
two classes $C,C'$ are neighbors if $pL<L'<L$ for some representatives $L\in C,L'\in C'$.
In this case, we say the directed edge $C'\rightarrow C$ has
\emph{color} $\log_{p}([L:L'])$. The \emph{Hecke operators} $A_{i}$ ($1\leq i\leq d-1$)
are the corresponding colored adjacency operators on vertices:
\[
\left(A_{i}f\right)(v)=\sum_{\col(v\rightarrow w)=i}f(w),\qquad f\in L^{2}(\mathcal{B}^{0}).
\]
If $\Gamma$ is a torsion-free lattice in $PGL_{d}(\mathbb{Q}_{p})$, the quotient
$X=\Gamma\backslash\mathcal{B}$ is a finite complex with universal
cover $\mathcal{B}$. In particular, if $X$ is Ramanujan, then for
every $i$ we have $\mathrm{Spec}_{0}(A_{i}|_{X})\subseteq\mathrm{Spec}(A_{i}|_{\mathcal{B}})$,
and 
\begin{equation}
\mathrm{Spec}(A_{i}|_{\mathcal{B}})=\left\{ p^{\frac{i(d-i)}{2}}\sigma_{i}(z_{1},\ldots,z_{d})\,\middle|\,\forall j:|z_{j}|=1\text{ and }z_{1}\cdot\ldots\cdot z_{d}=1\right\} ,\label{eq:Ai_building}
\end{equation}
where $\sigma_{i}$ is the $i$-th elementary symmetric polynomial. The degree of $A_i$ is $\left[\begin{smallmatrix}d\\
i
\end{smallmatrix}\right]_{p}$ (the Gaussian binomial coefficients), and the trivial eigenvalues of $A_i$ are its degree multiplied by any $d$-root of unity. It is not hard to see that $\col(v\rightarrow w)=d-\col(w\rightarrow v)$, so that $A_i^*=A_{d-i}$; a nontrivial fact is that all Hecke operators $A_i$ commute with each other, and are in particular normal.

From this point onward we assume that $d\geq3$. For $X=\Gamma\backslash\mathcal{B}$ as above, we define a bigraph $G_{X}=\left(L_{X}\sqcup R_{X},E_{X}\right)$ as follows:

\begin{itemize}
\item $L_{X}$ are the vertices of $X$.
\item $R_{X}$ are the (nondirected) edges of color $\pm2$ in $X$.
\item For $v\in L_{X}$ and $\left\{u,w\right\}\in R_X$,
$\left\{v,\left\{u,w\right\}\right\}\in E_{X}$ if $\left\{ v,u,w\right\} $
is a triangle in $X$ with edge colors 
$\vcenter{\xymatrix@1@R=1.5pt@C=15pt{ & u\ar[ld]_{1}\ar[dd]^{2}\\v\ar[dr]_{1}\\ & w}}$. 
\end{itemize}
We denote by $m_{X}$ the \textit{partiteness} of $X$, which is $m_{X}=\left[\Gamma\colon\Gamma\cap G'\right]$ (where $G'=[G,G]=PSL_d(\mathbb{Q}_p)$).

\begin{prop}
The graph $G_{X}$ is a $(\left[d\right]_{p}\left[d-1\right]_{p},p+1)$-bigraph\footnote{Here $[n]_q$ denotes the $q$-number $[n]_{q}=\frac{q^{n}-1}{q-1}$.}
with $m_{X}$ connected components. Each component $C$ of $G_{X}$
has $\mathcal{E}_C=0$, and 
\[
\mathrm{Spec}\left(A_{C}\right)\subseteq\left\{ 0\right\} \cup\pm\left[\sqrt{K}-\sqrt{k},\sqrt{K}+\tfrac{d^{2}-1}{2}\sqrt{k}\right]\cup\left\{ \pm\mathfrak{pf}\right\} .
\]
\end{prop}

\begin{proof}
Consider the geometric operator $T$ on $\mathcal{B}$
which induces $A_{G_{X}}$, namely, $T$ acts on vertices and color-$\pm 2$ edges
in $\mathcal{B}$, with incidence given by 
$\vcenter{\xymatrix@1@R=1.5pt@C=15pt{ & u\ar[ld]_{1}\ar[dd]^{2}\\v\ar[dr]_{1}\\ & w}}$ triangles. The operator $T^{2}|_{\mathcal{B}^{0}}$ takes a vertex $v$ once
to any $v'\neq v$ for which 
$\vcenter{\xymatrix@1@R=1.5pt@C=15pt{ & u\ar[ld]_{1}\ar[rd]^{1}\ar[dd]^{2} & \\ v\ar[dr]_{1} & & v'\ar[dl]^{1} \\ & w}}$ for some $u,w$, and in addition it takes $v$ $\left[d\right]_{p}\left[d-1\right]_{p}$
times back to itself, as this is the number of $\vcenter{\xymatrix@1@R=1.5pt@C=15pt{ & \ast\ar[ld]_{1}\ar[dd]^{2}\\v\ar[dr]_{1}\\ & \ast}}$ triangles.

Next, observe that 
the operator $A_{d-1}A_{1}$ takes $v$ to any $v'\neq v$
such that $v\underset{1}{\rightarrow}w\underset{1}{\leftarrow}v'$,
and also $\deg A_{1}=\left[d\right]_{p}$ times back to itself.
A path $v\underset{1}{\rightarrow}w\underset{1}{\leftarrow}v'$
corresponds to $v,v'$ being sublattices of $w$ of index $p$ containing
$pw$, so there exists a unique $u$ which completes this path to
$\vcenter{\xymatrix@1@R=1.5pt@C=15pt{ & u\ar[ld]_{1}\ar[rd]^{1}\ar[dd]^{2} & \\ v\ar[dr]_{1} & & v'\ar[dl]^{1} \\ & w}}$, namely $u=v\cap v'$. Thus, $A_{d-1}A_{1}$ and $T^{2}|_{\mathcal{B}^{0}}$
agree, except for the number of backtracking:
\[
T^{2}|_{\mathcal{B}^{0}}=A_{d-1}A_{1}-\left[d\right]_{p}I+\left[d\right]_{p}\left[d-1\right]_{p}I=A_{d-1}A_{1}+\left[d\right]_{p}(\left[d-1\right]_{p}-1)I.
\]
This relation continues to hold on quotients of the building, such as
$X$. The partiteness number $m_{X}$
is the number of $1$-dimensional representations of $G$ which factor through $\Gamma\backslash G$, which account for the Perron-Frobenius eigenvalues of the connected components of $G_{X}$. When $X$ is Ramanujan, the remaining eigenvalues of $A_{d-1}A_{1}$ are bounded
by $d^{2}p^{d-1}$ due to \eqref{eq:Ai_building}, and below by $0$, as $A_{d-1}=A_{1}^{*}$.
Thus, any nontrivial eigenvalue $\lambda$ of $A_{G_{X}}^{2}|_{L_{X}}$
satisfies
\begin{align*}
\sqrt{\lambda}-\sqrt{K} & \leq\sqrt{d^{2}p^{d-1}+[d]_{p}([d-1]_{p}-1)}-\sqrt{[d]_{p}[d-1]_{p}-1}\leq\tfrac{d^{2}-1}{2}\sqrt{p}\\
\sqrt{K}-\sqrt{\lambda} & \leq\sqrt{[d]_{p}[d-1]_{p}-1}-\sqrt{[d]_{p}([d-1]_{p}-1)}\leq\sqrt{p}.\qedhere
\end{align*}
\end{proof}

\begin{rem}
\begin{enumerate}
\item It is necessary for $X$ to be a Ramanujan complex in order to get almost-Ramanujan bigraphs. For example, the non-Ramanujan quotients of $\mathcal{B}_{3,p}$ which we construct in Section \ref{subsec:complexes} have $\varepsilon\approx\left(\sqrt{2}-1\right)\sqrt{p^{3}}$-biexpansion rather
than $\varepsilon=O(\sqrt{p})$.
\item It is natural to look at the incidence bigraph of the vertices and walls (codimension one cells) in $X$, with incidence
given by chambers. By similar arguments one can show that these are still $\varepsilon$-biexpanders (of
degrees $\left(\left[d\right]!_{p},p+1\right)$), but not almost Ramanujan
-- in fact, they have $\varepsilon\approx\frac{d^{2}-1}{2}\sqrt{k^{1+{d-2 \choose 2}}}$,
so only for $d=3$ one has $\varepsilon=O(\sqrt{k})$.
\end{enumerate}
\end{rem}

\subsection{Cutoff phenomena} \label{subsec:Cutoff}

In this section we study the total variation mixing time on bigraphs,
and show that under various assumptions we obtain Diaconis' cutoff
phenomenon with varying window size. 
Except for Corollary \ref{cor:cutoff-SRW} which addresses the simple random walk (SRW) on vertices, we focus on the non-backtracking random walk (NBRW) on directed edges $E=E_{X}$, which moves from an edge $v\rightarrow w$ to a uniformly chosen random edge of the form $w\rightarrow u$ with $u\neq v$. 

We continue with the notations from Section \ref{sec:Spectral-analysis}.
In particular, $X=\left(L\sqcup R,E\right)$ is a $(K\!+\!1,k\!+\!1)$-biregular graph with $|L|=n$. Since the NBRW is 2-periodic (alternating between $\overrightarrow{LR}$ and $\overleftarrow{LR}$), it suffices to observe the walk at even (or at odd) times. We assume w.l.o.g.\ that the starting edge $e_0$ is in $\overrightarrow{LR}$, so that the NBRW has distribution $\mathbf{p}_{e_0}^{2t}=\left(\tfrac{B^2}{Kk}\right)^{t}\one_{e_0}$ at time $2t$, and (restricted to even times) it has sample space of size $N:=\left|\overrightarrow{LR}\right|$, and stationary distribution
\[
\mathbf{u}=\tfrac1{N}\one_{\overrightarrow{LR}}=\mathbb{P}_{\Span(\one_{E},\widetilde{\one_{E}})}(\one_{e_0}),
\]
where $\mathbb{P}$ denotes orthogonal projection. The $\varepsilon$\textit{-mixing time} of NBRW on $X$ is 
\[
t_{mix}\left(\varepsilon\right) = t_{mix}\left(\varepsilon,X\right) = \min\left\{ 2t\in \N\,\middle|\,\forall e_0\in \overrightarrow{LR},\ \Vert\mathbf{p}_{e_0}^{2t}-\mathbf{u}\Vert_{TV}<\varepsilon\right\} ,
\]
where $\left\Vert \cdot\right\Vert _{TV}$ is the total-variation
norm:
\begin{equation}
\left\Vert \mu-\nu\right\Vert _{TV}=\max_{A\subseteq\overrightarrow{LR}}\left|\mu\left(A\right)-\nu\left(A\right)\right|=\tfrac12\left\Vert \mu-\nu\right\Vert _1.\label{eq:TV}
\end{equation}
A family of graphs $\left\{ X_{n}\right\} $ is said to exhibit cutoff
if $\frac{t_{mix}\left(\varepsilon,X_{n}\right)}{t_{mix}\left(1-\varepsilon,X_{n}\right)}\overset{{\scriptscriptstyle n\rightarrow\infty}}{\longrightarrow}1$
for every $0<\varepsilon<1$. The cutoff is said to occur at time
$t\left(n\right)$, if for every $\varepsilon>0$ there exists a window
of size $w\left(n,\varepsilon\right)=o\left(t\left(n\right)\right)$,
such that $\left|t_{mix}\left(\varepsilon,X_{n}\right)-t(n)\right|\leq w(n,\varepsilon)$
for $n$ large enough. The average degree of NBRW is $\sqrt{Kk}$ (in the sense that two
consecutive steps take an edge to $Kk$ edges), and if $t\left(n\right)=\log_{\sqrt{Kk}}|X_{n}|$
we say that the cutoff is \emph{optimal}. By the next claim, a $\sqrt{Kk}$-regular
walk cannot mix in less steps regardless of the graph spectrum:
\begin{lem}
\label{lem:mix-lower}For any $\varepsilon>0$, 
\[
t_{mix}\left(1-\varepsilon\right)\geq\log_{\sqrt{Kk}}N-\log_{\sqrt{Kk}}(\tfrac1{\varepsilon}).
\]
\end{lem}

\begin{proof}
For $S=\mathrm{supp}\left(B^{2t}\one_{e}\right)$ we have 
\[
\left\Vert \mathbf{p}_{e}^{2t}-\mathbf{u}\right\Vert _{TV}\geq\left|\mathbf{p}_{e}^{2t}(E\backslash S)-\mathbf{u}(E\backslash S)\right|=\mathbf{u}(E\backslash S)=1-\frac{|S|}{N}\geq1-\frac{(Kk)^{t}}{N}
\]
so at $2t=\log_{\sqrt{Kk}}N-\log_{\sqrt{Kk}}(\frac1{\varepsilon})$
we have $\left\Vert \mathbf{p}_{e}^{2t}-\mathbf{u}\right\Vert _{TV}\geq1-\varepsilon$.
\end{proof}
For the upper bound on the mixing time we use the spectral analysis
of Section \ref{sec:Spectral-analysis}. Recall that $X$ has positive adjacency eigenvalues
$\mathfrak{pf}=\lambda_1>\ldots\geq\lambda_{n-\mE}$ (where
$\mE=\mE_{X}$). By Theorem \ref{thm:B-decomp}, for $2\leq j\leq n-\mE$ we have
(in the appropriate orthonormal basis) 
\[
B^2|_{W_{f_{j}}}=\left(\begin{smallmatrix}0 & 0 & \alpha_{j}\mu_{j}^{+} & \gamma_{j}\\
0 & 0 & 0 & \beta_{j}\mu_{j}^{-}\\
\mu_{j}^{+}/\alpha_{j} & \delta_{j} & 0 & 0\\
0 & \mu_{j}^{-}/\beta_{j} & 0 & 0
\end{smallmatrix}\right)^2=\left(\begin{smallmatrix}\mu_{j}^{+^2} & \eta_{j} & 0 & 0\\
0 & \mu_{j}^{-^2} & 0 & 0\\
0 & 0 & \mu_{j}^{+^2} & \eta_{j}'\\
0 & 0 & 0 & \mu_{j}^{-^2}
\end{smallmatrix}\right)
\]
for $\mu_{j}^{\pm}=\mu_{\lambda_{j}}^{\pm}$, and some $\eta_{j},\eta_{j}'\in\C$,
which satisfy $|\eta_{j}|,|\eta_{j}'|\leq\left\Vert B^2|_{\smash{W_{f_{j}}}}\right\Vert _2\leq\left\Vert B^2\right\Vert _2=Kk$.
Denoting $D_{j}=\left(\begin{smallmatrix}\mu_{j}^{+^2} & \eta_{j}\\
0 & \mu_{j}^{-^2}
\end{smallmatrix}\right)$ and $D_{j}'=\left(\begin{smallmatrix}\mu_{j}^{+^2} & \eta_{j}'\\
0 & \mu_{j}^{-^2}
\end{smallmatrix}\right)$, we see that $B^2$ is unitarily equivalent to a $2\times2$-block-diagonal
matrix: 
\begin{equation}
B^2\sim\mathrm{diag}(Kk^{(\times2)},D_2,D_2',\ldots,D_{n-\mE},D_{n-\mE}',\left(-K\right)^{\times2\mE},\left(-k\right)^{\times2\mathcal{N}},1^{\times2\chi}),\label{eq:B2}
\end{equation}
where $\chi=\chi\left(X\right)=\frac{|E|}2-|V|+1$. We introduce
a notation for the corresponding decomposition of $F\in L^2(E)$:
\begin{equation}
F=F^1+F^2+F^{'2}+\ldots+F^{n-\mE}+F^{'n-\mE}+F^{L}+F^{R}+F^{\chi},\label{eq:proj_decomp}
\end{equation}
so that
\begin{equation}
B^{2t}F=(Kk)^{t}F^1+\left[{\textstyle \sum_{j=2}^{n-\mE}}D_{j}^{t}F^{j}+D{}_{j}'^{t}F^{'j}\right]+(-K)^{t}F^{L}+(-k)^{t}F^{R}+F^{\chi}.\label{eq:B2t_act}
\end{equation}
We denote by $L_0^2(E):=\{\one_{E},\widetilde{\one_{E}}\}^{\bot}=\{\one_{\overrightarrow{LR}},\one_{\overleftarrow{LR}}\}^{\bot}$
the ``non-trivial'' part of $L^2(E)$, and observe that $\mathbf{u}=\one_{e_0}^1$. We also observe that for any $F$ supported on $\overrightarrow{LR}$ we have $F^{'j}=0$ (for $2\leq j\leq n-\mE$), as the bases $\mathscr{B}'$ for $W_{f_j}$ in Theorem \ref{thm:B-decomp} are comprised of two vectors supported on $\overrightarrow{LR}$, and two supported on $\overleftarrow{LR}$.

The next Lemma bounds the operator and the Frobenius norms of powers
of $D_{j}$, in the (adj-)Ramanujan case.
\begin{lem}
\label{lem:Dt-norm}Let $\vartheta\in\left[0,\pi\right]$, $\mu^{\pm}=\sqrt{e^{\pm i\vartheta}\sqrt{Kk}}$,
and $D=\left(\begin{smallmatrix}\mu^{_{+}2} & \eta\\
0 & \mu^{_{-}2}
\end{smallmatrix}\right)$ with $\ensuremath{|\eta|\leq Kk}$. Then for any $t\geq0$
\begin{equation}
\left\Vert D^{t}\right\Vert _2^2\leq\left\Vert D^{t}\right\Vert _{F}^2=(Kk)^{t}\left(2+\tfrac{|\eta|^2}{Kk}\left(\tfrac{\sin t\vartheta}{\sin\vartheta}\right)^2\right)\leq(Kk)^{t}\left(2+Kkt^2\right).\label{eq:D_norm_bounds}
\end{equation}
\end{lem}

\begin{proof}
For any $\vartheta\in\C$ we have
\begin{align*}
D^{t}& = \left(\begin{matrix}\mu^{_{+}2t} & \eta\sum_{m=0}^{t-1}\mu^{_{+}2m}\mu^{_{-}2(t-m-1)}\\
0 & \mu^{_{-}2t}
\end{matrix}\right)\\
& =(Kk)^{t/2}\left(\begin{matrix}e^{it\vartheta} & e^{-i(t-1)\vartheta}\frac{\eta}{\sqrt{Kk}}\sum_{m=0}^{t-1}e^{2im\vartheta}\\
0 & e^{-it\vartheta}
\end{matrix}\right)=(Kk)^{t/2}\left(\begin{matrix}e^{it\vartheta} & \frac{\eta}{\sqrt{Kk}}\frac{\sin t\vartheta}{\sin\vartheta}\\
0 & e^{-it\vartheta}
\end{matrix}\right),
\end{align*}
and for $\vartheta\in\R$ also $|e^{\pm it\vartheta}|=1$.
\end{proof}
Since $X$ is Ramanujan when $\mE=0$ and $\vartheta_j\in\R$ for all $j$, we obtain from \eqref{eq:B2} and Lemma \ref{lem:Dt-norm}:
\begin{cor}
If $X$ is Ramanujan then 
\begin{align}
\left\Vert B^{2t}|_{L_0^2(E)}\right\Vert _2 & \leq\max_{j}\left\Vert D_{j}^{t}|_{L_0^2(E)}\right\Vert _{F}\leq\sqrt{(Kk)^{t}\left(2+Kkt^2\right)}\leq t\sqrt{2(Kk)^{t+1}},\text{  and}\label{eq:B2t_2norm}\\
\left\Vert B^{2t}\right\Vert _{F}^2 & =2(Kk)^{2t}+2\mathcal{N}k^{2t}\\ \nonumber &+2\chi(X)+(Kk)^{t-1}\sum_{j=2}^{n-\mE}\Big(4Kk+(|\eta_{j}|^2+|\eta_{j}'|^2)\big(\tfrac{\sin t\vartheta_{j}}{\sin\vartheta_{j}}\big)^2\Big).\label{eq:B2tF-full}
\end{align}
\end{cor}

Note that if $X$ is only adj-Ramanujan, then we get the much worse $\left\Vert B^{2t}|_{\smash{L_0^2(E)}}\right\Vert _2=K^{t}$,
but the Frobenius norm $\left\Vert B^{2t}\right\Vert _{F}^2$ is
only increased by $2\mE K^{2t}$; this will be used when we replace the Ramanujan assumption by a density hypothesis.

\begin{thm}[Logarithmic NBRW cutoff]
\label{thm:cutoff-NBRW}Restricted to even times, NBRW on $(K\!+\!1,k\!+\!1)$-regular
Ramanujan bigraphs exhibits optimal cutoff (at time $\log_{\sqrt{Kk}}N$)
with window size bounded by $3\log_{\sqrt{Kk}}\log N$ (for any base
in the inner log).
\end{thm}

\begin{proof}
Since $B^2\mathbf{u}=Kk\mathbf{u}$, we have\textbf{ }
\begin{align*}
\left\Vert \mathbf{p}_{e}^{2t}-\mathbf{u}\right\Vert _{TV} & \leq\tfrac{\sqrt{N}}2\left\Vert \mathbf{p}_{e}^{2t}-\mathbf{u}\right\Vert _2
\\&=\tfrac{\sqrt{N}}2\left\Vert \left(\tfrac{B^2}{Kk}\right)^{t}(\one_{e}-\mathbf{u})\right\Vert _2=\tfrac{\sqrt{N}}2\left\Vert \left(\tfrac{B^2}{Kk}\right)^{t}\mathbb{P}_{L_0^2(E)}(\one_{e})\right\Vert _2\\
 & \leq\text{\ensuremath{\tfrac{\sqrt{N}}2}}\left\Vert \left(\tfrac{B^2}{Kk}\right)^{t}\big|_{L_0^2(E)}\right\Vert _2\leq\sqrt{\frac{Nt^2Kk}{2(Kk)^{t}}}
\end{align*}
by (\ref{eq:B2t_2norm}). For any $\varepsilon>0$, at $2t=\log_{\sqrt{Kk}}N+3\log_{\sqrt{Kk}}\log N$
we obtain 
\[
\left\Vert \mathbf{p}_{e}^{2t}-\mathbf{u}\right\Vert _{TV}\leq\sqrt{\frac{N\left(\log_{Kk}N+3\log_{Kk}\log N\right)^2Kk}{2N\left(\log N\right)^3}}<\varepsilon
\]
once $N$ is large enough, and the lower bound on the mixing time follows from Lemma \ref{lem:mix-lower}.
\end{proof}

\begin{cor}[SRW cutoff] \label{cor:cutoff-SRW}
Restricted to even times, SRW on $(K\!+\!1,k\!+\!1)$-regular Ramanujan bigraphs exhibits cutoff at time $\frac{(K+1)(k+1)}{Kk-1}\log_{\sqrt{Kk}}n$ with window of size $O\left(\sqrt{\log n}\right)$.
\end{cor}

Interestingly, even though we study SRW, we need full Ramanujan and
not only adj-Ramanujan for the proof to work. The proof itself is
the same as the deduction of SRW-cutoff from NBRW-cutoff on regular
graphs in \cite{lubetzky2016cutoff}, and from geodesic flow on quotients of buildings
in \cite{Chapman2019CutoffRamanujancomplexes}. Let us give a informal
sketch: Choose a covering map $\varphi\colon\T\twoheadrightarrow X$,
where $\T=\T_{K\!+\!1,k\!+\!1}$ is the $(K\!+\!1,k\!+\!1)$-regular
tree. For a vertex $v_0\in V_{X}$, choose $\xi\in\varphi^{-1}(v_0)$.
Any walk $\left(X_{t}\right)$ on $X$ starting at $v_0$ lifts
via $\varphi$ to a unique walk $\left(\mX_{t}\right)$ on
$\T$ starting at $\xi$. Denoting $\rho\left(t\right):=\mathrm{dist}_{\T}(\xi,\mX_{t})$
and $Y_{t}=\rho(t)-\rho(t-2)$, we observe that except for the times
when $\mX_{t-2}=\xi$ (which occur only finitely many times
with probability one), the variables $Y_{t}$ are i.i.d.\ with expectancy
\[
\E\left(Y_{t}\,\middle|\,\mX_{t-2}\neq\xi\right)=\frac{2Kk-2}{(K+1)(k+1)}.
\]
It follows that $\rho(2t)=\sum_{j=1}^{t}Y_{2j}$ obeys a central limit
theorem with expectancy $\frac{Kk-1}{(K+1)(k+1)}\cdot t$, so at time
$t_0=\frac{(K+1)(k+1)}{Kk-1}\log_{\sqrt{Kk}}N$ the lifted walk
$\mX_{t_0}$ is roughly located at distance $\log_{\sqrt{Kk}}N$
from $\xi$. For simplicity we shall pretend that $\rho(t_0)=\log_{\sqrt{Kk}}N$,
and we can then deduce that $\mX_{t_0}$ distributes evenly
on the $\log_{\sqrt{Kk}}N$-sphere around $\xi$ in $\T$,
by symmetry considerations. Thus, $X_{t_0}=\varphi\left(\mX_{t_0}\right)$
distributes precisely as does NBRW starting from $v_0$ at time
$\log_{\sqrt{Kk}}N$, for which we have cutoff by Theorem \ref{thm:cutoff-NBRW}.

In \cite{Nestoridi2021Boundedcutoffwindow}, it is shown that for
regular Ramanujan graphs an additional assumption on the girth of the
graphs reduces the cutoff window to a bounded size. We next show
that the same holds for Ramanujan bigraphs. Our proof is
different and gives slightly stronger results, as we analyze the non-normal operator $B$ acting on edges,
whereas \cite{Nestoridi2021Boundedcutoffwindow} use Chebycheff polynomials
to relate the spectrum of $A=A_{X}$ to a self-adjoint non-backtracking
walk on vertices. With the necessary changes, our approach can be applied to the regular case as well.

\begin{defn}
\label{def:log-girth}A family of $(K\!+\!1,k\!+\!1)$-bigraphs has\emph{
$m$-logarithmic girth} (for $m\geq2$) if its members satisfy $\mathrm{girth}(X)\geq\frac2{m-1}\log_{\sqrt{Kk}}N_{X}$
for $N$ large enough. 

For example, the Ramanujan graphs $X_{\EE}^{p,q}$ constructed from
the Eisenstein lattice in Section \ref{sec:applications} have $6$-logarithmic
girth (Theorem \ref{thm:main-Eis}).
\end{defn}

\begin{thm}[Bounded cutoff for logarithmic girth]
\label{thm:bounded-cutoff}A family of $(K\!+\!1,k\!+\!1)$-regular
Ramanujan bigraphs with $m$-logarithmic girth exhibits cutoff
at time $\log_{\sqrt{Kk}}N$ and window size bounded by $2\log_{\sqrt{Kk}}\left(\tfrac1{\varepsilon}\right)+\log_{\sqrt{Kk}}7m^2$.\footnote{For technical convenience, we limit the allowed times to multiples
of $m$.}
\end{thm}

\begin{proof}
Since $e\in\overrightarrow{LR}$, $\one_{e}^{'j}=0$ for all $j$,
and write $\one_{e}^{j}=\left(\begin{smallmatrix}\alpha_{j}\\
\beta_{j}
\end{smallmatrix}\right)$ in the basis which corresponds to the matrix representation of $D_{j}$.
Since we assume the graph is Ramanujan, we also have $\one_{e}^{L}=0$. We take
\[
2t=\log_{\sqrt{Kk}}N+2\log_{\sqrt{Kk}}\left(\tfrac1{\varepsilon}\right)+\log_{\sqrt{Kk}}7m^2,
\]
(rounded up to a multiple of $m$) and let $\tau=\frac{t}{m}$. For
large enough $N$ we have 
\[
2\tau\leq\frac{\log_{\sqrt{Kk}}N+2\log_{\sqrt{Kk}}\left(\tfrac1{\varepsilon}\right)+\log_{\sqrt{Kk}}7m^2}{m}\leq\frac{\log_{\sqrt{Kk}}N}{m-1}\leq\frac{\mathrm{girth\,}X}2,
\]
which implies that $2\tau$ steps of NBRW take any edge to $(Kk)^{\tau}$
different edges. Thus, (\ref{eq:B2t_act}) gives 
\begin{equation}
\begin{alignedat}1(Kk)^{\tau} & =\left\Vert B^{2\tau}\one_{e}\right\Vert _2^2=\left\Vert (Kk)^{\tau}\one_{e}^1+\left[{\textstyle \sum_{j=2}^{n-\mE}}D_{j}^{\tau}\one_{e}^{j}\right]+(-k)^{\tau}\one_{e}^{R}+\one_{e}^{\chi}\right\Vert _2^2\\
 & \geq{\textstyle \sum_{j=2}^{n-\mE}}\left\Vert D_{j}^{\tau}\one_{e}^{j}\right\Vert ^2+k^{2\tau}\left\Vert \one_{e}^{R}\right\Vert ^2+\left\Vert \one_{e}^{\chi}\right\Vert ^2.
\end{alignedat}
\label{eq:girth_to_2norm}
\end{equation}
Let 
\begin{equation}
\boldsymbol{J}=\left\{ 2\leq j\leq n-\mE\,:\,\left|{\textstyle \frac{\eta_{j}}{\sqrt{Kk}}\frac{\sin\tau\vartheta_{j}}{\sin\vartheta_{j}}}\beta_{j}\right|\leq c\left|\alpha_{j}\right|\right\} ,\label{eq:in_J}
\end{equation}
where $c>2$ is to be determined later. If $j\notin\boldsymbol{J}$,
then 
\begin{align}
(Kk)^{-\tau}\left\Vert D_{j}^{\tau}\one_{e}^{j}\right\Vert ^2 & =\left\Vert \left(\begin{smallmatrix}e^{i\tau\vartheta_{j}} & \frac{\eta_{j}}{\sqrt{Kk}}\frac{\sin\tau\vartheta_{j}}{\sin\vartheta_{j}}\\
 & e^{-i\tau\vartheta_{j}}
\end{smallmatrix}\right)\left(\begin{smallmatrix}\alpha_{j}\\
\beta_{j}
\end{smallmatrix}\right)\right\Vert ^2
\nonumber \\&=\left|e^{i\tau\vartheta_{j}}\alpha_{j}+{\textstyle \frac{\eta_{j}}{\sqrt{Kk}}\frac{\sin\tau\vartheta_{j}}{\sin\vartheta_{j}}}\beta_{j}\right|^2+\left|\beta_{j}\right|^2\nonumber \\
 & \geq\left|\left|\alpha_{j}\right|-\left|{\textstyle \frac{\eta_{j}}{\sqrt{Kk}}\frac{\sin\tau\vartheta_{j}}{\sin\vartheta_{j}}}\beta_{j}\right|\right|^2+\left|\beta_{j}\right|^2\nonumber \\
 & =\left|\alpha_{j}\right|^2+\left|{\textstyle \frac{\eta_{j}}{\sqrt{Kk}}\frac{\sin\tau\vartheta_{j}}{\sin\vartheta_{j}}}\beta_{j}\right|^2-2\left|\alpha_{j}\right|\left|{\textstyle \frac{\eta_{j}}{\sqrt{Kk}}\frac{\sin\tau\vartheta_{j}}{\sin\vartheta_{j}}}\beta_{j}\right|+\left|\beta_{j}\right|^2\nonumber \\
(\text{as }j\notin\boldsymbol{J})\qquad & \geq\left|\alpha_{j}\right|^2+\left|{\textstyle \frac{\eta_{j}}{\sqrt{Kk}}\frac{\sin\tau\vartheta_{j}}{\sin\vartheta_{j}}}\beta_{j}\right|^2-\tfrac2{c}\left|{\textstyle \frac{\eta_{j}}{\sqrt{Kk}}\frac{\sin\tau\vartheta_{j}}{\sin\vartheta_{j}}}\beta_{j}\right|^2+\left|\beta_{j}\right|^2\nonumber \\
 & =\left|\alpha_{j}\right|^2+\left|\beta_{j}\right|^2+\tfrac{c-2}{c}\left|{\textstyle \frac{\eta_{j}}{\sqrt{Kk}}\frac{\sin\tau\vartheta_{j}}{\sin\vartheta_{j}}}\right|^2\left|\beta_{j}\right|^2.\label{eq:not_in_J}
\end{align}
Moving back to time $t=\tau m$, we have
\begin{align*}
(Kk)^{-t}{\textstyle \sum_{j=2}^{n-\mE}}\left\Vert D_{j}^{t}\one_{e}^{j}\right\Vert ^2 & ={\textstyle \sum_{j=2}^{n-\mE}}\left\Vert \left(\begin{smallmatrix}e^{it\vartheta_{j}} & \frac{\eta_{j}}{\sqrt{Kk}}\frac{\sin t\vartheta_{j}}{\sin\vartheta_{j}}\\
 & e^{-it\vartheta_{j}}
\end{smallmatrix}\right)\left(\begin{smallmatrix}\alpha_{j}\\
\beta_{j}
\end{smallmatrix}\right)\right\Vert ^2\\
 & ={\textstyle \sum_{j=2}^{n-\mE}}\left|e^{it\vartheta_{j}}\alpha_{j}+{\textstyle \frac{\eta_{j}}{\sqrt{Kk}}\frac{\sin\tau m\vartheta_{j}}{\sin\vartheta_{j}}}\beta_{j}\right|^2+\left|\beta_{j}\right|^2\\
 & \leq{\textstyle \sum_{j=2}^{n-\mE}}2\left|\alpha_{j}\right|^2+2\left|{\textstyle \frac{\eta_{j}}{\sqrt{Kk}}\frac{\sin\tau m\vartheta_{j}}{\sin\vartheta_{j}}}\right|^2\left|\beta_{j}\right|^2+\left|\beta_{j}\right|^2.
\end{align*}
Using the bound $\left|\frac{\sin\tau m\vartheta}{\sin\vartheta}\right|=\left|\frac{\sin\tau m\vartheta}{\sin\tau\vartheta}\right|\left|\frac{\sin\tau\vartheta}{\sin\vartheta}\right|\leq m\left|\frac{\sin\tau\vartheta}{\sin\vartheta}\right|$
we obtain 
\begin{align}
(Kk)^{-t}{\textstyle \sum_{j=2}^{n-\mE}}\left\Vert D_{j}^{t}\one_{e}^{j}\right\Vert ^2 & \leq{\textstyle \sum_{j=2}^{n-\mE}}2\left|\alpha_{j}\right|^2+\left|\beta_{j}\right|^2+2m^2\left|{\textstyle \frac{\eta_{j}}{\sqrt{Kk}}\frac{\sin\tau\vartheta_{j}}{\sin\vartheta_{j}}}\right|^2\left|\beta_{j}\right|^2\label{eq:bound_on_Dt}\\
\left(\text{using }\eqref{eq:in_J},\eqref{eq:not_in_J}\right)\quad & \leq\sum_{j\notin\boldsymbol{J}}\tfrac{2cm^2}{(c-2)(Kk)^{\tau}}\left\Vert D_{j}^{\tau}\one_{e}^{j}\right\Vert ^2+\sum_{j\in\boldsymbol{J}}2\left|\alpha_{j}\right|^2+\left|\beta_{j}\right|^2+2m^2\left(c\left|\alpha_{j}\right|\right)^2\nonumber \\
 & \leq\tfrac{2cm^2}{(c-2)(Kk)^{\tau}}\left[{\textstyle \sum_{j=2}^{n-\mE}}\left\Vert D_{j}^{\tau}\one_{e}^{j}\right\Vert ^2\right]+(2c^2+2)m^2{\textstyle \sum_{j=2}^{n-\mE}}\left\Vert \one_{e}^{j}\right\Vert ^2\nonumber \\
 & \leq\tfrac{2cm^2}{(c-2)(Kk)^{\tau}}\left[{\textstyle \sum_{j=2}^{n-\mE}}\left\Vert D_{j}^{\tau}\one_{e}^{j}\right\Vert ^2\right]+(2c^2+2)m^2.\nonumber 
\end{align}
Combining everything we obtain
\begin{align*}
\left\Vert \mathbf{p}_{e}^{2t}-\mathbf{u}\right\Vert _2^2 & =\left\Vert \left(\tfrac{B^2}{Kk}\right)^{t}\mathbb{P}_{L_0^2(E)}(\one_{e})\right\Vert ^2
\\&=\tfrac1{(Kk)^{2t}}{\textstyle \sum_{j=2}^{n-\mE}}\left\Vert D_{j}^{t}\one_{e}^{j}\right\Vert ^2+\tfrac1{K^{2t}}\left\Vert \one_{e}^{R}\right\Vert ^2+\tfrac1{(Kk)^{2t}}\left\Vert \one_{e}^{\chi}\right\Vert ^2\\
\left(\text{using }\eqref{eq:bound_on_Dt}\right)\quad & \leq\tfrac{2cm^2}{(c-2)(Kk)^{t+\tau}}\left[{\textstyle \sum_{j=2}^{n-\mE}}\left\Vert D_{j}^{\tau}\one_{e}^{j}\right\Vert ^2\right]+\tfrac{(2c^2+2)m^2}{(Kk)^{t}}+\tfrac1{K^{2t}}\left\Vert \one_{e}^{R}\right\Vert ^2+\tfrac1{(Kk)^{2t}}\left\Vert \one_{e}^{\chi}\right\Vert ^2\\
 & \leq\tfrac{2cm^2}{(c-2)(Kk)^{t+\tau}}\left[{\textstyle \sum_{j=2}^{n-\mE}}\left\Vert D_{j}^{\tau}\one_{e}^{j}\right\Vert ^2+k^{2\tau}\left\Vert \one_{e}^{R}\right\Vert ^2+\left\Vert \one_{e}^{\chi}\right\Vert ^2\right]+\tfrac{(2c^2+2)m^2}{(Kk)^{t}}\\
\left(\text{using }\eqref{eq:girth_to_2norm}\right)\quad & \leq\tfrac{\left(\frac{2c}{c-2}+2c^2+2\right)m^2}{(Kk)^{t}}.
\end{align*}
Taking $c=3$ we obtain $\left\Vert \mathbf{p}_{e}^{2t}-\mathbf{u}\right\Vert _2^2\leq\frac{26m^2}{(Kk)^{t}}$,
hence at $2t=\log_{\sqrt{Kk}}N+2\log_{\sqrt{Kk}}\left(\tfrac1{\varepsilon}\right)+\log_{\sqrt{Kk}}7m^2$
we obtain
\[
\left\Vert \mathbf{p}_{e}^{2t}-\mathbf{u}\right\Vert _{TV}\leq\tfrac{\sqrt{N}}2\left\Vert \mathbf{p}_{e}^{2t}-\mathbf{u}\right\Vert _2\leq\sqrt{\frac{N\cdot26m^2}{4(Kk)^{t}}}\leq\sqrt{\frac{N\cdot26m^2}{4N\varepsilon^{-2}\cdot7m^2}}\leq\varepsilon.
\]
As before, the lower bound follows from Lemma \ref{lem:mix-lower}.
\end{proof}
In the next theorem, we establish cutoff for adj-Ramanujan bigraphs
which satisfy Sarnak's Density Hypothesis. This is precisely the case
for principal arithmetic quotients of the Bruhat-Tits tree of $U_3$,
by Theorem \ref{thm:SXDH}. In addition, we shall assume that our
bigraphs are \emph{left-transitive}, namely, that their automorphism
groups act transitively on their left (smaller) side -- this is the
case for Cayley bigraphs as defined in Section \ref{sec:bicayley},
and in particular, for the Cayley quotients of the simply-transitive
lattices constructed in Section \ref{sec:lattices}. We remark that one could also remove
the left-transitivity assumption, and instead prove cutoff for $\Theta\left(N\right)$
many starting vertices, as done in \cite{Golubev2022CutoffgraphsSarnakXue}.

\begin{thm}[Cutoff under density hypothesis]
\label{thm:cutoff-density}Let $\FF$ be a family of left-transitive
adj-Ramanujan $(K\!+\!1,k\!+\!1)$-bigraphs, which satisfy the density
hypothesis $\mE_X<N_{X}^{\delta}$ with $\delta=\frac2{1+\log_{k}(K)}$.\footnote{For inert $U_3$ quotients, we have $\delta=\frac2{1+\log_q(q^3)}=\frac12$.}
Then $\FF$ exhibits cutoff at time $\log_{\sqrt{Kk}}N$
and window size bounded by
\begin{enumerate}
\item $\left(\tfrac2{1-\delta}\right)\log_{\sqrt{Kk}}\log N$ in general,
and
\item $\frac1{\delta}\log_{\sqrt{Kk}}\frac{K}{\varepsilon^2}$ if $\FF$
has $m$-logarithmic girth, and $\varepsilon\leq m^{\frac{\delta}{\delta-1}}\sqrt{K}$.
\end{enumerate}
\end{thm}

\begin{proof}
Let $X\in\FF$ and $\Gamma=\mathrm{Aut}\,X$. We restrict
our attention to $L^2(\overrightarrow{LR})$ (so henceforth $B^2$
stands for $B^2\big|_{L^2(\overrightarrow{LR})}$). We denote
$L_0^2(\overrightarrow{LR})=\one^{\bot}$, and\textbf{ $\mathbf{J}$}
the all-one matrix.\textbf{ }As $\Gamma$ acts by automorphisms and
$\left|\Gamma e_0\right|\geq\frac{N}{K+1}$, we have
\begin{align*}
\left\Vert \mathbf{p}_{e_0}^{2t}-\mathbf{u}\right\Vert _2^2 & =\frac1{\left|\Gamma e_0\right|}\sum_{e\in\Gamma e_0}\left\Vert \mathbf{p}_{e}^{2t}-\mathbf{u}\right\Vert _2^2\\&\leq\frac{K+1}{N}\sum_{e\in\Gamma e_0}\left\Vert \mathbf{p}_{e}^{2t}-\mathbf{u}\right\Vert _2^2\leq\frac{K+1}{N}\sum_{e\in\overrightarrow{LR}}\left\Vert \mathbf{p}_{e}^{2t}-\mathbf{u}\right\Vert _2^2\\
 & =\frac{K+1}{N(Kk)^{2t}}\sum_{e\in\overrightarrow{LR}}\left\Vert B^{2t}\one_{e}-(Kk)^{t}\mathbf{u}\right\Vert _2^2
 \\&=\frac{K+1}{N(Kk)^{2t}}\left\Vert B^{2t}-\tfrac{(Kk)^{t}}{N}\mathbf{J}\right\Vert _{F}^2
\end{align*}
Furthermore, $B^{2t}-\tfrac{(Kk)^{t}}{N}\mathbf{J}$ and $\left(B^{2t}\right)\big|_{\one^{\bot}}$
have the same non-zero singular values, hence the same Frobenius norm,
so that
\begin{align} \label{eq:density-L2-bound}
\left\Vert \mathbf{p}_{e_0}^{2t}-\mathbf{u}\right\Vert _2^2 & \leq\frac{K+1}{N(Kk)^{2t}}\left\Vert B^{2t}\big|_{\one^{\bot}}\right\Vert _{F}^2
\\ \nonumber &=\frac{K+1}{N(Kk)^{2t}}\left[\mathcal{N}k^{2t}+\mE K^{2t}+\chi(X)+\sum_{j=2}^{n-\mE}\left\Vert D_{j}^{t}\right\Vert _{F}^2\right]\\
\left(\text{using }\eqref{eq:D_norm_bounds}\right) \; & \leq\frac{K+1}{N(Kk)^{2t}}\left[\mathcal{N}k^{2t}+\mE K^{2t}+\chi(X)+(n-\mE-1)(Kk)^{t}\left(2+Kkt^2\right)\right]\nonumber \\
 & \leq\frac{K+1}{(Kk)^{t}}\left[\frac{\mE}{N}\left(\frac{K}{k}\right)^{t}+2Kkt^2\right]\nonumber 
\end{align}
(where we have used $\mathcal{N}+\chi(X)+(n-\mE-1)=Kn\leq N$).
The density hypothesis give $\tfrac{\mE}{N}<N^{\frac{1-\log_{k}K}{1+\log_{k}K}}$,
and taking $t_0=\log_{Kk}N+\left(\tfrac2{1-\delta}\right)\log_{Kk}\log N$
(with any basis $b>1$) we obtain 
\[
\frac{\mE}{N}\left(\frac{K}{k}\right)^{t_0}\leq N^{\frac{1-\log_{k}K}{1+\log_{k}K}}\left(\frac{K}{k}\right)^{t_0}=\left(\log N\right)^2.
\]
Thus, 
\[
\left\Vert \mathbf{p}_{e_0}^{2t_0}-\mathbf{u}\right\Vert _2^2\leq\frac{K+1}{(Kk)^{t_0}}\left(\log^2N+2Kkt_0^2\right)\leq\frac{4K^2k\left(\log^2N+t_0^2\right)}{(Kk)^{t_0}},
\]
which implies cutoff at time $2t_0=\log_{\sqrt{Kk}}N+\left(\tfrac2{1-\delta}\right)\log_{\sqrt{Kk}}\log N$,
since
\begin{align*}
\left\Vert \mathbf{p}_{e}^{2t_0}-\mathbf{u}\right\Vert _{TV} & \leq\tfrac{\sqrt{N}}2\left\Vert \mathbf{p}_{e}^{2t_0}-\mathbf{u}\right\Vert _2=\tfrac{\sqrt{N}}2\sqrt{\frac{4K^2k\left(\log^2N+t_0^2\right)}{(Kk)^{t_0}}}\\
 & =\sqrt{\frac{K^2k\left(\log^2N+\left(\log_{Kk}N+\frac2{1-\delta}\log_{Kk}\log N\right)^2\right)}{(\log N)^{\frac2{1-\delta}}}},
\end{align*}
and the latter goes to zero as $N\rightarrow\infty$ since $\frac2{1-\delta}>2$.
This concludes (1), and we now add the assumption that $\FF$
has $m$-logarithmic girth. Taking $t=\log_{Kk}N+\frac1{\delta}\log_{Kk}\frac{K}{\varepsilon^2}$
and $\tau=\frac{t}{m}$, we have 
\[
2\tau\leq\frac{\log_{\sqrt{Kk}}N+\frac1{\delta}\log_{\sqrt{Kk}}\frac{K}{\varepsilon^2}}{m}\leq\frac{\log_{\sqrt{Kk}}N}{m-1}\leq\frac{\mathrm{girth\,X}}2
\]
for large enough $N$. This means that $\tau$ steps of $B^2$ take
any edge in $\overrightarrow{LR}$ to $(Kk)^{\tau}$ distinct edges,
hence every column of $B^{2\tau}$ (w.r.t.\ the standard basis of
$L^2(\overrightarrow{LR})$) has $(Kk)^{\tau}$ ones and zeros elsewhere,
so that $\left\Vert B^{2\tau}\right\Vert _{F}^2=N(Kk)^{\tau}.$
From the spectral analysis of $B^2$ (restricted to $\overrightarrow{LR}$),
we have 
\[
N(Kk)^{\tau}=\left\Vert B^{2\tau}\right\Vert _{F}^2=(Kk)^{2\tau}+\mE K^{2\tau}+\mathcal{N}k^{2\tau}+\chi(X)+\sum\nolimits _{j=2}^{n-\mE}\left\Vert D_{j}^{\tau}\right\Vert _{F}^2.
\]
We turn to Lemma \ref{lem:Dt-norm} to observe that for each $D=D_{j}$
\begin{align*}
\left\Vert D^{t}\right\Vert _{F}^2&=(Kk)^{t}\left(2+\tfrac{|\eta|^2}{Kk}\left(\tfrac{\sin t\vartheta}{\sin\vartheta}\right)^2\right)\\&\leq(Kk)^{t}m^2\left(2+\tfrac{|\eta|^2}{Kk}\left(\tfrac{\sin\tau\vartheta}{\sin\vartheta}\right)^2\right)=(Kk)^{t-\tau}m^2\left\Vert D^{\tau}\right\Vert _{F}^2.
\end{align*}
Thus, returning to \eqref{eq:density-L2-bound} we obtain
\begin{align*}
\left\Vert \mathbf{p}_{e_0}^{2t}-\mathbf{u}\right\Vert _2^2&\leq\frac{K+1}{N(Kk)^{2t}}\left[\mE K^{2t}+\mathcal{N}k^{2t}+\chi(X)+\sum\nolimits _{j=2}^{n-\mE}\left\Vert D_{j}^{t}\right\Vert _{F}^2\right]\\
&\leq\frac{(K+1)}{N(Kk)^{2t}}\left[\mE K^{2t}+\mathcal{N}k^{2t}+\chi(X)+(Kk)^{t-\tau}m^2\sum\nolimits _{j=2}^{n-\mE}\left\Vert D_{j}^{\tau}\right\Vert _{F}^2\right]
\\
&=\frac{(K+1)}{N(Kk)^{2t}}\left(\mE K^{2t}+\mathcal{N}k^{2t}+\chi(X)\right>
\\ & \left. +(Kk)^{t-\tau}m^2\left(N(Kk)^{\tau}-(Kk)^{2\tau}-\mE K^{2\tau}-\mathcal{N}k^{2\tau}-\chi(X)\right)\right)\\
&=\frac{(K+1)}{N(Kk)^{2t}}\left( \mE\left(K^{2t}-m^2K^{t+\tau}k^{t-\tau}\right)+\mathcal{N}\left(k^{2t}-m^2K^{t-\tau}k^{t+\tau}\right)\right.
\\& \left. +\chi(X)\left(1-m^2(Kk)^{t-\tau}\right)+m^2N(Kk)^{t}-m^2(Kk)^{t+\tau}\right)\\
&\leq\frac{(K+1)}{(Kk)^{t}}\left[\frac{\mE}{N}\left(\frac{K}{k}\right)^{t}+m^2\right].
\end{align*}
\textcolor{lightgray}{}For $t=\log_{Kk}N+\frac1{\delta}\log_{Kk}\frac{K}{\varepsilon^2}$
we have $\frac{\mE}{N}\left(\frac{K}{k}\right)^{t}<\left(\frac{K}{\varepsilon^2}\right)^{\frac{1-\delta}{\delta}}$
(by $\mE<N^{\delta}$), and thus 
\[
\left\Vert \mathbf{p}_{e}^{2t}-\mathbf{u}\right\Vert _{TV}\leq\tfrac{\sqrt{N}}2\left\Vert \mathbf{p}_{e}^{2t}-\mathbf{u}\right\Vert _2<\tfrac{\sqrt{N}}2\sqrt{\frac{(K+1)\left(\left(K/\varepsilon^2\right)^{\frac{1-\delta}{\delta}}+m^2\right)}{(Kk)^{t}}}.
\]
The assumption $\varepsilon\leq m^{\frac{\delta}{\delta-1}}\sqrt{K}$
gives $\left(K/\varepsilon^2\right)^{\frac{1-\delta}{\delta}}\geq m^2$,
so that 
\[
\left\Vert \mathbf{p}_{e}^{2t}-\mathbf{u}\right\Vert _{TV}\leq\tfrac{\sqrt{N}}2\sqrt{\frac{2(K+1)\left(K/\varepsilon^2\right)^{\frac{1-\delta}{\delta}}}{(Kk)^{t}}}=\sqrt{\frac{(K+1)\left(K/\varepsilon^2\right)^{\frac{1-\delta}{\delta}}}{2\left(K/\varepsilon^2\right)^{1/\delta}}}<\varepsilon.
\]
\end{proof}

\subsection{Geodesic prime number theorem} \label{subsec:Geodesic-prime-number}

We now delve a bit into the prehistory of Ramanujan graphs. For a
lattice $\Gamma\leq SL_2\left(\R\right)$, Selberg introduced
a zeta function $\zeta_{\Gamma}\left(u\right)$ which counts primitive
geodesic cycles in the Riemann surface $\mathbb{H}/\Gamma$, and proved
some properties analogous to Riemann's zeta function \cite{Selberg1956Harmonicanalysisdiscontinuous}.
The analogue of the Riemann hypothesis in these settings is Selberg's
famous $\nicefrac1{4}$-conjecture. Selberg's trace formula expresses
$\zeta_{\Gamma}\left(u\right)$ in terms of the algebraic structure
of $\Gamma$, which prompted Ihara to suggest a $p$-adic analogue
for lattices $\Gamma$ in $SL_2\left(\Q_p\right)$ \cite{ihara1966discrete}.
It was pointed out by Serre that Ihara's zeta has a parallel geometric
interpretation -- it counts non-backtracking cycles in a finite graph,
which in modern language is the quotient of the Bruhat-Tits tree of
$SL_2\left(\Q_p\right)$ by $\Gamma$ \cite{serre1980trees}.
This means the Ihara zeta function can be defined for any graph,
and Sunada observed that the Riemann Hypothesis for a regular graph
is equivalent to being a Ramanujan graph \cite{sunada1986functions}.
Hashimoto \cite{hashimoto1989zeta} has introduced the non-backtracking
perspective, which uses $B$ rather than $A$, and makes the story
much simpler. For more details we refer the reader to \cite{Terras2011Zetafunctionsgraphs}. 

Let us denote by $N_{m}=N_{m}(X)$ the number of cyclically-non-backtracking
(CNB) cycles in a graph $X$. A CNB cycle is called \emph{primitive
}if it is not a proper power of a shorter one, and two CNB cycles which
differ by their starting point are considered equivalent. A \emph{prime} in $X$ is an equivalence class of primitive CNB cycles, and Ihara's
zeta function is the function $\zeta_{X}\left(u\right)=\prod_{\left[\gamma\right]}\frac1{1-u^{\mathrm{length}\,\gamma}}$,
where the product is over all primes in $X$. We have $N_m=\mathrm{tr}(B^m)$ by the definition of $B$, and every CNB cycle can be uniquely expressed as a power of a primitive one; it is a standard exercise in generating functions (see e.g.\ \cite{kotani2000zeta}) to see that these two observations yield
\[
\zeta_{X}\left(u\right)=\exp\left(\sum_{m=1}^{\infty}\frac{N_{m}}{m}u^{m}\right)=\prod_{\mu\in\Spec B}\frac1{1-\mu u}=\frac1{\det\left(I-uB\right)}.
\]
For a $(K\!+\!1,k\!+\!1)$-bigraph, using the change of variable $u=\sqrt{Kk}^{-s}$
(where $s\in\C$) we note that $\zeta_{X}$ has poles at $s\in\left\{ \log_{\sqrt{Kk}}\mu\,\middle|\,\mu\in\Spec B\right\} $.
 In particular, as was observed in \cite{hashimoto1989zeta}, $\zeta_{X}$ has no poles with real part in $\left(\tfrac12,1\right)$ (which is the right analogue of the Riemann Hypothesis in this case),
if and only if $X$ is NB-Ramanujan. As in the number theoretic setting, the zeta function can be related to prime counting:
\begin{thm}[PNT for bigraphs]\label{thm:PNT}
If $\pi(m)$ is the number of primes of length $m$ in an adj-Ramanujan
$(K\!+\!1,k\!+\!1)$-bigraph $X$ with $N$ edges, then
\[
\left|\pi(2m)-\frac{(Kk)^{m}}{m}-\mE_X\frac{\left(-K\right)^{m}}{2m}\right|\leq2N(Kk)^{m/2}.
\]
\end{thm}

\begin{proof}
As every CNB-cycle is a power of a unique primitive one we have $N_{2m}=\sum_{d\mid2m}d\pi(d)$,
so that Möbius inversion gives 
\begin{equation}
\left|\pi(2m)-\frac{N_{2m}}{2m}\right|=\frac1{2m}\left|\sum\nolimits _{2m\neq d\mid2m}\boldsymbol{\mu}\left(\tfrac{2m}{d}\right)N_{d}\right|\leq N\sqrt{Kk}^{m},\label{eq:pi-N}
\end{equation}
where $\boldsymbol{\mu}$ is the Möbius function and the r.h.s.\ is a crude bound
for the number of cycles of length at most $m$. By the definition of
$B$ we have $N_{2m}=\mathrm{tr}(B^{2m})$, and the spectral analysis
of $B$ (Theorem \ref{thm:B-decomp}) yields 
\begin{equation}
N_{2m}=\sum\nolimits _{\mu\in\Spec\left(B\right)}\mu^{2m}=2(Kk)^{m}+\mE_X(-K)^{m}+\sum\nolimits _{{\mu\in\Spec B\atop \left|\mu\right|\leq\sqrt[4]{Kk}}}\mu^{2m}.\label{N-spec}
\end{equation}
Combining (\ref{eq:pi-N}) and (\ref{N-spec}) yields the theorem:
\begin{align*}
&\left|\pi(2m)-\frac{(Kk)^{m}}{m}-\mE_X\frac{\left(-K\right)^{m}}{2m}\right|\\
\leq & N\sqrt{Kk}^{m}+\tfrac1{2m}\sum\nolimits _{{\mu\in\Spec B\atop \left|\mu\right|\leq\sqrt[4]{Kk}}}\left|\mu\right|^{2m}\leq2N(Kk)^{m/2}.\qedhere
\end{align*}
\end{proof}
We remark that the non-Ramanujan graphs which we construct in Theorem \ref{thm:main-Eis} have $K=q^3,k=q$
for some prime $q$. As they are still adj-Ramanujan, their zeta
function satisfy the R.H.\ with the exception of $\mE_X$ additional  poles at 
\[
s=\frac3{4}\pm\frac{\pi}{4\ln q}i.
\]

Going back to prime counting, it is interesting to observe that for
adj-Ramanujan graphs which are not Ramanujan, the deviation of the
number of primes from the ``Riemann hypothesis'' prediction ($\frac{(Kk)^{m}}{m}$)
fluctuates between over and under-counting, according to $2m\Mod4$.
\section{Simply-transitive Lattices in Unitary Groups} \label{sec:lattices}

In this section we begin our arithmetic construction of explicit Ramanujan bigraphs. 
In \ref{subsec:unitarygroups} we define unitary groups schemes, and describe the Bruhat-Tits tree of $p$-adic unitary groups.
In \ref{subsec:lattices-main} we construct arithmetic lattices and state our main Theorem \ref{thm-simply-transitive}, which claims that these arithmetic lattices act simply-transitively on the left hand side of the Bruhat-Tits trees.
In \ref{subsec:lattices-parahorics}, \ref{subsec:lattices-transitive} and \ref{subsec:lattices-simply} we prove Theorem \ref{thm-simply-transitive}. 
Finally in \ref{subsec:lattices-strong-approx} we collect some auxiliary results which will be needed in later Sections (e.g. strong approximation).

\subsection{Unitary groups in three variables} \label{subsec:unitarygroups}

Let $E/F$ be a separable quadratic extension of global fields, and $g\mapsto g^*$ the associated involution on $GL_3(E)$, i.e.\ $(g^{*})_{i,j}=\tau(g_{j,i})$ for $\left\langle \tau\right\rangle =\mathrm{Gal}(E/F)$. Classical unitary groups (which give rise to lattices of type (I) on page \pageref{type-I-unitary}) arise from a choice of an Hermitian form $\Phi\in GL_3(E)$ (i.e.\ $\Phi^{*}=\Phi$). For $\mO_{F}$ the ring of integers of $F$, the unitary group scheme over $\mO_{F}$ associated with $E,\Phi$ is defined by
\begin{equation} \label{eq:U_de}
U_3\left(E,\Phi\right)(R)=\left\{ g\in GL_3\left(\mO_{E}\otimes_{\mO_{F}}R\right)\,\middle|\,g^{*}\Phi g=\Phi\right\} 
\end{equation}
for any $\mO_{F}$-algebra $R$ (note that one can always assume by scaling that $\Phi$ has coefficients in $\mO_E$). For example, for $E=\Q[\sqrt{-3}]$, $F=\Q$ and $\Phi=I$, we obtain for any commutative ring $R$:
\[
U_3\left(\Q\left[\sqrt{-3}\right],I\right)\left(R\right) = \left\{ g\in GL_3\left(R\left[\tfrac{1+\sqrt{-3}}2\right]\right)\,\middle|\,g^{*}g=I\right\}.
\]
The quotient of $ U_3(E,\Phi)$ by its center is called the \emph{projective unitary group scheme}, and denoted by $ PU_3(E,\Phi)$.


For any place $v$ of $F$, and either $G=U_3(E,\Phi)$ or $G=PU_3(E,\Phi)$, we denote by $G_v = G(F_v)$ the group of $F_v$-rational points of $G$. 
When $v$ is a finite place which does not split in $E$, $E_v = \mO_E \otimes F_v$ is a quadratic extension of the local field $F_v$, and $G_v$ acts on a Bruhat-Tits tree. Let us give now a brief description of this tree, for the case that $E_v/F_v$ is unramified (the ramified case is slightly different - see e.g.\ \cite[\S3]{Evra2018RamanujancomplexesGolden}).

Let $\varpi$ be a uniformizer in $F_v$ (and in $E_v$) and $q=|\mO_{F_v}/\varpi|$ the size of the residue field of $F_v$ (so that $|\mO_{E_v}/\varpi|=q^2$). 
Let $\wtB$ be the two-dimensional Bruhat-Tits building of $\wtG = PGL_3\left(E_v\right)$, whose vertices correspond to cosets $\wtG/\wtK$, where $\wtK = PGL_3\left(\mO_{E_v}\right)$ (see e.g.\ \cite[\S V(8)]{Brown1989}). 
If $v_0\in \wtB$ is the vertex with stabilizer $\wtK$, there is a simplicial involution $\#\colon\wtB\rightarrow\wtB$ defined on vertices by $\left(gv_0\right)^{\#}=\Phi^{-1}\left(g^{*}\right)^{-1}v_0$ (for all  $g\in \wtG$).
By \cite[\S2.6.1]{Tits1979Reductivegroupsover}, the Bruhat-Tits tree $\B$ of $G_v$ is a $\left(q^3+1,q+1\right)$-regular tree, which can be identified with the fixed-section of the involution $\#$:
the $\B$-vertices of degree $q^3+1$ are the $\wtB$-vertices fixed by $\#$, those of degree $q+1$ are midpoints of $\wtB$-edges flipped by $\#$, and the $\B$-edges are medians of $\tB$-triangles reflected by $\#$. 
The hyperspecial vertices $\B^{hs}\subseteq\B^0$ are precisely those which were vertices in $\wtB$, namely those of degree $q^3+1$.

By Bruhat-Tits theory (see \cite[\S2]{Tits1979Reductivegroupsover}), $G_v$ acts simplicially on $\B$ and transitively on its maximal faces (the undirected edges). 
Since $\B$ is a biregular tree it follows that $G_v$ acts transitively on the hyperspecial vertices $\B^{hs}\subseteq\B^0$, as well as on the non-hyperspecial ones.
For any $x\in \B$, there is an associated subgroup $P_x \leq G_v$, called a \textit{parahoric} subgroup, which is precisely the stabilizer of $x$ in $G_v$ by the work of \cite{Haines2008Parahoric} (see Lemma \ref{lem:parahorics-equiv}). If $x \in \B^1$, then $P_x$ is called an \textit{Iwahori} subgroup, which is a minimal parahoric. For $x \in \B^0$, $P_x$ is a maximal parahoric, and if $x \in \B^{hs}$, then $P_x$ is called a hyperspecial maximal parahoric. Since $v_0^\#=\Phi^{-1}v_0$, whenever $\Phi\in \wtK$, $v_0$ itself is a hyperspecial vertex in $\B$.\footnote{This is in fact the case in all the examples which we shall consider, with the exception of the CMSZ lattice (Theorem \ref{thm-simply-transitive}($\CC$)) over $p=2$, which requires extra caution.} As its $G_v$-stabilizer is $G(\mO_{F_v}) = G_v\cap\wtK$, and since $G_v$ acts transitively on $\B^{hs}$, we obtain an identification $\B^{hs}\cong G(F_v)/G(\mO_{F_v})$.

\subsection{Simply transitive lattices} \label{subsec:lattices-main}

The main goal of this section is to construct four definite unitary group schemes over $\Z$ and prove that they give rise to $p$-arithmetic congruence subgroups which act simply-transitively on the hyperspecial vertices of the Bruhat-Tits building of the corresponding $p$-adic group (Theorem \ref{thm-simply-transitive}).   

\begin{defn}\label{def-unitary-datum}
A unitary datum (over $\Q$) is a pair $(E,\Phi)$, of an imaginary quadratic field $E/\Q$, and a definite hermitian matrix $\Phi \in GL_3(E)$.\footnote{This extends to any totally real global field $F$ replacing $\Q$, where $E$ is now a CM extension of $F$, and $\Phi$ is definite at every real place of $F$.}
Let $G=U_3(E,\Phi)$ be the unitary group scheme over $\Z$ associated to  $(E,\Phi)$.
An arithmetic datum is a pair $(R,\mK)$, where $R$ is a finite set of primes such that for each $p\not\in R$, the subgroup $K_p := G(\Z_p) \leq G_p$ is maximal parahoric (and hyperspecial for almost all $p$), and a finite set of maximal parahoric subgroups for the remaining primes $\mK = \{K_p \leq G(\Q_p)\}_{p\in R}$.
For any prime $p$, define the principal $p$-arithmetic group associated to $G/\Q$ and the family of parahoric subgroups $\{K_\ell \leq G_\ell \}_\ell$,
\begin{equation}
\Gamma^p = G(\Q) \bigcap \prod_{\ell \ne p} K_\ell \leq G_p.
\end{equation}
A congruence datum is a pair $(M,H)$, of an ideal $M \trianglelefteq \mO_E$, whose norm is a power of a single non-split prime $m \not \in R$, and a subgroup $H$ of the finite group $G[M] := G(\Z_m) \Mod M \leq GL_3(\mO_E/M)$.
Note that the modulo $M$ map is a well defined homomorphism from $K_m = G(\Z_m)$ to $G[M]$.
Denote the level $(M,H)$ congruence subgroup of $K_m$ to be $K_m(H) = \{g \in K_m \; : g \mod M \in H \}$.
For any prime $p\ne m$, since $\Gamma^p \leq K_m$, the modulo $M$ map induces an homomorphism from $\Gamma^p$ to $G[M]$.
Define the associated level $(M,H)$ congruence $p$-arithmetic subgroup to be
\begin{equation}
\tilde\Lambda^p = G(\Q) \bigcap \prod_{\ell \ne p,m} K_\ell \cdot K_m(H) = \{g \in \Gamma^p \; : g \mod M \in H \}.
\end{equation}
A strong unitary datum is a sextuple $(E,\Phi, R, \mK, M, H)$, where $(E,\Phi)$ is a unitary datum (over $\Q$), $(R,\mK)$ is an arithmetic datum and $(M,H)$ is a congruence datum. 
\end{defn}

Note that for the arithmetic datum $(R,\mK)$ with $R=\emptyset$ and $\mK = \emptyset$, then $\Gamma^p = G(\Q) \bigcap \prod_{\ell \ne p}G(\Z_l) = G(\Z[1/p])$.
Also note that for the congruence datum $(M,H)$ with $M=m$ a rational prime and $H = \{I\}$ the trivial subgroup, then $K_m(H)$ is the principal congruence subgroup of level $m$, which we shall also denote by $K_m(m)$.

The main purpose of this section is to prove the following Theorem. 
\begin{thm} \label{thm-simply-transitive}
For each of following strong unitary datum, ($\EE$), ($\GG$), ($\MM$), and ($\CC$), and each prime $p\ne m$, the associated congruence $p$-arithmetic subgroup modulo its center $\Lambda^p = \tilde\Lambda^p/Z(\tilde\Lambda^p)$, where $Z(\tilde\Lambda^p) = \{I\}$ for ($\EE$), ($\GG$),  and ($\CC$), and  $Z(\tilde\Lambda^p) = \{\pm I\}$ for ($\MM$), acts simply-transitively on the hyperspecial vertices on the Bruhat-Tits building of $G(\Q_p)$.
\begin{itemize}
\item[($\EE$)] $E=\Q[\omega]$, $\omega = \frac{-1+\sqrt{-3}}2$, $\Phi= \bmx 1 & 0 &  0 \\ 0 & 1 &  0 \\ 0 & 0 & 1 \emx$, $R=\emptyset$, $\mK=\emptyset$, $M=(3)$, $m=3$ and 
\begin{equation} \label{eq:Eis_H}
H = \left\{ \bmx 1 & * & * \\ * & 1 & * \\ * & * & 1 \emx \right\} \leq G[3].
\end{equation}
Call $\Lambda^p_\EE$ the Eisenstein lattice.

\item[($\GG$)]  $E = \Q[i]$, $i = \sqrt{-1}$, $\Phi= \bmx 1 & 0 &  0 \\ 0 & 1 &  0 \\ 0 & 0 & 1 \emx$, $R=\emptyset$, $\mK=\emptyset$, $M=(2 + 2i)$, $m=2$ and 
\begin{equation} \label{eq:GG_H}
H = \left\{ \bmx 1 & * & * \\ * & 1 & * \\ * & * & 1 \emx \right\} \leq G[2+2i].
\end{equation}
Call $\Lambda^p_\GG$ the Gauss lattice.

\item[($\MM$)] $E=\Q[\lambda]$, $\lambda = \frac{-1+\sqrt{-7}}2$, $\Phi=\bmx 3 & \bar{\lambda} & \bar{\lambda}\\ \lambda & 3 & \bar{\lambda}\\ \lambda & \lambda & 3 \emx$, $R=\emptyset$, $\mK=\emptyset$, $M=(2)$, $m=2$ and 
\begin{equation} \label{eq:MM_H}
H = \left\{ \bmx * & * & * \\  & * & * \\  &  & * \emx \right\} \leq GL_3(\F_2) \cong G[2].
\end{equation}
Call $\Lambda^p_{\MM}$ the Mumford lattice.

\item[($\CC$)] $E=\Q[\eta]$, $\eta = \frac{1-\sqrt{-15}}2$, 
$\Phi=\bmx 10 & -2(\eta+2) & \eta+2\\ -2(\bar{\eta}+2) & 10 & -2(\eta+2)\\ \bar{\eta}+2 & -2(\bar{\eta}+2) & 10 \emx$, $R=\{2\}$, $\mK = \{K_2\}$ where $K_2$ is defined in Definition \ref{def-parahoric}, $M=(3,1+\eta)$, $m=3$ and 
\begin{equation} \label{eq:CMSZ_H}
H= \left\langle \bmx 1 & 1 & 0 \\ -1 & 1 & 0 \\ 0 & 0 & 1 \emx, \bmx \pm I_2 & 0 \\ & 1 \emx, \bmx I_2 & b \\ & 1 \emx \mid  b\in \F_3^2 \right\rangle \leq G[M].
\end{equation}
Call $\Lambda^p_\CC$ the CMSZ lattice.
\end{itemize}
\end{thm}

We note that it is not obvious why the $H$'s from equations \eqref{eq:Eis_H} and \eqref{eq:GG_H} are in fact subgroups.
For case ($\GG$) this was proved in \cite[Proposition 34]{Evra2018RamanujancomplexesGolden}, while for case ($\EE$) we shall prove it below in Lemma \ref{lem-(E)}.
We also note that in equation \eqref{eq:MM_H}, the isomorphism $G[2] \cong GL_3(\F_2)$ is obtained from $G[2] \le GL_3(\mO_E/(2)) \cong GL_3(\mO_E/(\lambda)) \times GL_3(\mO_E/(\bar{\lambda}))$, $\mO_E/(\lambda) \cong \F_2$, and projecting to the first component.

In this paper we shall mostly be interested in the above theorem when the prime $p$ is inert in $E$, in which case the Bruhat-Tits building is a $(p^3+1,p+1)$-biregular tree.
This happens when $p \equiv 3 \mod 4$ for ($\GG$), $p \equiv 2 \mod 3$ for ($\EE$), $p \equiv 3,5,6 \mod 7$ for ($\MM$) and $p \equiv 7,11,13,14 \mod 15$ for ($\CC$). 
However we note that the above Theorem holds also for split primes, in which case it yields new $\tilde{A}_2$-groups in the sense of \cite{cartwright1993groups,Cartwright1993Groupsactingsimply}, i.e. discrete groups which act simply-transitively on the vertices of the $2$-dimensional building of $PGL_3(\Q_p)$, for infinitely many primes (and not just for the small primes $p=2,3$, as was extensively investigated in \cite{Cartwright1993Groupsactingsimply}). This will be used in Section \ref{subsec:complexes} to give new examples of Ramanujan and non-Ramanujan Cayley $\tilde{A}_2$-complexes.

Our proof of the above Theorem differs from those of \cite{Cartwright1993Groupsactingsimply} and \cite{mumford1979algebraic} and proceeds as follows:
First we prove the class number one property for the unitary groups associated to our unitary and arithmetic data, and use it to deduce the transitivity of the action of the $p$-arithmetic groups, for almost all $p$.
Secondly we use our congruence condition to remove non-trivial elements stabilizing hyperspecial vertices, while still maintaining the transitivity property.
The first step differs from the proofs of \cite{Cartwright1993Groupsactingsimply} and \cite{mumford1979algebraic} and this is what enables us to prove our result for primes other than $p=2,3$.
The construction of the lattices ($\MM$) and ($\CC$) is strongly motivated by, and draws intuition from, the works of \cite{mumford1979algebraic} and \cite{Kato2006ArithmeticstructureCMSZ}.
However, we stress that our congruence conditions are different from the ones appearing in the above mentioned works.
More specifically, in \cite{mumford1979algebraic} the congruence condition $M$ was over the prime $m=7$, while ours is over $m=2$, and in \cite{Kato2006ArithmeticstructureCMSZ} the congruence condition was $M=(3)$, while ours is over $M = (3,1+\eta)$.
The reason to find new congruence conditions which guarantee a simply transitive action has to do with the Ramanujan conjecture, as will become apparent in Section \ref{section:automorphic}.

This section and the proof of Theorem \ref{thm-simply-transitive} is organized as follows: 
In subsection \ref{subsec:lattices-parahorics}, we construct maximal parahoric subgroups at each prime and for each of the four unitary group schemes from Theorem \ref{thm-simply-transitive} (Proposition \ref{pro-parahoric}).
In subsection \ref{subsec:lattices-transitive}, we give a criterion for the transitivity of the action of the $p$-arithmetic groups, for almost all $p$, in terms of the unitary and the arithmetic datum (Proposition \ref{prop-transitivity}).
In subsection \ref{subsec:lattices-simply}, we give a criterion for the simply transitive action of certain congruence subgroups, assuming the transitivity of the action, and depending on the congruence datum (Proposition \ref{prop-simplicity}), and we also prove that the conditions in the previous two steps holds for each of the four strong unitary datum in Theorem \ref{thm-simply-transitive} (Proposition \ref{prop-(G,E,M,C)}).

\subsection{Parahorics} \label{subsec:lattices-parahorics}

\begin{defn}\label{def-parahoric}
Let $(E,\Phi)$ be a unitary datum of type ($\EE$), ($\GG$), ($\MM$) or ($\CC$) as in Theorem \ref{thm-simply-transitive}, let $G=U(E,\Phi)$, and let $p$ be a prime.
Except for the special cases ($\CC$,$(p=)2$) and ($\GG,2$), define the following subgroup of $G_p$, 
\begin{equation*}
K_p = G(\Z_p) = \{g\in GL_3(\mO_{E_p}) \;:\; g^* \Phi g = \Phi \},
\end{equation*}
and for the case ($\CC,2$), define the following subgroup of $G(\Q_2)$,
\begin{equation*}
K_2 \cong \{ (g_1,g_2) \in GL_3(\Z_2) \times GL_3(\Q_2) \;:\;  g_2^t = \phi g_1^{-1} \phi^{-1}  \}
\end{equation*}
using the splitting $E_2 \cong \Q_2 \times \Q_2$, $\mO_{E_2} \cong \Z_2 \times \Z_2$ and $\Phi \cong (\phi , \phi^t) \in M_3(\Z_p)\times M_3(\Z_p)$.
\end{defn}

\begin{rem}
We shall not deal with the case ($\GG, 2$) in which a wildly ramified extension arises.
For a complete proof of case ($\GG$) of Theorem \ref{thm-simply-transitive} see \cite{Evra2018RamanujancomplexesGolden}.
\end{rem}

\begin{rem}
The unique case under consideration in which $K_p \ne G(\Z_p)$ is ($\CC,2$).
In this case, note that $K_2$ contains $G(\Z_2)$ as a proper subgroup of finite index
\begin{equation}
G(\Z_2) \cong \{(g_1,g_2) \in GL_3(\Z_2)\times GL_3(\Z_2) \;:\; g_2 = (\phi^t)^{-1} (g_1^t)^{-1} \phi^t  \}.
\end{equation}
\end{rem}

\begin{prop}\label{pro-parahoric}
The subgroups $K_p \leq G_p =  G(\Q_p)$ defined in Definition \ref{def-parahoric} are maximal parahorics.
Moreover, except for the special cases ($\MM,7$), ($\CC,3$), ($\MM,5$), ($\CC,2$) and ($\GG,2$), they are hyperspecial.
\end{prop}

The proof of Proposition \ref{pro-parahoric} follows from the following Lemmas:
Lemma \ref{lem:parahorics-equiv} gives a characterization of maximal parahoric subgroups in terms of stabilizer of vertices in the Bruhat-Tits building.
Lemma \ref{lem:parahorics-unramified-case} constructs hyperspecial maximal parahoric subgroups in all cases ($\EE$), ($\GG$), ($\MM$), ($\CC$), and for all primes $p$,  except for the following the special cases in which $p$ divides the determinant $\det\Phi$: ($\MM,7$), ($\CC,3$), ($\CC,5$), ($\CC,2$) and ($\GG,2$).
Lemma \ref{lem:parahorics-ramified-odd-case} constructs maximal parahoric subgroups in the special cases ($\MM,7$), ($\CC,3$), ($\CC,5$).
Lemma \ref{lem:parahorics-2-case} deals with the  special case ($\CC,2$).

\begin{lem}\label{lem:parahorics-equiv}
Let $(E,\Phi)$ be a unitary datum, $G = U_3(E,\Phi)$, $p$ a prime and $G_p = G(\Q_p)$.
Then $K_p \leq G_p$ is maximal parahoric if and only if it is the stabilizer of a vertex in the Bruhat-Tits building of $G_p$. 
\end{lem}

\begin{proof}
The proof follows from the main result of \cite{Haines2008Parahoric}, together with the fact that kernel of the Kottwitz homomorphism, $\kappa_G \,:\, G_p \rightarrow X^*(\hat{Z}(G)^{\mbox{Gal}(\bar{\Q}_p/\Q_p)})$, is the entire group $G_p$, since the group $\hat{Z}(G)^{\mbox{Gal}(\bar{\Q}_p/\Q_p)} \cong U(1)$ is compact, hence $X^*(\hat{Z}(G)^{\mbox{Gal}(\bar{\Q}_p/\Q_p)})$ is trivial.
\end{proof}

A maximal parahoric subgroup is called hyperspecial if it stabilizes a hyperspecial vertex in the Bruhat-Tits building.
If $p$ splits then any vertex is hyperspecial, if $p$ ramifies then no vertex is hyperspecial, and if $p$ is inert then every vertex of degree $p^3+1$ is hyperspecial.

\begin{lem}\label{lem:parahorics-unramified-case}
Let $(E,\Phi)$ be a unitary datum and $G = U_3(E,\Phi)$. For any prime $p \nmid \mathrm{disc} \Phi$, the group
\begin{equation*}
K_p = G(\Z_p) = \{g\in GL_3(\mO_{E_p}) \;:\; g^* \Phi g = \Phi \}
\end{equation*}
is a maximal parahoric subgroup of $G_p = G(\Q_p)$.
Moreover, if $p\nmid \mbox{disc}(E)$, then $K_p$ is hyperspecial.
\end{lem}

\begin{proof}
Let $\tG = \mbox{Res}_{E/ \Q} G$ be the Weil restriction of scalars of $G$ from $E$ to $\Q$ and let $\mbox{Gal}(E/\Q) = \{1,\sigma\}$ be the Galois group. 
Since $E\otimes_{\Q} E = E \oplus E$, we obtain that $\tG(\Q)=G(E)=\{g=(g_1,g_2)\in GL_3(E)\times GL_3(E) \,:\, g_2 = \Phi^{-1} (g_1^*)^{-1} \Phi \}$.
By projecting to the first component we get an isomorphism $\iota \,:\, \tG(\Q) \rightarrow GL_3(E)$, $\iota(g)=g_1$. 
The Galois group $\mbox{Gal}(E/\Q)$ acts on $\tG$ by $\sigma((g_1,g_2))=(g_2,g_1)$, and $G$ is the $\mbox{Gal}(E/\Q)$-fixed points of $\tG$.
Under the $\iota$ isomorphism, $\mbox{Gal}(E/\Q)$ acts on $GL_3(E)$ through the following involution $\theta(g) = \Phi^{-1} (g^*)^{-1} \Phi$.
Note that $G$ is equal the $\theta$-fixed points of $GL_3(E)$, i.e. $G = GL_3(E)^{\theta} =  \{g \in GL_3(E) \,:\, \theta(g) = g \}$.

Let $\B_p = \B(G_p)$ be the Bruhat-Tits building of $G_p= G(\Q_p)$ and $\tilde{\B}_p = \B(\tG_p)$ the Bruhat-Tits building of $\tG_p= \tG(\Q_p) \cong GL_3(E_p)$.
By Section 2.6.1 of Tits' survey article \cite{Tits1979Reductivegroupsover} (excluding the case $(E,p)=(\Q[i],2)$, which is wildly ramified), $\B_p$ is the $\mbox{Gal}(E/\Q)$-fixed points of $\tilde{\B}_p$, where $\mbox{Gal}(E/\Q)$ acts on $\tilde{\B}_p$ through its action on $\tG_p$.
The Bruhat-Tits building $\tilde{\B}_p = \B(GL_3(E_p))$ is a well understood object, for instance $\tK_p = GL_3(\mO_{E_p})$  is the stabilizer of a hyperspecial vertex $v_0$.

Since $p \nmid \det \Phi$ and $\Phi \in M_3 (\mO_E)$, we have $\Phi_3 \in GL_3(\mO_{E_p})$ and therefore $\theta(\tK_p) = \Phi^{-1} \tK_p \Phi = \tK_p$.
Since $\theta(\tK_p) $ is the stabilizer in $GL_3(E_p)$ of $\theta(v_0)$, we get that $\theta(v_0) = v_0$, i.e. $v_0 \in \tilde{\B}_p^{\mbox{Gal}(E/\Q)}$.
Hence $v_0$ belongs to $\B_p$, and $K_p = G_p \cap \tK_p$, which is the stabilizer of $v_0$ in $G_p$, is a maximal parahoric subgroup of $G_p$, and it is hyperspecial if $E_p/\Q_p$ is an unramified extension, which is when $p\nmid \mbox{disc}(E)$.
\end{proof}

\begin{lem}\label{lem:parahorics-ramified-odd-case}
In each of the following cases the group
\begin{equation*}
K_p = G(\Z_p) = \{g\in GL_3(\mO_{E_p}) \;:\; g^* \Phi g = \Phi \},
\end{equation*}
is a maximal parahoric subgroup of $G_p = G(\Q_p)$.
\begin{itemize}
 \item[($\MM,7$)] $E=\Q[\lambda = \frac{-1+\sqrt{-7}}2]$,
$\Phi=\left(\begin{array}{ccc} 3 & \bar{\lambda} & \bar{\lambda}\\ \lambda & 3 & \bar{\lambda}\\ \lambda & \lambda & 3 \end{array}\right)$, and $p=7$.
\item[($\CC,3$)] $E=\Q[\eta = \frac{1-\sqrt{-15}}2]$,
$\Phi=\left(\begin{array}{ccc} 10 & -2(\eta+2) & \eta+2\\ -2(\bar{\eta}+2) & 10 & -2(\eta+2)\\ \bar{\eta}+2 & -2(\bar{\eta}+2) & 10 \end{array}\right)$ and  $p=3$.
\item[($\CC,5$)] $E=\Q[\eta = \frac{1-\sqrt{-15}}2]$,
$\Phi=\left(\begin{array}{ccc} 10 & -2(\eta+2) & \eta+2\\ -2(\bar{\eta}+2) & 10 & -2(\eta+2)\\ \bar{\eta}+2 & -2(\bar{\eta}+2) & 10 \end{array}\right)$ 
and  $p=5$.
\end{itemize}
\end{lem}

\begin{proof}
We continue with the notation in the proof of Lemma \ref{lem:parahorics-unramified-case}.
Recall that $\mbox{Gal}(E/\Q)$ acts on the Bruhat-Tits building $\B(GL_3(E_p))$ through the involution $\theta$, which acts on $GL_3(E_p)$ by $\theta(g) = \Phi^{-1} (g^*)^{-1} \Phi$, and it acts on the building by sending the vertex $v_0$ to the vertex $\theta(v_0)$ whose stabilizer is $\theta(GL_3(\mO_{E_p})) =  \Phi^{-1} GL_3(\mO_{E_p}) \Phi$.
Namely, since $v_0 = [\mO_{E_p}^3] \in \B(GL_3(E_p))$, then $\theta(v_0) = \Phi^{-1}.v_0 =  [\Phi^{-1}\mO_{E_p}^3]$.

Note that $K_p = G_p \cap GL_3(\mO_{E_p})$ is the stabilizer in $G_p$ of $v_0 = [\mO_{E_p}^3] \in \B(GL_3(E_p))$.
If $\theta(v_0)$ is adjacent to $v_0$, then the edge $\{v_0,\theta(v_0)\}$ is $\mbox{Gal}(E/\Q)$-fixed, hence it is a vertex in $\B(G_p)$. 
Therefore $K_p = G_p \cap GL_3(\mO_{E_p}) = G_p \cap \theta(GL_3(\mO_{E_p}))$ is a maximal parahoric subgroup of $G_p$.
Hence it suffice to prove that $\theta(v_0)$ is adjacent to $v_0$.

Let $\Phi = k_1 \cdot a \cdot k_2$, where $k_1,k_2\in GL_3(\mO_{E_p})$, $a=\mbox{diag}(\varpi^{e_1},\varpi^{e_2},\varpi^{e_3})$, $\varpi \in \mO_{E_p}$ a uniformizer, $e_1\geq e_2 \geq e_3$, be a Cartan decomposition of $\Phi$.
Since $GL_3(\mO_{E_p})$ is the stabilizer of $v_0$, the distance between $\theta(v_0) = \Phi^{-1}.v_0$ and $v_0$, is the same as the distance between $a.v_0$ and $v_0$.
So it suffices to prove that $a.v_0$ and $v_0$ are of distance $1$, which means that $|e_i - e_j| \leq 1$ for any $i,j=1,2,3$.
Hence we check that in each of the three cases ($\MM,7$), ($\CC,3$) and ($\CC,5$), the Cartan decomposition of $\Phi$ satisfies the above condition.

\begin{itemize}
\item[($\MM,7$)] The Cartan decomposition of $\Phi$ satisfies the above condition, 
\begin{equation*}
\Phi = k_1 \bmx \varpi &  &  \\  & \varpi &  \\  &  & 1 \emx k_2,
\end{equation*}
where $\varpi=\sqrt{-7}$ is a uniformizer in $\mO_{E_7}$ and
\begin{equation*}
k_1 = (B^{-1})^* \in GL_3(\mO_{E_7}) , \qquad k_2 = \bmx & 1 & \\ -1 & & \\ & & -1 \emx \cdot B^{-1}  \in GL_3(\mO_{E_7}),
\end{equation*}
\begin{equation*}
B= \bmx 1 & -\frac{\bar\lambda}3 & -\frac{\bar\lambda}3 \\ & 1 & \\  & & 1 \emx \bmx 1 & & \\ & 1 & 1 \\  & \beta & \beta^{-1} \emx \bmx \alpha^{-1} & & \\  & \gamma & \\  & & 1 \emx
\bmx  &  & 1\\ & 1& \\ 1&  &  \emx \in GL_3(\mO_{E_7}),
\end{equation*}
$\alpha \in \Z_7^\times$ satisfies $\alpha^2 = -3, \beta = \frac{1+\alpha}2\in\Z_7^\times$, and $\gamma = \frac{-\alpha-\varpi}2 \in \mO_7^\times$.

\item[($\CC,3$)]  The Cartan decomposition of $\Phi$ satisfies the above condition, 
\begin{equation*}
\Phi = k_1 \bmx \varpi &  &  \\  & \varpi &  \\  &  & 1 \emx k_2,
\end{equation*}
where $\varpi=\sqrt{-15}$ is a uniformizer in $\mO_{E_3}$ and
\begin{equation*}
k_1 = (B^{-1})^\ast \in GL_3(\mO_{E_3}) , \qquad k_2 =  \bmx &1 & \\ -1 & & \\  & & 1 \emx \cdot B^{-1} ,
\end{equation*}
\begin{equation*}
B=  \bmx 1 & \frac{\eta+2}{5} & -\frac{\eta+2}{10} \\ & 1 & \\ & & 1 \emx \bmx 1 & & \\ & \beta & \tilde\beta \\ & 1 & 1 \emx
\bmx \sqrt{10} & & \\ & 1 & \\ & & (\alpha+\bar\varpi)^{-1} \emx \bmx &&1\\&1&\\1&& \emx \in GL_3(\mO_{E_3}),
\end{equation*}
$\alpha\in\Z_3^\times$ satisfies $\alpha^2=-5$ and $\alpha\equiv 1\mod 3$, $\beta=\frac{1+\alpha}2\in\Z_3^\times$, and $\tilde\beta=\frac{1-\alpha}2\in p\Z_3^\times$.

\item[($\CC,5$)]  The Cartan decomposition of $\Phi$ satisfies the above condition, 
\begin{equation*}
\Phi = k_1 \bmx \varpi^2 &  &  \\  & \varpi &  \\  &  & \varpi \emx k_2,
\end{equation*}
where $\varpi=\sqrt{-15}$ is a uniformizer in $\mO_{E_5}$ and
\begin{equation*}
k_1 = (B^\ast)^{-1} \in GL_3(\mO_{E_5}) , \qquad k_2 = \bmx  -\frac2{9} & & \\ &  &1 \\  & 1&  \emx \cdot B^{-1}  \in GL_3(\mO_{E_5}),
\end{equation*}
\begin{equation*}
B= \bmx \eta+2 & 1 & \\ 2+\alpha & \frac{2-\alpha}{\eta+2} & \\ & & 1 \emx \bmx \beta^{-1} & &  \\ & 1 & \\ & & 1 \emx
\bmx 1 &  & \frac1{\varpi}(\eta+2-\eta(2-\alpha)) \\ & 1 & \frac{10-2(\eta+2)(2+\alpha)}{\varpi\bar\beta} \\ & & 1 \emx \bmx & & 1\\ & 1&  \\ 1&  &  \emx \in GL_3(\mO_{E_5}),
\end{equation*}
$\alpha\in\Z_5^\times$ satisfies $\alpha^2=-6$ and $\alpha\equiv 2\mod 5$, and $\beta=\frac{12(\eta+2)}{\bar\varpi}\in\Z_5^\times$.
\end{itemize}
\end{proof}

We are left to deal with the last remaining case ($\CC,2$). 
Note that $p=2$ splits in the quadratic field $E = \Q[\eta = \frac{1-\sqrt{-15}}2]$ of case ($\CC$).
The following Lemma deals with this case, and more generally, it constructs a maximal parahoric subgroup for any unitary datum $(E,\Phi)$ and any prime $p$ that splits in $E$.

\begin{lem}\label{lem:parahorics-2-case}
Let $(E,\Phi)$ be a unitary datum, $G=U_3(E,\Phi)$ and let $p$ be prime that splits in $E$.
Consider the splitting $E_p \cong \Q_p \times \Q_p$, $\mO_{E_p} \cong \Z_p \times \Z_p$ and let $\Phi \cong (\phi , \phi^t) \in M_3(\Z_p)\times M_3(\Z_p)$.
Under this splitting
\begin{equation*}
G(\Q_p) \cong \{(g_1,g_2) \in GL_3(\Q_p)\times GL_3(\Q_p) \;|\; g_2^t = \phi g_1^{-1} \phi^{-1} \},
\end{equation*}
and the following is a hyperspecial maximal parahoric subgroup of $G(\Q_p)$,  
\begin{equation*}
K_p  \cong  \{(g,( \phi g_1^{-1} \phi^{-1})^t) \in GL_3(\Z_p)\times (\phi^t)^{-1} GL_3(\Z_p) \phi^t \}.
\end{equation*}
\end{lem}

\begin{proof}
For a split prime $p$, the isomorphism $E_p \cong \Q_p\times\Q_p$ is defined by two conjugate idempotents $\nu$ and $\nu^\ast$ (i.e. $\nu^2=\nu$, $1=\nu+\nu^\ast$ and $\nu\nu^\ast=0$) together with $\nu\mapsto (1,0)$ resp.\ $\nu^\ast\mapsto(0,1)$. 
This induces the following decompositions $GL_3(E_p) \cong GL_3(\Q_p)\times GL_3(\Q_p)$, where $g=g_1\nu+g_2\nu^\ast \mapsto (g_1,g_2)$,  and $\Phi \cong (\phi,\phi^T)$, where $\Phi=\phi\nu+\phi^T\nu^\ast$. 
An element $g=g_1\nu+g_2\nu^\ast\in G(E_p)$ belongs to $G(\Q_p)$  if and only if $(g_1,g_2)\in GL_3(\Q_p)\times GL_3(\Q_p)$ satisfies $g_2^t=\phi g_1^{-1}\phi^{-1}$.
This induces the following isomorphism $G(\Q_p) \cong GL_3(\Q_p)$, $g=g_1\nu+g_2\nu^\ast \mapsto g_1$.
The subgroup $GL_3(\Z_p)$ is a hyperspecial maximal parahoric subgroup of $GL_3(\Q_p)$.
By the above isomorphism between $G(\Q_p)$ and $GL_3(\Q_p)$, the preimage of $GL_3(\Z_p)$ in $G(\Q_p)$, which is precisely $K_p$, is a hyperspecial maximal parahoric subgroup.
\end{proof}

\begin{proof}[Proof of Proposition \ref{pro-parahoric}]
Follows from Lemmas \ref{lem:parahorics-unramified-case}, \ref{lem:parahorics-ramified-odd-case} and \ref{lem:parahorics-2-case}.
\end{proof}

\subsection{Transitive actions} \label{subsec:lattices-transitive}

Let $(E,\Phi)$ be a unitary datum of one of the four types in Theorem \ref{thm-simply-transitive} and $G=U_3(E,\Phi)$.

\begin{defn}
For any prime $p$, let $K_p$ be the maximal parahoric subgroup of $G_p = G(\Q_p)$ in Definition \ref{def-parahoric}.
Define the following maximal open compact adelic subgroup of $G(\A)$
\begin{equation}
K = G(\R) \prod_p K_p \leq G(\A).
\end{equation}
Define the group of globally integral elements of $G$,
\begin{equation}
\Gamma = G(\Q) \bigcap K = G(\Q) \bigcap \prod_p K_p.
\end{equation}
For any prime $p$, define the principal $p$-arithmetic subgroup of $G$,
\begin{equation}
\Gamma^p = G(\Q) \bigcap K^p = G(\Q) \bigcap \prod_{\ell \ne p} K_\ell.
\end{equation}
\end{defn}

\begin{rem}
Note that for all cases except ($\CC,2$), $K_p = G(\Z_p)$, and in case ($\CC,2$), $G(\Z_p)$ is of finite index in $K_p$.
Hence in all cases except ($\CC$), $K = G(\R\hat{\Z})$, $\Gamma = G(\Z)$ and $\Gamma^p = G(\Z[1/p])$, for any $p$.
In case ($\CC$), $G(\Z)$ is of finite index in $\Gamma$ and $G(\Z[1/p])$ is of finite index in $\Gamma^p$, for any $p$.
\end{rem}

\begin{prop} \label{prop-transitivity}
Let $(E,\Phi)$ be a unitary datum, $G= U_3(E,\Phi)$ and $R_E = \{p \mbox{ prime} \,:\, p \mid \mbox{disc}(E) \}$.
Assume that the discriminant of $\Phi$ is a product of primes from $R_E$.
Let $L(s,\chi_E)$ be the (continuation of the) Dirichlet L-function of $\chi_E$, the Dirichlet character associated to $E/\Q$ by class field theory.
For any prime $p$, let $K_p \leq G(\Q_p)$  be the maximal parahoric subgroup in Definition \ref{def-parahoric}.
If 
\begin{equation}
|\Gamma|^{-1} = 2^{-2 - |R_E|} 12^{-1}  L(0,\chi_E) L(-2,\chi_E),
\end{equation}
then for any unramified prime $p$, the principal $p$-arithmetic subgroup of $G$, $\Gamma^p = G(\Q) \bigcap \prod_{\ell \ne p} K_\ell$, acts transitively on the hyperspecial vertices of the Bruhat-Tits building of $G(\Q_p)$.
\end{prop}

The proof of the above transitivity criterion comprise of the following steps:
First, we introduce the mass invariant of $G$ and $K$ (Definition \ref{def-mass}).
Second, we give a mass formula relating the mass of $G$ to the subgroup of globally integral elements of $G$ and its twists (Lemma \ref{lem-GHY}).
Third, we show that the assumption in Proposition \ref{prop-transitivity} implies that the class number of the group is one (Lemma \ref{lem-cls1}).
Finally, we show that class number one property implies that the $p$-arithmetic subgroup of $G$ acts transitively on the corresponding Bruhat-Tits buildings (Lemma \ref{lem-transitivity}).

\begin{defn}\label{def-mass}
Define the mass of $G/\Q$ w.r.t. the maximal open compact adelic subgroup $K$, to be 
\begin{equation}
\mbox{Mass}(G,K) := \frac{\mu(G(\Q)\backslash G(\A))}{\mu(K)} = 
 \sum_{g \in G(\Q)\backslash G(\A) / K} |G(\Q) \cap g^{-1}K g|^{-1}.
\end{equation}
where $\mu$ is a Haar measure on the adelic group $G(\A)$. The class number of $G$ is
\begin{equation}
\# \mathrm{Class}(G,K) = |G(\Q)\backslash G(\A) / K|.
\end{equation}
\end{defn}

\begin{lem} \label{lem-GHY}
In the notation of Proposition \ref{prop-transitivity},
\begin{equation}
\mathrm{Mass}(G,K) = 2^{-2 - |R_E|} \cdot 12^{-1} \cdot L(0,\chi_E) \cdot L(-2,\chi_E).
\end{equation}
\end{lem}

\begin{proof}
This formula follows directly from Proposition 2.13 of \cite{gan2001exact}, which gives
\begin{equation}
\mathrm{Mass}(G,K) = 2^{ - \ell \cdot d} \cdot L(M) \cdot  \tau(G) \cdot \prod_{v \in S} \lambda_v,
\end{equation}
where in our case, $d=[\Q:\Q]=1$ (the dimension of the number field $\Q$), $\ell=\mbox{rank}_{\bar\Q}(G)=3$ (the absolute rank of $G$), $\tau(G)=2$ (the Tamagawa number of $G/\bar\Q$), $L(M) = \prod_{i=1}^{\ell}  L(1-i,\chi_E^i) = 12^{-1} \cdot L(0,\chi_E) \cdot L(-2,\chi_E)$ (the special value of the Artin-Tate motive associated to $G$), $S=R_E$ (the ramified primes of $G$) and $\lambda_p = 1/2$ for $p\in R_E$   (the lambda factors at the ramified places). 
The last fact follows from the local calculation in Section 3 of \cite{gan2001exact} and since for any $p\in R_E$, $K_p$  has maximal reductive quotient modulo $p$ which is isomorphic to $O_1\times Sp_2$.
\end{proof}

\begin{lem} \label{lem-cls1}
In the notation of proposition \ref{prop-transitivity}, if  
\begin{equation}
|\Gamma|^{-1} = 2^{-2 - |R_E|} 12^{-1}  L(0,\chi_E) L(-2,\chi_E),
\end{equation}
then $G$ is of class number one, i.e.
\begin{equation}
 G(\A) = G(\Q) \cdot K.
\end{equation}
\end{lem}

\begin{proof}
By Lemma \ref{lem-GHY}, we get that
\begin{equation}
0 = \mbox{Mass}(G,K) - |G(\Q)\cap K|^{-1} 
=  \sum_{1 \ne g \in G(\Q)\backslash G(\A) / K} |G(\Q) \cap g^{-1}K g|^{-1}.
\end{equation}
Since all the finite groups $G(\Q) \cap g^{-1}K g$ are of size $\geq 1$, we get that $G(\Q)\backslash G(\A) / K = \{1\}$.
\end{proof}

\begin{lem} \label{lem-transitivity}
In the notation of Proposition \ref{prop-transitivity}, if $G(\A) = G(\Q) \cdot K$, then for any unramified prime $p$, the principal $p$-arithmetic subgroup of $G$, $\Gamma^p  = G(\Q) \cap \bigcap_{\ell \ne p} K_\ell$, acts transitively on the hyperspecial vertices of the Bruhat-Tits building of $G(\Q_p)$.
\end{lem}

\begin{proof}
Since $G_p = G(\Q_p)$ acts transitively on the hyperspecial vertices of its Bruhat-Tits building and  $K_p$ is the stabilizer of a hyperspecial vertex, we get that $\Gamma^p$ will act transitively if we can prove that
\begin{equation}
G_p = \Gamma^p \cdot K_p.
\end{equation}  
Let $g\in G_p$, and let $\tG \in G(\A)$ be such that $\tG_p = g$ and $\tG_v = 1$ if $v \ne p$.
By the class number one property, there exists $q \in G(\Q)$, considered as $\tilde{q} = (q,q,q,\ldots) \in G(\A)$, and $\tilde{k}  = (k_\infty,k_2,k_3,\ldots) \in K$, such that $\tG = \tilde{q} \cdot \tilde{k}$. For any $\ell \ne p$, we get $1 = \tG_\ell = \tilde{q}_\ell \cdot \tilde{k}_\ell = q \cdot k_\ell$, hence $q = k_\ell^{-1} \in K_\ell$. 
Therefore, $q \in G(\Q) \cap \bigcap_{\ell \ne p} K_\ell = \Gamma^p$. 
Finally, from $g = \tG_p = \tilde{q}_p \cdot \tilde{k}_p = q \cdot k_p$, and the fact that $k_p \in \Gamma^p$, we obtain the claim.
\end{proof}

\begin{proof}[Proof of Proposition \ref{prop-transitivity}]
Follows from Lemmas \ref{lem-GHY}, \ref{lem-cls1} and \ref{lem-transitivity}.
\end{proof}

\subsection{Simply transitive actions} \label{subsec:lattices-simply}

\begin{defn}
Fix a group $G$ and a subgroup $S$. 
A subgroup $H$ of $G$ is called transversal to $S$ if 
\begin{equation}
G = H \cdot S := \{ h\cdot s \;:\; h\in H, s\in S\} \qquad \mbox{and} \qquad  H \cap S = \{1\}.
\end{equation}
A subgroup $H$ of $G$ is called central transversal to $S$ if 
\begin{equation}
G = H \cdot S := \{ h\cdot s \;:\; h\in H, s\in S\} \qquad \mbox{and} \qquad  H \cap S \subset Z(G),
\end{equation}
where $Z(G)$ is the center of $G$.
\end{defn}

\begin{prop} \label{prop-simplicity}
Let $M$ be an ideal of $\mO_E$ above a  prime $m \in \Z$ (assume $m \ne 2$ in case ($\CC$)) and denote, 
\begin{equation}
G[M] = G(\Z_m) \mod M \leq GL_3(\mO_E/M).
\end{equation}
Let $H \leq G[M]$ be transversal to $\Gamma \mod{M} \leq G[M]$, and assume that the kernel of $\Gamma$ modulo $M$ is trivial (resp.\ equal to the center $Z(\Gamma)$).
Define for any unramified prime $p \nmid m$, the level $H$ congruence $p$-arithmetic subgroup and its quotient modulo the center
\begin{equation}
\tilde\Lambda^p= \{ g\in  \Gamma^p \; :\; (g \mod m) \in H\}, \qquad \Lambda^p = \tilde\Lambda^p/Z(\tilde\Lambda^p).
\end{equation}
If $\Gamma^p$ acts transitively on the hyperspecial vertices of the Bruhat-Tits building associated to $G(\Q_p)$, then $\Lambda^p $  acts simply-transitively on the hyperspecial vertices of the Bruhat-Tits building.
\end{prop}

We note in passing that $\Gamma$ contains the center of $\Gamma^p$ and therefore that the center of $\Gamma$ and of $\Gamma^p$ coincides.
This follows from the fact that $\Gamma = \Gamma \cap K_p$ is a stabilizer of a vertex in $\Gamma^p$ and that the center of $\Gamma^p$ is the point-wise stabilizer of the entire building.
A similar argument shows that $\Gamma \cap \tilde\Lambda^p$ contains the centers of $\tilde\Lambda^p$ and that the center of $\Gamma \cap \tilde\Lambda^p$ and of $\tilde\Lambda^p$ coincide.

\begin{lem} \label{lem-transversal}
In the notation of Proposition \ref{prop-simplicity}, assume that the kernel of $\Gamma$ modulo $M$ is trivial (resp.\ equals the center $Z(\Gamma)$).
If $H \leq G[M]$ is transversal to $\Gamma \mod{M} \leq G[M]$, then for any prime $p \nmid m$, $\tilde\Lambda^p \leq \Gamma^p$ is (resp.\ central) transversal to $\Gamma$.
\end{lem}

\begin{proof}
If $\gamma \in \Gamma  \cap \tilde\Lambda^p$, then modulo $M$, $\bar{\gamma} \in \Gamma \cap H=1$, and since the kernel of $\Gamma$ modulo $M$ is trivial (resp.\ equal the center $Z(\Gamma)$), we get that $\gamma =1$ (resp.\ $\gamma \in Z(\Gamma  \cap \tilde\Lambda^p)= Z(\tilde\Lambda^p)$).
Let $\gamma \in \Gamma^p$, and consider it modulo $M$, $\bar{\gamma} \in G[M]$. 
Since $G[M]= H \cdot \Gamma$, there exist $\gamma_0 \in \Gamma$, such that $\bar{\gamma} \gamma_0^{-1} \in H$.
Hence $\gamma \gamma_0^{-1} \in \tilde\Lambda^p$, and therefore $\gamma = \gamma \gamma_0^{-1} \cdot \gamma_0 \in \tilde\Lambda^p \cdot \Gamma$.
\end{proof}

\begin{lem} \label{lem-simplicity}
In the notation of Proposition \ref{prop-simplicity}, if $\tilde\Lambda^p \leq \Gamma^p$ is central transversal to $\Gamma$, then $ \Lambda^p=\tilde\Lambda^p/Z(\tilde\Lambda^p)$  acts simply-transitively on the hyperspecial vertices of the Bruhat-Tits building.
\end{lem}

\begin{proof}
Note that $\Gamma = \Gamma^p \cap K_p$ is the stabilizer in $\Gamma^p$ of a hyperspecial vertex $v_0$. 
Since $\Gamma^p$ acts transitively, for any other hyperspecial vertex $v$, there exists $\gamma \in \Gamma^p$ such that $\gamma.v_0 = v$.
Write $\gamma = \gamma_1 \gamma_0$, where $\gamma_1 \in \Lambda^p$ and $\gamma_0 \in \Gamma$.
Then $\gamma_1.v_0 = \gamma \gamma_0^{-1}.v_0 = \gamma.v_0 = v$, hence $\Lambda^p$ acts transitively.
By $\Gamma  \cap \tilde\Lambda^p = \{I\}$ (resp.\ central), we get that $\tilde\Lambda^p$ acts simply modulo its center. 
\end{proof}

\begin{proof}[Proof of Proposition \ref{prop-simplicity}]
Follows from Lemmas \ref{lem-transversal} and \ref{lem-simplicity}.
\end{proof}

The next Proposition shows that the four cases of Theorem \ref{thm-simply-transitive} satisfies the conditions of Propositions \ref{prop-transitivity} and \ref{prop-simplicity}.

\begin{prop} \label{prop-(G,E,M,C)}
In the notation of Theorem \ref{thm-simply-transitive}, each of the four strong unitary datum, ($\EE$), ($\GG$), ($\MM$) and ($\CC$), satisfies the following two properties:
\begin{itemize}
\item[(i)] $|\Gamma|^{-1} = 2^{-2 - |R_E|} 12^{-1}  L(0,\chi_E) L(-2,\chi_E)$.
\item[(ii)] $\Gamma$ embeds (resp.\ its kernel is the center for ($\MM$)) in $G[M]$, and $H\leq G[M]$ is transversal to $\Gamma$.
\end{itemize}
\end{prop}

Before proceeding with the proof of the above proposition, we first need to show that case ($\EE$) is well defined, in the sense that $H$ is indeed a subgroup (case ($\GG$) is well defined by Proposition 3.4 in [EP]).

\begin{lem} \label{lem-(E)}
Let $G = U_3( \Q[\sqrt{-3}],I)$ and $H = \left\lbrace \bmx 1 & * & * \\ * & 1 & * \\ * & * & 1 \emx \right\rbrace \subset G[3]$.  
Then $H$ is a normal subgroup of $G[3]$.
Moreover, $G[3] = \Omega \ltimes H$, where $\Omega$ is the subgroup of monomials of matrices with coefficients in $\langle \zeta \rangle$, $\zeta = \frac{1+\sqrt{-3}}2$ a sixth root of unity.
\end{lem}

\begin{proof}
First observe that $G[3] = \{g \in GL_3(\F_3[z]/(z^2)) \,:\, g^*g = I \}$, where $(g^*)_{i,j} = \overline{g_{j,i}}$ and $\overline{a + bz} = a - bz$. 
Writing $g = a +bz$, $a,b \in M_3(\F_3)$, then $g \in G[3]$ if and only if $a^ta = I$ and $b^t = a^t b a^t$.
Let $N = \ker\left(\mbox{mod}z\right) \leq G[3]$ and note that $N = \{ I + bz \,:\, b^t = b\}$ is an abelian normal subgroup.
Let $g = a + bz \in H$. Since $a_{i,i} = 1$ for any $i$, combined with the fact that a sum of two non-zero squares in $\F_3$ is non-zero, we get that $a = I$.
Therefore $H = \{g = I + bz \,:\, b^t = b,\; b_{i,i} = 0\; \forall i \}$, which is a subgroup of $G[3]$.
Next note that $\Omega \cap H = \{1\}$ and that $\omega H \omega^{-1} \subset H$ for any $\omega \in \Omega$ (both claims are easy to verify when $\omega$ is either a permutation matrix or a diagonal matrix).
Finally note that for any $g \in G[3]$, in any column of $g$ there exists an element from $(\F_3[z]/(z^2))^* = \{\pm 1 +bz \,:\, b\in \F_3\}$, hence there exists $\omega \in \Omega$ such that $h = \omega^{-1}g$ satisfies $h_{i,i} = 1$ for any $i$, i.e. $h \in H$.
Therefore $G[3] = \Omega \cdot H$ as sets, which implies that $H$ is normal and the fact that $G[3] = \Omega \ltimes H$, as claimed.
\end{proof}

\begin{lem}\label{lem-(G,E,M,C)-(I)} 
For each of the four strong unitary datum, ($\EE$), ($\GG$), ($\MM$) and ($\CC$), we have
\begin{itemize}
\item[($\EE$)] $\Gamma \cong S_3 \ltimes  C_6^3$ and $2^{-2 - |R_E|} 12^{-1}  L(0,\chi_E) L(-2,\chi_E) = \frac1{3! \cdot 6^3} $.
\item[($\GG$)] $\Gamma \cong S_3 \ltimes C_4^3$ and $2^{-2 - |R_E|} 12^{-1}  L(0,\chi_E) L(-2,\chi_E) =  \frac1{3! \cdot 4^3}$.
\item[($\MM$)] $\Gamma \cong C_2 \times C_3 \ltimes  C_7$ and $2^{-2 - |R_E|} 12^{-1}  L(0,\chi_E) L(-2,\chi_E) = \frac1{2\cdot 3 \cdot 7}$.
\item[($\CC$)] $\Gamma \cong C_2 \times C_3 $ and $2^{-2 - |R_E|} 12^{-1}  L(0,\chi_E) L(-2,\chi_E) = \frac1{2 \cdot 3}$.
\end{itemize}
In particular property (i) of Proposition \ref{prop-(G,E,M,C)} holds for ($\EE$), ($\GG$), ($\MM$) and ($\CC$).
\end{lem}

\begin{proof}
Evaluating the expression $2^{-2 - |R_E|} 12^{-1}  L(0,\chi_E) L(-2,\chi_E)$ is a simple and direct calculation of the special values of the Dirichlet $L$-functions.
Therefore what we are left to do is to compute the global integer group $\Gamma$ for each of the four cases ($\EE$), ($\GG$), ($\MM$) and ($\CC$).

Cases ($\EE$) and ($\GG$) follows from the fact for $\Phi = I$, the identity matrix, the group $\Gamma$ is generated by permutation matrices and diagonal matrices with entries roots of unity in $\mO_E$, i.e. $\Gamma \cong S_3 \ltimes (\mO_E^*)^3$.

For cases ($\MM$) and ($\CC$), where $\Phi \ne I$, because $G(\R)$ is compact, there exist bounds for the absolute values of the coordinates of a matrix $g \in G(\R)$, $\|g\| = \max_{i,j}|g_{i,j}|$, which in particular hold for any $g\in \Gamma = \bigcap_p K_p$.  
Explicitly, one obtains such bounds by  giving base change matrices $B \in GL_3(\R)$ such that
\[
B^\ast \Phi B = I.
\]
Then, $G(\R)=U_3(E,\Phi)(\R) = B \cdot U(E,I)(\R) \cdot B^{-1}$. 
Note that $\|g\| \leq 1$ for any $g \in U(E,I)(\R)$.
Therefore the absolute values of the coordinates of any $g\in G(\R) = U(E,\Phi)(\R)$, are bounded by 
\[
m(\Phi):=\max_j\{\sum_k \lvert b_{jk}\rvert\}\cdot \max_k\{\sum_j\lvert c_{jk}\rvert\},
\]
where $B=(b_{jk})_{jk}$ and $B^{-1}=(c_{jk})_{jk}$.
This information is fed to a computer to compute elements of $\Gamma$.

In case ($\MM$), we particularly find
\[
B = \bmx \frac1{\sqrt3} & -\frac{\bar\lambda}{\sqrt{21}} & \frac{\bar\lambda(\bar\lambda-3)}{7} \\ 0 & \sqrt{\frac3{7}} & \frac{2-3\bar\lambda}{7} \\ 0 & 0 & 1 \emx, \quad \mbox{ and }\quad B^{-1} = \bmx {\sqrt 3} & \frac{\bar\lambda}{\sqrt3} & \frac{\bar\lambda}{\sqrt 3} \\0 & \sqrt{\frac{7}3} & \frac{3\bar\lambda-2}{\sqrt{21}} \\ 0 & 0 & 1 \emx,
\]
from which we obtain $m(\Phi) \approx 7.46$. For an element $a+b\lambda\in \mO$ ($a,b\in\R$) the norm is given by $N(a+b\lambda)=a^2-ab+2b^2$. Looking at this as a parabola in $a$, its minimum is in $a=\frac{b}2$, so $N(a+b\lambda)\geq \frac{7}{4}b^2$. Similarly, the minimum with respect to the variable $b$ is in $b=\frac{a}{4}$, so $N(a+b\lambda)\geq \frac{7}{8}a^2$. 
In turn, in order for $a+b\lambda$ to occur as a component of $g\in U_3(\Phi,\R)$, it is necessary  that $a^2\leq \frac{8}{7}m(\Phi)^2$ and $b^2\leq \frac{4}{7}m(\Phi)^2$.
In particular, for integer components must have range $a\in\{-7,5,\dots,7\}$ and $b\in\{-5,-3,\dots,5\}$. These yields finitely many coordinates to test for the entries of $g\in \Gamma$. Fed to a computer, we find that the group $\Gamma$ is a group with $42$ elements
\[
\Gamma \cong C_2\times \tilde G,
\] 
where $C_2$ is generated by $-I$, and the group $\tilde G$ of order $21$ has a unique $7$-Sylow group $\mathrm{Syl}_7$ which is generated by 
$g=\left(\begin{smallmatrix} \lambda&1&0\\-\bar\lambda&0&1\\1&0&0\end{smallmatrix}\right)$.
The element $h=\left(\begin{smallmatrix}\bar\lambda&0&-1\\-1&1&-1\\-1&0&\lambda\end{smallmatrix}\right)$ in  $\Gamma$
generates a $3$-Sylow group $\mathrm{Syl}_3$, which acts by $h^{-1}xh\mapsto x^2$ on $\mathrm{Syl}_7$.
Altogether, this describes the group $\Gamma$ completely
\[\Gamma \cong C_2\times\bigl(\mathrm{Syl}_3\ltimes \mathrm{Syl}_7\bigr).\]

In case ($\CC$), we proceed analogy. A matrix satisfying the above is
\[
 B = \bmx \frac1{\sqrt{10}} & \frac{\eta+2}{5\sqrt 6} & \frac{\eta-2}{6\sqrt 5} \\ \frac1{\sqrt{10}} & \frac{\eta+7}{5\sqrt 6} & \frac{\eta}{2\sqrt 5} \\ \frac1{\sqrt{10}} & \frac{\eta+7}{5\sqrt 6} & \frac{\eta+2}{2\sqrt 5} \emx \quad \mbox{ and } \quad B^{-1} = \bmx \frac{2(\eta+7)}{\sqrt{10}} & \frac{-3(\eta+2)}{\sqrt{10}} & \frac{\eta+2}{\sqrt{10}} \\ -\sqrt 6 & \frac{2(\eta+4)}{\sqrt{6}} & \frac{-2(\eta+1)}{\sqrt 6} \\ 0 & -\sqrt 5 & \sqrt 5 \emx,
\]
which gives the bound $m(\Phi)\approx 15.3$.
Notice here that because $K_2$ is larger than $G(\Z_2)$, we have to allow factors $\frac12$ in some of the entries of
\[
\Gamma = \bigcap_p K_p=\left\lbrace g\in  U(E,\Phi)(\Q) \bigcap
\bmx \mO_E & \frac12\mO_E & \frac12\mO_E \\ \mO_E & \frac12\mO_E & \frac12\mO_E \\ \mO_E & \frac12\mO_E & \frac12\mO_E \emx \mid \det g=\pm 1 \right\rbrace.
\]
Here we obtain 
\[
\Gamma = \left\langle \bmx 2 & -2 & 1 \\ 3 & -2+\frac{\eta}2 & 1-\frac{\eta}2 \\ 3 & -1+\frac{\eta}2 & -\frac{\eta}2 \emx \right\rangle
= \left\langle \pm \bmx 1 & -(1+\frac{\eta}2) & \frac{\eta}2 \\ 3 & -(3+\eta) & 1+\eta \\ 3 & -(4+\eta) & 2+\eta \emx \right\rangle \cong C_6.
\]
Observe that  $G(\Z)=\{\pm I\}$.
\end{proof}

\begin{lem}\label{lem-(G,E,M,C)-(II)} 
In the notation of Theorem \ref{thm-simply-transitive}, for each of the four strong unitary datum, ($\EE$), ($\GG$), ($\MM$) and ($\CC$), property 2 of Proposition \ref{prop-(G,E,M,C)} holds, i.e. $\Gamma$ embeds (resp.\ its kernel is the center for ($\MM$)) in $G[M]$ and $H\leq G[M]$ is transversal to $\Gamma$.
\end{lem}

\begin{proof}
Case ($\EE$): Note that $G[3] = \{g\in GL_3(R)\mid g^\ast \cdot g=\mathbf 1_3\}$, where $R=\mO_E/3\mO_E\cong \mathbb F_3(t)/(t^2)$, identifying $\zeta_6 $ (the sixth root of unity in $\mO_E$) with $-(1+t)$ and complex conjugation is given by $\bar t=-t$.
By Lemma \ref{lem-(G,E,M,C)-(I)}, $\Gamma$ is the group of monomial matrices whose non-zero coefficients belongs to $\langle \zeta_6 \rangle$.
Because $\zeta_6^j\equiv 1 \mod 3$ holds if and only if $j\in 6\Z$, and because in each non-trivial permutation matrix at least one of the diagonal elements is zero, we see that the residue class homomorphism $\mod 3$ is injective on $\Gamma$, as well as $H\cap \Gamma=\{I\}$.
What is more, for each element of $G[3]$ it is possible to rearrange the columns such that every diagonal entry is a unit. 
This is done by multiplying with an appropriate permutation matrix from the right.
Then, by multiplying with a diagonal matrix from the right (the diagonal entries belonging to the units $R^\times\cong <\zeta_6>$), we obtain a matrix with all diagonal elements equal to one. This shows $G[3]=H\cdot \Gamma$.

Case ($\GG$): This was already proved in \cite{Evra2018RamanujancomplexesGolden}. 

Case ($\MM$): Since $M=(2)$ splits in $\mO_E=\Z[\lambda = \frac{-1+\sqrt{-7}}2]$, we choose zeros $\mu,\bar\mu\in\Z_2$ of $X^2+X+2$ get the idempotents $e_1=-\frac{\sqrt{-7}}{7}(\lambda-\bar\mu)$ and $\frac{\sqrt{-7}}{7}(\lambda-\mu)$
in $E_2\cong\Q_2[X]/(X^2+X+1)$, where $\lambda$ is identified with the equivalence class of $X$. The splitting of $E_2$ is explicitly given by the  $\Q_2$-algebra isomorphism $\Psi:E_2\to\Q_2\oplus \Q_2$, $e_1\mapsto (1,0)$, $e_2\mapsto (0,1)$,
in particular $\Psi(1)=(1,1)$ and $\Psi(\lambda)=(\mu,\bar\mu)$.
Imposing on $E_2$ the valuation $v$ such that $v(\lambda)=1$, there is a unique choice of valuation $v_1$ on the first component of $\Q_2\oplus\Q_2$ which is compatible with $v$, that is $v_1(\mu)=v_1\left(pr_1(\Psi(\lambda))\right)=1$.
\\
The isomorphism $G_2\cong PGL_3(\Q_2)$ is induced by $\Psi$,  as $\Psi(G_2)=\left\{(g_1,g_2): g_1\in GL_3(\Q_2) \text{ and } g_2^t=\phi g_1^{-1}\phi^{-1}\right\}$, where $\Psi(\Phi)=(\phi,\phi^t)$, followed by projection onto the first component. Because $\phi\in GL_3(\Z_2)$, we obtain $K_2\cong GL_3(\Z_2)$, and any $\Mod\lambda$-condition on $G(\Z_2)$ translates to a $\Mod 2$-condition on $GL_3(\Z_2)$. In particular $G[2] \cong GL_3(\F_2)$, and the congruence condition \eqref{eq:MM_lambda} may be replaced by \eqref{eq:MM_H}.
\\
By Lemma \ref{lem-(G,E,M,C)-(I)}, $\Gamma \cong C_2 \times C_3 \ltimes C_7$, where $C_2 \cong \langle \pm I \rangle = Z(\Gamma)$ is the center of $\Gamma$, 
$C_3 \cong \left\langle h = \bmx \bar\lambda & 0 & -1 \\ -1 & 1 & -1 \\ -1 & 0 & \lambda \emx \right\rangle$ and $C_7 \cong \left\langle g = \bmx \lambda & 1 & 0 \\ -\bar\lambda & 0 & 1 \\ 1 & 0 & 0 \emx \right\rangle$.
Then the kernel of $\Gamma$ under the $\mod 2$ map is precisely the center group $Z(\Gamma)=\{\pm I\}$.
Finally the subgroup $H$ is a $2$-Sylow subgroup in $G[2]\cong GL_3(\F_2)$ whose index is $21$, hence by size consideration, $H$ is transversal to $\Gamma \mod 2 \leq G[2]$.

Case ($\CC$): The following lemma completes the proof of Proposition \ref{prop-(G,E,M,C)} in case ($\CC$).
\end{proof}

\begin{lem}\label{lem:CMSZ-congruence-condition}
In the notation of Theorem \ref{thm-simply-transitive} case ($\CC$), we have:
(a) The group $G[M]$ is equal to 
\[
G[M] = \left\langle \bmx A & b \\ & \pm 1 \emx \mid A \in GL_2(\F_3), b \in \F_3^2 \right\rangle.
\]
(b) The kernel of $\Gamma$ modulo $M$ is trivial and the image of $\Gamma$ modulo $M$ is equal to
\[
\Gamma \mod M =  \left\langle \pm \bmx 1 & 1 & 1 \\ 0 & 1 & 0 \\ 0 & 0 &1 \emx \right\rangle =: N \leq G[M].
\]
(c) Define the following subgroups 
\[
V = \left\lbrace \bmx I_2 & b \\ & 1 \emx \mid b \in \F_3^2\right\rbrace, \quad Syl_2 = \left\langle \bmx 1 & 1 & 0 \\ -1 & 1 & 0 \\ 0 & 0 & 1 \emx, \bmx \pm I_2 & 0 \\ & 1 \emx \right\rangle.
\]
Then $H = \left\langle Syl_2, V \right\rangle$ is transversal to $N$.
\end{lem}

\begin{proof}
(a) follows from the fact that $\Phi \mod M = \bmx 0 & 0 & 0 \\ 0 & 0 & 0 \\ 0 & 0 & 1 \emx =: \bar\Phi$, and by unfolding the unitarity condition w.r.t. $\bar\Phi$. 
(b) follows from the fact that the generator $\tau=\bmx 1&-(1+\frac{\eta}2)&\frac{\eta}2\\ 3&-(3+\eta)&1+\eta\\ 3&-(4+\eta)&2+\eta \emx$ of $\Gamma$ obtained in the proof of Lemma \ref{lem-(G,E,M,C)-(I)} is equal modulo $M$  to the corresponding generator of $N$
\[
\tau\mod M \equiv \begin{pmatrix}1&1&1\\0&1&0\\0&0&1\end{pmatrix}\in N.
\]
(c) by (a), $G[M] \cong \{\pm I\} \times \left( V \rtimes GL_2(\F_3)\right)$, and the group $Syl_2$ is isomorphic to a $2$-Sylow-group of $GL_2(\F_3)$, 
which together with $\bmx 1 & 1 \\ 0 & 1 \emx$ generates $GL_2(\F_3)$. 
\end{proof}

\begin{proof}[Proof of Proposition \ref{prop-(G,E,M,C)}]
Follows from Lemmas \ref{lem-(G,E,M,C)-(I)} and \ref{lem-(G,E,M,C)-(II)}.
\end{proof}

We are finally in a position to prove our main result in this section.

\begin{proof}[Proof of Theorem \ref{thm-simply-transitive}]
Follows from combining Propositions \ref{pro-parahoric}, \ref{prop-transitivity}, \ref{prop-simplicity} and \ref{prop-(G,E,M,C)}. 
\end{proof}

We end this subsection with the following consequence of Theorem \ref{thm-simply-transitive}, which we shall use in Section \ref{section:automorphic}.

\begin{cor}\label{cor:adelic-simply-transitive}
For any one of the four cases of Theorem \ref{thm-simply-transitive}, define the open compact adelic subgroup $K'=\prod_v K_v' \leq G(\A)$, where $K_v'=K_v$ for all $v\neq m$, and $K'_m = K_m(H)$.
Then $G(\A)=G(\Q)\cdot K'$.
\end{cor}

\begin{proof}
By Lemma \ref{lem-cls1} and claim (1) of Proposition \ref{prop-(G,E,M,C)} we get that $G(\A) = G(\Q)\cdot K$, where $K = \prod_v K_v$, and by claim (2) of Proposition \ref{prop-(G,E,M,C)} we get that $G(\Q)\cdot K = G(\Q)\cdot K'$, which completes the proof.
\end{proof}



\subsection{Strong approximation} \label{subsec:lattices-strong-approx}

We end this section with some consequences of the strong approximation property, and a few auxiliary results which were used earlier in Section \ref{sec:bicayley}, and which will be used in later sections.

Let $(E,\Phi)$ be a unitary datum and $G = U_3(E,\Phi)$ the unitary group scheme.
Let $K' = G(\R) \cdot \prod_\ell K'_\ell \le G(\A)$ be an open compact subgroup, and let $\mbox{Ram}(K')$ be the finite set of primes $\ell$ such that either $K_\ell = G(\Z_\ell)$ is not hyperspecial or $K'_\ell \ne K_\ell$.
Let $p \not\in \mbox{Ram}(K')$ be a prime, let $K'^p = G(\R) \cdot \prod_{\ell \ne p} K'_\ell$ and let $\tilde{\Lambda}^p = G(\Q) \cap K'^p$.
Let $G^s = SU_3(E,\Phi) = \{g\in G\,:\, \det(g)=1\}$ be the special unitary group scheme, let $G^s_\ell = G^s(\Q_\ell)$, $K^s_\ell = K_\ell \cap G^s_\ell$ and $K'^s_\ell = K'_\ell \cap G^s_\ell$ for any prime $\ell$, and let $K'^{s,p} = G^s(\R) \cdot \prod_{\ell \ne p} K'^s_\ell$ and $\tilde{\Lambda}^{s,p} = G^s(\Q) \cap K'^{s,p}$. 
 
A prime $q$ is called {\it good} (w.r.t. $K'$) if $q \notin \mbox{Ram}(K')$.
By Hensel's Lemma, if $q$ is good then $K^s_q = G^s(\Z_q)$ projects onto $G^s(\F_q)$ via the modulo $q$ map.
Note that $G^s(\F_q) \cong SL_3(\F_q)$ if $q$ splits in $E$, and $G^s(\F_q) \cong SU_3(\F_q)$ if $q$ is inert.

\begin{prop} \label{prop:SA-G^s}
For any good prime $p$
\[
G^s(\A) = G^s(\Q) \cdot G^s_\infty \cdot G^s_p \cdot K'^{s,p},
\] 
and for any  good prime $q \ne p$,
\[
\tilde\Lambda^{s,p} \Mod{q} = G^s(\F_q).
\] 
\end{prop}

\begin{proof}
By the strong approximation theorem (see e.g. \cite{Platonov1994Algebraicgroupsand}) applied to the special unitary group $G^s$, which is a simply-connected almost-simple algebraic group, for the set $\{\infty,p\}$, for which $G_\infty \cdot G_p$ is non-compact, we get that $G^s(\Q) \cdot G^s_\infty \cdot G^s_p$ is dense in $G^s(\A)$.
The first claim follows from this by multiplying this dense set with the open subgroup $K^s_p \cdot K'^{s,p}$.
The second claim follows from multiplying this dense set with the open subgroup $K^s_p \cdot K'^{s,p}(q)$, where $K'^{s,p}(q) =  G^s_\infty \cdot \prod_{\ell \ne p,q} K'^s_\ell \cdot K^s_q(q)$ and $K^s_q(q)$ is the kernel of the modulo $q$ map on $K^s_q$, getting that $K^s_q$ is in the image of this product, and concluding by the above consequence of Hensel's Lemma.
\end{proof}

\begin{cor} \label{cor:SA-classical}
Let $p \ne q$ be good primes. 
Let $S$ be a generating set of $\tilde\Lambda^p$, $S_q = S \mod{q} \subset G(\F_q)$ and $D = \langle \det s \,:\, s \in S_q \rangle \leq \det G(\F_q)$. 
Then 
\[
\tilde\Lambda^p \Mod{q} = \{g\in G(\F_q) \,:\, \det(g) \in D\}.
\] 
Moreover, if $p$ is an inert prime, then
\[
\tilde\Lambda^p \Mod{q} = G^s(\F_q).
\] 
\end{cor}

\begin{proof}
By Proposition \ref{prop:SA-G^s}, the image of $\tilde\Lambda^p$ modulo $q$, which is a subgroup of $G(\F_q)$, contains $G^s(\F_q) = \{g\in G(\F_q)\,:\, \det g = 1\}$ as a subgroup. 
Hence, the image of $\tilde\Lambda^p$ modulo $q$, is completely determined by the values of the determinant of the elements of $\tilde\Lambda^p$, and the first claim follows since $S$ generates $\tilde\Lambda^p$.

For an inert prime $p$, we note that $\tilde\Lambda^p$ acts type-preserving on the Bruhat-Tits biregular tree.
Since  $G^s_p K_p \leq G_p$ is the subgroup of type preserving elements in $G_p$, and since $\tilde\Lambda^p \cap K_p = \{I\}$, we get that $\tilde\Lambda^p \subset G^s_p$.
combining this with the first claim, we get the second claim.
\end{proof}

Denote the projective unitary group scheme by $\bar{G} = PU_3(E,\Phi) = U_3(E,\Phi)/Z$, where $Z = \{cI\,:\,c\in U_1(E)\}$ is the center of $G$.\footnote{We note that in Sections \ref{sec:local-rep} and \ref{sec:applications}, we denote $PU_3(E,\Phi)$ by $G$.} 
As before we denote $\bar{G}_v = \bar{G}(\Q_v)$ for any place $v$, $\bar{K}_\ell = \bar{G}(\Z_\ell)$ for any prime $\ell$, for any $K'_\ell \leq K_\ell$ denote by $\bar{K}'_\ell \leq \bar{K}_\ell$, and let $\bar{K}'^p = \bar{G}_\infty \cdot \prod_{\ell \ne p} \bar{K}'_\ell$. 
We now give a variant of the strong approximation property (Proposition~\ref{prop:SA-G^s}) for the projective unitary group scheme $\bar{G}$.

\begin{prop} \label{prop:SA-adeles}
Let $\bar{G} = PU_3(E,\Phi)$ and denote $\bar{G}^0 := \bar{G}(\Q) \cdot \bar{G}_\infty \cdot \bar{G}_p \cdot \bar{K}'^p \subset \bar{G}(\A)$. 
Then $\bar{G}^0$ is a finite index normal subgroup of $\bar{G}(\A)$, and there exists $c_1 = 1, \ldots, c_h \in \bar{G}(\A)$, $h = [\bar{G}(\A) : \bar{G}^0]$, such that 
\[
\bar{G}(\A) =\bigsqcup_{i=1}^h c_i \bar{G}^0 = \bigsqcup_{i=1}^h \bar{G}(\Q) c_i \bar{G}_p \bar{K}'^p.
\]
Moreover, denote the $\Q$-algebraic abelian group $C = \det \bar{G}$, then
\[
h \leq 3^{\mbox{Ram}(K')} \cdot [C(\A) : C(\Q) \cdot C(\R \hat{\Z})] < \infty.
\]
\end{prop}

\begin{proof}
Let $G^0 := G(\Q) \cdot G_\infty \cdot G_p \cdot K'^p \subset G(\A)$.
Then by Proposition~\ref{prop:SA-G^s}, 
\[
G^s(\A) = G^s(\Q) \cdot G^s_\infty \cdot G^s_p \cdot K'^{s,p} \subset G(\Q) \cdot G_\infty \cdot G_p \cdot K'^p = G^0 \subset G(\A).
\]
Since $G^s$ is the derived subgroup of $G$ we get that $G^0$ is a normal subgroup of $G(\A)$.
Therefore $\bar{G}^0$, which is the image of $G^0$ in $\bar{G}(\A)$, is normal in $\bar{G}(\A)$.
Denote $C^0 = \det \bar{G}^0 \leq C(\A)$ and $C' = \det \bar{K}'$, and note that $h = [\bar{G}(\A) : \bar{G}^0] = [C(\A) : C^0]$ and therefore $h \leq [C(\A) : C(\Q) \cdot C'] \leq [C(\A) : C(\Q) C(\R \hat{\Z})] \cdot [C(\R \hat{\Z}) : C']$.
Note that $C(\A) / C(\Q) C(\R \hat{\Z})$ is the class group of $C$, hence it is finite.
Also note that $C(\Q_v) \cong U_1(\Q_v) /U_1(\Q_v)^3$, which is of size $\leq 3$.
Hence $[C(\R \hat{\Z}) : C'] = \prod_{\ell \in \mbox{Ram}(K')} [C(\Q_\ell) : \det \bar{K}_\ell] \leq \prod_{\ell \in \mbox{Ram}(K')} |C(\Q_\ell)| = 3^{\mbox{Ram}(K')}$, which completes the proof.
\end{proof}

We end this section with the following Lemma, which we shall use in Section~\ref{section:automorphic}.
Let $G = U_3(E,\Phi)$ and $G^* = U_3(E,J)$, where $E$ is a quadratic imaginary field and $\Phi \in GL_3(E)$ is a non-degenerate Hermitian matrix and $J = (\delta_{i,3-j})_{i,j} \in GL_3(E)$. 
By \cite{Landherr1935quivalenzHF} (or \cite{Jacobowitz1962Hermitianformsover}) the groups $G(\Q_p)$ and $G^*(\Q_p)$ are isomorphic for any prime $p$, and we shall identify them using the following Lemma in such a way that $G(\Z_p)$ is sent to $G^*(\Z_p)$, for any odd prime $p \nmid \mathrm{disc}\Phi$.

\begin{lem} \label{lem:Landherr}
\begin{itemize}
\item [(a)] The groups $G(\Z_p)$   and $G^*(\Z_p)$ are isomorphic if and only if $G(\Z_p)$ is a hyperspecial maximal compact subgroup of $G(\Q_p)$. 
If so, there is a matrix $A\in GL_3(E_p)$ such that an isomorphism is given by $g\mapsto A^{-1}gA$.
In this case, because  changing $\Phi$ by some factor doesn't change the unitary group, we may assume $\Phi\in GL_3(\mO_p)$.
\item[(b)] Assume $\Phi\in M_3(\mO_p)$, and  $p^{-1}\Phi\notin M_3(\mO_p)$. Then $G(\Z_p)$ is hyperspecial if and only if $p \nmid \det(\Phi)$.
\item[(c)] Let $p$ be either an odd prime or a split prime in $E$. If $G(\Z_p)$ is hyperspecial, then the matrix $A$ can be chosen such that it belongs to $GL_3(\mO_p)$ canonically, and such that $A^\ast\Phi A =J$.
\end{itemize}
\end{lem}

\begin{rem}
In the remaining cases of either a ramified prime $p$, or $p=2$ inert and $G(\Z_p)$ hyperspecial, we will give explicit matrices $A\in GL_3(\mO_p)$ with $A^\ast\Phi A=J$ when needed.
\end{rem}

\begin{proof}
(a) There is a matrix $B\in GL_3(E_p)$ inducing the isomorphism $G(\Q_p)\to G^\ast(\Q_p)$, $g\mapsto B^{-1}gB$. The compact subgroup $B^{-1} G(\Z_p)B$ of $G^\ast(\Q_p)$ is conjugate to $G^\ast(\Z_p)$ if and only if it is hyperspecial, i.e. the stabilizer of the  lattice chain given by the multiples of the standard lattice $L_0=\mO_p^3$. So there is $C\in GL_3(E_p)$ such that $(BC)^{-1} G(\Z_p)BC=G^\ast(\Z_p)$. 
Then $A=BC$ is as claimed. 
In this case, it holds $\Phi L_0=\mu L_0$ for some $\mu\in \Q_p^\ast$, so $\mu^{-1}\Phi\in \mathrm{Stab}_{GL_3(E_p)}(L_0)=GL_3(\mO_p)$.
 
(b) If $(p,\det(\Phi))=1$, then $\Phi\in GL_3(\mO_p)$ stabilizes  $L_0$, i.e. $L_0$ is selfdual with respect to the Hermitian form given by $\Phi$. So $G(\Z_p)=\mathrm{Stab}_{U_3(\Phi,\Q_p)}(L_0)$ is hyperspecial.
In turn, if $G(\Z_p)$ is hyperspecial, then by (a), $\Phi L_0=\mu L_0$ for some $\mu\in \Q_p$. Because $\Phi\in M_3(\mO_p)$, we have $v_p(\mu)\geq 0$. Because $\mu^{-1}\Phi\in GL_3(\mO_p)$, by assumption we obtain $v_p(\mu)=0$. 
So actually $\Phi\in GL_3(\mO_p)$.
 
(c) We may assume $\Phi \in GL_3(\mO_p)$. 
In case $p$ splits in $E$, i.e. $E_p\simeq\Q_p\oplus\Q_p$, let $\Phi=(\phi,\phi^T)$ be the corresponding decomposition with $\phi\in GL_3(\Z_p)$. Then $(g_1,g_2)\in G(\Z_p)$ if and only if $g_1\in GL_3(\Z_p)$ and $g_2^T=\phi g_1^{-1}\phi^{-1}$. 
Choosing $A=(\phi^{-1}J,\mathbf{1_3})\in GL_3(\mO_p)$, we obtain $A^\ast\Phi A=(J,J)$ as well as $A^{-1}G(\Z_p)A=G^\ast(\Z_p)$.
 
In case $p$ is non-split in $E$, the manipulations in \cite{Landherr1935quivalenzHF} can be chosen such that they correspond to base changes of $\mO_p^3$, so there is $D\in GL_3(\mO_p)$ such that $D^\ast\Phi D=diag(a,b,c)$ is a diagonal matrix. 
Because the norm $\mO_p^\times\to\Z_p^\times$ is surjective for $p$ inert in $E$, we may assume $a=b=c=1$ with respect to the basis $e_1,e_2,e_3$. 
Choosing $\mu,\nu\in\mO_p^\times$ such that $\mu\bar\mu=-1$ and $\nu\bar\nu=\frac12$, we obtain that with respect to the basis $\nu(e_1+\mu e_3),\mu e_2, \nu(e_1-\mu e_3)$ of $\mO_p^3$, the Hermitian form  becomes $J$. 
This amounts to changing $D$ further by the base change matrix to obtain $A$.
\end{proof}

\section{Local Representation Theory} \label{sec:local-rep}

In this Section we develop the representation theoretic machinery needed to analyze arithmetic quotients of the Bruhat-Tits tree of unramified $p$-adic unitary groups. 
In Section \ref{subsec:Local-rank-one-representation} we describe the irreducible, Iwahori-spherical representations of these groups. 
In Section \ref{subsec:Spectral-theory-revisited} we relate the representation theory to the spectral theory of bigraphs developed in Section \ref{sec:Spectral-analysis}, and use this to give a local criterion for quotients of the Bruhat-Tits tree to be (adj-)Ramanujan. 
In Section \ref{subsec:Ramanujan-global-criterion} we lift this criterion to a global criterion on adelic representations, which will be used in Sections \ref{section:automorphic} and \ref{sec:applications}.

\subsection{Representations of $PU_3$ over local fields} \label{subsec:Local-rank-one-representation}

Let $E/F$ and $\Phi$ be as in Section \ref{subsec:unitarygroups}, and let $v$ be a finite place of $F$ which is inert in $E$. 
In this section we focus on the group $G=PU_3(E,\Phi)(F_v)$, and since all non-degenerate Hermitian forms on $E_v^3$ give rise to isomorphic unitary groups \cite{Landherr1935quivalenzHF,Jacobowitz1962Hermitianformsover}), we let ourselves assume that $\Phi=J=\left(\begin{smallmatrix} &  & 1\\ & -1\\ 1 \end{smallmatrix}\right)$, which makes computations considerably simpler.
We shall use the action of $G$ on its Bruhat-Tits tree $\B$ to classify the irreducible unitary Iwahori-spherical (I.S.) representations of $G$. 
This classification is not new (it appears in \cite{Ballantine2011Ramanujanbigraphsassociated}, and related results appear also in \cite{borel1976admissible,Moy1986RepresentationsU21,hashimoto1989zeta,Lusztig1989AffineHeckealgebras}), but we give an elementary analysis which in addition highlights the connection to the spectral theory of bigraphs developed in Section \ref{sec:Spectral-analysis}. 
We remark that the representations of $G=PU_3$ are the representations of $U_3$ with trivial central character, and they suffice for our purposes as the center of $U_3$ acts trivially on $\B$. 

Denote $\F=F_v$, $\E=E_v$, $\kappa_{\F}=\mO_{\F}/\varpi$, $\kappa_{\E}=\mO_{\E}/\varpi$ (for a uniformizer $\varpi\in\F$), and let $q=\left|\kappa_{\F}\right|$ (so that $\left|\kappa_{\E}\right|=q^2$).\footnote{In fact, the results of this section apply to local fields $\F,\E$ of positive characteristic as well.}
Recall that $\B$ was described in Section \ref{subsec:unitarygroups} as the fixed-section of the involution $\#$ of the Bruhat-Tits building $\widetilde{\B}$ of $PGL_3(\E)$. Since $J\in GL_3(\mO_{\E})$, the $\widetilde{\B}$-vertex $v_0$ is a hyperspecial vertex in $\B$ with stabilizer $K=G\cap PGL_3(\mO_{\E})$. 
The midpoint of the $\widetilde{\B}$-edge
$\left(\begin{smallmatrix}1\\ & 1\\  &  & \varpi \end{smallmatrix}\right)v_0-\left(\begin{smallmatrix}1\\  & \varpi\\  &  & \varpi \end{smallmatrix}\right)v_0$ is a non-hyperspecial $\B$-vertex, which we denote by $v_1$. 
We denote by $K'$ the stabilizer of $v_1$ in $G$, and 
\[
\boldsymbol{I} = K\cap K' = G\cap\left(\begin{smallmatrix}\mO_{\E}^{\times} & \mO_{\E} & \mO_{\E}\\ \varpi\mO_{\E} & \mO_{\E}^{\times} & \mO_{\E}\\ \varpi\mO_{\E} & \varpi\mO_{\E} & \mO_{\E}^{\times} \end{smallmatrix}\right)
\]
is the $G$-stabilizer of the $\B$-edge $e_0$, which connects $v_0$ and $v_1$; $\boldsymbol{I}$ is an Iwahori subgroup of $G$ (see \cite[Lem.\ 5.4-5.6]{Evra2018RamanujancomplexesGolden}\footnote{contrary to popular belief, the stabilizer of a chamber in a Bruhat-Tits building is not always Iwahori - see \cite[Lecture 2]{Yu2009BruhatTitstheory}.}). As $\E/\F$ is unramified, there exists $\varepsilon\in\mO_{\E}^{\times}$ with $Tr_{\E/\F}\left(\varepsilon\right)=0$. The upper-triangular matrices in $G$ form a Borel group, which can be parameterized by
\[
P:=\left\{ p_{\alpha,x,y}:=\overline{\alpha}^{-1}\left(\begin{matrix}\alpha\overline{\alpha} & \alpha x & x\overline{x}/2+y\varepsilon\\ & \alpha & \overline{x}\\ &  & 1 \end{matrix}\right)\,\middle|\,\alpha\in \E^{\times},x\in \E,y\in \F\right\} .
\]

The Borel group $P$ acts transitively on hyperspecial vertices, and the $2$-sphere around $v_0$ in $\B$ is composed of $q^{4}+q$ vertices, which are the translations of $v_0$ by 
\begin{equation} \label{eq:Sp-borel}
\mS_p:=\left\{ p_{\varpi,x,y}\,\middle|\,x\in\kappa_{\E},\ y\in\nicefrac{\mO_{F}}{\varpi^2}\right\} \cup\left\{ p_{1,0,y/\varpi}\,\middle|\,y\in\kappa_{\F}^{\times}\right\} \cup\left\{ p_{1/\varpi,0,0}\right\} .
\end{equation}

The representations which concern us arise from characters (one-dimensional representations) of $P$. 
For $z\in\C^{\times}$, denote by $\chi_z:P\rightarrow\C^{\times}$
the unramified character $\chi_z\left(p_{\alpha,x,y}\right)=z^{\ord_{\varpi}\alpha}$.
The modular character of $P$ is $\Delta_{P}\left(p_{\alpha,x,y}\right)=q^{4\ord_{\varpi}\alpha}$,
and we denote $\widetilde{\chi_z}=\chi_z\cdot\Delta_{P}^{-1/2}$,
namely $\widetilde{\chi_z}\left(p_{\alpha,x,y}\right)=\left(z/q^2\right)^{\ord_{\varpi}\alpha}$.
We denote the normalized parabolic induction of $\chi_z$ by
\begin{equation} \label{eq:Vz-def}
V_z:=\mbox{n-ind}_P^G \chi_z=\left\{ f\in C^{\infty}\left(G,\C\right)\,\middle|\,\forall p\in P,g\in G:f\left(pg\right)=\widetilde{\chi_z}\left(p\right)f\left(g\right)\right\},
\end{equation}
on which $G$ acts by $(gf)(x)=f(xg)$.
If $W$ is an irreducible representation of $G$ which embeds in $V_z$, then $z$ is called a \emph{Satake parameter} for $W$.

Let $\mathcal{H}_{\boldsymbol{I}}=C_{c}\left(\boldsymbol{I}\backslash G/\boldsymbol{I}\right)$
be the Iwahori-Hecke algebra of $G$, namely, compactly supported
doubly-$\boldsymbol{I}$-invariant functions on $G$, w.r.t.\ convolution.
For any I.S.\ representation $V$ of $G$, its space of Iwahori-fixed
vectors $V^{\boldsymbol{I}}$ is an $\mathcal{H}_{\boldsymbol{I}}$-module
via $\varphi v=\int_{G}\varphi(g)gv\,dg$, and we have the following:
\begin{thm}
\cite{borel1976admissible,casselman1980unramified,Barbasch2013Unitaryequivalencesreductive}
Let $V,W$ be representations of $G$. 
\begin{enumerate}
\item If $V$ is irreducible then it is I.S.\ if and only if it embeds in $V_z$
for some $z\in\C^{\times}$.
\item If $V$ is generated by $V^{\boldsymbol{I}}$ then $V$ is irreducible
if and only if $V^{\boldsymbol{I}}$ is.
\item If $V,W$ are generated by $V^{\boldsymbol{I}}$, $W^{\boldsymbol{I}}$
then $V\cong_{G}W$ if and only if $V^{\boldsymbol{I}}\cong_{\mathcal{H}_{\boldsymbol{I}}}W^{\boldsymbol{I}}$.
\item If $V$ is irreducible and I.S.\ then it is unitary if and only if $V^{\boldsymbol{I}}$
is (where $\mathcal{H}_{\boldsymbol{I}}$ is a $*$-algebra by $\varphi^{*}\left(g\right)=\overline{\varphi\left(g^{-1}\right)}$).
\end{enumerate}
\end{thm}

In particular, every I.S.\ irreducible representation has Satake parameter(s), and we denote by $Sat(W)$ the set of Satake parameters of $W$. We now
show the following:
\begin{prop}\label{prop:satake}
The Satake parametrization identifies the I.S.\ dual of $G$ with
the non-Hausdorff space 
\[
\frac{\C\backslash\left\{ -q^{\pm1},0,q^{\pm2}\right\} }{z\sim1/z}\cup\left\{ -q^{\pm1},q^{\pm2}\right\} ,
\]
and the unitary irreducible representations are those with Satake parameter in $S^1\cup\left[-q,-\tfrac1{q}\right]\cup\left[\tfrac1{q^2},q^2\right]$.
\end{prop}

\begin{proof}
The Weyl element $w=\left(\begin{smallmatrix} &  & 1\\
 & 1\\
1
\end{smallmatrix}\right)\in G$ reflects $e_0$ around $v_0$, taking $v_1$ to the midpoint
of $\diag(\varpi,1,1)v_0-\diag(\varpi,\varpi,1)v_0$. Similarly,
$s=\left(\begin{smallmatrix} &  & 1/\varpi\\
 & 1\\
\varpi
\end{smallmatrix}\right)\in G$ reflects $e_0$ around $v_1$, taking $v_0$ to $p_{\frac1{\varpi},0,0}v_0$.
The algebra $\mathcal{H}_{\boldsymbol{I}}$ is generated by $\tau=\one_{\boldsymbol{I}w\boldsymbol{I}}$
and $\sigma=\one_{\boldsymbol{I}s\boldsymbol{I}}$. Observing the
action of $G$ on edges we have
\[
\boldsymbol{I}w\boldsymbol{I}=\bigsqcup\nolimits _{{x\in\kappa_{\E}\atop y\in\kappa_{\F}}}p_{1,x,y}w\boldsymbol{I},\qquad\boldsymbol{I}s\boldsymbol{I}=p_{\frac1{\varpi},0,0}w\boldsymbol{I}\sqcup\bigsqcup\nolimits _{y\in\kappa_{\F}^{\times}}p_{1,0,\frac{y}{\varpi}}\boldsymbol{I},
\]
and it follows that for any representation $V$ of $G$, $\mathcal{H}_{\boldsymbol{I}}$
acts on $V^{\boldsymbol{I}}$ by
\begin{equation}
\tau v=\sum\nolimits _{{x\in\kappa_{\E}\atop y\in\kappa_{\F}}}p_{1,x,y}wv,\qquad\sigma v=p_{\frac1{\varpi},0,0}wv+\sum\nolimits _{y\in\kappa_{\F}^{\times}}p_{1,0,\frac{y}{\varpi}}v\label{eq:tau-sigma}
\end{equation}
(here we have fixed Haar measures satisfying $\mu_{G}\left(\boldsymbol{I}\right)=\mu_{P}\left(P\cap\boldsymbol{I}\right)=1$).
For $V=V_z$, the Iwasawa decomposition $G=PK=P\boldsymbol{I}\sqcup Pw\boldsymbol{I}$
implies that $f\in V_z^{\boldsymbol{I}}$ is determined by $f(1),f(w)$;
upon verifying that the functions $f_1,f_{w}$ given by
\[
\begin{matrix}f_1\left(pb\right)= & \widetilde{\chi_z}\left(p\right), &  &  & f_1\left(pwb\right)= & 0,\\
f_{w}\left(pb\right)= & 0, &  &  & f_{w}\left(pwb\right)= & \widetilde{\chi_z}\left(p\right)/q^{3/2},
\end{matrix}\qquad\left(p\in P,b\in\boldsymbol{I}\right)
\]
are well defined, one obtains that $\left\{ f_1,f_{w}\right\} $
forms a basis for $V_z^{\boldsymbol{I}}$ (and we shall later see
it is orthonormal when $\left|z\right|=1$). Denoting $f^{K}:=f_1+q^{3/2}f_{w}$
we have that $wf^{K}=f^{K}$; thus, $K=\boldsymbol{I}\sqcup\boldsymbol{I}w\boldsymbol{I}$
implies that $f^{K}$ spans $V_z^{K}$. Similarly $K'=\boldsymbol{I}\sqcup\boldsymbol{I}s\boldsymbol{I}$,
and $f^{K'}=f_1+\frac{z}{\sqrt{q}}f_{w}$ is $s$-fixed so that
it spans $V_z^{K'}$. Fixing the basis $\left\{ f_1,f_{w}\right\} $,
the structure map $\rho=\rho_z\colon\mathcal{H}_{\boldsymbol{I}}\rightarrow End\left(V_z^{\boldsymbol{I}}\right)$
can be directly computed from (\ref{eq:Vz-def}) and (\ref{eq:tau-sigma}),
yielding
\begin{equation}
\rho_z\left(\tau\right)=\left(\begin{matrix}0 & q^{3/2}\\
q^{3/2} & q^3-1
\end{matrix}\right),\quad\rho_z\left(\sigma\right)=\left(\begin{matrix}q-1 & \frac{\sqrt{q}}{z}\\
z\sqrt{q} & 0
\end{matrix}\right).\label{eq:rho-sigma-tau}
\end{equation}
The eigenvectors of $\rho_z\left(\tau\right)$ are $f^{K}\negthickspace=\negthickspace\left(\begin{smallmatrix}1\\
q^{3/2}
\end{smallmatrix}\right),\left(\begin{smallmatrix}1\\
-q^{-3/2}
\end{smallmatrix}\right)$ and of $\rho_z\left(\sigma\right)$ are $f^{K'}\negthickspace=\negthickspace\left(\begin{smallmatrix}1\\
z/\sqrt{q}
\end{smallmatrix}\right),\left(\begin{smallmatrix}1\\
-z\sqrt{q}
\end{smallmatrix}\right)$, so $V_z^{\boldsymbol{I}}$ (and thus $V_z$) is reducible if
and only if $z\in\left\{ q^{\pm2},-q^{\pm1}\right\} $. Denoting by
$W_z$ the unique irreducible I.S.\ subrepresentation of $V_z$,
we obtain the four representations $W_{q^{\pm2}},W_{-q^{\pm1}}$, which
are described in Table \ref{tab:subreps}; these are the trivial and
Steinberg representations, and two other representations which we
call A and B.\footnote{In \cite{Ballantine2011Ramanujanbigraphsassociated} these four are
called sph, St, nt and ds respectively. We chose the name ``A''
as these are the components of Rogawski's A-packets (see Section \ref{section:automorphic}),
and ``B'' accordingly. In \cite[\S5.8(ii)]{borel1976admissible} A
and B appear as $\left(1,-1\right)$ and $\left(-1,1\right)$ (under
$l=1$).}

\begin{table}[h]
\hspace*{-2cm}
\begin{tabular}{|c|c|c|c|c|c|c|c|}
\hline 
name & $z$ (Satake) & basis for $W_z^{\boldsymbol{I}}$ & $\rho_z\left(\tau\right)$ & $\rho_z\left(\sigma\right)$ & $\rho_z\left(\sigma\tau\right)$ & $K$-spher & $K'$-spher\tabularnewline
\hline 
\hline 
Trivial & $q^2$ & $f^{K}=f^{K'}=f_1+q^{3/2}f_{w}$ & $q^3$ & $q$ & $q^{4}$ & \Checkmark{} & \Checkmark{}\tabularnewline
\hline 
Steinberg & $\frac1{q^2}$ & $f_1-q^{3/2}f_{w}$ & $-1$ & $-1$ & $1$ & \XSolidBrush{} & \XSolidBrush{}\tabularnewline
\hline 
A-type & $-q$ & $f^{K}=f_1+q^{3/2}f_{w}$ & $q^3$ & $-1$ & $-q^3$ & \Checkmark{} & \XSolidBrush{}\tabularnewline
\hline 
B-type & $-\frac1{q}$ & $f^{K'}=f_1-q^{-3/2}f_{w}$ & $-1$ & $q$ & $-q$ & \XSolidBrush{} & \Checkmark{}\tabularnewline
\hline 
\end{tabular}
\par
\caption{\label{tab:subreps}The irreducible representations $W_z$ of $PU_3$
with $\dim W_z^{\boldsymbol{I}}=1$.}
\end{table}

We move to the irreducible $V_z$, i.e.\ $z\not\in\left\{ q^{\pm2},-q^{\pm1}\right\} $, for which $W_z=V_z$ and $\dim V_z^{\boldsymbol{I}}=2$. The
eigenvalues of 
\[
\rho_z\left(\sigma\tau\right)=\left(\begin{matrix}\frac{q^2}{z} & \sqrt{q}\left(\frac{q^3-1}{z}+q^2-q\right)\\
0 & q^2z
\end{matrix}\right)
\]
are $zq^2,q^2/z$, so  the only possible isomorphism between
two such representations is $V_z\cong V_{1/z}$. They are indeed
isomorphic: 
\begin{equation}
Q=\left(\begin{matrix}\left(q^2z+qz+q+z\right)(q-1)z & q^{3/2}\left(1-z^2\right)\\
q^{3/2}\left(1-z^2\right) & \left(q^2+qz+q+1\right)(q-1)
\end{matrix}\right)\label{eq:P}
\end{equation}
gives $Q\rho_z\left(\tau\right)Q^{-1}=\rho_{1/z}\left(\tau\right)$
and $Q\rho_z\left(\sigma\right)Q^{-1}=\rho_{1/z}\left(\sigma\right)$. 

Next comes the matter of unitarity. We have $\tau^{*}=\overline{\one_{\left(\boldsymbol{I}w\boldsymbol{I}\right)^{-1}}}=\tau$,
and likewise $\sigma^{*}=\sigma$. Assume first that $V\cong V_z$
for some $z\notin\left\{ q^{\pm2},-q^{\pm1}\right\} $. Since $\tau\sigma+\sigma\tau$
is self-adjoint, if $V$ is unitary then 
\[
\rho_z\left(\tau\sigma+\sigma\tau\right)=\left(\begin{array}{cc}
q^2\left(z+\frac1{z}\right) & \frac{\sqrt{q}\left(q^3+zq^2-qz-1\right)}{z}\\
\sqrt{q}\left(q^3z+q^2-q-z\right) & q^2\left(z+\frac1{z}\right)
\end{array}\right)
\]
is self-adjoint, and in particular has real trace, so that $z\in S^1\cup\R^{\times}$.
If $z\in S^1$, then $\chi_z$ is unitary, hence so is $V_z$
with the inner product $\left\langle f,f'\right\rangle =\int_{P\backslash G}f\left(x\right)\overline{f'\left(x\right)}dx$;
these $V_z$ are called the \emph{(unitary) principal series}. For $V_z^{\boldsymbol{I}}$ with
$z\in\R^{\times}\backslash\left\{ q^{\pm2},-q^{\pm1}\right\}$ to be unitary,
we need a Hermitian matrix $H\in GL_2\left(\C\right)$ such that $\rho=\rho_z$ is a $*$-homomorphism with respect to $X^{*}:=H^{-1}\overline{X}^{t}H$, namely:
\begin{align*}
\left(\begin{smallmatrix}0 & q^{3/2}\\
q^{3/2} & q-1
\end{smallmatrix}\right) & =\rho\left(\tau\right)=\rho\left(\tau^{*}\right)=\rho\left(\tau\right)^{*}=H^{-1}\left(\begin{smallmatrix}0 & q^{3/2}\\
q^{3/2} & q-1
\end{smallmatrix}\right)H\\
\left(\begin{smallmatrix}q-1 & \sqrt{q}/z\\
z\sqrt{q} & 0
\end{smallmatrix}\right) & =\rho\left(\sigma\right)=\rho\left(\sigma^{*}\right)=\rho\left(\sigma\right)^{*}=H^{-1}\left(\begin{smallmatrix}q-1 & \overline{z}\sqrt{q}\\
\sqrt{q}/\overline{z} & 0
\end{smallmatrix}\right)H.
\end{align*}
The unique solution (up to scaling) is $H=Q$ from (\ref{eq:P}), which is invertible as
\begin{equation} \label{eq:Qdet}
\det Q=-q^3\left(z-q^2\right)\left(z-\tfrac1{q^2}\right)\left(z+q\right)\left(z+\tfrac1{q}\right).    
\end{equation}
This means that $V_z^{\boldsymbol{I}}$ admits a Hermitian
structure for all $z\in\R^{\times}\backslash\left\{ q^{\pm2},-q^{\pm1}\right\}$; it is unitary if in
addition $H$ is definite, which by \eqref{eq:Qdet} happens if and only if $z\in\left(-q,-\frac1{q}\right)\cup\left(\frac1{q^2},q^2\right)$;
these $V_z$ form the \emph{complementary series}. Finally, for
$z\in\left\{ q^{\pm2},-q^{\pm1}\right\} $, the one-dimensional representation
$W_z^{\boldsymbol{I}}$ is unitary since $\rho_z\left(\sigma\right),\rho_z\left(\tau\right)\in\R$
(and $\sigma^{*}=\sigma,\tau^{*}=\tau$).
\end{proof}
We shall also need the notion of temperedness: An irreducible representation $W$
is called tempered if it weakly contained in $L^2\left(G\right)$,
which is equivalent to its matrix coefficients being in $L^{2+\varepsilon}\left(G\right)$
for every $\varepsilon>0$ \cite{Haagerup1988}. A criterion of Casselman
for temperedness is that all Satake parameters of $W$ satisfy $\left|z\right|\leq1$
(see \cite[Def.\ 3.3]{Ballantine2011Ramanujanbigraphsassociated},
where $\chi$ denotes the Satake parameter), and Table \ref{tab:rep-spec}
indicates which $W_z$ satisfy this.

\subsection{Spectral theory revisited} \label{subsec:Spectral-theory-revisited}

If $\Lambda$ is a cocompact lattice in $G$ which acts on $\B$
without fixed points, then $X=X_{\Lambda}=\Lambda\backslash\B$
is a finite $(q^3+1,q+1)$-bigraph, and our next goal is to relate
the analysis of $A=A_{X}$ and $B=B_{X}$ in Section \ref{sec:Spectral-analysis}
to the representation theory of $G$. 
\begin{prop}\label{prop:Bspec_reps}
If $L^2\left(\Lambda\backslash G\right)=\bigoplus_iW_i\oplus\widehat{\bigoplus}_iU_i$
is the decomposition of $L^2\left(\Lambda\backslash G\right)$ as
a $G$-rep., where $W_i$ are the Iwahori spherical components and
the $U_i$ are the rest, then
\begin{equation}
\Spec\left(B_{X}\right)=\bigcup\nolimits _i\left\{ \pm q\sqrt{z}\,\middle|\,z\in\mathrm{Sat}\left(W_i\right)\right\} .\label{eq:spec-B}
\end{equation}
\end{prop}

\begin{proof}
We note that $G$ acts on $E_{\B}$ (the directed edges in
$\B$) with two orbits, $E_{\B}=Ge_0\sqcup Ge_0'$
where $e_0=v_0\shortrightarrow v_1$ and $e_0'=v_0\shortleftarrow v_1$.
Both $e_0$ and $e'_0$ have stabilizer $\boldsymbol{I}$, so
that $E_{\B}\cong G/\boldsymbol{I}\sqcup G/\boldsymbol{I}$
as $G$-sets, and thus $E_{X}=\Lambda\backslash E_{\B}\cong\Lambda\backslash G/\boldsymbol{I}\sqcup\Lambda\backslash G/\boldsymbol{I}$.
We obtain an identification 
\begin{equation}
L^2\left(E_{X}\right)\cong L^2\left(\Lambda\backslash G\right)^{\boldsymbol{I}}\oplus L^2\left(\Lambda\backslash G\right)^{\boldsymbol{I}},\label{eq:L2E-reps}
\end{equation}
where $\left(f,f'\right)$ on the r.h.s.\ corresponds to $\big\{{\Lambda ge_0\mapsto f\left(\Lambda g\right)\phantom{'}\atop \Lambda ge_0'\mapsto f'(\Lambda g)}\big\}$
on the l.h.s. The r.h.s.\ is naturally a $M_2(\mathcal{H}_{\boldsymbol{I}})$-module,
and the crux of the matter is that the operator $B=B_{X}$ on $L^2\left(E_{X}\right)$
corresponds under (\ref{eq:L2E-reps}) to an element of $M_2(\mathcal{H}_{\boldsymbol{I}})$
-- specifically, to $\left(\begin{smallmatrix}0 & \sigma\\
\tau & 0
\end{smallmatrix}\right)$. From the decomposition $L^2\left(\Lambda\backslash G\right)=\bigoplus_iW_i\oplus\widehat{\bigoplus}_iU_i$
we obtain an orthogonal decomposition of $M_2\left(\mathcal{H}_{\boldsymbol{I}}\right)$-modules
\begin{equation}
L^2\left(E_{X}\right)\cong L^2\left(\Lambda\backslash G\right)^{\boldsymbol{I}}\oplus L^2\left(\Lambda\backslash G\right)^{\boldsymbol{I}}=\bigoplus\nolimits _iW_i^{\boldsymbol{I}}\oplus W_i^{\boldsymbol{I}}\label{eq:L2E-decomp}
\end{equation}
(note there are finitely many $W_i$ since $\left|E_{X}\right|<\infty$).
Each summand $W_i^{\boldsymbol{I}}\oplus W_i^{\boldsymbol{I}}$
corresponds via (\ref{eq:L2E-decomp}) to a two- or four-dimensional
subspace of $L^2\left(E_{X}\right)$ which is invariant under $B$,
and furthermore, the isomorphism type of $W_i$ already determines
the spectrum of $B=\left(\begin{smallmatrix}0 & \sigma\\
\tau & 0
\end{smallmatrix}\right)$ on this subspace. We obtain: 
\begin{align*}
\Spec\left(\rho_{W_z^{\boldsymbol{I}}\oplus W_z^{\boldsymbol{I}}}\left(\begin{smallmatrix}0 & \sigma\\
\tau & 0
\end{smallmatrix}\right)\right) & =\Spec\left(\begin{matrix}0 & \rho_{W_z^{\boldsymbol{I}}}(\sigma)\\
\rho_{W_z^{\boldsymbol{I}}}(\tau) & 0
\end{matrix}\right)\\
 & =\left.\begin{cases}
\Spec\left(\begin{smallmatrix}0 & 0 & q-1 & \sqrt{q}/z\\
0 & 0 & z\sqrt{q} & 0\\
0 & q^{3/2} & 0 & 0\\
q^{3/2} & q^3-1 & 0 & 0
\end{smallmatrix}\right) & \dim W_z^{\boldsymbol{I}}=2\\
\pm\sqrt{\rho_z(\sigma\tau)} & \dim W_z^{\boldsymbol{I}}=1
\end{cases}\right\}
\\&=\left\{ \pm q\sqrt{z}\,\middle|\,z\in\mathrm{Sat}\left(W_z\right)\right\} ,
\end{align*}
using (\ref{eq:rho-sigma-tau}) for $\dim W_z^{\boldsymbol{I}}=2$
(i.e.\ $z\notin\left\{ q^{\pm2},-q^{\pm1}\right\} $), and Table
\ref{tab:subreps} otherwise. In total, from $L^2\left(\Lambda\backslash G\right)=\bigoplus_iW_i\oplus\widehat{\bigoplus}_iU_i$
we obtain \eqref{eq:spec-B}. 
\end{proof}
The block decomposition of $L^2\left(E_{X}\right)$ in (\ref{eq:L2E-decomp})
agrees with that appearing in Theorem \ref{thm:B-decomp}, and Table
\ref{tab:rep-spec} gives the precise connection between the two presentations.
We note that the Satake parameter is slightly stronger than the parameter
$\vartheta$ used in Section \ref{sec:Spectral-analysis}, as the
latter does not distinguish between $W_z$ and $W_{1/z}$. 

\begin{table}[h]
\hspace*{-2cm}
\begin{tabular}{|c|c|c|c|c|c|c|c|}
\hline 
Type & $z$ (Satake) & temp & $\vartheta$ & $\negthickspace\negmedspace{\scriptscriptstyle \lambda=\pm\sqrt{q^3+\big(z+\tfrac1{z}\big)q^2+q}}\negthickspace\negmedspace$ & $\mu=\pm q\sqrt{z}$ & \# & Thm.\ \ref{thm:B-decomp}\tabularnewline
\hline 
\hline 
triv. & $q^2$ & \XSolidBrush{} & \multirow2{*}{$-i\log q^2$} & ${\scriptstyle \pm\mathfrak{pf}=\pm\sqrt{(q^3+1)(q+1)}}$ & $\pm q^2$ & 1 & \emph{(1)}\tabularnewline
\cline{1-3} \cline{2-3} \cline{3-3} \cline{5-8} \cline{6-8} \cline{7-8} \cline{8-8} 
Stein. & $\frac1{q^2}$ & $\checkmark$ &  & none & $\pm1$ & $\chi(X)$ & \emph{(5)}\tabularnewline
\hline 
prin. & $z^{\pm1}\in S^1$ & $\checkmark$ & $\left[0,\pi\right]$ & ${\scriptstyle \pm\left[q^{3/2}-\sqrt{q},q^{3/2}+\sqrt{q}\right]}$ & $qS^1$ & \multirow3{*}{${\scriptstyle n-1-\mE_X}\!$} & \emph{(2)(a)}\tabularnewline
\cline{1-6} \cline{2-6} \cline{3-6} \cline{4-6} \cline{5-6} \cline{6-6} \cline{8-8} 
\multirow2{*}{comp.} & $z^{\pm1}\in{\scriptstyle \left(\frac1{q^2},q^2\right)\backslash\{1\}}$ & \XSolidBrush{} & ${\scriptstyle \left(-i\log q^2,0\right)}$ & $\pm(q^{3/2}+\sqrt{q},\mathfrak{pf})$ & ${\scriptstyle \pm\left(\left(1,q^2\right)\backslash\{q\}\right)}$ &  & \multirow2{*}{\emph{(2)(b)}}\tabularnewline
\cline{2-6} \cline{3-6} \cline{4-6} \cline{5-6} \cline{6-6} 
 & $z^{\pm1}\in{\scriptstyle \left(-q,-\frac1{q}\right)\backslash\{-1\}}$ & \XSolidBrush{} & ${\scriptstyle \left(\pi,\pi+i\log q\right)}$ & $\pm(0,q^{3/2}-\sqrt{q})$ & ${\scriptstyle \negthickspace\negmedspace\pm\left(\left(i\sqrt{q},iq^{3/2}\right)\backslash\{iq\}\right)\negthickspace\negmedspace}$ &  & \tabularnewline
\hline 
A-type & $-q$ & \XSolidBrush{} & \multirow2{*}{$\pi\!+\!i\log q$} & $0$ ($\ker A\big|_{L}$) & $\pm iq^{3/2}$ & $\mE_X$ & \emph{(3)}\tabularnewline
\cline{1-3} \cline{2-3} \cline{3-3} \cline{5-8} \cline{6-8} \cline{7-8} \cline{8-8} 
B-type & $-\frac1{q}$ & $\checkmark$ &  & $0$ ($\ker A\big|_{R}$) & $\pm i\sqrt{q}$ & $\mathcal{N}_X$ & \emph{(4)}\tabularnewline
\hline 
\end{tabular}
\par
\caption{\label{tab:rep-spec}The unitary Iwahori-spherical irreducible representations of unramified
$U_3$, their Satake parameter(s) $z$, temperedness, the corresponding
$\vartheta$ parameter from Section \ref{sec:Spectral-analysis},
adjacency eigenvalues $\lambda$ and non-backtracking eigenvalues
$\mu$, the number of representations of this type in $L^2\left(\Lambda\backslash G\right)$
for $X=\Lambda\backslash\B$ with $n$ vertices, and the
corresponding section in Theorem \ref{thm:B-decomp}.}
\end{table}

\begin{prop}\label{prop:Aspec_reps}
With $L^2\left(\Lambda\backslash G\right)=\bigoplus_iW_i\oplus\widehat{\bigoplus}_iU_i$
as in Proposition \ref{prop:Bspec_reps},
\begin{equation}
\Spec\left(A_{X}\right)=\{0\}^{\#\left\{ i\,\middle|\,-q^{\pm1}\in\mathrm{Sat}(W_i)\right\} }\cup\bigcup_i\left\{ \pm\sqrt{q^3+\big(z_i+\tfrac1{z_i}\big)q^2+q}\right\} \label{eq:A-spec}
\end{equation}
where the union is over $i$ such that $q^{-2},-q^{\pm1}\notin\mathrm{Sat}\left(W_i\right)$,
and $z_i$ is any choice of Satake parameter in $\mathrm{Sat}(W_i)$. Moreover,
\begin{equation}
\begin{alignedat}1\mE_X & =\dim\mbox{Hom}_{G}\left(W_{-q},L^2\left(\Lambda\backslash G\right)\right)=\#\left\{ i\,\middle|\,-q\in\mathrm{Sat}\left(W_i\right)\right\} ,\\
\mathcal{N}_{X} & =\dim\mbox{Hom}_{G}\left(W_{-1/q},L^2\left(\Lambda\backslash G\right)\right)=\#\left\{ i\,\middle|\,-1/q\in\mathrm{Sat}\left(W_i\right)\right\} .
\end{alignedat}
\label{eq:E-N-reps}
\end{equation}
\end{prop}

\begin{proof}
The analysis of $A=A_{X}$ is a bit trickier than that of $B$. We
now have $V_{\B}=Gv_0\sqcup Gv_1$, which gives $L^2\left(V_{X}\right)\cong L^2\left(\Lambda\backslash G\right)^{K}\oplus L^2\left(\Lambda\backslash G\right)^{K'}$.
The latter is a module over the subalgebra 
\[
\mA=\left\{ \left(\begin{matrix}\alpha & \beta\\
\gamma & \delta
\end{matrix}\right)\in M_2(\mathcal{H}_{\boldsymbol{I}})\,\middle|\,{\alpha\in C_{c}\left(K\backslash G/K\right),\beta\in C_{c}\left(K\backslash G/K'\right)\atop \gamma\in C_{c}\left(K'\backslash G/K\right),\delta\in C_{c}\left(K'\backslash G/K'\right)}\right\} 
\]
of $M_2(\mathcal{H}_{\boldsymbol{I}})$, and the operator $A_{X}$
acting on $L^2\left(V_{X}\right)$ coincides with the element $\left(\begin{smallmatrix}0 & \tfrac1{\mu(K')}\one_{KK'}\\
\tfrac1{\mu(K)}\one_{K'K} & 0
\end{smallmatrix}\right)\in\mA$. We obtain an $A_{X}$-stable decomposition $L^2\left(V_{X}\right)\cong\bigoplus_i\big(W_i^{K}\oplus W_i^{K'}\big)$,
and the summands have dimension one for $z=-q^{\pm1}$, zero for $z=q^{-2}$
and two otherwise (see Table \ref{tab:subreps}). We recall that $f^{K}:=f_1+q^{3/2}f_{w}$ and $f^{K'}=f_1+\frac{z}{\sqrt{q}}f_{w}$
span $W_i^{K}$ and $W_i^{K'}$ respectively, and observe that
$\mu(K)=\mu(\boldsymbol{I})[K:\boldsymbol{I}]=q^3+1$ and similarly $\mu(K')=q+1$.
We have
\[
KK'=\left(\boldsymbol{I}\sqcup\boldsymbol{I}w\boldsymbol{I}\right)\left(\boldsymbol{I}\sqcup\boldsymbol{I}s\boldsymbol{I}\right)=\boldsymbol{I}\sqcup\boldsymbol{I}w\boldsymbol{I}\sqcup\boldsymbol{I}s\boldsymbol{I}\sqcup\boldsymbol{I}ws\boldsymbol{I},
\]
which implies $\one_{KK'}=1+\tau+\sigma+\tau\sigma$ and $\one_{K'K}=\smash{\one_{KK'}^{^{*}}}=1+\tau+\sigma+\sigma\tau$.
Using (\ref{eq:rho-sigma-tau}) we can now compute that $\mu(K')^{-1}\one_{KK'}f^{K'}=(qz+1)f^{K}$
and $\mu(K)^{-1}\one_{K'K}f^{K}=q(1+q/z)f^{K'}$. Altogether, we have
obtained 
\begin{align*}
&\Spec\left(\rho_{W_z^{K}\oplus W_z^{K'}}\left(\begin{smallmatrix}0 & \tfrac1{\mu(K')}\one_{KK'}\\
\tfrac1{\mu(K)}\one_{K'K} & 0
\end{smallmatrix}\right)\right) \\& =\begin{cases}
\Spec\left(\begin{smallmatrix}0 & qz+1\\
q(1+q/z) & 0
\end{smallmatrix}\right) & z\notin\{q^{-2},-q^{\pm1}\}\\
\Spec\left(\left(0\right)\right) & z\in\{-q^{\pm1}\}
\end{cases}\\
 & =\begin{cases}
\pm\sqrt{q^3+\big(z+\tfrac1{z}\big)q^2+q} & z\notin\{q^{-2},-q^{\pm1}\}\\
0 & z\in\{-q^{\pm1}\}
\end{cases},
\end{align*}
from which (\ref{eq:A-spec}) follows. Moreover, $z=-q$ and $z=-1/q$
give eigenfunctions in $W_z^{K}$ and $W_z^{K'}$ respectively,
so in the notations of Section \ref{sec:Spectral-analysis}, $W_{-q}$
contributes to $\ker A|_{L}$ and $W_{-1/q}$ to $\ker A|_{R}$, yielding
(\ref{eq:E-N-reps}).
\end{proof}

\begin{thm}
\label{thm:ram-local-crit}Let $\Lambda\leq G$ be a cocompact lattice,
and $X=X_{\Lambda}=\Lambda\backslash\B$.
\begin{enumerate}
\item $X$ is adj-Ramanujan if and only if every $K$-spherical irreducible representation $W\leq L^2\left(\Lambda\backslash G\right)$
is one-dimensional, tempered or of A-type.
\item The following are equivalent:
\begin{enumerate}
\item $X$ is Ramanujan.
\item $X$ is NB-Ramanujan.
\item $X$ satisfies the Riemann Hypothesis.
\item Every $K$-spherical irreducible representation $W\leq L^2\left(\Lambda\backslash G\right)$
is one-dimensional or tempered.
\end{enumerate}
\end{enumerate}
\end{thm}

\begin{proof}
\emph{1.} By (\ref{eq:aram-def}) $X_{\Lambda}$ is adj-Ramanujan if and only if every $\lambda\in\mathrm{Spec}\left(A\right)$ satisfies either (i) $\lambda=\pm\mathfrak{pf}$, (ii) $\lambda=0$ or (iii) $\left|\lambda^2-q^3-q\right|\leq2q^2$. From Table \ref{tab:rep-spec} we see that (i) corresponds to a one-dimensional irreducible representation, (ii) to the A and B types (where the latter is tempered), and (iii) to the (tempered) principal series.

\emph{2.} $(a)\Rightarrow(b)$ is immediate and $(b)\Leftrightarrow(c)$
was shown in Corollary \ref{cor:ram-by-theta}.\\
$(b)\Rightarrow(d)$: Any $K$-spherical representation is also I.S., and by
\eqref{eq:ram-bg-def} every $\mu\in\Spec\left(B_{X}\right)$ satisfies
$\left|\mu\right|\in\left\{ 1,\sqrt{q},q,q^2\right\} $, so by Table
\ref{tab:rep-spec} any I.S.\ $W\leq L^2\left(\Lambda\backslash G\right)$
is one-dimensional (trivial) or tempered. \\
$(d)\Rightarrow(a)$: If
every $K$-spherical $W\leq L^2\left(\Lambda\backslash G\right)$
is tempered then so is every I.S.\ $W\leq L^2\left(\Lambda\backslash G\right)$,
since by Table \ref{tab:subreps} the remaining ones are only the Steinberg
and B-type representations, which are tempered. Since $\B=Gv_0\sqcup Gv_1\sqcup Ge_0\sqcup Ge_1$,
any geometric operator $T$ on (all or some) cells of $\B$
can be identified with an operator in an appropriate subalgebra of
$M_{4}\left(\mathcal{H}_{\boldsymbol{I}}\right)$ acting on a subspace
of $\bigoplus_iW_i^{K}\oplus W_i^{K'}\oplus W_i^{\boldsymbol{I}}\oplus W_i^{\boldsymbol{I}}$
(see \cite[proof of Prop.\ 4.1]{Lubetzky2017RandomWalks} for more
details). If $W$ is one-dimensional then the corresponding eigenfunctions
of $T|_{X}$ are invariant under $G'=PSU_3(\Q_p)$, so
the corresponding eigenvalues are trivial. If $W$ is tempered, then
it is weakly contained in $L^2\left(G\right)$ \cite{Haagerup1988},
so its $T$-eigenfunctions are approximate eigenfunctions in $L^2\left(\B\right)$,
and thus $\Spec_0\left(T|_{W}\right)\subseteq\Spec(T|_{L^2\left(\B\right)})$.
\end{proof}

\begin{rem}
\label{rem:Ram-NBRam}In Corollary \ref{cor:ram-by-theta} we have
shown $\left(b\right)\Leftrightarrow\left(c\right)$ for all bigraphs,
and not only quotients of $\T_{p^3+1,p+1}$ by lattices
in $PU_3(\Q_p)$. Let us indicate how to show that $\left(a\right)\Leftrightarrow\left(b\right)$
for all bigraphs. The structure and representation theory of the Hecke-algebra
$\mathcal{H}_{\boldsymbol{I}}$ both follow from the structure of
the Bruhat-Tits tree and $\mathcal{H}_{\boldsymbol{I}}$'s
action on it. Taking $G=\Aut(\T_{K+1,k+1})$ and $\boldsymbol{I}\leq G$
to be an edge stabilizer, one obtains completely analogous results for the
structure and representation theory of the algebra $\mathcal{H}_{\boldsymbol{I}}=C_{c}\left(\boldsymbol{I}\backslash G/\boldsymbol{I}\right)$
in this case. Every finite $(K\!+\!1,k\!+\!1)$-bigraph $X$ is a quotient
of $\T$ by a co-compact lattice $\Lambda$ in $G$, and
one can again verify that the non-backtracking eigenvalues (unlike adjacency eigenvalues) already distinguish between tempered and non-tempered representations. Thus, $X$ is NB-Ramanujan if and only if every irreducible $\mathcal{H}_{\boldsymbol{I}}$-subrepresentation of $L^2(\Lambda\backslash G)^{\boldsymbol{I}}$ is trivial or tempered. As any
geometric operator can be represented by (a sub-matrix algebra over) the
Iwahori-Hecke algebra (see \cite[Prop.\ 4.1]{Lubetzky2017RandomWalks}),
this shows that in the case of bigraphs, NB-Ramanujan already implies Ramanujan for all geometric operators.
\end{rem}

\subsection{Ramanujan global criterion} \label{subsec:Ramanujan-global-criterion}

Let $E/\Q$ be a quadratic imaginary field, $\Phi\in GL_3\left(E\right)$
a definite Hermitian form, and 
$G=PU\left(E,\Phi\right)$
the associated projective unitary group scheme over $\Z$.
Denote $K_{\infty}=G\left(\R\right)=G_{\infty}$, $K_p=G\left(\Z_p\right)\leq G\left(\Q_p\right)=G_p$
for any prime $p$, and $G\left(\A\right)=\Pi'_{v}G_{v}$,
the adelic group. This is a locally compact group, with $K=\prod_{v}K_{v}=G(\R\hat{\Z})$
a maximal compact open subgroup, and $G\left(\Q\right)$ embeds diagonally
as a discrete cocompact lattice in $G(\A)$. Let $K'=\prod_{v}K'_{v}$ be a finite
index subgroup of $K$, i.e.\ $K'_{\infty}=K_{\infty}$ and $K'_{\ell}=K_{\ell}$
for almost all primes $\ell$. Let $p$ be an $E$-inert prime such that
$K'_p=K_p$, and denote $K'^p=\prod_{\ell\ne p}K'_{\ell}$.
Define the $p$-arithmetic congruence subgroup corresponding to $K'$
to be
\[
\Lambda_{K'}^p:=G\left(\Q\right)\bigcap K'^p\leq G\left(\Q\right)\bigcap K^p=G\left(\Z[1/p]\right).
\]
By the strong approximation property (Proposition \ref{prop:SA-adeles}),
there exist $c_1=1,\ldots,c_{h}\in G(\A)$, such that 
\[
G\left(\A\right)=\bigsqcup\nolimits_{i=1}^{h}G\left(\Q\right)c_iG_pK'^p.
\]
For any $1\leq i\leq h$, denote $\Lambda_{K',i}^p=G\left(\Q\right)\bigcap c_iK'^pc_i^{-1}$,
and note that $\Lambda_{K',1}^p=\Lambda_{K'}^p$. By a Theorem of Borel--Harish-Chandra
\cite[Thm.\ 5.7(2)]{Platonov1994Algebraicgroupsand}, $\Lambda_{K',i}^p$
is a cocompact lattice in $G_p$. Denote the finite quotient of
the Bruhat-Tits tree $\B_p$ by the congruence subgroup
$\Lambda_{K'}^p$ by
\[
X_{K'}^p:=\Lambda_{K'}^p\backslash\B_p.
\]

An automorphic representation of $G\left(\A\right)$ is an
irreducible subrepresentation of the right regular $G\left(\A\right)$-representation $L^2\left(G\left(\Q\right)\backslash G\left(\A\right)\right)$.
Denote by $\mA_{G}$ the set of irreducible automorphic representations
of $G$. By \cite{flath1979decomposition}, any $\pi\in\mA_{G}$
is a restricted tensor product $\pi=\otimes_{v}\pi_{v}$, where $\pi_{v}$
is an irreducible admissible representation of $G_{v}$ for each $v$,
and $\pi_{\ell}^{K_{\ell}}\ne\{0\}$ for almost all prime $\ell$.
For $K'\leq K$ of finite index, define the set of automorphic representations
of level $K'$ by
\[
\mA_{G}\left(K'\right)=\left\{ \pi\in\mA_{G}\,\middle|\,\pi^{K'}\ne\{0\}\right\} .
\]

We remark that the automorphic representations of $PU\left(E,\Phi\right)$
are in natural bijection with the automorphic representations of $U\left(E,\Phi\right)$
with a trivial central character, which will be denoted by $\mA_{U\left(E,\Phi\right),\mathbf1}$ in Section \ref{automorphic:SXDH}.
 
\begin{prop} \label{prop:SA-rep} 
The following is an isomorphism of $G_p$-representations (by right translations)
\[
L^2\left(G\left(\Q\right)\backslash G\left(\A\right)\right)^{K'^p}\cong\bigoplus_{i=1}^{h}L^2\left(\Lambda_{K',i}^p\backslash G_p\right).
\]
\end{prop}

\begin{proof}
By Proposition \ref{prop:SA-adeles} and by projecting to the $p$-factor,
we get the following $G_p$-equivariant bijection
\[
G\left(\Q\right)\backslash G\left(\A\right)/K'^p=\bigsqcup_{i=1}^{h}G\left(\Q\right)\backslash G\left(\Q\right)c_iG_pK'^p/K'^p\cong\bigsqcup_{i=1}^{h}\Lambda_{K',i}^p\backslash G_p,
\]
which gives rise to an isomorphism of $G_p$-representations 
\begin{align*}
L^2\left(G\left(\Q\right)\backslash G\left(\A\right)\right)^{K'^p}&\cong\bigoplus_i^{h}L^2\left(G\left(\Q\right)\backslash G\left(\Q\right)c_iG_pK'^p/K'^p\right)
\\&\cong\bigoplus_{i=1}^{h}L^2\left(\Lambda_{K',i}^p\backslash G_p\right).\qedhere
\end{align*}
\end{proof}
\begin{defn}
Say that $\mA_{G}\left(K'\right)$ is Ramanujan (resp.\ A-Ramanujan)
at $p$ if for any $\pi\in\mA_{G}\left(K'\right)$, either
$\pi_p$ is one-dimensional or $\pi_p$ is tempered (resp.\ or
$\pi_p$ is tempered or of A-type).
\end{defn}

The following Theorem gives an automorphic representation theoretic
criterion for $X_{K'}^p$ being Ramanujan.
\begin{thm}
\textcolor{red}{\label{thm:ram-global-crit}} (i) If $\mA_{G}\left(K'\right)$
is Ramanujan (resp.\ A-Ramanujan) at $p$ then $X_{K'}^p$ is Ramanujan
(resp.\ adj-Ramanujan).

(ii) Let $K''\trianglelefteq K'\leq K$ such that $G\left(\A\right)=G\left(\Q\right)\cdot K'$.
If $X_{K''}^p$ is Ramanujan (resp.\ adj-Ramanujan) then $\mA_{G}\left(K''\right)$
is Ramanujan (resp.\ A-Ramanujan) at $p$.
\end{thm}

\begin{proof}
(i) If $\mA_{G}\left(K'\right)$ is Ramanujan (resp.\ A-Ramanujan)
at $p$ then any irreducible $G_p$-representation of $L^2\left(G\left(\Q\right)\backslash G\left(\A\right)\right)^{K'^p}$
is either one-dimensional or tempered (resp.\ or A-type). By Proposition
\ref{prop:SA-rep}, $L^2\left(\Lambda_{K'}^p\backslash G_p\right)$
is a sub-$G_p$-representation of $L^2\left(G\left(\Q\right)\backslash G\left(\A\right)\right)^{K'^p}$,
hence its irreducible representations are either tempered or one-dimensional
(resp.\ or A-type), so by Theorem \ref{thm:ram-local-crit}, 
$X_{K'}^p$ is Ramanujan (resp.\ adj-Ramanujan).

(ii) If $X_{K''}^p$ is Ramanujan (resp.\ adj-Ramanujan) then by
Theorem \ref{thm:ram-local-crit}, $L^2\left(\Lambda_{K''}^p\backslash G_p\right)$
is comprised of either tempered or one-dimensional (resp.\ or A-type)
irreducible representations. By Proposition \ref{prop:SA-rep}, $L^2\left(G\left(\Q\right)\backslash G\left(\A\right)\right)^{K''^p}\cong\bigoplus_{i=1}^{h}L^2\left(\Lambda_{K'',i}^p\backslash G_p\right)$.
Since $G\left(\A\right)=G\left(\Q\right)\cdot K'$,
we may choose $c_2,\ldots,c_{h}$ from $K'$, and since $K''\trianglelefteq K'$,
we get that $c_iK''^pc_i^{-1}=K''^p$ for any $i$. Therefore
$L^2\left(G\left(\Q\right)\backslash G\left(\A\right)\right)^{K''^p}\cong L^2\left(\Lambda_{K''}^p\backslash G_p\right)^{\oplus h}$,
hence its irreducible $G_p$-representations are either tempered
or one-dimensional (resp.\ or A-type), which completes the proof. 
\end{proof}

In fact, in the settings of Theorem \ref{thm:ram-global-crit}(ii),
we have shown something a bit stronger. For an irreducible $K_p$-spherical
representation $\rho$ of $G_p$, recall that $\rho\leq V_{\mathrm{Sat}\left(\rho\right)}$. 

\begin{prop}
Under the assumptions of Theorem \ref{thm:ram-global-crit}(ii), 
\[
\mathrm{Spec}\left(A^2|_{L_{X_{K'}^p}}\right)=\left\{ \mathrm{Spec}\left(A^2\big|_{V_{\mathrm{Sat}(\pi_p)}^{K}}\right)\,\middle|\,\pi\in\mA_{G}\left(K'\right)\right\} ,
\]
and counting eigenvalues with multiplicities, we have: 
\[
\mathrm{Spec}\left(A^2|_{L_{X_{K'}^p}}\right)^{\times h}=\left\{ \mathrm{Spec}\left(A^2\big|_{\smash{V_{\mathrm{Sat}(\pi_p)}^{K}}}\right)^{\times\dim\pi^{K'}}\,\middle|\,\pi\in\mA_{G}\left(K'\right)\right\} ,
\]
where $1\leq h < \infty$ is as in Proposition \ref{prop:SA-adeles}. 
In particular,
\begin{equation}
\mE\left(X_{K'}^p\right)=\frac1{h}\sum_{{\pi\in\mA_{G}\left(K'\right)\atop \mathrm{Sat}(\pi_p)=-p}}\dim\pi{}^{K'}.\label{eq:exces-autom}
\end{equation}
\end{prop}

\section{Automorphic Representation Theory} \label{section:automorphic}

In this section we prove both positive and negative results about families of automorphic representations satisfying and violating the Ramanujan property.

This section is organized as follows: 
In Section \ref{automorphic:theorems} we state the main automorphic results we will be proving. 
In Section \ref{automorphic:rogawski} we recall the main results of Rogawski from \cite{Rogawski1990Automorphicrepresentationsunitary} and \cite{Rogawski1992Analyticexpressionnumber} that we need. 
In Section \ref{automorphic:CFT} we carry out explicit calculations of class field theory. 
In Section \ref{automorphic:level} we compute levels of relevant automorphic representations. 
In Section \ref{automorphic:proofs} we put everything together to prove the main theorems.
In Section \ref{automorphic:SXDH} we prove the Sarnak-Xue Density Hypothesis for unitary group associated to a definite form in three variables. 

Throughout this section we use the following notations. 
Let $G=U_3(E,\Phi)$ denote the definite unitary group scheme over $\Z$ associated to $E$ a quadratic imaginary field and $\Phi \in GL_3(E)$ a definite Hermitian matrix, and  $G^*=U_3(E,J)$ denote its quasi-split inner form. 
For any place $v$ of $\Q$, denote $G_v = G(\Q_v)$, $G^*_v = G^*(\Q_v)$, $K_\infty = G_\infty = U(3)$, $K_p=G(\Z_p)$ and $K^*_p = G^*(\Z_p)$ for any prime $p$.
Fix an isomorphism $G_p \cong G^*_p$ for any prime $p$, and if $p \nmid 2\mathrm{disc}(\Phi)$, choose it such that $K_p$ is mapped to $K^*_p$ (see Lemma \ref{lem:Landherr}).
Let $G(\A) = \prod'_v G_v$ be the adelic group and let $K = \prod_v K_v = G(\R \hat{\Z}) \leq  G(\A)$. 
Denote by $K' = \prod_v K'_v$ a finite index subgroup of $K$ and denote its finite set of ramified places by $\mathrm{Ram}(K') = \{p \mbox{ prime} \;:\; K'_p \ne K_p \}$.

Let $\mA_G$ be the set of automorphic representations of $G$.
For any $\pi \in \mA_G$, $\pi = \otimes_v \pi_v$, where $\pi_v$ is the $v$-local factor of $\pi$ for any place $v$.
For any $K' \leq K$, let $\mA_G(K') = \{ \pi \in \mA_G \,:\, \pi^{K'} \ne 0 \} $ be the set of level $K'$ automorphic representations of $G$.

\begin{defn} \label{defn:Ram}
We say that $\pi \in \mA_G$ is {\it Ramanujan} if $\pi_p$ is tempered for all $p$.

We say that $\mA_G(K')$ is {\it Ramanujan} if for any $\pi \in \mA_G(K')$, either $\pi$ is one-dimensional, or $\pi$ is Ramanujan.

We say that $\mA_G(K')$ is {\it A-Ramanujan} if for any $\pi \in \mA_G(K')$, either: 
(i) $\pi$ is one-dimensional, (ii) $\pi$ is Ramanujan, or (iii) $\pi$ belongs to a global A-packet (see Definition \ref{defn:A-packet}).
\end{defn} 

Note that in the terminology of the previous section, if $\mA_G(K')$ is {\it Ramanujan} (resp.\ {\it A-Ramanujan}) then it is {\it Ramanujan} (resp.\ {\it A-Ramanujan}) at $p$, for all $p$.

\subsection{Statement of the main results} \label{automorphic:theorems}

In this subsection we record the main results of this section which we then prove in Section \ref{automorphic:proofs}, with the exception of our first main result which we prove in Section \ref{automorphic:rogawski}.

\begin{thm} \label{thm:A-Ram}
$\mA_G(K')$ is A-Ramanujan for any $K'$.
\end{thm}

Theorem \ref{thm:A-Ram} follows from Rogawski's classification of automorphic representations of unitary groups in three variables, combined with the generalized Ramanujan-Petersson Theorem (see Section \ref{automorphic:rogawski} for more details).

Our second main result gives some criteria for proving that $\mA_G(K')$ is Ramanujan for certain $K' \leq K$.

\begin{thm} \label{thm:Ram-gen}
 $\mA_G(K')$ is Ramanujan if either:
\begin{enumerate}
\item there exists a prime $p$ which ramifies in $E$, such that $K'_p$ contains an Iwahori subgroup, or
\item there exists $ K'' \leq K$, such that  $K' \leq K''$, $\mA(K'')$ is Ramanujan, and for any prime $p$ for which $K'_p \neq K''_p$, $K'_p$ contains an Iwahori subgroup.
\end{enumerate}
\end{thm}

Claim (1) of Theorem \ref{thm:Ram-gen} is a slight strengthening of Theorem 1.4 of \cite{Evra2018RamanujancomplexesGolden}.
It follows from Theorem \ref{thm:A-Ram} above combined with Proposition 5.7 of \cite{Evra2018RamanujancomplexesGolden}.
Note that if $p\nmid disc(\Phi)$, and $p \notin \mathrm{Ram}(K')$, then $K'_p$ contains an Iwahori subgroup.
Claim (2) of Theorem \ref{thm:Ram-gen} requires a more careful analysis of Rogawski's A-packets.

Next we specialize to the Eisenstein case $E=\Q[\sqrt{-3}]$ and $\Phi = I$.
We recall that by Lemma \ref{lem-(E)} of Section \ref{sec:lattices}, the following is a normal subgroup of $K_3$:
\begin{equation} \label{eq:K3C}
K_3(C) :=  \left\lbrace g \in K_3 \;:\; g \equiv \bsmx 1 & * & * \\ * & 1 & * \\ * & * & 1 \esmx \mod3 \right\rbrace.
\end{equation}

The following corollary of Theorem \ref{thm:Ram-gen} gives positive Ramanujan results in the Eisenstein case.

\begin{cor} \label{cor:Ram-Eis}
Let $G=U_3(\Q[\sqrt{-3}],I)$.
Then $\mA_G(K')$ is Ramanujan if either:
\begin{enumerate}
\item $K'_3$ contains an Iwahori subgroup, in particular if $K'_3 = K_3$, or
\item $K'_3 = K_3(C)$ and for any prime $3 \ne p \in \mathrm{Ram}(K')$, $K'_p$ contains an Iwahori subgroup.
\end{enumerate}
\end{cor}

\begin{proof}
(1) follows from Theorem \ref{thm:Ram-gen}(1) and the fact that $3$ ramifies at $\Q[\sqrt{-3}]$.

(2) follows from Theorem \ref{thm:Ram-gen}(2) and the following fact:
If $K' = \prod_v K'_v$, $K_3' = K_3(C)$ and $K'_v = K_v$ for any $v\ne 3$, then $\mA_G(K') = \{\mathbf1 \}$, and in particular is Ramanujan.
Indeed by Corollary \ref{cor:adelic-simply-transitive}, $G(\A) = G(\Q) \cdot K'$. 
Therefore for any $\pi \in \mA_G(K')$ there is a $K'$-invariant vector $0 \ne  f \in \pi^{K'} \leq L^2(G(\Q)\backslash G(\A))^{K'}$. 
By the assumption on $K'$, $f$ is the constant function, hence $\mathbf1 \leq \pi$, and since $\pi$ is irreducible, we get that $\pi = \mathbf1$. 
\end{proof}

\begin{rem}
We  also prove a local result at $3$ showing certain A-packet supercuspidal representations have no $K_3(C)$-invariant vectors.
 For the complete statement see Proposition \ref{cor:supercuspidal}.
\end{rem}

The following theorem and the subsequent corollary of it provide counterexamples of non-Ramanujan levels in the Eisenstein case. 
For any prime $q$, denote $K_q(q) = \{g \in K_q \,:\, g \equiv I \mod{q}\}$ and see the paragraph after Proposition \ref{prop:level-Langlands} for the definition of the subgroups $\boldsymbol{I}_3(3) \leq K_3(3)$.

\begin{thm} \label{thm:NonRam-Eis}
Let $G=U_3(\Q[\sqrt{-3}],I)$. 
Then $\mA_G(K')$ is non-Ramanujan if either: 
\begin{enumerate}
\item $\mathrm{Ram}(K') = \{3\}$ and $K'_3 = \boldsymbol{I}_3(3)$, or
\item $\mathrm{Ram}(K')= \{3,q\}$, $q \geq 5$, $K'_3 = K_3(C)$ and $K'_q = K_q(q)$.
\end{enumerate}
\end{thm}

\begin{cor} \label{cor:NonRam-Eis}
Let $G=U_3(\Q[\sqrt{-3}],I)$ and $K'\leq G(\A)$ such that $K'_3 \subseteq K_3(C)$ and $K'_q \subseteq K_q(q)$ for some prime $q > 5$. 
Then $\mA_G(K')$ is non-Ramanujan.
\end{cor}

\begin{proof}
The corollary follows directly from Theorem \ref{thm:NonRam-Eis} combined with the following simple fact:
if $K'' \leq K' \leq K$ and $\mA_G(K'')$ is Ramanujan then $\mA_G(K')$ is Ramanujan (since $\mA_G(K') \subset \mA_G(K'')$). 
\end{proof}

The conjecture below is needed to state the next theorem, which gives another positive result towards Ramanujan in the Eisenstein case. 
For the notation used in the conjecture see Definition \ref{defn:A-packet}. 

\begin{conjecture} \label{conj:level-packet} 
For any character $\rho_p$ of $U_1(\Q_p)^2$ and finite index subgroup  $K_p'$ of $K_p$,
\[
\text{ if }\pi^n(\rho_p)^{K_p'}= 0 \text{ then } \pi_p^{K_p'}= 0 \text{ for all } \pi_p\in \Pi'(\rho_p).
\]
\end{conjecture}

While we only use the conjecture for $p$ ramified in $E$, we state it for all $p$. 
The conjecture holds trivially for split primes $p$ since $\Pi'(\rho_p)=\{\pi^n(\rho_p)\}$.
In Proposition \ref{prop:conj} we prove the conjecture holds for inert primes  and principal congruence subgroups as well as  Iwahori principal congruence subgroups.

\begin{thm} \label{thm:Ram-Eis-kq}
Let $G=U_3(\Q[\sqrt{-3}],I)$ and $K'\leq G(\A)$, where $\mathrm{Ram}(K')=\{3, q\}$,   $q\equiv 1 \pmod{12}$, $K_3' = K_3(C)$ and $K'_q=\{g\in K_q : \det(g)\mod q\in \langle \zeta \rangle \}$. 
Assume Conjecture \ref{conj:level-packet} holds for $G$ and $p=3$.
Then $\mA_G(K')$ is Ramanujan.  
\end{thm}

Our final main result is a proof of the Sarnak-Xue Density Hypothesis (SXDH) for definite unitary $3\times 3$ matrix groups $G=U_3(E,\Phi)$.
For the classical SXDH see \cite{sarnak1991bounds} and for its variant for algebraic groups defined over number fields which are compact at the infinite places see \cite{evra2023cohomological}.

Let $S$ be a finite set of places of $\Q$ which contains $\{\infty, 2, 3, 5, 7\}$ and the primes in which $G$ ramifies. 
Fix $K'_\ell \leq K_\ell$ a finite index subgroup for any $\ell \in S$, where $K'_\infty = K_\infty$.
For any integer $N  = \prod_i p_i^{e_i}$, such that $p_i \not\in S$, denote $K_{p_i}(p_i^{e_i}) = \{g\in K_{p_i} \,:\, g \equiv I \mod{p_i^{e_i}} \}$, and define $K'(N) = \prod_{v \in S} K'_v \prod_{p_i \ne v \not\in S} K_v \prod_i K_{p_i}(p_i^{e_i}) \leq G(\A)$.
Let $\mA_{G,\b1}$ be the subset of $\pi \in \mA_G$ with trivial central character and let $\mA^R_{G,\b1}$ (resp.\ $\mA^A_{G,\b1}$, resp.\ $\mA^F_{G,\b1}$) be the subset of $\pi \in \mA_{G,\b1}$ which are Ramanujan (resp.\ A-type, resp.\ one-dimensional).
Denote $V(N) := \bigoplus_{\pi \in \mA_{G,\b1}} \pi^{K'(N)}$ and $V_X(N) := \bigoplus_{\pi \in \mA^X_{G,\b1}} \pi^{K'(N)}$ for $X= R, A$ or $F$.
By Theorem \ref{thm:A-Ram}, $V(N) = V_R(N) \oplus V_A(N) \oplus V_F(N)$. 

\begin{conjecture}[SXDH]
For any $\e >0$ there exists $C_\e > 0$, such that for any $N$ coprime to $S$,
\[
\dim V_A(N) \leq C_\e \cdot \dim V(N)^{\frac12 + \e}.
\]
\end{conjecture}

Adjusting Marshall's endoscopic arguments \cite{Marshall2014Endoscopycohomologygrowth} from the cohomological to the definite settings, we get the following result which is a stronger version of the (SXDH).

\begin{thm} \label{thm:SXDH}
For any $\e >0$ there exists $C_\e > 0$, such that for any $N$ coprime to $S$,
\[
\dim V_A(N)  \leq C_\e \cdot \dim V(N)^{\frac3{8} + \e}.
\]
\end{thm}

\subsection{Automorphic representations of unitary groups in three variables} \label{automorphic:rogawski}

The purpose of this subsection is to characterize the automorphic representations of unitary groups in three variables, which are non-Ramanujan. We outline a proof and consequences of
Theorem \ref{thm:A-Ram} which states that the non-Ramanujan infinite-dimensional automorphic representations of definite unitary groups are precisely the ones appearing in Rogawski's A-packets \cite{Rogawski1990Automorphicrepresentationsunitary}. 

Recall that $G=U_3(E,\Phi)$,  $G^*=U_3(E,J)$ the quasi-split inner form and that $G_p \cong G^*_p$, for any prime $p$, while $G_\infty \not \cong G^*_\infty$.
Let $P^* \leq G^*$ be the Borel subgroup of upper triangular matrices, $M^* \leq P^*$ the maximal torus of diagonal matrices, and denote $P^*_v = P^*(\Q_v)$ and $M^*_v = M^*(\Q_v)$ for any place $v$.

Let $|\cdot|=|\cdot|_E : E^\times \backslash \A_E^\times \rightarrow \C^\times$ be the adelic norm considered as an automorphic character of $GL_1/E$.
Let $\omega=\omega_{E/\Q}:\Q^\times \backslash \A^\times \rightarrow \C^\times$ be the automorphic character of $GL_1/\Q$, associated by class field theory to the field extension $E/\Q$. 
Fix $\mu : E^\times \backslash \A_E^\times \rightarrow \C^\times$ an automorphic character of $GL_1/E$, which extends $\omega$, i.e. $\mu |_{\A^\times} = \omega$.
See Definition \ref{defn:mu} below for an explicit construction of such a $\mu$.

\begin{defn}\label{defn:A-packet}
For any $\rho=(\rho_1,\rho')$, where $\rho_1, \rho' : U_1(\Q) \backslash U_1(\A) \rightarrow \C^\times$ are automorphic characters of  $U_1(E)/\Q$, define:
\begin{enumerate} 

\item Let $\phi = \phi(\rho)$ be the automorphic character of $GL_1/E$  associated to $\rho$ (and $\mu$) defined by 
\[
\phi \,:\, E^\times \backslash \A_E^\times \rightarrow \C^\times, \qquad \phi(\alpha) = \mu(\alpha)\rho_1(\alpha/\bar{\alpha}).
\]
Let $\varepsilon(s,\phi)$  be the epsilon factor of $\phi$, and denote the root number of $\phi$ by $\varepsilon(1/2,\phi) = \pm 1$.

\item Let $\eta=\eta(\rho)$ be the automorphic character of $M^*/\Q$ associated to $\rho$ (and $\mu$) defined by 
\[
\eta \,:\, M^*(\Q) \backslash M^*(\A) \rightarrow \C^\times, \qquad \eta \bmx \alpha & 0 & 0\\ 0 & \beta & 0\\0 & 0 & \bar{\alpha}^{-1} \emx 
= |\alpha|^{1/2}\phi(\alpha)\rho'((\alpha/\bar{\alpha})\beta).
\]
For any place $v$ of $\Q$, let $\delta_v$ be the modular character of $P^*_v$, and $\mbox{ind}(\eta_v) = \mbox{Ind}(\delta_v^{1/2}\eta_p)$ be the unitary parabolic induction of $\eta_v$ from $P^*_v$ to $G^*_v$, defined by
\begin{align} \label{eq:ind}
&\mbox{ind}_{P^*_v}^{G^*_v}(\eta_v) = \mbox{ind}(\eta_v)  
\\ \nonumber =& \left\lbrace f\,:\, G^*_v \rightarrow \C \;:\; 
\begin{array}{c}
(i) \quad \exists K \leq_{f.i.} K^*_v,\; \forall k\in K,\; \forall g\in G^*_v, \qquad f(gk) = f(g) \\ 
(ii) \quad \forall p \in P^*_v,\; \forall g\in G^*_v, \qquad f(pg) = \delta_v(p)^{1/2} \eta_v(p) f(g)
\end{array} 
\right\rbrace.
\end{align}

\item For a non-split prime $v=p$ (resp.\ $v=\infty$), there is a certain supercuspidal (resp.\ discrete series) representation of $G^*_v$, denoted by $\pi^s(\rho_v)$, which is characterized via endoscopic character relations and depends only on $\rho_v$. For all finite primes $v$, there is a non-tempered representation $\pi^n(\rho_v)$ that is the unique irreducible quotient, or Langlands quotient, of $\mbox{ind}(\eta_v)$. 
Define the local A-packet of $G^*_v$ associated to $\rho_v$, to be 
\[
\Pi(\rho_v)=\begin{cases} \{\pi^n(\rho_v)\} & v \mbox{ split}\\ \{\pi^n(\rho_v),\pi^s(\rho_v)\} & \mbox{otherwise.} \end{cases}
\]
Define the global A-packet of $G^*$ associated to $\rho$, to be
\[
\Pi(\rho) = \{ \pi = \otimes_v \pi_v \in \Pi(\rho_v) \mid \mbox{for almost all places } v, \pi_v = \pi^n(\rho_v) \}. 
\]
\item Recall $G_p \cong G^*_p$ for any finite place $v = p$, and $G_\infty = U(3)$ is compact for the infinite place $v=\infty$.
For $v=\infty$, there is a finite-dimensional irreducible representation of $G_\infty$, denoted by $\pi'^s(\rho_v)$, which depends only on $\rho_v$.
For any place $v$, define the local A-packet of $G_v$ associated to $\rho$, to be
\[
\Pi'(\rho_v)=\begin{cases} \Pi(\rho_v) & v =p \ne \infty \\ \{\pi'^s(\rho_v)\} & v=\infty .\end{cases}
\]
Define the global A-packet of $G/\Q$ associated to $\rho$, to be
\[
\Pi '(\rho)=\left\{ \pi=\otimes_v \pi_v \in \otimes_v \Pi'(\rho_v) \mid \mbox{for almost all places } v, \pi_v = \pi^n(\rho_v) \right\} 
\]
An automorphic representation $\pi$ of $G$ is said to belong to some A-packet, and more explicitly, that $\pi$ belong to the A-packet of $\rho$, if $\pi \in \Pi'(\rho)$, i.e. $\pi_v \in \Pi'(\rho_v)$  for any place $v$.
\end{enumerate}
\end{defn}

We now state some useful facts about A-packets which were proven by Rogawski in \cite{Rogawski1990Automorphicrepresentationsunitary} and \cite{Rogawski1992Analyticexpressionnumber}. Similar results were also proven by Flicker (see \cite{flicker:2006}).

\begin{thm} \label{thm:Rogawski-A-packets} \cite{Rogawski1992Analyticexpressionnumber}
Let $\rho$, $\phi = \phi(\rho)$ and $\varepsilon(1/2,\phi)$ be as in Definition \ref{defn:A-packet}.
For $\pi \in \Pi(\rho)$, let $n(\pi) = \#\{v\mid\pi_v = \pi^s(\rho_v)\}$, and $m(\pi) = \dim \mathrm{Hom}\left(\pi , L^2_{disc}(G^*(\Q)\backslash G^*(\A))  \right)$, i.e. $m(\pi) \neq 0$ if and only if $\pi$ is a discrete automorphic representation of $G^*$.
Then 
\[
m(\pi)=\frac12 \left(1+\varepsilon(1/2,\phi)(-1)^{n(\pi)}\right).
\]
For $\pi \in \Pi'(\rho)$, let $n'(\pi) = \#\{v \ne \infty \mid \pi_v = \pi^s(\rho_v)\}$, and $m(\pi) = \dim \mathrm{Hom}\left(\pi , L^2(G(\Q)\backslash G(\A))  \right)$, i.e. $m(\pi) \neq 0$ if and only if $\pi$ is an automorphic representation of $G$.
Then 
\[
m(\pi)=\frac12\left(1+\varepsilon(1/2,\phi)(-1)^{1+n'(\pi)}\right).
\]
\end{thm}

\begin{prop} \label{prop:archtrivial} \cite[\S~14]{Rogawski1990Automorphicrepresentationsunitary}
Let $\rho=(\rho_1,\rho')$ be as in Definition \ref{defn:A-packet} and assume  $\mu_\infty(z) = (\frac{z}{|z|})^{-1}$. 
Then $\pi'^s(\rho_\infty)$ is the trivial representation of $G_\infty \cong U(3)$ if and only if either:
\begin{enumerate}
\item $\rho_{1,\infty}(\frac{\alpha}{\bar \alpha})=(\frac{\alpha}{\bar \alpha})^2$ and $\rho'_\infty(\beta)= \beta^{-1}$, or
\item $\rho_{1, \infty}(\frac{\alpha}{\bar \alpha})=(\frac{\alpha}{\bar \alpha})^{-1}$ and $\rho'_\infty(\beta)=\beta$.
\end{enumerate}
where $|z|=\sqrt{z\bar z}$ and for $s\in \frac12\Z$, $(\frac{z}{\bar z})^s:=\left(\frac{z}{|z|}\right)^{2s}$.
\end{prop}

We now prove Theorem \ref{thm:A-Ram} which characterizes automorphic representations of $G = U_3(E,\Phi)$, for $\Phi$ definite, as either one-dimensional, or Ramanujan, or of A-type (i.e. sits in global A-packets).
A slightly weaker result follows from Section 5 of \cite{Evra2018RamanujancomplexesGolden}, where the Ramanujan property holds only for the unramified places. 
It was pointed out to us by S.W. Shin that using Rogawski's explicit description of base change, one can prove the Ramanujan property at all places. 
In particular, this confirms the expectation stated in Remark 5.8 of \cite{Evra2018RamanujancomplexesGolden}.

\begin{proof}[Proof of Theorem \ref{thm:A-Ram}]
Let $\pi \in \mA_G$, which is not one-dimensional or of A-type.
By Section 14 of \cite{Rogawski1990Automorphicrepresentationsunitary} (transfer of inner forms of $U_3$), there exists a unique $\pi^* \in \mA_{G^*}$, such that $\pi_p \cong \pi^*_p$ for any finite place $p$ and $\pi^*_\infty$ is cohomological.
By Section 13 of \cite{Rogawski1990Automorphicrepresentationsunitary} (endoscopic classification of $U_3$), see also Theorem 5.2 of \cite{Evra2018RamanujancomplexesGolden}, we get that $\pi^*$ is either a stable cuspidal representation of $G^*$ or an endoscopic transfer of a stable cuspidal representation from $U_2 \times U_1$ or $U_1 \times U_1 \times U_1$.
Let $\tilde{\pi}$ be the global base change of $\pi^*$, which is a cohomological and (conjugate) self-dual automorphic representation of $GL_3/E$, that is either cuspidal (when $\pi^*$ is stable) or belongs to the automorphic parabolic induction from cuspidals on $GL_2 \times GL_1$ or $GL_1 \times GL_1 \times GL_1$.
By the Generalized Ramanujan-Petersson Conjecture (GRPC) due to \cite{Shin2011Galoisrepresentationsarising} (which relies on Deligne's proof of the Weil conjectures \cite{deligne1974conjecture}, Ngo's proof of the fundamental Lemma \cite{Ngo2010Lelemmefondamental}, as well as previous techniques developed by Eichler, Shimura, Langlands, Kottowitz, Clozel, Harris-Taylor and other experts), any cohomological, self-dual, cuspidal representation of $GL_n/E$ is Ramanujan.
Combining the (GRPC) with the fact that parabolic induction preserves temperedness, we get that $\tilde{\pi}$ is Ramanujan.
By the paragraph before Theorem 13.3.3 of \cite{Rogawski1990Automorphicrepresentationsunitary}, 
we get that by definition, $\tilde{\pi}_p$ is the local base change of $\pi^*_p$, for any $p$.
If $\pi^*_p$ was not tempered, then by Sections 11.4 and 13.2 of \cite{Rogawski1990Automorphicrepresentationsunitary}, we get that $\pi^*_p$ and $\tilde{\pi}_p$ are the Langlands quotients of parabolic induction from the Borel subgroups of characters $\chi$ and $\chi\circ N$, where $N$ is the norm map from the subgroup of diagonal matrices of $GL_3$ to $G^*$ defined in Section 3.10 of \cite{Rogawski1990Automorphicrepresentationsunitary}.
Since $\pi^*_p$ is non-tempered the character $\chi$ is non-unitary, hence $\chi\circ N$ is also non-unitary, and therefore $\tilde{\pi}_p$ is non-tempered, in contradiction to the above. Hence $\pi_p = \pi^*_p$ is tempered for any finite place $p$.
\end{proof}


As a consequence of Theorem \ref{thm:A-Ram}, together with the classification result of the Iwahori-spherical unitary representations of $G_p$, for an inert prime $p$  (Table \ref{tab:rep-spec}), we get that the A-type representations are precisely the Iwahori-spherical members of the Rogawski's A-packets.

\begin{cor} \label{cor:A-type} 
Let $\pi \in \mA_G$. Then the following are equivalent:
\begin{enumerate}
\item[(1)] $\pi$ belongs to a global A-packet.
\item[(2)] For all primes $p$, $\pi_p$ belongs to a local A-packet.
\item[(3)] For all primes $p$, if $\pi_p$ is Iwahori-spherical then $\pi_p = \pi^n(\rho_p)$ for some unramified character $\rho_p$.
\item[(4)] For all inert primes $p$, if $\pi_p$ is Iwahori-spherical then $\mathrm{Sat}(\pi_p) = - p$.
\item[(5)] For some inert prime $p$, $\pi_p$ is Iwahori-spherical and $\mathrm{Sat}(\pi_p) = - p$.
\end{enumerate}
\end{cor}

\begin{proof}
By the definition of a A-packets, (1) implies (2).
Since the only Iwahori-spherical member in a local A-packet can be $\pi^n(\rho_p)$ (see Lemma \ref{lem:level-supercuspidal} below), (2) implies (3).
For an unramified character $\rho_p$, $\pi^n(\rho_p)$ is the Langlands quotient of an unramified principal series representation $\mbox{ind}_{P_p}^{G_p}(\eta_p)$, and since
\[ \eta_{p}\left(\diag(p,1,p^{-1})\right)=\left|p\right|_{E}^{1/2}\mu\left(p\right)\rho_{1}(p/\overline{p})\rho'\left(p/\overline{p}\right)=p^{-1}\omega(p)=-p^{-1}, \] the Satake parameter of any subquotient of $\mbox{ind}_{P_p}^{G_p}(\eta_p)$ is either $-p$ or its inverse $-1/p$. However, since the Langlands quotient $\pi^n(\rho)$ is non-tempered (see \cite[\S12.2]{Rogawski1990Automorphicrepresentationsunitary}), $-1/p$ is ruled out (see Table \ref{tab:rep-spec}) and we get that (3) implies (4).
Clearly, (4) implies (5) since $\pi_p$ is Iwahori-spherical for almost all $p$.
Finally, since $\mathrm{Sat}(\pi_p) = - p$ implies that $\pi_p$ is neither tempered nor one-dimensional, by Theorem \ref{thm:A-Ram}, (5) implies (1).
\end{proof}

\subsection{Class field theory} \label{automorphic:CFT}

Throughout this subsection $E=\Q[\sqrt{-3}]$ is the Eisenstein field.
Our goal in this section is to construct a specific automorphic character $\mu$ of $GL_1/E$, whose restriction to an automorphic character of $GL_1/\Q$ gives the quadratic character $\omega_{E/\Q}$ coming from class field theory, and to calculate its epsilon factor. 

For any place $v$ of $\Q$, denote $E_v = E \otimes_\Q \Q_v$ the étale quadratic extension of $\Q_v$.
Note that $E_\infty = \C$, and for a prime $p$, $E_p$ is a field if and only if $p$ does not split in $E$, otherwise $E_p \cong \Q_p \times \Q_p$.

\begin{lem} \label{lem:omega}
The class field character $\omega=\omega_{E/\Q} \,:\, \Q^\times \backslash \A^\times \rightarrow \C^\times$, $\omega=\otimes_v \omega_v$, is uniquely defined as follows: 
\begin{enumerate}
    \item 
 $\omega_\infty \mid_{\R^\times_{>0}} \equiv 1$ and $w_\infty(-1)=-1$, 
 \item $\omega_3 \mid_{3^\Z (1+3 \Z_3)} \equiv 1$ and $\omega_3(-1)=-1$,  
 \item $\omega_p \mid_{\Q_p^\times} \equiv 1$  for any prime $p \equiv 1 \mod3$, and 
 \item $\omega_p \mid_{\Z_p^\times} \equiv 1$ and $\omega_p(p) = -1$, for any prime $p \equiv 2 \mod3$.
\end{enumerate}
\end{lem}

\begin{proof}
Denote by $\omega_v = \omega_{E_v/\Q_v}$ the local class field character. 
Then as a character of the ideles of $\Q$, we get  that $w=\otimes_v w_v$. 
By definition \[\omega_v(x)=\begin{cases} 1 & x \in \mathrm{Nm}_{E_v/\Q_v}(E_v\times) \\ -1 & x \in  \Q_v^\times \setminus \mathrm{Nm}(E_v^\times).\end{cases}\]
By direct calculation we get that 
\begin{enumerate} \item $\mathrm{Nm}_{E_\infty/\Q_\infty}(E_\infty^\times) = \R^\times_{>0}$, \item $\mathrm{Nm}_{E_3/\Q_3}(E_3^\times) = 3^\Z (1 + 3\Z)$, \item $\mathrm{Nm}_{E_p/\Q_p}(E_p^\times) = \Q_p^\times$ for any prime $p \equiv 1 \mod3$, and \item $\mathrm{Nm}_{E_p/\Q_p}(E_p^\times) = p^{2\Z}\Z_p^\times$ for any prime $p \equiv 2 \mod3$.\end{enumerate}
The claim now follows.
\end{proof}

Let  $\zeta = \frac{1+\sqrt{-3}}2 \in E$. Then $\mO= \Z[\zeta]$ is the ring of integers of $E$ and  $\mO^\times = \langle \zeta \rangle$ is its  group of units.
Also $E$ is embedded in $\C$ by sending $\zeta$ to $e^{\frac{2\pi i}{6}} \in \C$.

For any finite place $w$ of $E$, denote by $E_w$ the completion of $E$ w.r.t. $w$,  $\nu_w \, :\, E_w^\times \twoheadrightarrow \Z$ the normalized discrete valuation, $\mO_w = \{x\in E_w \,:\, \nu_w(x) \geq 0\}$ the ring of integers of $E_w$, $\mO_w^\times = \{x\in E_w \,:\, \nu_w(x) = 0\}$ the group of units of $\mO_w$,  and $\varpi_w \in \mO_w$, $\nu_w(\varpi_w) = 1$, a uniformizer of $E_w$.

\begin{defn} \label{defn:conductor}
Let $w$ be a finite place of $E$ and let $\chi$ be a character of $E_w^\times$.
Call $\chi$ unramified if  $\chi \mid_{\mO_w^*}  \equiv 1$, in which case define $c(\chi) = 0$.
If $\chi$ is not unramified then define its conductor to be
\[
c(\chi) = \min \{ c \in \N \,:\, \chi \mid (1+\varpi_w^c \mO_w ) \equiv 1\}.
\]
\end{defn}

The places of $E$, denoted by $w$, are as follows: 
The unique infinity place $w=\infty$, for which $E_\infty = \C$.
The unique ramified place $w=\sqrt{-3}$, for which $E_{\sqrt{-3}}$ is a ramified quadratic field extension of $\Q_3$ and $E_3 = E_{\sqrt{-3}}$.
The inert places $w=p$, one for each (positive) prime $p \equiv 2 \mod3$, for which $E_p$ is an unramified quadratic field extension of $\Q_p$.
The split places $w=\p$, two for each (positive) prime $p \equiv 1 \mod3$,  for which $E_\p$ is isomorphic to $\Q_p$ and $E_p = E_{\p}\times E_{\bar{\p}}$ where $p = \p \cdot \bar{\p}$.

Note that for any finite place $w$, we have $E_w^\times =  \mO_w^\times \cdot \varpi_w^\Z$, and that for any $n\in \N$, and any set of representatives $X \subset \mO_w^\times$ for the coset space $ \mO_w^\times / (1 + \varpi_w^n \mO_w)$, we have $\mO_w^\times = X \cdot (1 + \varpi_w^n \mO_w)$.
We shall need a slightly finer information regarding the place $\sqrt{-3}$, which we summarize in the following Lemma.

\begin{lem} \label{lem:unif-mod3}
(i) For $w = \sqrt{-3}$, the group of order six, $\langle \zeta \rangle$, is a set of representatives for $\mO_w^\times / (1+3\mO_w)$.
Therefore we have the decomposition $\mO_{\sqrt{-3}}^\times = (1+3\mO_{\sqrt{-3}}) \cdot \langle \zeta \rangle$.\\
(ii) For any finite place $w \ne \sqrt{-3}$, there exists a uniformizer $\varpi_w \in \mO$ for $E_w$ which belongs to $(1+3\mO)$.
\end{lem}

\begin{proof}
(i) follows from the fact that $\mO_{\sqrt{-3}}^\times/ (1+3\mO_{\sqrt{-3}}) $ is of order six, $\zeta$ is also of order six and $\zeta \not \in 1+3\mO_{\sqrt{-3}}$. 
(ii) is a consequence of (i), since by picking any uniformizer $\varpi'_w \in \mO$ of $E_w$, by (i) there exists a unique $a\in \Z/6\Z$ such that $\varpi'_w \equiv \zeta^a \mod3$, and then simply take $\varpi_w = \varpi'_w \cdot \zeta^{-a}$. 
\end{proof}

From now on, for any finite place $w \ne \sqrt{-3}$, let $\varpi_w \in \mO$ be a uniformizer of $E_w$ as in Lemma \ref{lem:unif-mod3}, i.e. $\varpi_w \equiv 1 \mod3$, and for $w = \sqrt{-3}$ we take $w_{\sqrt{-3}} = \sqrt{-3}$. 
Let us be more explicit.
If $w=p$ is an inert prime, then $p \equiv 2 \mod3$, and we can take $\varpi_w = -p$ since $-p \equiv 1 \mod3$.
If $w=\p$, where $\p\cdot\bar{\p} = p$ is a split prime, then $p \equiv 1 \mod3$, and we can take $\varpi_w = \p$ such that $\p \equiv 1 \mod3$.

\begin{lem}\label{lem:u1} Let  $U_1(\Z_q,q) = \ker\left( U_1(\Z_q) \rightarrow U_1(\F_q) \right)$, for any prime $q$. Then
 $$U_1(\Q_p)=\begin{cases} U_1(\Z_3,3)\cdot \langle \zeta \rangle, & p=3 
 \\ U_1(\Z_p,p)\cdot \langle \beta^{p-1}\rangle, \beta \text{ element of order $p^2-1$  in }U_1(\Z_p), & p \text{ inert }
 \\ \Q_p^\times=p^{\Z}U_1(\Z_p, p) \cdot \F_p^\times& p \text{ split.} 
 \end{cases}$$
\end{lem}
\begin{proof} 
For $p$ not split in $E$, by Hilbert Theorem 90, $U_1(\Q_p)=\{ \alpha /\bar \alpha : \alpha \in E_{w}^\times\}$. Thus by $E_w^\times =  \mO_w^\times \cdot \varpi_w^\Z$ and observing that $\frac{\varpi_{\sqrt{-3}}}{\overline{\varpi_{\sqrt{-3}}}}=-1=\zeta^3$ and for $w=p$ inert, $\frac{\varpi_{w}}{\overline{\varpi_{w}}}=1$ we see that
\[
U_1(\Q_p)=\begin{cases} U_1(\Z_p) & p \text{ inert or } p=3
 \\ \Q_p^\times\ &p\equiv 1\mod 3.
 \end{cases}
\]

For $p=3$, we finish by noting that $\mO_{\sqrt{-3}}^\times = (1+3\mO_{\sqrt{-3}}) \cdot \langle \zeta \rangle$ and $\zeta \in U_1(\Z_3)$. 

For $p\equiv 2 \mod 3$ we observe that $\mO_{w}^\times/(1+p\mO_{w})\cong \F_{p^2}^\times=\langle \beta \rangle$ and $\alpha=\beta^d \in \F_{p^2}^\times$ has norm $1$ if and only if $\beta^{d(p+1)}=1$ which is if and only if $(p-1)\mid d$. Hence $U_1(\F_p)=\langle \beta^{p-1}\rangle \cong \Z/(p+1)\Z$. 
\end{proof}

\begin{defn} \label{defn:mu}
Define $\mu \,:\, \A^\times_E \rightarrow \C^\times$, $\mu=\otimes_w \mu_w$  as follows:\\
First define the local character at the infinite place $w=\infty$, to be
\[
\mu_\infty \mid \R^\times_{>0} \equiv 1, \qquad \mu_\infty(e^{\theta i}) = e^{ - \theta i} \quad \textrm{ for } \quad 0 \leq \theta \leq 2\pi.
\]
Second, define the local character at the ramified place $w=\sqrt{-3}$, to be
\begin{align*}
\mu_{\sqrt{-3}} \mid (1+3 \mO_{\sqrt{-3}}) \equiv 1, \quad \mu_{\sqrt{-3}}(\zeta) = e^{\frac{2\pi i}{6}} , \\ \mu_{\sqrt{-3}}(\varpi_{\sqrt{-3}}) = \mu_{\infty}(\varpi_{\sqrt{-3}})^{-1} =  e^{\frac{2\pi i}{4}}.
\end{align*}
Finally, define the local character at any unramified place $w \ne\sqrt{-3}$, to be
\[
\mu_w \mid_{\mO_w^\times} \equiv 1, \qquad \mu_w(\varpi_w) = \mu_{\infty}(\varpi_w)^{-1} = \left\lbrace \begin{array}{cc} -1 & w=p \mbox{ an inert prime} \\ \frac{\p}{\sqrt{p}} & w=\p, \; \p\cdot \bar{\p} = p \mbox{ a split prime.} \end{array}  \right. 
\]
\end{defn}

\begin{prop}\label{prop:mu}
Let $\mu$ be as in Definition \ref{defn:mu}. Then:
\begin{enumerate}
\item $\mu$ is an automorphic character of $GL_1/E$, i.e. $\mu \mid_{E^\times} \equiv 1$.
\item $\mu$ extends the quadratic class field character associated to $E/\Q$, i.e. $\mu \mid_{\A^\times} \equiv \omega_{E/\Q}$.
\end{enumerate}
\end{prop}

\begin{proof}
(1) Because $E$ has class number one, we can work with unique prime factorization. 
So it is enough to show $\mu(\zeta)=1$ and  $\mu(\varpi_u)=1$ for each finite place $u$ of $E$.
By Definition \ref{defn:mu} we get, 
\[
\mu(\zeta) = \prod_w \mu_w(\zeta) = \mu_\infty(\zeta) \cdot \mu_{\sqrt{-3}}(\zeta) =e^{-\frac{2\pi i}{6}} \cdot e^{\frac{2\pi i}{6}} = 1,
\]
\[
\mu(\varpi_u) = \prod_w \mu_w(\varpi_u) = \mu_\infty(\varpi_u) \cdot \mu_u(\varpi_u) =\mu_\infty(\varpi_u) \cdot \mu_\infty(\varpi_u)^{-1} = 1.
\]
(2) Let $\mu' = \mu \mid \A^\times$ be the restriction of $\mu$ to the ideles of $\Q$, i.e. $\mu' = \otimes_v \mu'_v$, $\mu'_\infty = \mu_\infty \mid \R^\times$, $\mu'_3 = \mu_{\sqrt{-3}} \mid_{\Z_3^\times}$,  $\mu'_p = \mu_p \mid \Z_p^\times$ for any inert prime $p$, and $\mu'_p(x) = \mu_\p(x) \cdot \mu_{\bar{\p}}(x)$ for any $x\in \Z_p^\times$ and any split prime $p = \p \cdot \bar{\p}$.
By Definition \ref{defn:mu} at the infinite place we get 
\[
\mu'_\infty \mid \R_{>0} \equiv 1, \qquad \mu'_\infty(-1) = -1,
\]
at the ramified prime we get 
\begin{align*}
&\mu'_3 \mid (1 + 3\Z_3) \equiv 1, \qquad \mu'_3(-1) = -1, 
\\  &\mu'_3(3) = \mu_{\sqrt{-3}}(-\sqrt{-3}^2) = (-1) \mu_\infty(\sqrt{-3})^{-2} = (-1)(-1)=1.
\end{align*}
at an inert prime, i.e. $p \equiv 2 \mod3$, we get
\[
\mu'_p \mid \Z_p^\times \equiv 1, \qquad \mu'_p(p) = \mu_p(p) = \mu_p(-p) = \mu_\infty(-p)^{-1} = -1,
\]
and at a split prime, i.e. $p \equiv 1 \mod3$, and let $p = \p \cdot \bar{\p}$ where $\p,\bar{\p} \equiv 1 \mod3$, we get
\[
\mu'_p \mid \Z_p^\times \equiv 1,\]
\[\mu'_p(p) = \mu_{\p}(p)\mu_{\bar{\p}}(p) = \mu_{\p}(\p)\mu_{\bar{\p}}(\bar{\p}) = \mu_\infty(\p)^{-1} \mu_\infty(\bar{\p})^{-1} = \mu_\infty(p)^{-1} = 1.
\]
By Lemma \ref{lem:omega} this proves that $\mu' = \omega_{E/\Q}$.
\end{proof}

Fix from now on the automorphic character of $GL_1/E$ which extends $\omega_{E/\Q}$ to be $\mu$ from Proposition \ref{prop:mu}.
Fix the additive character $\psi \,:\, \Q \backslash \A \rightarrow \C^\times$, $\psi = \otimes_v \psi_v$, defined by $\psi_\infty(x) = e^{2\pi i x}$ for any $x \in \R$ and $\psi_p \mid \Z_p \equiv 1$ for every prime $p$.
Fix the additive character $\psi' \,:\, E \backslash \A_E \rightarrow \C^\times$, $\psi' = \otimes_w \psi'_w$, defined by $\psi'_w = \psi \circ \mbox{Tr}_{E_w/\Q_v}$, where $v$ is the unique place below $w$.
For any place $w$ of $E$, and any local character $\chi_w\,:\,E_w^\times\rightarrow \C$, denote by $\epsilon(1/2,\chi_w,\psi'_w)$ be the epsilon number of $\chi_w$ w.r.t. $\psi'_w$, evaluated at $1/2$.

Next we calculate the conductors and epsilon factors of a slightly more general family $\rho$.
Recall the notation of Definition \ref{defn:A-packet}, where $\rho = (\rho_1,\rho')$ is an automorphic character of $U_1^2(\Q)$, and  $\phi = \phi(\rho) = \mu\rho_{1,E}$, where $\rho_{1,E}(x) = \rho_1(x/\bar{x})$ is an automorphic character of $GL_1/E$, i.e. $\rho_{1,E}$ is the base change of $\rho_1$ from $U_1$ to $GL_1/E$.
Note that $\mu$ is fixed, and that $\phi$ depends only on $\rho_1$ (and not on $\rho'$).

\begin{defn} \label{defn:phi-b}
For $\rho_1$ be as above and $v$ a place of $\Q$, denote $b(\rho_{1,v}) \in \Z/6\Z$, defined by $\rho_{1,v}(\zeta) = \zeta^{b(\rho_{1,v})}$.
Note that $b(\rho_{1,v}) = 0$ whenever $\rho_{1,v}$ is unramified, i.e. $b(\rho_{1,v}) \ne 0$ for at most finitely many places $v$.
Moreover, since $\rho_1$ is automorphic, $1 = \rho_1(\zeta) = \prod_v \rho_{1,v}(\zeta) = \prod_v \zeta^{b(\rho_{1,v})}$, hence $\sum_v b(\rho_{1,v}) = 0$.
\end{defn}

The following Lemmas \ref{lem:phi-arch}, \ref{lem:phi-unram}, \ref{lem:phi-inert}, \ref{lem:epsilon-split} and \ref{lem:phi-ram}, describe the conductor and epsilon factor of $\phi$ at each local place, for certain $\rho$.

\begin{lem}\label{lem:phi-arch} 
For any $\rho$ such that $\pi'^s(\rho_\infty)$ (see Definition \ref{defn:A-packet}) is the trivial representation of $G_\infty = U(3)$.
Then
\[
\epsilon(1/2,\phi_\infty,\psi'_\infty) = - i .
\]
\end{lem}

\begin{proof}
By Proposition \ref{prop:archtrivial} we get that either $\rho_{1,\infty}(z) = z^2$  or $\rho_{1, \infty}(z) = z^{-1}$.
Hence $\phi_\infty(z) = z^3 (z\bar{z})^{-3/2}$ or $\phi_\infty(z) = \bar{z}^3 (z\bar{z})^{-3/2}$.
By \cite[Proposition 3.8(iv)]{kudla2004tate}, we get $\epsilon(1/2, \phi_\infty, \psi'_\infty) =  i^3 = - i$.
\end{proof}

We recall some facts of local epsilon factors used in  \cite[p.\ 404]{Rogawski1992Analyticexpressionnumber} and \cite{Tate1979Numbertheoreticbackground}. 

\begin{prop}[Tate] \label{prop:Tate}
Let $\chi$  be an idele character of $E^\times$, let $\psi'$ be as above, and let $w$ be a place of $E$.

(i)  \cite[(3.6.3) p. 17]{Tate1979Numbertheoreticbackground} The local epsilon factor is 
\begin{equation}\label{eq:def_epsilon}
 \epsilon(1/2,\chi_w,\psi'_{w})=\chi(d_w)\cdot\frac{\int_{\mO_w^\times}\chi_w^{-1}(y)\psi'_{w}(\frac{y}{d_w})~dy}{\vert\int_{\mO_w^\times}\chi_w^{-1}(y)\psi'_{w}(\frac{y}{d_w})~dy\vert},
\end{equation}
for an arbitrary $d_w$ of valuation $c(\chi_w)+n(\psi'_w)$, where $n(\psi'_w)$ is the largest integer $n$ such that $\psi'_w(\mathfrak p_w^{-n})=1$.

(ii) \cite[(3.2.6.3) p.14]{Tate1979Numbertheoreticbackground} If $\chi'$ is an unramified character of $E_w^\times$, then   
\begin{equation}\label{eq:epsilon-1}
 \epsilon_w(1/2,\chi_w \chi' ,\psi'_{w})=\chi'(\pi^{c(\chi_w)})\cdot \epsilon_w(1/2,\chi_w,\psi'_{w}).
\end{equation}

(iii)  \cite[(3.6.8) p.17]{Tate1979Numbertheoreticbackground} Since $\chi^{-1}$ is the contragradient of $\chi$,
\begin{equation}\label{eq:epsilon-2}
 \epsilon_w(1/2,\chi_w,\psi'_w)\cdot \epsilon_w(1/2,\chi_w^{-1},\psi'_w)=\chi_w(-1).
\end{equation}
\end{prop}

\begin{lem} \label{lem:phi-unram}
Let $w \ne \sqrt{-3}$ be a finite place of $E$. 
Then for any $\rho$,
\[
c(\phi_w) = c(\rho_{1,E,w}) \qquad \mbox{ and } \qquad \epsilon(1/2, \phi_w, \psi'_w) = \mu_w(\varpi_w^{c(\rho_{1,E,w})})\epsilon(1/2, \rho_{1,E,w}, \psi'_w).
\]
In particular, if $\rho_{1,w}$ is unramified, then 
\[
c(\phi_w) = 0 \qquad \mbox{ and } \qquad \epsilon(1/2, \phi_w, \psi'_w) = 1.
\]
\end{lem}

\begin{proof}
Recall that $\phi_w = \mu_w \rho_{1,E,w}$, and that $\mu_w$ is unramified for $w\ne \sqrt{-3}$.
The claim on the conductor follows immediately.
The claim on the epsilon factor follows from claim (ii) in Proposition \ref{prop:Tate}.
\end{proof}

\begin{lem} \label{lem:phi-inert}
Let $p$ be an odd prime which is inert in $E$. 
Then for any $\rho$,
\[
\epsilon(1/2, \phi_p, \psi'_p) =  (-1)^{b(\rho_{1,p})+c(\rho_{1,E,p})}.
\]
\end{lem}

\begin{proof}
By Proposition 4.1 of \cite{Rogawski1992Analyticexpressionnumber}, $\epsilon(1/2, \phi_p, \psi'_p) = \phi_p(\sqrt{-3}) (-1)^{c(\phi_p)}$. 
By Lemma \ref{lem:phi-unram}, $c(\phi_w) = c(\rho_{1,E,w})$.
Finally $\phi_p(\sqrt{-3}) = \rho_{1,E,p}(\sqrt{-3}) = \rho_{1,p}(-1) = (-1)^{b(\rho_{1,p})}$. 
\end{proof} 

\begin{lem} \label{lem:epsilon-split}
Let $p=w\bar w$ be a prime which is split in $E$.  Then for any $\rho$,
\[
 \epsilon(1/2,\phi_p,\psi_p):=\epsilon(1/2,\phi_w,\psi'_w)\cdot \epsilon(1/2,\phi_{\bar w},\psi'_{\bar w})=(-1)^{b(\rho_{1,p})}.
\]
\end{lem}

\begin{proof}
 Identifying $E_p=E_w\oplus E_{\bar w}$ with $\Q_p\oplus\Q_p$, for any $x\in\Q_p$ we have 
$ \rho_{1,E,w}(x)=\rho_{1,p}((x,x^{-1}))$ as well as $ \rho_{1,E,\bar w}(x)=\rho_{1,p}((x^{-1},x))$, i.e.
$\rho_{1,E,\bar w}=\overline{\rho_{1,E,w}}$. In particular, we have
\[
 \rho_{1,E,w}(-1)=\overline{\rho_{1,E,\bar w}}(-1)=\rho_{1,p}((-1,-1))=(-1)^{b(\rho_{1,p})}.
\]
By Lemma \ref{lem:phi-unram},
\[
 \epsilon_w(1/2,\phi_w,\psi'_{w})=\mu_w(\varpi_w^{c(\rho_{1,E,w})})\epsilon_w(1/2,\rho_{1,E,w},\psi'_{w}),
\]
as well as 
\[
 \epsilon_{1/2,\bar w}(\phi_{\bar w},\psi'_{\bar w})= \mu_{\bar w}(\varpi_{\bar w}^{c(\rho_{1,E,\bar w})})\epsilon_{\bar w}(1/2,\rho_{1,E,\bar w},\psi'_{\bar w}).
\]
Here we may interpret the epsilon factors at the places $w$ and $\bar w$ of $E$ as epsilon factors at $p$ of $\Q_p$. Then, by \eqref{eq:epsilon-2},
\begin{align*}
 \epsilon_v(1/2,\phi,\psi')&:=\epsilon_w(1/2,\phi_w,\psi'_{w})\cdot \epsilon_{1/2,\bar w}(\phi_{\bar w},\psi'_{\bar w})\\
 &=\mu_w(\pi_w^{c(\rho_{1,E,w})})\mu_{\bar w}(\pi_{\bar w}^{c(\rho_{1,E,w})})\cdot 
 \epsilon_p(1/2,\rho_{1,E,w},\psi_p)\cdot \epsilon_p(1/2,\rho_{1,E,w}^{-1},\psi_p)\\
 &= \mu_w(\pi_w^{c(\rho_{1,E,w})})\mu_{\bar w}(\pi_{\bar w}^{c(\rho_{1,E,w})})\rho_{1,E,w}(-1).
\end{align*}
By Definition \ref{defn:mu}, we obtain $\mu_w(\pi_w)\mu_{\bar w}(\pi_{\bar w})=\frac{\mathfrak p}{\sqrt{p}}\frac{\bar{\mathfrak p}}{\sqrt{p}}=1$, so  
 $\epsilon_v(1/2,\phi,\psi')=\rho_{1,E,w}(-1)=(-1)^{b(\rho_1,p)}$.
\end{proof}

\begin{lem} \label{lem:phi-ram}
Let  $\rho$ be such that $c(\rho_{1,E,\sqrt{-3}}) \leq 2$. 
Then  
\[
 c(\phi_{\sqrt{-3}}) = \begin{cases} 1 & b(\rho_{1,3}) \equiv 1 \mod3 \\ 2 & \mbox{otherwise} \end{cases}
\]
and
\[
\epsilon(1/2, \phi_{\sqrt{-3}}, \psi'_{ \sqrt{-3}}) = \begin{cases} i & b(\rho_{1,3}) \equiv 1,2,3,4 \mod{6} \\ -i   & b(\rho_{1,3}) \equiv 0,5 \mod{6} \end{cases}.
\]
\end{lem}

\begin{proof}
Since $c(\mu_{\sqrt{-3}}) = 2$ and $c(\rho_{1,E,\sqrt{-3}}) \leq 2$ we get that $c(\phi_{\sqrt{-3}})\leq 2$.
Now
\[
\phi_{\sqrt{-3}}(\zeta) = \mu_{\sqrt{-3}}(\zeta) \rho_{1,E,\sqrt{-3}}(\zeta)  = \zeta \cdot  \rho_{1,3}(\zeta^2) = \zeta^{1+2b_3}.
\]
Using the decompositions, $\mO_{\sqrt{-3}}^\times = (\mO_{\sqrt{-3}})_2 \cdot \langle \zeta \rangle$ and $ (\mO_{\sqrt{-3}})_1 = (\mO_{\sqrt{-3}})_2 \cdot \langle \zeta^2 \rangle$, we get that $c(\phi_{\sqrt{-3}}) \ne 0$ for any $b(\rho_{1,3})$, and $c(\phi_{\sqrt{-3}}) = 1$ if and only if $3 \mid 1+2b(\rho_{1,3})$, which proves the claim on the conductor. 

We use \eqref{eq:def_epsilon}. Here, $n(\psi'_{\sqrt{-3}})=1$, and we have $\psi_3(\pm\frac13)=\zeta^{\mp 2}$. We further know $\phi_{\sqrt{-3}}(\zeta)=\zeta^{1+2b(\rho_{1,3})}$, and $\phi_{\sqrt{-3}}({\sqrt{-3}})=i\cdot (-1)^{b(\rho_{1,3})}$.
 
First assume $c(\phi_{\sqrt{-3}})=1$, which is equivalent to $b(\rho_{1,3})\equiv 0\mod 3$. Putting $V_1=\mathrm{vol}(1+{\sqrt{-3}}\mO_{\sqrt{-3}})$, the integral in \eqref{eq:def_epsilon} is given by
\begin{align*}
&\int_{\mO_{\sqrt{-3}}^\times}\chi_{\sqrt{-3}}^{-1}(y)\psi'_{\sqrt{-3}}(\frac{y}{\sqrt{-3}^2})~dy
\\=&
V_1\cdot\left(\phi_{\sqrt{-3}}^{-1}(1)\psi'_{\sqrt{-3}}(-\frac13)+
\phi_{\sqrt{-3}}^{-1}(-1)\psi'_{\sqrt{-3}}(\frac13)  
\right)\\
=&V_1(\psi_3(-\frac23)-\psi_3(\frac23))=V_1\cdot(\zeta^4-\zeta^2)=V_1\cdot(-\sqrt{-3}).
\end{align*}
Accordingly,
\begin{align*}
\epsilon(1/2,\phi_{\sqrt{-3}},\psi'_{\sqrt{-3}})&=\phi_{\sqrt{-3}}({\sqrt{-3}}^2)\cdot
\frac{V_1\cdot(-\sqrt{-3})}{\lvert V_1\cdot(-\sqrt{-3})\rvert}=\left(i\cdot(-1)^{b(\rho_{1,3})}\right)^2\cdot(-i)=i.
\end{align*}
Now assume $c(\phi_{\sqrt{-3}})=2$. Then, writing $V_2=\mathrm{vol}(1+3\mO_{\sqrt{-3}})$, the integral in \eqref{eq:def_epsilon} is given by
\begin{align*}
&\int_{\mO_{\sqrt{-3}}^\times}\chi_{\sqrt{-3}}^{-1}(y)\psi'_{\sqrt{-3}}(\frac{y}{\sqrt{-3}^3})~dy
\\=&
V_2\cdot\sum_{j=0}^5\phi_{\sqrt{-3}}^{-1}(\zeta^j)\psi'_{\sqrt{-3}}(\frac{\zeta^j}{{\sqrt{-3}}^3})\\
=& V_2\cdot\left(1+\zeta^{1-2b(\rho_{1,3})}+\zeta^{2b(\rho_{1,3})}-1+\zeta^{-2b(\rho_{1,3})}+\zeta^{-1+2b(\rho_{1,3})}\right)\\
=& V_2\cdot \left\{\begin{array}{lc}3&b\equiv 0\mod 3\\ -3& b\equiv 2\mod 3 \end{array}\right..
  \end{align*}
So we obtain
\begin{align*}
\epsilon(1/2,\phi_{\sqrt{-3}},\psi'_{\sqrt{-3}})&=\phi_{\sqrt{-3}}({\sqrt{-3}}^3)\cdot
\left\{\begin{array}{lc}1&b(\rho_{1,3})\equiv 0\mod 3\\ -1& b(\rho_{1,3})\equiv 2\mod 3 \end{array}\right\}\\
&=i\cdot (-1)^{1+3b(\rho_{1,3})}\cdot 
\left\{\begin{array}{lc}1&b(\rho_{1,3})\equiv 0\mod 3\\ -1& b(\rho_{1,3})\equiv 2\mod 3,\end{array}\right.
 \end{align*}
 which implies the claim of the lemma.
\end{proof}

The next Lemma is a special case of calculating the conductors and epsilon factors of $\phi = \phi(\rho)$, $\rho = (\rho_1,\rho')$ as in Definition \ref{defn:A-packet}, where $\rho_1$ is trivial.
\begin{lem} \label{lem:mu-epsilon}
Let $\mu$ be as in Proposition \ref{prop:mu}. 
Then for any place $w$, 
\[
c(\mu_w) = \left\lbrace \begin{array}{cc} 2 & w = \sqrt{-3} \\ 0 & w \ne  \sqrt{-3},\infty \end{array} \right.
\]
 and 
 \[
\epsilon(1/2,\mu_w,\psi'_w) = \left\lbrace \begin{array}{cc} i & w = \infty \\ - i & w =  \sqrt{-3} \\ 1 & w \ne \sqrt{-3},\,\infty . \end{array} \right.
\]
\end{lem}

\begin{proof}
The claim about the conductor follows from Definition \ref{defn:mu}. 
Note that the epsilon factor is trivial at places where all the data ($\mu$ and $\psi'$) is unramified, which are $w \ne \infty, \sqrt{-3}$.
For $w = \infty$, note that $\mu_\infty(z) = (\frac{z}{\bar{z}})^{-1/2} = \bar{z} (z\bar{z})^{-1/2}$, and by \cite[Proposition 3.8(iv)]{kudla2004tate}, we get $\epsilon(1/2, \mu_\infty, \psi'_\infty) = i$.
\end{proof}

\begin{prop}\label{prop:phi-global-unram}
Let $\rho$ be such that: 
(1) $\pi'^s(\rho_\infty)$ is the trivial representation of $G_\infty = U(3)$,  (2) $c(\rho_{1,E,w}) = 0$ for any $w \ne \sqrt{-3}$, and (3) $c(\rho_{1,E,\sqrt{-3}})\leq 2$.
Then 
\[
 c(\phi_w) = \begin{cases} 1 & w = \sqrt{-3} \\ 0 & \mbox{otherwise} \end{cases} \qquad \mbox{ and } \qquad 
\epsilon(1/2, \phi_w, \psi'_w) = \begin{cases} i & w = \sqrt{-3} \\ -i & w = \infty \\ 1  & \mbox{otherwise} \end{cases}.
\]
In particular, 
\[
\epsilon(1/2, \phi) = \prod_w \epsilon(1/2, \phi_w, \psi'_w) = 1.
\]
\end{prop}

\begin{proof} 
By the definition of the conductor we get that it is sub-multiplicative, $c(\phi_w) = c(\mu_w\rho_{1,E,w}) \leq \max\{c(\mu_w),c(\rho_{1,E,w}) \}$.
Therefore, by Lemma \ref{lem:mu-epsilon} and the assumption on $\rho$, $c(\phi_w) = 0$ for any $w\ne \sqrt{-3}$ and $c(\phi_{\sqrt{-3}}) \leq 2$.
By Lemma \ref{lem:phi-unram}, $\epsilon(1/2, \phi_w, \psi'_w) = 1$ for any $w\ne \sqrt{-3},\infty$, by Lemma \ref{lem:phi-arch}, $\epsilon(1/2, \phi_\infty, \psi'_\infty) = -i$, and by Lemma \ref{lem:phi-ram}, $\epsilon(1/2, \phi_{\sqrt{-3}}, \psi'_{ \sqrt{-3}}) = i$ if $c(\phi_{ \sqrt{-3}}) = 1$.
Therefore we are left to prove $c(\phi_{ \sqrt{-3}}) = 1$, or equivalently, since $c(\phi_{ \sqrt{-3}}) \leq 2$ and $(\mO_{\sqrt{-3}})_1 = (\mO_{\sqrt{-3}})_2 \cdot \langle\zeta^2 \rangle$, it suffices to prove $\phi_{\sqrt{-3}}(\zeta^2) = 1$.  
Because $\phi$ is automorphic and unramified at $w \ne \sqrt{-3}$, we have $1 = \prod_w \phi_w(\zeta^2) = \phi_{\sqrt{-3}}(\zeta^2) \phi_\infty(\zeta^2)$, i.e. $\phi_{\sqrt{-3}}(\zeta^2) = \phi_\infty(\zeta^2)^{-1}$.
Using the fact that $\pi'^s(\rho_\infty)$ is the trivial representation of $G_\infty = U(3)$, by Proposition \ref{prop:archtrivial}, we have $\rho_{1,\infty}(\alpha) = \alpha^2 \mbox{ or } \alpha^{-1}$.
In either case we get that $\phi_\infty(\zeta^2) =  \zeta^{-2} \rho_{1,\infty}(\zeta^4) = \zeta^{\pm 6} = 1$, which proves the claim.
\end{proof}

\begin{prop}\label{prop:phi-global-p}
Fix a prime $p>5$. 
Let $\rho$ be such that: 
(1) $\pi'^s(\rho_\infty)$ is the trivial representation of $G_\infty = U(3)$,  (2) $c(\rho_{1,E,w}) = 0$ for any $w \ne p, \sqrt{-3}$,  (3) $c(\rho_{1,E,\sqrt{-3}})\leq 2$, and (4) $c(\rho_{1,E,p})=1$. 
For any place $v$ of $\Q$, let $b_v = b(\rho_{1,v}) \in \Z/6\Z$ be as in Definition \ref{defn:phi-b}.
Then
\[
 \forall v \ne \infty, 3, p: \quad b_v = 0, \qquad b_\infty = -1 \mbox{ or } 2 \qquad \mbox{ and } \qquad b_\infty + b_3 + b_p = 0.
\]
The local conductors and epsilon factors of $\phi = \mu \rho_{1,E}$ are given by
\[
 c(\phi_w) = \begin{cases} 1 & w = \sqrt{-3},\; 3 \mid b_3 - 1 \\  2 & w = \sqrt{-3},\; 3 \nmid b_3 - 1 \\ 1 & w = p \\ 0 & \mbox{otherwise} \end{cases}\]
and
\[
\epsilon(1/2, \phi_w, \psi'_w) = \begin{cases} -i & w = \infty \\ i & w = \sqrt{-3},\;  b_3 \equiv 1, 2, 3, 4\mod{6} \\ -i  & w = \sqrt{-3},\;  b_3 \equiv 0, 5 \mod 6  \\ (-1)^{b_p +1} & w = p\equiv 2\mod 3 
\\ (-1)^{b_p} & w=p\equiv 1 \mod 3\\
1  & \mbox{otherwise.} \end{cases}
\]
In particular, since $b_3 = - b_\infty - b_p$, we get that for $p\equiv 2 \mod 3$
\begin{align*}
\epsilon(1/2, \phi) &= \prod_w \epsilon(1/2, \phi_w, \psi'_w)  
\\&= \begin{cases} 1 & (b_\infty,b_p) =  (-1,2),\, (-1,3),\, (-1,5),\, (2,1),\, (2,3) \mbox{ or } (2,4)  
\\ -1 & (b_\infty,b_p) = (-1,0),\, (-1,1),\, (-1,4),\, (2,0),\, (2,2) \mbox{ or } (2,5) \end{cases}
\end{align*}
and for $p\equiv 1 \mod 3$, 
\[
\epsilon(1/2, \phi) = \begin{cases} 1 & (b_\infty,b_p) = (-1,0),\, (-1,1),\, (-1,4),\, (2,0),\, (2,2) \mbox{ or } (2,5)
\\ -1 & (b_\infty,b_p) =
 (-1,2),\, (-1,3),\, (-1,5),\, (2,1),\, (2,3) \mbox{ or } (2,4) . \end{cases}
\]
\end{prop}

\begin{proof} 
The claim about the $b_v$ follows from the fact that $\rho_1$ is automorphic and unramified at $v \ne \infty, 3, p$, together with Proposition \ref{prop:archtrivial}.
The claim about the conductors follows from the assumption on $\rho$, combined with Lemma \ref{lem:phi-ram}.
The claim about the epsilon factors follows from Lemmas \ref{lem:phi-arch}, \ref{lem:phi-unram},  \ref{lem:phi-inert}, \ref{lem:epsilon-split}, and \ref{lem:phi-ram}.
As a consequence, we get that the global epsilon factor of $\phi$ is 
\begin{align*}
\epsilon(1/2, \phi) &= \prod_w \epsilon(1/2, \phi_w, \psi'_w)  
\\&= \begin{cases} (-1)^{b_p} & b_3\equiv 1, 2, 3, 4 \mod 6, p\equiv 1 \mod 3 \text{ or } 
\\ & b_3\equiv 0, 5 \mod 6, p\equiv 2\mod 3 \\ (-1)^{b_p+1} & b_3 \equiv 0, 5 \mod 6, p\equiv 1\mod 3 \text{ or }
\\ & b_3\equiv 1, 2, 3, 4 \mod 6, p\equiv 2 \mod 3.\end{cases}
\end{align*}
Note that $b_3 + b_p + b_\infty = 0$ and $b_\infty = 2 \mbox{ or } -1$, hence for any $(b_\infty,b_p) \in \{2,-1\}\times \Z/6\Z$ there exists a unique triplet  $(b_\infty,b_3,b_p)$ satisfying the above conditions, and a direct calculation shows the final claim.
\end{proof}

\begin{lem} \label{lem:class1}
Let  $U_1(\Z_q,q) = \ker\left( U_1(\Z_q) \rightarrow U_1(\F_q) \right)$, for any prime $q$.
Then
\[
\forall g\in U_1(\A), \quad \exists ! r \in U_1(\Q), \quad \exists ! k \in U_1(\R) U_1(\Z_3,3) \prod_{p \ne 3} U_1(\Z_p) \quad : \qquad g= r \cdot k.
\]
Hence, for any prime $q \ne 3$, there is a bijection between the following two sets
\[
\left\{ \mbox{characters of }U_1(\F_q)\right\} \quad\longleftrightarrow\quad\left\{ {\mbox{characters of }U_1(\Q)\backslash U_1(\A)\mbox{ trivial on}\atop U_1(\R)U_1(\Z_3,3)U_1(\Z_q,q)\prod_{p\ne3,q}U_1(\Z_p)}\right\} ,\]
\begin{align*}
\chi' \mapsto \chi \quad : \quad &\chi(r \cdot k) = \chi'(k_q \mod{q}), \quad r \in U_1(\Q), 
\\ &k = (k_v)_v \in U_1(\R) U_1(\Z_3,3) \prod_{p \ne 3} U_1(\Z_p).
\end{align*}
\end{lem}

\begin{proof}
Since $E= \Q[\sqrt{-3}]$ has class number one, by arguing similarly to Proposition 3.6 in \cite[]{Evra2018RamanujancomplexesGolden} (see Equation 3.7), we get that $U_1(\A) = U_1(\Q) \cdot \left( U_1(\R) \prod_p U_1(\Z_p) \right)$.
Also note that $U_1(\Q) \cap \left( U_1(\R) \prod_p U_1(\Z_p) \right) = U_1(\Z) = \left\langle \zeta \right\rangle$, where $\zeta = \frac{1+ \sqrt{-3}}2$.
Note that $U_1(\F_3) \cong U_1(\Z)$ and therefore we have the unique decomposition $U_1(\Z_3) \cong U_1(\Z_3,3) \times U_1(\Z)$.
This proves the first claim.
The second claim follows from the first.
\end{proof}

\subsection{Level of local representations} \label{automorphic:level}

The purpose of this subsection is to describe the levels of local factors of automorphic representations of unitary groups in three variables, which belongs to A-packets.

\begin{lem}\label{lem:level-supercuspidal}
If $p$ does not split in $E$, then for any $\rho_p$, 
\[
\pi^s(\rho_p)^{\boldsymbol{I}^*_p} = 0.
\]
\end{lem}

\begin{proof}
On the one hand, since $\pi = \pi^s(\rho_p)$ is supercuspidal it does not embed into any parabolic induction from $P^*_p$ to $G^*_p$ of a character of $M^*_p$.
On the other hand, any irreducible representation with a non-zero Iwahori-fixed vector embeds in a parabolic induction from $P^*_p$ to $G^*_p$ of a character $M^*_p$ which is trivial on $K^*_p \cap M^*_p$, by \cite{casselman1980unramified} when $p$ is unramified in $E$, and by claim 1 in the proof of Proposition 5.3 in \cite{Evra2018RamanujancomplexesGolden} when $p$ is ramified in $E$.
Combined together we get the claim.
\end{proof}

\begin{lem} \label{lem:level-Frobenius}
Let $K' \leq K^*_p$ be a finite index subgroup.
Then for any $\rho_p$, 
\[
\mathrm{ind}(\eta_p)^{K'} \neq 0 \qquad \Longleftrightarrow \qquad  \exists g \in K^*_p \quad : \quad  \eta_p^{gK'g^{-1} \cap P^*_p} \neq 0.
\]
In particular, if $K'$ is a normal subgroup of $K^*_p$, then
\[
\mathrm{ind}(\eta_p)^{K'} \neq 0 \qquad \Longleftrightarrow \qquad \eta_p^{K' \cap P^*_p} \neq 0. 
\]
\end{lem}

\begin{proof}
For any group $H$, subgroup $S \leq H$, element $h\in H$ and character $\chi\,:\,S\rightarrow \C^\times$, denote $^hS = hSh^{-1}$ and $^h\chi\,:\, ^hS \rightarrow \C^\times$, $^h\chi(x) = \chi(h^{-1}xh)$.
Then by Frobenius reciprocity, Mackey intertwining theorem and the Iwasawa decomposition, we get
\begin{align*}
 \mbox{ind}(\eta)^{K'} &\cong  \Hom_{K'}(\C, \mbox{Res}_{K'}^{G_p^*} \mbox{Ind}_{P_p^*}^{G_p^*}(\delta^{1/2}\eta_p)) 
 \\
 &\cong \Hom_{K'}(\C, \bigoplus_{g \in K'\backslash G^*_p/P^*_p} \mbox{Ind}_{K'\cap ^gP^*_p}^{K'} \mbox{Res}_{K'\cap ^g P^*_p}^{^gP^*_p}(^g(\delta^{1/2}\eta_p))) 
 \\
&\cong \bigoplus_{g \in K' \backslash K^*_p}  \Hom_{K' \cap ^gP^*_p}( \C, \mbox{Res}_{K' \cap ^g P^*_p}^{^gP^*_p}(^g(\delta^{1/2}\eta_p))) 
\\
&\cong \bigoplus_{g \in K' \backslash K^*_p}\, ^g(\delta^{1/2} \eta_p)^{K' \cap ^gP^*_p} \cong \bigoplus_{g \in K' \backslash K^*_p}\, (\delta^{1/2} \eta_p)^{^gK' \cap P^*_p}.
\end{align*}
Since $\delta$ and $\eta_p$ factor through $P^*_p \twoheadrightarrow M^*_p$, we get the first claim.
The second claim follows from the first and the fact that $^hK' = K'$ for any $h \in K^*_p$.
\end{proof}

\begin{prop}\label{prop:level-spherical}
(i) If $p$ is unramified in $E$, then
\[
\pi^n(\rho_p)^{K^*_p} \ne 0 \qquad \Longleftrightarrow \qquad \eta_p^{K^*_p \cap P^*_p} \ne 0.
\]
(ii) If $p$ is ramified in $E$, then
\[
\pi^n(\rho_p)^{K^*_p} = 0.
\]
\end{prop}

\begin{proof}
(i) By Lemma \ref{lem:level-Frobenius} $\eta_p^{K^*_p \cap P^*_p} \ne 0$ if and only if $\mbox{ind}(\eta_p)^{K^*_p} \ne 0$.
By  \cite[Section 12.2]{Rogawski1990Automorphicrepresentationsunitary}, $\mbox{ind}(\eta_p)$ is reducible with two components, $\pi^2(\rho_p)$ and $\pi^n(\rho_p)$, where $\pi^2(\rho_p)$ is square integrable and $\pi^n(\rho_p)$ is non-tempered.
The Satake parameter associated to the unique irreducible spherical component of $\mbox{ind}(\eta_p)^{K^*_p} \ne 0$ is $\diag(p^{1/2},1,p^{-1/2})$.
Since they are not all roots of unity, this component is non-tempered, hence it must be $\pi^n(\rho_p)$, namely $\mbox{ind}(\eta_p)^{K^*_p} \ne 0$ if and only if $\pi^n(\rho_p)^{K^*_p} \ne 0$. 

(ii) Since $p$ is ramified, the class field character $\omega_p = \omega_{E_p/\Q_p}$ is ramified, i.e. there exists $x \in \Z_p^*$ such that $\omega_p(x) \ne 1$, and therefore $\eta_p(\diag(x,1,1/x)) = \mu_p(x) = \omega_p(x) \ne 1$.
Hence $\eta_p^{K^*_p \cap P^*_p} = 0$, which implies $\mbox{ind}(\eta_p)^{K^*_p} = 0$, which in turn implies $\pi^n(\rho_p)^{K^*_p} = 0$. 
\end{proof}

\begin{lem} \label{lem:level-Invariant}
Let $K' \leq K^*_p$ be a finite index subgroup and let $\rho_p$ be such that $\eta_p^{K' \cap P^*_p} \ne 0$.
Define 
\[
f\,:\, G^*_p \rightarrow \C, \qquad f(g) = \left\lbrace \begin{array}{cc} (\delta^{1/2} \eta_p)(b) & g = bk \in P^*_p \cdot K' \\ 0 & g \not \in P^*_p \cdot K'. \end{array} \right. 
\]
Then $f$ is a well defined, non-zero, $K'$-invariant vector in $\mathrm{ind}(\eta_p)$.
\end{lem}

\begin{proof}
The function $f$ is well defined, since if $b_1k_1 = b_2k_2$, for some $b_1,b_2 \in P^*_p$ and $k_1,k_2 \in K'$, then $b_1^{-1}b_2 = k_1k_2^{-1} \in P^*_p \cap K'$, and because $\delta$ and $\eta_p$ are trivial on $P^*_p \cap K'$, we get 
\[
f(b_1k_1) = (\delta^{1/2} \eta_p)(b_1) = (\delta^{1/2} \eta_p)(b_2) = f(b_2k_2).
\]
The function $f$ is obviously $K'$-invariant from the right and non-zero since $f(1)=1$.
The $K'$-invariance of $f$ also proves condition (i) in Equation \ref{eq:ind}.
The function $f$ belongs to $\mbox{ind}(\eta_p)$, since if $g \not \in P^*_p \cdot K'$, then $bg \not \in P^*_p \cdot K'$ for any $b\in P^*_p$, hence $f(bg) = 0 = (\delta^{1/2} \eta_p)(b) f(g)$, and if $g = b'k' \in P^*_p \cdot K'$, then
\[
f(bg) = f(bb'k') = (\delta^{1/2} \eta_p)(bb') = (\delta^{1/2} \eta_p)(b) (\delta^{1/2} \eta_p)(b') = (\delta^{1/2} \eta_p)(b) f(g).\qedhere
\]
\end{proof}

\begin{lem} \label{lem:level-intertwining}
Let $w =\bsmx & & 1 \\ & -1 & \\ 1 &  &\esmx \in K^*_p$, let $K' \leq K^*_p$ and let $0 \ne f \in \mathrm{ind}(\eta_p)^{K'}$.
Then
\[
\int_{N^*_p} f(w^{-1}nw)dn \ne 0 \qquad \Rightarrow \qquad \pi^n(\rho_p)^{K'} \ne 0.
\]
\end{lem}

\begin{proof}
Denote $\bar{P}^*_p = w^{-1} P^*_p w$ the subgroup of lower triangular matrices and $\bar{N}^*_p = w^{-1} N^*_p w$ the subgroup of unipotent lower triangular matrices.
Note $M^*_p = w^{-1}M^*_pw$, and define $w.\eta_p := \eta_p(w^{-1}mw)$ for any $m\in M^*_p$.
Let $\mbox{ind}(\eta_p) = \mbox{ind}^{G^*_p}_{P^*_p}(\eta_p)$ and $\mbox{Ind}^{G^*_p}_{\bar{P}^*_p}(w.\eta_p)$ be the unitary parabolic induction from $P^*_p$ and $\bar{P}^*_p$, respectively.
Define the intertwining operator
\[
J_{\bar{P}^*_p | P^*_p} \,:\, \mbox{ind}^{G^*_p}_{P^*_p}(\eta_p) \rightarrow \mbox{Ind}^{G^*_p}_{\bar{P}^*_p}(w.\eta_p), \qquad J_{\bar{P}^*_p | P^*_p}f(g) = \int_{\bar{N}^*_p} f(ng)dn.
\]
By \cite[Corollary 3.2]{konno2003note} the unique irreducible quotient of $\mbox{ind}^{G^*_p}_{P^*_p}(\eta_p)$, i.e. the Langlands quotient $\pi_p^n(\rho_p)$, is isomorphic to the image of the intertwining operator $J_{\bar{P}^*_p | P^*_p}$.
Since $J_{\bar{P}^*_p | P^*_p}$ is an intertwining operator and $f$ is $K'$-invariant, $J_{\bar{P}^*_p | P^*_p}f \in \pi^n_p(\rho_p)$ is $K'$-invariant. 
Observe that $J_{\bar{P}^*_p | P^*_p}f(1) = \int_{\bar{N}^*_p} f(n)dn$, therefore if $\int_{\bar{N}^*_p} f(n)dn \neq 0$ then $J_{\bar{P}^*_p | P^*_p}f$ is non-zero $K'$-invariant vector of $\pi^n_p(\rho_p)$, which proves the claim.
\end{proof}

\begin{lem} \label{lem:level-integral}
Let $K'$ be one of the following subgroups of $K_p^\ast$:
\begin{itemize}
 \item  $K'=K^*_p(p^m) =  \{g \in K^*_p \,:\, g \equiv I \mod{p^m}\}$ for some $m\geq 1$, or
 \item  $K'=\{g\in K_p^* \,:\, \det g\equiv \langle \zeta^a \rangle \mod p\}$, for some $a\in \Z/6\Z$.
\end{itemize}
Then
\[
w^{-1}N^*_p w \cap P^*_p\cdot K'  = w^{-1}N^*_p w \cap K'.
\]
In particular, if $\eta^{K' \cap P^*_p} \ne 0$ and  $0 \ne f \in \mathrm{ind}(\eta_p)^{K'}$ is as in Lemma \ref{lem:level-Invariant}, then
\[
\int_{N^*_p} f(w^{-1}nw)dn = [N^*_p \cap K^*_p : N^*_p \cap K']^{-1}.
\]
\end{lem}

\begin{proof}
Let $K'$ be one of the above subgroups.
Let $ \bar{n} = w^{-1} n(x,y) w \in w^{-1}N^*_p w$, and assume $\bar{n} \in P^*_p K'$, namely there exists 
\[b  = \bmx \alpha & * & * \\ & \beta & * \\ & & \bar{\alpha}^{-1} \emx \in P^*_p,\] 
such that $b\bar{n} \in K'$.
Since the last row of $b\bar{n}$ is $(\bar{\alpha}^{-1}y, \bar{\alpha}^{-1}\bar{x}, \bar{\alpha}^{-1})$, by  $K'\subset K_p^*$, we obtain that  $\alpha \in  \mO_p^\times$ and that $x,y \in \mO_p$, i.e. $\bar n \in K_p^*$. 
In case $K'=K_p^*(m)$ we even get the stronger conditions $\alpha \in  1+p^m\mO_p^\times$ and  $x,y \in p^m\mO_p$, which proves the first claim. In the second case, the same holds true because $\det \bar n=1$.

The second claim follows from the fact that $\mbox{supp}(f) = P^*_p \cdot K'$ and $f | K' \equiv 1$, hence
\begin{align*}
\int_{N^*_p\cap K^*_p} f(w^{-1}nw)dn &= \int_{N^*_p\cap K'} f(w^{-1}nw)dn \\&= \int_{N^*_p \cap K'} dn = [N^*_p \cap K^*_p : N^*_p \cap K']^{-1}.\qedhere
\end{align*}
\end{proof}

Combining the above lemmas we get the following criterion for the existence of a non-zero invariant vector in the Langlands quotient $\pi^n(\rho_p)$ of $\mbox{ind}(\eta_p)$, in terms of the level of the character $\eta_p = \eta(\rho_p)$, for levels which are principal congruence subgroups of $K^*_p$, or satisfy a determinant congruence condition.

\begin{prop}\label{prop:level-Langlands}
For each of the subgroups $K'$ of Lemma~\ref{lem:level-integral}, and for any character $\rho_p$, 
\[
\pi^n(\rho_p)^{K' } \ne 0 \qquad \Longleftrightarrow \qquad \eta_p^{K'  \cap P^*_p} \ne 0.
\]
\end{prop}

\begin{proof}
$(\Rightarrow)$ If $\pi^n(\rho_p)^{K'} \ne 0$, then $\mbox{ind}(\eta_p)^{K'} \ne 0$, and by Lemma \ref{lem:level-Frobenius}, $\eta_p^{K' \cap P^*_p} \ne 0$.\\
$(\Leftarrow)$ If $\eta_p^{K' \cap P^*_p} \ne 0$, then by Lemmas \ref{lem:level-integral} and \ref{lem:level-intertwining}, $\pi^n(\rho_p)^{K'} \ne 0$.
\end{proof}

Let $K^*_p = GL_3(\mO_p) \cap G^*_p$, $K^{**}_p = aGL_3(\mO_p)a^{-1} \cap G^*_p$ and $\boldsymbol{I}^*_p = K^*_p \cap K^{**}_p$, where $\mO_p$ is the ring of integers of $E_p$, $\varpi_p$ a uniformizer of $E_p$ and $a = \mbox{diag}(\varpi_p,1,1)$, i.e. $K^*_p$ and $K^{**}_p$ are two non-conjugate maximal parahoric subgroups and $\boldsymbol{I}^*_p$ is an Iwahori subgroup  of $G^*_p$.
Let $GL_3(\mO_p,\varpi_p^e)$ be the kernel of the modulo $\varpi_p^e$ map, and let  $\boldsymbol{I}^*_p(p^e) = GL_3(\mO_p,p^e) \cap G^*_p$, $K^{**}_p(p^e) = a GL_3(\mO_p,p^e) a^{-1} \cap G^*_p$ and $\boldsymbol{I}^*_p(p^e) = K^*_p(p^e) \cap K^{**}_p(p^e)$.
Let $\boldsymbol{I}_p$, $K_p(p^e)$ and $\boldsymbol{I}_p(p^e)$ be the  images  of $\boldsymbol{I}^*_p$, $K^*_p(p^e)$, and $\boldsymbol{I}^*_p(p^e)$, respectively, under the isomorphism from $G^*_p$ to $G_p$ which sends $K^*_p$ to $K_p$.
Define the Moy-Prasad depth of a representation $\pi$ of $G_p$ to be
\[
d(\pi) = \inf\{ r \geq 0 \;:\; \forall e > r, \quad  \pi^{\boldsymbol{I}_p(p^e)} \ne 0 \}
\]
Let $\mu$ be an extension of $\omega = \omega_{E/Q}$, the class field character associated to $E/\Q_p$.

\begin{prop}\label{prop:depth-preservation}
For any character $\rho = (\rho_1,\rho')$ of $U_1(\Q_p)^2$, we have  $d(\pi^n(\rho)) = d(\pi^s(\rho))$. 
Furthermore, if $d(\rho) \geq d(\omega)$ then $d(\pi) = d(\rho)$ for any $\pi \in \Pi(\rho)$.
\end{prop}


For $p>5$ the claim follows directly from Oi's depth preservation result between L-packets and L-parameters in the local Langlands correspondence \cite{oi2023depth}, combined with the observation that the L-parameters of  $\pi^n(\rho)$ and $\pi^s(\rho)$ differ by how they act on the Deligne $SL_2$ factor, hence in particular are of the same depth (\cite[A.10]{bellaiche:2009}). 
However, we are interested in the case $p=3$, so we provide a different proof using the theta correspondence.

\begin{proof}
By \cite[Lemma 5.1.2]{Gelbart1991LfunctionsFourier}, $\Pi(\rho)$ is equal to the set of local Weil representations or local theta lifts, $\{\omega(\psi, \gamma, \chi) \,:\, \psi \}$, where $(\gamma,\chi)$ are determined uniquely from  $\rho$, by the relations $\gamma = \mu \rho_{1,E} \rho'_{E}$ and $\chi = \gamma^1 \rho'$, where $\gamma^1$ is the restriction of $\gamma$ to $U_1(\Q_p)$, and where $\psi$ runs over the additive characters of $\Q_p$ modulo $N_{E/\Q_p}(E^\times)$.
The theta correspondence is the one defined in \cite[Section 3]{Gelbart1991LfunctionsFourier} according to their choice of splitting.
By \cite[Main Corollary, p. 533]{pan2002depth}, the theta correspondence, according to the splitting defined in \cite{pan2001splittings}, is depth preserving.
Two splittings differ by a character (\cite[Section 2.5]{gan2023automorphic}).
Combining the results of Gelbart-Rogawski, Pan and Gan we get that  all the members in a given A-packet are of the same depth, i.e. $d(\pi^n(\rho)) = d(\pi^s(\rho))$.
Finally, we note that $\pi^n(\rho)$ is the Langlands quotient of the parabolic induction ind$(\eta)$. 
By \cite[Theorem 5.2]{prasad1994unrefined}, the Moy-Prasad depth is preserved under parabolic induction, hence $d(\pi^n(\rho)) = d(\eta)$.
If $d(\rho) \geq d(\omega)$, we have $d(\eta)=d(\rho)$, and hence $d(\pi) = d(\rho)$ for any $\pi \in \Pi(\rho)$.
\end{proof}

For the remainder of this section $E=\Q[\sqrt{-3}]$, $\Phi = I$ and $G = U_3(\Q[\sqrt{-3}],I)$.
Let $\sqrt{-2} \in \mO_3$ be such that $\sqrt{-2} \equiv 1\mod3$, and define
\begin{align*}
A &:= \frac12 \bmx 2\sqrt{-2}+\sqrt{-3} & 1+\sqrt{-3} & -1 \\ (1+\sqrt{-2})+(1-\sqrt{-2})\sqrt{-3} & 2 & (1-\sqrt{-2})+(1+\sqrt{-2})\sqrt{-3} \\ -1 & 1+\sqrt{-3} & -2\sqrt{-2}+\sqrt{-3} \emx 
\\ &\in  GL_3(\Z_3[\sqrt{-3}]).
\end{align*}
Then $AA^* = J$, and therefore conjugation by $A$ gives the following isomorphisms $G^*_3 = A \cdot G_3 \cdot A^{-1}$, $K^*_3 = A \cdot K_3 \cdot A^{-1}$ and $K^*_3(3) = A \cdot K_3(3) \cdot A^{-1}$.
Let $K_3(C)$ be as in \eqref{eq:K3C}, and let $K^*_3(C) = A \cdot K_3(C) \cdot A^{-1}$.

\begin{lem} \label{lem:K3C}
In the above notations,
\[
K^*_3(C) \cap P^*_3 = \left\lbrace g\in K^*_3 \cap P^*_3 \;|\; \exists a \in 2\Z/6\Z \;:\; g \equiv \zeta^a \cdot \bmx 1 &  & \zeta^{-a} - 1 \\  & 1 &  \\  & & 1 \emx \mod 3 \right\rbrace.
\]
\end{lem}

\begin{proof}
First note that
\[
A  \equiv \bmx \zeta^2 & \zeta & 1 \\ 1 & 1 & \sqrt{-3} \\ 1 & \zeta & \zeta \emx \mod3
\]
and
\[
A^{-1} =  A^* J \equiv  \bmx 1 & -1 & \zeta^4 \\ \zeta^5 &-1 & \zeta^5 \\ \zeta^5 & \sqrt{-3} & 1 \emx \mod3.
\]
Next we prove both containments in the claim.

$(\supset)$ - 
Note that any element in the set on the right hand side of the claim, can be written as a product $g = b_a \cdot g'$, for some  $a \in \{0,2,4\}$, where 
\[b_a = \zeta^a \cdot \bmx 1 &  & \zeta^{-a} - 1 \\  & 1 &  \\  & & 1 \emx\] 
and 
\[
g' \in K^*_3(3) \cap P^*_3 = \{h \in K^*_3 \cap P^*_3 \;:\; h \equiv I \mod3 \}.\]
Since $g' \in K^*_3(3) \cap P^*_3 $ is clearly contained in $K^*_3(C) \cap P^*_3$, it suffices to show that $b_a \in K^*_3(C) \cap P^*_3$, which follows from the following direct computation
\begin{align*}
A^{-1}b_a^*A &\equiv \zeta^a \bmx \zeta^2 - 1 + \zeta^4 + \zeta^{-a} - 1 & * & * \\ * & 1 - 1 + \zeta^{-a}  & * \\ * & * & \zeta^5 - 3 + \zeta +  \zeta^{-a} - 1 \emx 
\\&\equiv  \bmx 1 & * & * \\ * & 1 & * \\ * & * & 1 \emx \mod 3.
\end{align*}

$(\subset)$ - 
Let $g \in K^*_3(C) \cap P^*_3$, and note that since $g \in K^*_3 \cap P^*_3$, $g = \bmx \alpha & \alpha x & \alpha y \\  & \beta & \beta \bar x \\ & & \bar{\alpha}^{-1} \emx$ for some $\alpha, \beta \in \mO_{\sqrt{-3}}^\times, y+\bar y=x\bar x$, and since $g \in K^*_3(C)$, $g = Ag'A^{-1}$ for some $g' \in K_3(C)$. 
Hence $\alpha + \beta + \bar{\alpha}^{-1} = \mbox{Trace}(g) = \mbox{Trace}(g') \equiv  \mbox{Trace}(\bmx 1 & * & * \\ * & 1 & * \\ * & * & 1 \emx) \equiv 0 \mod3$.
As $\mO_{\sqrt{-3}}^\times = (1+3\mO_{\sqrt{-3}}) \cdot \zeta^{\Z/6\Z}$, there is an $a \in \Z/6\Z$ such that $\alpha \equiv \zeta^a \mod 3$.
Then $\bar{\alpha}^{-1} \equiv \alpha \mod 3$ and combined with the trace condition we see that $\beta \equiv \alpha \mod 3$. 
Hence  $g \equiv \zeta^a \bmx 1 & x &y \\  & 1 & \bar{x} \\  & & 1 \emx  \mod 3$, for some $a \in \Z/6\Z$, and some $x,y \in \Z[\sqrt{-3}]/3$, $x\bar{x} = y + \bar{y}$.
A direct calculation shows
\begin{align*}
&\bmx 1 & * & * \\ * & 1 & * \\ * & * & 1 \emx \equiv A^{-1} g A  
\\ \equiv &
\zeta^a \bmx 1 + y + x - \bar{x} & * &* \\  * & 1 + y + \zeta^5 x - \zeta \bar{x} & * \\ * & * & 1 + y +\zeta^5 \sqrt{-3} x + \zeta \sqrt{-3} \bar{x} \emx \mod{3}.
\end{align*}
Since the diagonal elements are congruent mod $3$ we have $x- \bar{x} \equiv \zeta^5 x -\zeta \bar{x} \equiv \zeta^5 \sqrt{-3} x + \zeta \sqrt{-3} \bar{x} \mod 3 $ which holds if and only if $x\equiv \zeta^4\bar x \mod 3$.
Now we use that $y \equiv \zeta^{-a} - 1-x+\bar x\mod 3$ hence $y+\bar y\equiv 1+\zeta^a+\zeta^{-a}\mod 3$. Since $y + \bar{y} = x \bar{x}$, we get that 
$a \in \{ 0,2,4\}$ and thus $x\bar x \equiv 0\mod 3$, which, along with the previous condition that $x\equiv \zeta^4 \bar x\mod 3$, implies $x\equiv 0 \mod 3$ and completes the proof.
\end{proof}

\begin{prop}\label{prop:level-K3C} 
For any character $\rho_3$,
\[ 
\pi^n(\rho_3)^{K^*_3(3, C)} \neq 0  \qquad \Longleftrightarrow \qquad \eta_3^{K^*_3(C) \cap P^*_3} \neq 0. 
\]
\end{prop}

\begin{proof}
$(\Rightarrow)$ 
If $\pi^n(\rho_3)^{K_3^*(C)} \ne 0$, then $\mbox{ind}(\eta_3)^{K_3^*(C)} \ne 0$, and by Lemma \ref{lem:level-Frobenius} and \ref{lem-(E)}, $\eta_3^{K_3^*(C) \cap P^*_3} \ne 0$.

$(\Leftarrow)$ 
By Lemma \ref{lem:level-intertwining},  we have to show that for the function $f$ of Lemma \ref{lem:level-Invariant}, the following integral is non-zero
\[
F := \int_{\bar{N}^*_3} f(n)dn = \int_{\bar{N}^*_3 \cap K^*_3} f(n)dn + \int_{\bar{N}^*_3 \setminus K^*_3} f(n)dn =: F_1 +F_2.
\]
Because $K_3^\ast(3)\subset K_3^\ast(3,C)\subset K_3^\ast(\sqrt{-3})$ and because, by Lemma \ref{lem:level-integral}, $(\bar N_3^\ast\setminus K_3^\ast)\cap K_3^\ast(\sqrt{-3}^m)=\emptyset$ for all $m$, we know $F_2=0$. 
For the integral $F_1$ we fix notation:
\[
\bar{n}(x,y)=\bmx 1 & & \\ x & 1 & \\ y & \bar{x} & 1 \emx \in \bar{N}_3^\ast, 
\]
\[n(x,y) = \bmx 1 & x & y \\ & 1 & \bar{x} \\ & & 1 \emx \in N_3^\ast, \textrm{ whenever } x\bar{x}=y+\bar{y}.
\]
The condition $x\bar x=y+\bar y$ implies that $y$ has the form $y=\frac{x\bar x}2+y_2\sqrt{-3}$, where $y_2\in\Q_3$. 
This yields  an obvious isomorphism $\bar N_3^\ast\cong E_3\times \Q_3\cong\Q_3^3$, under which the choice of Haar measure on $\bar N_3^\ast$ coincides with that on $\Q_3^3$ giving $\Z_3^3$ measure one.

By Lemma \ref{lem:K3C} we see that $N^*_3 \cap K^*_3(C)$ is the kernel of the modulo $3$ map on $N^*_3 \cap K^*_3$.
Since  $\bar{N}^*_3  = w N^*_3 w$ and $w \in K^*_3$, we get also that $\bar{N}^*_3 \cap K^*_3(C)$ is the kernel of the modulo $3$ map on $\bar{N}^*_3 \cap K^*_3$.

The set $ \Sigma=\{\bar{n}(x_1+x_2\sqrt{-3},\frac{x_1^2 + 3x_2^2}2 + y_2\sqrt{-3}) \,:\, x_1,x_2,y_2 \in \{-1,0,1\} \}$ is a transversal set of  $\bar{N}^*_3 \cap K^*_3 / \bar{N}^*_3 \cap K^*_3(C)$. 
Breaking $\bar{N}^*_3 \cap K^*_3$ into $\bigsqcup_{g\in \Sigma}g \left(\bar{N}^*_3 \cap K^*_3(C)\right)$, and using the fact that $f$ is $K^*_3(C)$-invariant, we get
\[
F_1 =  \frac1{27} \sum_{x_1,x_2,y_2=-1,0,1}^3 f(\bar{n}(x_1+x_2\sqrt{-3},\frac{x_1^2 + 3x_2^2}2 + y_2\sqrt{-3})).
\]
 
By Lemma~\ref{lem:level-Invariant}, we know $f(\bar n(x,y))\neq 0$ if and only if there is $b\in P^\ast_3\cap K^\ast_3$ such that $b\bar n(x,y)\in K_3^\ast(3,C)$. (And in this case, $f(\bar n(x,y))=\xi(b^{-1}))$, where $\xi=\delta_3^\frac12\eta_3$.)
Equivalently, if and only if $A^{-1}\cdot b\cdot \bar n(x,y) \cdot A\in K_3(C)$, i.e.
\begin{equation}\label{eq:F_1-condition}
 A^{-1}\cdot b\cdot \bar n(x,y) \cdot A \equiv \begin{pmatrix} 1 & \ast & \ast \\ \ast & 1 & \ast \\ \ast & \ast & 1 \end{pmatrix} \mod 3.
\end{equation}
We evaluate this mod-$3$-condition  for $\bar n(x,y)\in\Sigma$ above. 
Putting $b=\begin{pmatrix} \alpha & \alpha v & \alpha w \\ & \beta & \beta \bar v \\ & & 1 /\bar\alpha \end{pmatrix}\in P_3^\ast\cap K_3^\ast$ we obtain the following:
\\
{\bf Case 1:} Let $x\equiv 1\mod \sqrt{-3}$, i.e $x\equiv 1+x_2\sqrt{-3}$, $y\equiv -1+y_2\sqrt{-3}\mod 3$.\\
Then there is no such matrix $b$ in case $x_2\neq y_2$.
But in case $x_2=y_2$, such $b$ is given, for example, by
\[
\alpha\equiv -1,\: \beta\equiv 1+(y_2-1)\sqrt{-3},\: v\equiv 0,\: w\equiv 1+\sqrt{-3} \mod 3,
\]
where we obtain $f(\bar n(x,y))=\xi^{-1}(\diag(-1,1+(y_2-1)\sqrt{-3},-1))$.\\
{\bf Case 2:} Let $x\equiv 2\mod \sqrt{-3}$, i.e $x\equiv -1+x_2\sqrt{-3}$, $y\equiv -1+y_2\sqrt{-3}\mod 3$.
\\
Here, there is no such matrix $b$ if $x_2+y_2\neq 1$.
But in case $x_2+y_2=1$, such $b$ is given, for example, by
\[
\alpha\equiv -1,\: \beta\equiv 1+y_2t,\: v\equiv 0,\: w\equiv 1 \mod 3,
\]
 where we obtain $f(\bar n(x,y))=\xi^{-1}(\diag(-1,1+y_2\sqrt{-3},-1))$.
 \\
{\bf Case 3:} Let $x\equiv 3\mod \sqrt{-3}$, i.e $x\equiv -x_2\sqrt{-3}$, $y\equiv y_2\sqrt{-3}\mod 3$. Then $b$ being defined by $\alpha=\beta=1$, $v=x$ and $w=-y$ is a solution.
We obtain $ f(\bar n(x,y))=\xi^{-1}(\diag(1,1,1))=1$.
Now we sum up the integral:
\begin{eqnarray*}
F_1&=&\frac1{27}\sum_{x_1,x_2,y_2\in\mathbb F_3} f(x_1,x_2,y_2)\\
&=\frac1{27}\Bigl(9&+\sum_{y_2\in\{-1,0,1\}}\bigl(\xi^{-1}(\diag(-1,1+(y_2-1)\sqrt{-3},-1))
\\
&&+ \xi^{-1}(\diag(-1,1+y_2\sqrt{-3},-1))\bigr) \Bigr)\\
& =\frac1{27}\Bigl(9&+2\cdot\bigl(\xi^{-1}(\diag(-1,1,-1))
\\&&+\xi(\diag(-1,\zeta^2,-1))+2\xi(\diag(-1,\zeta^{4},-1))\bigr)\Bigr)\\
&=\frac1{27}\bigl(9&+2\cdot\xi^{-1}(\diag(-1,-1,-1)\cdot(\xi^{-1}(\diag(1,-1,1)
\\&&+\xi(\diag(1,\zeta,1))+\xi(\diag(1,\zeta^{-1},1))\bigr), 
\end{eqnarray*}
in particular, $F_1$ is non-zero. In  case $\xi(\diag(1,\zeta,1))=\phi_3'(\zeta)=\zeta^{\pm 1}$, we obtain $F_1=\frac13$.
\end{proof}

The following is a Corollary of Proposition \ref{prop:depth-preservation}.

\begin{cor}\label{cor:level-3-depth}
If $E = \Q[\sqrt{-3}]$ and $p = 3$, then
\[
\pi^s(\rho_3)^{\boldsymbol{I}^*_3(3)} \ne 0 \qquad \Longleftrightarrow \qquad \eta_3^{K^*_3(3) \cap P^*_3} \ne 0.
\]
\end{cor}

\begin{proof}
We note that $\eta_3^{K_3^*(3)\cap P_3^*}\neq 0$ if and only if $\rho_3=(\rho_{1,3}, \rho'_3)$ is trivial on $U_1(\Z_3,3)^2$ and this is if and only if $\gamma_3^{1+3\mO_3}\neq 0$ and $\chi_3^{U_1(\Z_3,3)}\neq 0$. 
Now we apply Proposition \ref{prop:depth-preservation}.
\end{proof}

\subsection{Proofs of main results} \label{automorphic:proofs}

In this subsection we prove the main results stated in Section \ref{automorphic:theorems}.

\begin{proof}[Proof of Theorem \ref{thm:Ram-gen}]

(2) Assume in contradiction that $\pi \in \mA(K')$ is non-Ramanujan. 
Since $\mA(K'')$ is Ramanujan,  $\pi \not \in \mA(K'')$, hence there exists $p \in \mathrm{Ram}(K') \setminus \mathrm{Ram}(K'')$, such that $\pi_p^{K''_p} = 0$.
Since $\boldsymbol{I}^*_p \subset K'_p \subset K''_p \subset K^*_p$, where $\boldsymbol{I}^*_p$ is an Iwahori subgroup, we get that $\pi_p^{K^*_p} = 0$ and $\pi_p^{\boldsymbol{I}^*_p} \ne 0$.
By Theorem \ref{thm:A-Ram}, $\pi$ belongs to an A-packet, i.e. there exists an automorphic character $\rho$ of $U_1(\Q)^2$, such that $\pi_v \in \Pi'(\rho_v)$ for any place $v$.
Therefore, $\pi_p$ is equal to either $\pi^n(\rho_p)$ or $\pi^s(\rho_p)$, for some $\rho_p$. However, by Lemma \ref{lem:level-supercuspidal}, $\pi^s(\rho_p)^{\boldsymbol{I}^*_p} = 0$. 
Thus we must have $\pi_p=\pi^n(\rho_p)$. 
Since $\pi_p^{\boldsymbol{I}^*_p} \ne 0$, by Theorem \ref{thm:Ram-gen}(1), $p$ must be unramified in $E$. 
Then by Proposition \ref{prop:level-spherical} $\eta_p^{K_p^*\cap P_p^*}=0$ but this is impossible because  by \cite{casselman1980unramified}, $\pi_p^{\boldsymbol{I}^*_p} \ne 0$ implies $\eta_p^{K_p^*\cap P_p^*}\neq 0$.
\end{proof}

\begin{proof}[Proof of Theorem \ref{thm:NonRam-Eis}]
(1) Define the automorphic character $\rho = (\mu_1, \mu_1^{-1})$ of $U_1(\Q)^2$, where $\mu_1 = \mu|_{U_1(\Q)}$ is the restriction from $GL_1(E)$ to $U_1(\Q)$ of $\mu$ from Definition \ref{defn:mu}.
Then $\rho$ satisfies the following: 
(i)  $\rho_\infty(x) = (x^{-1},x)$ for any $x \in U_1(\R)$, (ii) $\eta_p^{K^*_p \cap P^*_p} \ne 0$ for any $p\ne 3$, and (iii) $\eta_3^{K^*_3(3) \cap P^*_3} \ne 0$.
Define the adelic representation $\pi = \pi'^s(\rho_\infty) \otimes \pi^s(\rho_3) \otimes \bigotimes_{p\ne 3} \pi^n(\rho_p) \in \Pi'(\rho)$.
By Proposition \ref{prop:phi-global-unram}, $\epsilon(1/2,\phi) = 1$, where $\phi = \phi(\rho) = \mu^2 \bar{\mu}^{-1}$, and by Theorem \ref{thm:Rogawski-A-packets}, we get that $\pi$ is automorphic.
By Theorem \ref{thm:A-Ram}, we get that $\pi$ is non-Ramanujan.
By Proposition \ref{prop:archtrivial} and property (i), $\pi'^s(\rho_\infty)$ is the trivial representation of $G_\infty$, hence $\pi_\infty^{K_\infty} \ne 0$ for $K_\infty = G_\infty$.
By Proposition \ref{prop:level-spherical} and property (ii),  $\pi^n(\rho_p)^{K^*_p} \ne 0$, for any $p\ne 3$.
By Corollary \ref{cor:level-3-depth} and property (iii), $\pi^s(\rho_3)^{\boldsymbol{I}^*_3(3)} \ne 0$.
Combining all of the above, we get that $\pi \in \mA(K')$, hence $\mA(K')$ is non-Ramanujan for $K' = K_\infty \otimes \boldsymbol{I}_3(3) \otimes \bigotimes_{p\ne 3} K_p$.

(2) Let $q \geq 5$ be a prime and let $a_q \in U_1(\F_q)$ be a generator of the cyclic group of order $n_q = |U_1(\F_q)|$ chosen such that $\zeta \,\mbox{mod}\,q\, = a_q^{\frac{n_q}{6}}$. Let $\chi' \,:\, U_1(\F_q) \rightarrow \C^\times$ by $\chi'(a_q) = \left(e^{2 \pi i / n_q}\right)^{b_q}$, where $b_q=6$ if $q\equiv 2\mod 3$ and $b_q=3$ if $q\equiv 1\mod 3$,  and let $\chi$ be the automorphic character of $U_1(\Q)$ associated to $\chi'$ by Lemma \ref{lem:class1}. Hence $\chi_q(\zeta)=\left(e^{2\pi i /6}\right)^{b_q}$ and $\chi_q$ is nontrivial.
Define the automorphic character $\rho = (\chi\mu_1, \mu_1^{-1})$ of $U_1(\Q)^2$, where $\mu_1 = \mu|_{U_1(\A)}$ is the restriction from $GL_1(\A_E)$ to $U_1(\A)$ of $\mu$ from Definition \ref{defn:mu}.
Then $\rho$ and the associated automorphic character $\eta=\eta(\rho)$ of $M^\times/\Q$  satisfies the following: 
(i) $\rho_\infty(x) = (x^{-1},x)$ for any $x \in U_1(\R)$, (ii) $\eta_p^{K^*_p \cap P^*_p} \ne 0$ for any $p\ne 3,q$, (iii) $\eta_3^{K^*_3(C) \cap P^*_3} \ne 0$, and (iv) $\eta_q^{K^*_q(q) \cap P^*_q} \ne 0$.
Define the adelic representation $\pi = \pi'^s(\rho_\infty) \otimes \bigotimes_p \pi^n(\rho_p) \in \Pi'(\rho)$.
By Proposition \ref{prop:phi-global-p}, $\epsilon(1/2,\phi) = -1$, where $\phi = \phi(\rho) = \mu^2 \bar{\mu}^{-1} \chi \bar{\chi}^{-1}$, we get that from Theorem \ref{thm:Rogawski-A-packets}, $\pi$ is automorphic.
By Theorem \ref{thm:A-Ram}, we get that $\pi$ is non-Ramanujan.
By Proposition \ref{prop:archtrivial} and property (i), $\pi'^s(\rho_\infty)$ is the trivial representation of $G_\infty$, hence $\pi_\infty^{K_\infty} \ne 0$ for $K_\infty = G_\infty$.
By Proposition \ref{prop:level-spherical} and property (ii),  $\pi^n(\rho_p)^{K^*_p} \ne 0$, for any $p\ne 3,q$.
By Proposition \ref{prop:level-K3C} and property (iii), $\pi^n(\rho_3)^{K^*_3(C)} \ne 0$.
By Proposition \ref{prop:level-Langlands} and property (iv), $\pi^n(\rho_q)^{K^*_q(q)} \ne 0$.
Combining all of the above, we get that $\pi \in \mA(K')$, hence $\mA(K')$ is non-Ramanujan for $K' = K_\infty \cdot K_3(C) \cdot K_q(q) \cdot \prod_{p\ne 3,q} K_p$.
\end{proof}

In the Proposition below we show by global means that two  local supercuspidal representations do not have non-zero $K_3(C)$-invariant vectors. This is precisely the missing piece we need to complete the proof of Theorem \ref{thm:Ram-Eis-kq}.
\begin{prop}\label{cor:supercuspidal}   Let $\rho_3=(\rho_{1,3},\rho_3')$ where either
\begin{enumerate}
\item\label{case1} $b(\rho_{1,3})=4$,  $\rho'_3(\zeta)=\zeta,$ and $\rho_{1,3},\rho_3'|_{1+3\mathcal O_3}\equiv 1$, or
\item \label{case2} $b(\rho_{1,3})=1$, $\rho'_3(\zeta)=\zeta^{-1},$ and $\rho_{1,3},\rho_3'|_{1+3\mathcal O_3}\equiv 1$.
\end{enumerate}   
Then $\pi^s(\rho_3)^{K_3(C)}=0$. 
\end{prop}
\begin{proof}
First we note that by Lemma \ref{lem:u1} $U_1(\Q_3)\cong U_1(\Z_3,3)\times \langle \zeta \rangle$, thus if a local character on $U_1(\Z_3)$ is trivial on $U_1(\Z_3,3)$ then it is uniquely determined by where it sends $\zeta$. Hence  characters that satisfy the properties described in \eqref{case1} and \eqref{case2} are uniquely defined on $U_1(\Q_3)$.
We now construct automorphic characters $\rho=(\rho_1,\rho')$  of $U_1(\Q)^2$ that at the prime $3$ are equal to the $\rho_3$ in the statement of the corollary.

For case \eqref{case1}
let   $\rho=(\mu_1^{-1}\times \mu_1^{-1}, \mu_1)$. Then 
\begin{enumerate}
\item $\rho_{1,\infty}(x)=x^2$, $\rho_\infty'(x)=x^{-1}$ 
\item $\rho_{1,3}(\zeta)=\zeta^4$, $\rho_3'(\zeta)=\zeta$, $\rho_{1,3}|_{1+3\mathcal O_3}\equiv 1, \rho_3'|_{1+3\mathcal O_3}\equiv 1$ so this is indeed the $\rho_3$ in \eqref{case1} above. 
\item $\rho_p^{K^*_p \cap P^*_p} \ne 0$ for any $p\ne 3$
\end{enumerate}

For case \eqref{case2} let $\rho=(\mu_1, \mu_1^{-1})$. Then 
\begin{enumerate}
\item $\rho_{1, \infty}(x)=(x)^{-1}$, $\rho'_\infty(x)={x}$, 
\item $\rho_{1,3}(\zeta)=\zeta$, $\rho'_3(\zeta)=\zeta^{-1}$, $\rho_{1,3}|_{1+3\mathcal O_3}\equiv 1, \rho_3'|_{1+3\mathcal O_3}\equiv 1$ so this is indeed the $\rho_3$ in \eqref{case2} above,
\item $\rho_p^{K^*_p \cap P^*_p} \ne 0$ for any $p\ne 3$
\end{enumerate}

In both cases  $\rho$ is automorphic, by Proposition \ref{prop:archtrivial} $\pi'^s(\rho_\infty)$ is the trivial representation,  and by Proposition \ref{prop:phi-global-unram}, $\epsilon(1/2, \phi)=1$.
Denote $\pi = \pi'^s(\rho_\infty) \otimes \bigotimes_{v\neq \infty, 3} \pi^n(\rho_v)\otimes \pi^s(\rho_3)$.
By Theorem \ref{thm:Rogawski-A-packets}, $m(\pi)=1$, i.e. $\pi$ is automorphic. 
Let $K'=\otimes_{v\neq 3}K_v \otimes K_3(C)$. Then by Corollary \ref{cor:Ram-Eis}, Part (2), $\pi^{K'}=0$ but by Proposition \ref{prop:level-spherical} $\pi_v^{K_v}\neq 0$ for $v\nmid 3$, hence we must have $\pi^s(\rho_3)^{K_3(C)}=0$.
\end{proof}

\begin{proof}[Proof of Theorem \ref{thm:Ram-Eis-kq}] 
Throughout this proof let $b_v = b(\rho_{1,v}) \in \Z/6\Z$ be as in Definition \ref{defn:phi-b}. Assume in contradiction that there exists $\pi=\otimes_v \pi_v \in \mA(K')$ that is non-Ramanujan. 
Then by definition of   $\mA(K')$,  $\pi_v^{K_v'}\neq 0$ for each $v$ and by Theorem \ref{thm:A-Ram} there exists an automorphic character $\rho$ of $U_1(\Q)^2$ such that  $\pi_v \in \Pi'(\rho_v)$, i.e., for each place $v$, $\pi_v=\pi^n(\rho_v)$ or $\pi^s(\rho_v)$. 

For $v$ split in $E$, $\Pi'(\rho_v)=\{\pi^n(\rho_v)\}$. For $v$ not split in $E$, by Lemma \ref{lem:level-supercuspidal} $\pi^s(\rho_v)^{\boldsymbol{I}^*_v} = 0$ and hence $\pi^s(\rho_v)^{K^*_v}=0$.  
Thus for $v\neq 3, \infty$ we must have $\pi_v(\rho_v)=\pi^n(\rho_v)$. By Proposition \ref{prop:level-spherical} for $v\neq 3$, $\pi^n(\rho_v)^{K^*_v} \ne 0$ if and only if $\eta_v^{K_v^*\cap P_v^*}\neq 0$ so for $v\neq 3, q, \infty$, $\rho_v$ must be unramified. 

At $\infty$, $\pi_\infty=\pi'^s(\rho_\infty)$ and by Proposition \ref{prop:archtrivial} at $\infty$ we must have either
\begin{enumerate}
\item $\rho_{1,\infty}(\frac{\alpha}{\bar \alpha})=(\frac{\alpha}{\bar \alpha})^2$, i.e., $b_\infty=2$, and $\rho'_\infty(\beta)= \beta^{-1}$, or
\item $\rho_{1, \infty}(\frac{\alpha}{\bar \alpha})=(\frac{\alpha}{\bar \alpha})^{-1}$, i.e., $b_\infty=5$, and $\rho'_\infty(\beta)=\beta$.
\end{enumerate}
We  have shown we must have
$$\pi = \pi'^s(\rho_\infty) \otimes \bigotimes_{v\neq \infty, 3} \pi^n(\rho_v)\otimes \pi_3.$$

By Proposition \ref{prop:level-Langlands}, $\pi^n(\rho_q)^{K'_q}\neq 0$ if and only if $\eta_q^{K_q'\cap P_q^*}\neq 0$.  Since $q$ is split in $E$,  $U_1(\Q_q)\cong \Q_q^\times$ and $\Z_q^\times/(1+q\Z_q)\cong \F_q^\times$. We fix a choice of $\beta \in \Z_q^\times$ such that its image in $\F_q^\times$ is a generator there and $\beta^{(q-1)/6}\equiv \zeta \mod q$. 
Recall our restriction that $q\equiv 1\mod{12}$, hence we can let $\xi=\beta^{(q-1)/12}\in E_q^\times$ . Then $\xi^2=\zeta$  and $\diag(\xi, \xi^{-2}, \xi)\in K_q'$. We note that $\eta_q(\xi, \xi^{-2}, \xi)=\rho_{1,q}(\xi^2)\rho_q'(1)=\rho_{1,q}(\zeta)=\zeta^{b_q}$. We also note that $\diag(1,\zeta, 1)\in K_q'$ and  $\eta_q(1,\zeta, 1)=\rho_q'(\zeta)$.  Hence $\pi^n(\rho_q)^{K'_q}\neq 0 \Rightarrow b_q=0$ and $\rho_q'(\zeta)=1$. 

By Theorem \ref{thm:Rogawski-A-packets}, $m(\pi)=1$ if and only if either 
$\epsilon(1/2,\phi)=-1$ and $\pi_3=\pi^n(\rho_3)$
or
 $\epsilon(1/2,\phi)=1$ and $\pi_3=\pi^s(\rho_3)$.
By Conjecture \ref{conj:level-packet} if $\pi_3^{K_3(C)^*}\neq 0$ then $\pi^n(\rho_3)^{K_3(C)^*}\neq 0$ and   by Proposition \ref{prop:level-K3C}, $\eta_3^{K_3(C)^*\cap P^*_3} \ne 0$. 
By Lemma \ref{lem:K3C}, $\eta_3^{K_3^*(C)\cap P_3^*}\neq 0$ if and only if $\rho_3=(\rho_{1,3}, \rho'_3)$ is trivial on $U_1(\Z_3,3)^2$ and $\eta_3(diag(\zeta^a, \zeta^a,\zeta^a))=1$ for $a=0, 2, 4$. For $a$ even, $\eta_3(diag(\zeta^a, \zeta^a,\zeta^a))=\mu_3(\zeta^a)\rho_{1,3}(\zeta^{2a})\rho_3'(\zeta^{3a})=\zeta^{a(1+2b_3)}$ and this equals $1$ for even values of $a$ if and only if $b_3=1$ or $4$. Hence $\pi_3^{K_3(C)^*}\neq 0$ implies $\rho_3=(\rho_{1,3}, \rho'_3)$ is trivial on $U_1(\Z_3,3)^2$
and $b_3=1$ or $4$.
By Propositions \ref{prop:phi-global-unram} and \ref{prop:phi-global-p} we now have $m(\pi)=1$ if and only if $\pi_3=\pi^s(\rho_3)$ and
$(b_3, b_q, b_\infty)\in \{(1, 0, -1), (4,0,2)\}$.

If $(b_3, b_q, b_\infty)=(1, 0, -1),$  then we must have that $\rho'_\infty(\beta)=\beta$ and since we must have $\pi^s(\rho_3)^{K_3(C)}\neq 0$ then by Proposition \ref{cor:supercuspidal} $\rho'_3(\zeta)\neq \zeta^{-1}$. For $\rho'$ to be automorphic we need $\rho'(\zeta)=1$ and $\rho'(\zeta)=\rho_\infty'(\zeta)\rho_3'(\zeta)
\rho_q'(\zeta)=\zeta \rho_3'(\zeta)\rho_q'(\zeta)$ so we must have $\rho'_q(\zeta)\neq 1$ but then $\pi_q^{K_q'}=0$.

If $(b_3, b_q, b_\infty)=(4, 0, 2),$  then we must have that $\rho'_\infty(\beta)=\beta^{-1}$ and by Proposition \ref{cor:supercuspidal} $\rho'_3(\zeta)\neq \zeta$. For $\rho'$ to be automorphic we need $\rho'(\zeta)=1$ and $\rho'(\zeta)=\rho_\infty'(\zeta)\rho_3'(\zeta)
\rho_q'(\zeta)=\zeta^{-1} \rho_3'(\zeta)\rho_q'(\zeta)$ so we must have $\rho'_q(\zeta)\neq 1$ but then $\pi_q^{K_q'}=0$.
Hence we can conclude that $\pi^{K'}=0$ so $\pi \not \in \mA(K')$. 
\end{proof}

In the proposition below we prove the following special cases of Conjecture \ref{conj:level-packet}. 

\begin{prop} \label{prop:conj}
Let $G = U_3$ and $p$ be a prime which is inert in $E$.
Let $\rho_p = (\rho_{1,p},\rho_p')$ be a  character of $U_1(\Q_p)^2$, and let $K'_p$  equal either one of the two cases below
\begin{enumerate}
\item $K'_p=K_p(p^m) = G(\Z_p,p^m)$ where $m \in \N$ and $p>10$, or
\item $K'_p=\boldsymbol{I}_p(p^e)$.
\end{enumerate}
Then
\[
\pi^n(\rho_p)^{K'_p}= 0 \qquad \Rightarrow \qquad \pi^s(\rho_p)^{K'_p}= 0.
\]
\end{prop}

\begin{proof} 
(1) Let $G_p = G(\Q_p)$ and 
\[H_p = \left\lbrace \bmx a & 0 & b \\ 0 & \beta& 0 \\ c& 0 & d \emx \in G_p \;:\; \bmx a& b \\ c & d \emx \in U_2(\Q_p), \; \beta \in U_1(\Q_p)  \right\rbrace.\]
Identify the character $\rho_p$  on $U_1(\Q_p)^2$ with the character on $H_p$ via $\rho_p \bmx a & 0 & b \\ 0 & \beta& 0 \\ c& 0 & d \emx=\rho_{1,p}(ad-bc)\rho'_p((ad-bc)\beta)$.
We shall use the notations of \cite{Marshall2014Endoscopycohomologygrowth}, to denote by $C_c^\infty(G_p, \omega)$ the space of compactly supported functions $f$ on $G_p$ satisfying $f(zg) = \omega(z) f(g)$ for any $g\in G_p$ and $z \in Z(G_p)$, and similarly for $C_c^\infty(H_p, \omega\mu^{-1})$.
For an admissible representation $\pi_p$ of $G_p$ and $f \in C_c^\infty(G_p, \omega)$, denote the trace class operator $\pi_p(f) = \int_{G_p} f(g) \pi_p(g)dg$, where $dg$ is the Haar measure which gives $K_p$ measure $1$.
Note that $\tr (\pi_p(\b1_{K'_p})) = \vol(K'_p)\dim(\pi_p^{K'_p})$, and in particular 
\begin{equation}\label{eq:dim-tr}
\tr (\pi_p(\b1_{K'_p})) = 0 \qquad \Leftrightarrow \qquad \pi_p^{K'_p} = 0.
\end{equation}
Say that $f \in C_c^\infty(G_p, \omega)$ and $f^H\in C_c^\infty(H_p, \omega\mu^{-1})$ form a transfer pair if the unstable orbital integrals of $f$ match the stable orbital integrals of $f^H$. 
For more details see \cite[(4.9.1)]{Rogawski1990Automorphicrepresentationsunitary}.
Then by \cite[Thm.\ 3.2.3]{Ferrari2007Theoremedelindice} (see also  \cite[Prop.\ 2]{Marshall2014Endoscopycohomologygrowth} and \cite[Lem.\ 7.1.2]{Gerbelli-Gauthier2019Growthcohomologyarithmetic}), $f = {\b1_{K_p^*(p^m)}}$ and $f^H = p^{-4m}{\b1_{K^H_p(p^m)}}$ are a transfer pair where $K^H_p(p^m) = U_2(\Z_p,p^m)\times U_1(\Z_p,p^m)$.   
Hence by \cite[Cor.\ 12.7.4]{Rogawski1990Automorphicrepresentationsunitary} (see also \cite[(5)]{Marshall2014Endoscopycohomologygrowth}),
\begin{equation}\label{eq:tr-A-packet}
\tr(\pi^n(\rho_p)(\b1_{K_p(p^m)})) + \tr(\pi^s(\rho_p)(\b1_{K_p(p^m)}))=\rho_p(p^{-4m}{\b1_{K_p^H(p^m)}}).
\end{equation}
By Proposition \ref{prop:level-Langlands}, if $\pi^n(\rho_p)^{K_p^*(p^m)} = 0$ then $\eta_p^{K_p^*(p^m)  \cap P^*_p} = 0$.
Since $\mu_p$ is unramified and $|\alpha |_p=1$, this is equivalent to $\rho_p^{K_p^H(p^m)} = 0$ and therefore $\rho_p(p^{-4m} \b1_{K_p^H(p^m)}) = 0$. 
Combined with equations \eqref{eq:dim-tr} and \eqref{eq:tr-A-packet} we get that $\tr(\pi^s(\rho_p)(\b1_{K_p^*(p^m)} ) = 0$, i.e. $\pi^s(\rho_p)^{K_p^*(p^m)} = 0$ as claimed.

(2) If $\pi^n(\rho_p)^{\boldsymbol{I}_p(p^e)} = 0$ then $d(\pi^n(\rho_p))\geq t$ so by Proposition \ref{prop:depth-preservation}, $d(\pi^s(\rho_p))\geq t$ and hence by definition $\pi^s(\rho_p)^{\boldsymbol{I}_p(p^e)} = 0$.
\end{proof}

\subsection{Sarnak-Xue Density Hypothesis} \label{automorphic:SXDH}

Let $H = U_2 \times U_1$ be the unique proper elliptic endoscopic group of $G^*$, the quasi-split inner form of $G$. 
Let $\mA^F_{H,\mu}$ be the set of $1$-dimensional automorphic representations of $H$ with central character $\mu$.
For each $\rho \in \mA^F_{H,\mu}$, let $\Pi(\rho)$ be its corresponding Rogawski A-packet of $G$.
Denote
\[
V_A(N) := \bigoplus_{\rho \in \mA^F_{H,\mu}} \bigoplus_{\pi \in \Pi(\rho) \cap \mA_{G,\b1}} \pi^{K(N)} \leq \bigoplus_{\pi \in \mA_{G,\b1}} \pi^{K(N)} = V(N).
\]

We shall use the following asymptotic notations:
Let $f$ and $g$ be two positive real valued functions in the variable $N$. 
Denote $f(N) \lesssim g(N)$ if for any $\e>0$ there exists $C_\e>0$ such that $f(N) \leq C_\e \cdot g(N)^{1+\e}$ for any $N$, and denote $f(N) \asymp g(N)$ if $f(N) \lesssim g(N)$ and $g(N) \lesssim f(N)$.

Therefore the Sarnak-Xue Density Hypothesis claims that
\[
\dim V_A(N)  \lesssim \dim V(N)^{1/2},
\]
and Theorem \ref{thm:SXDH} claims that
\[
\dim V_A(N)  \lesssim \dim V(N)^{3/8}.
\]
Note that $\dim V(N) =  \mbox{vol}(K(N)) \asymp N^{\dim G/Z} = N^8$.

\begin{proof}[Proof of Theorem \ref{thm:SXDH}]
For any $v$, fix a Haar measure on $G_v = G(\Q_v)$ which gives $K_v$ measure  $1$, and consider their product as a Haar measure on $G(\A)$ which gives $K(1)$ measure $1$.
Then for any representation $\pi$ of $G(\A)$ with trivial central character, 
\[
\tr(\pi(1_{K(N)})) = \mbox{vol}(K(N)) \cdot \dim (\pi^{K(N)}) \asymp N^{-8} \dim (\pi^{K(N)}),
\]
and similarly for any representation $\sigma$ of $H(\A)$ with central character $\mu$, 
\[
\tr(\sigma(1_{K^H(N)})) = \mbox{vol}(K^H(N)) \cdot \dim (\sigma^{K(N)}) \asymp N^{-4} \dim (\pi^{K^H(N)}).
\]
Therefore 
\begin{align*}
\dim V_A(N) &= \sum_{\rho \in \mA^F_{H,\mu}} \sum_{\pi \in \Pi(\rho) \cap \mA_{G,\b1}} \dim (\pi^{K(N)}) \\&= N^8 \sum_{\rho \in \mA^F_{H,\mu}} \sum_{\pi \in \Pi(\rho) \cap \mA_{G,\b1}}\prod_v \tr(\pi_v(1_{K_v(N)})).
\end{align*}
Note that $\tr(\pi_v(1_{K_v})) \leq 1$ for any $v\not \in S$, and $\max_{v \in S} \tr(\pi_v(1_{K_v})) \lesssim 1$. 
Then
\begin{align*}
\dim V_A(N) &\lesssim N^8 \cdot \sum_{\rho \in \mA^F_{H,\mu}} \sum_{\pi \in \Pi(\rho)} \prod_{\ell \mid N} \tr(\pi_\ell(1_{K_\ell(N)})) 
\\&= N^8 \sum_{\rho \in \mA^F_{H,\mu}}  \prod_{\ell \mid N} (\sum_{\pi_\ell \in \Pi(\rho_\ell)} \tr(\pi_\ell(1_{K_\ell(N)}) ) .
\end{align*}
By the endoscopic character relation \cite[Thm.\ 3.2.3]{Ferrari2007Theoremedelindice},  \cite[Prop.\ 2]{Marshall2014Endoscopycohomologygrowth}, \cite[Lem.\ 7.1.2]{Gerbelli-Gauthier2019Growthcohomologyarithmetic}, for any $\ell \mid N$, 
\[
\sum_{\pi_\ell \in \Pi(\rho_\ell)} \tr(\pi_\ell(1_{K_\ell(N)}) = \ell^{-2 \ord_\ell(N)} \tr(\rho_\ell(1_{K^H_\ell(N)}) ).
\]
Therefore
\[
\dim V_A(N) \lesssim N^6 \sum_{\rho \in \mA^F_{H,\mu}} \tr(\rho(1_{K^H(N)}) ) \asymp N^2 \sum_{\rho \in \mA^F_{H,\mu}} \dim (\rho^{K^H(N)}).
\]
Note that $\mA^F_{H,\mu}$ is the set of Hecke characters of $H$ with central character $\mu$, hence
\[
\sum_{\rho \in \mA^F_{H,\mu}} \dim (\rho^{K^H(N)}) = |\{\rho \in \mA^F_{H,\mu} \,:\, \rho^{K^H(N)} \ne 0 \}| \asymp N.
\]
Combined we get $\dim V_A(N) \lesssim N^3 \asymp \dim V(N)^{3/8}$, as claimed.
\end{proof}

\section{Ramanujan Bigraphs and Applications} \label{sec:applications}

Let $E$ be an imaginary quadratic extension of $\Q$, $p$ a prime inert in $E$, and $\Phi\in GL_3(E)$ a definite Hermitian form, such that $p\nmid\mathrm{disc}\,\Phi$, or equivalently, $\Phi\in GL_3(\mO_{E_p})$. We denote by $g^{\#}=\Phi^{-1}g^{*}\Phi$ the corresponding Hermitian involution. In this section we return to denote by $G$ the projective unitary group scheme $G=PU_3(E,\Phi)$.
We denote $G_p = G(\Q_p)$, $K_p = G(\Z_p)$, $\B_p$ the $(p^3+1,p+1)$-biregular Bruhat-Tits tree of $G_p$ and $\B_p^{hs}$ the set of hyperspecial vertices. 
We identify $\B_p^{hs}$ with $G_p/K_p$, with $K_p$ stabilizing the vertex $v_0$. 

For $g\in U_3(E,\Phi)(\Q_p)$, define the \emph{level} of $g$ to be 
\[
\ell(g) = -2\min_{i,j}\ord_pg_{ij} = 2\min\left\{ t\,\middle|\,p^t\cdot g\in M_3(\mO_{E_p})\right\}. 
\]
In particular, if $g\in U_3(E,\Phi)(\Z[1/p])$, then $\ell(g)$ is the minimal $\ell$ such that $p^{\ell/2}g\in M_3(\mO_{E})$.
Since $Z(G_p)=\{ \alpha I\,|\,\alpha\in U_1^E(\Q_p)\}$ and $U_1^E(\Q_p) = U_1^E(\Z_p)\le \mO_{E_p}^\times$ (due to the fact that $p$ is inert), we have $\ell(zg)=\ell(g)$ for any $z\in Z(G_p)$, hence the level is well defined on $G_p=PU_3\left(\Q_p\right)$.
Finally the level relates to the graph distance in $\B_p$ by 
\begin{equation} \label{eq:dist-level}
\forall h\in G_p,\qquad \dist_{\B_p}(hv_0,gv_0) = \dist_{\B_p}(v_0,h^{-1}gv_0) = \ell(h^{-1}g).
\end{equation}

This is proved in \cite[Prop.\ 3.3]{Evra2018RamanujancomplexesGolden} for $E=\Q[\sqrt{-1}]$ and $\Phi=I$, but up to replacing $g^{*}$ by $g^{\#}$, the proof follows verbatim to the general case. 

\begin{rem}
It is sometimes convenient to work with the larger group scheme of
\textit{unitary similitudes}, which is defined, using the notations of Section \ref{subsec:unitarygroups}, by
\[
GU_{3}\left(E,\Phi\right)(R)=\left\{ g\in GL_{3}\left(\mO_{E}\otimes_{\mO_{F}}R\right)\,\middle|\,g^{*}\Phi g=\lambda_{g}\Phi\text{ for some }\lambda_{g}\in R^{\times}\right\} .
\]
For example, the matrices described in \eqref{eq:set-Sp} below live naturally in the
group $GU_{3}(E,\Phi)(\Z[1/p])$. The distance equation \eqref{eq:dist-level}
remains true upon defining the
level of a similitude matrix $g\in GU_{3}(E,\Phi)(\mathbb{Q}_{p})$ by $\ell(g)=\ord_{p}(\lambda_{g})-2\min_{i,j}\ord_{p}g_{ij}$.
We remark that any similitude matrix can be scaled to be unitary, so that the projective similitude group coincides with the
projective unitary group; indeed, in any \textit{odd} dimension
$d$, for $g\in GU_{d}(E,\Phi)(R)$ we have $\frac{\lambda_{g}^{(d-1)/2}}{\det g}g\in U_{d}(E,\Phi)(R)$,
so that $g\mapsto\frac{\lambda_{g}^{(d-1)/2}}{\det g}g$ induces an
isomorphism $PGU_{d}(E,\Phi)\cong PU_{d}(E,\Phi)$.
\end{rem}

It follows from \eqref{eq:dist-level} that if $\Lambda\leq G(\Z[1/p])$ acts simply-transitively on $\B_p^{hs}$ then $S=\{ s\in\Lambda\,|\,\ell(s)=2\} $ takes $v_0$ once to each closest hyperspecial vertex. 
Note that $ps\in M_3\left(\mO_{E}\right)$ for any $s\in S$. 
To reconstruct the tree as a Cayley bigraph we need to determine when do $s,s'\in S$ take $v_0$ to neighbors of a common non-hyperspecial vertex. 

\begin{prop} \label{prop:nbs-crit} 
If $s,s'\in G(\Z[1/p])$ with $\ell(s),\ell(s')=2$, then 
\[
\dist(sv_0,s'v_0)\leq 2 \qquad \Leftrightarrow \qquad p\mid p^2s^{\#}s'.
\]
\end{prop}

\begin{proof}
Note first that $p^2s^{\#}s'$ is in $M_3(\mO_E)$. Now, by \eqref{eq:dist-level} and the fact that $s^{-1}=s^{\#}$ for $\Phi$-unitary matrices, we have
\begin{align*}
\dist\left(sv_0,s'v_0\right)&=\dist\left(v_0,s^{-1}s'v_0\right)=\dist\left(v_0,s^{\#}s'v_0\right)=\ell\left(s^{\#}s'\right) \\
& =-2\min_{i,j}\ord_p\left(s^{\#}s'\right)_{ij}=4-2\min_{i,j}\ord_p\left(p^2s^{\#}s'\right)_{ij}
\end{align*}
so that $\dist\left(sv_0,s'v_0\right)\leq 2$ if and only if $p\mid p^2s^{\#}s'$. 
\end{proof}


\subsection{Eisenstein case}

We now prove the main theorem for the Eisenstein case ($\EE$), and the differences in the other three cases ($\GG,\MM,\CC$) will be explained in the next subsection. 

\begin{thm}\label{thm:main-Eis}
Let $p,q$ be primes with $p\equiv2\Mod3$, $q\notin\left\{ 3,p\right\} $, and $\omega=\tfrac{-1+\sqrt{-3}}2$.
Let
\begin{equation}
S_p:=\Bigg\{ g\in M_3\left(\Z\left[\omega\right]\right)\,\Bigg|\,{g^{*}g=p^2I,\;g\text{ is not scalar},\atop g\equiv\left(\begin{smallmatrix}1 & * & *\\
* & 1 & *\\
* & * & 1
\end{smallmatrix}\right)\Mod3}\Bigg\},\label{eq:set-Sp}
\end{equation}
and let $S_p=\bigsqcup_iS_p^i$ be the partition induced by the equivalence relation 
\[
g\sim h\text{ if and only if }p\mid g^{*}h.
\]
Denote 
\[
\mathbf{G}_q:=\begin{cases}
PSL_3\left(\F_q\right) & q\equiv1\Mod3\\
PSU_3\left(\F_q\right) & q\equiv2\Mod3,
\end{cases}\quad\text{and}\quad S_{p,q}^i:=S_p^i\Mod{q}\overset{{\scriptscriptstyle (\star)}}{\subseteq}\mathbf{G}_q
\]
where $(\star)$ implies mapping $\omega$ to a root of $x^2+x+1$ in $\F_q$ or in $\F_{q^2}$ according to $q\Mod3$.
The Cayley bigraphs
\[
X_{\EE}^{p,q}=CayB\left(\mathbf{G}_q,\left\{ S_{p,q}^i\right\} _i\right)
\]
(see Definition \ref{def:biCay}) satisfy:
\begin{enumerate}
\item $X_{\EE}^{p,q}$ is an adj-Ramanujan $(p^3\!+\!1,p\!+\!1)$-regular
bigraph, with left side of size 
\[
\left|L_{X_{\EE}^{p,q}}\right|=\left|\mathbf{G}_q\right|=\begin{cases}
\left|PSL_3(q)\right|=\tfrac13\left(q^{8}-q^{6}-q^{5}+q^3\right) & q\equiv1\Mod3\\
\left|PSU_3(q)\right|=\tfrac13\left(q^{8}-q^{6}+q^{5}-q^3\right) & q\equiv2\Mod3.
\end{cases}
\]
\item $X_{\EE}^{p,q}$ is \uline{non}-Ramanujan: $\pm ip^{3/2}$
is in $\Spec B_{X_{\EE}^{p,q}}$, while $\rho\left(B_{\T_{p^3\!+\!1,p\!+\!1}}\right)=p$.
\item $X_{\EE}^{p,q}$ satisfies Sarnak-Xue density. In fact, $\Spec_0(B_{X_{\EE}^{p,q}})\subseteq\Spec B_{\T_{p^3\!+\!1,p\!+\!1}}\cup\left\{ \pm ip^{3/2}\right\} $,
and for any $\varepsilon>0$ there is $C_{\varepsilon}>0$ (not depending on $p$ and $q$) such that the multiplicity of $\pm ip^{3/2}$ in $\Spec B_{X_{\EE}^{p,q}}$ is bounded by  $C_{\varepsilon}\left|X_{\EE}^{p,q}\right|^{3/8+\varepsilon}$.\footnote{In fact, for Sarnak-Xue $C_{\varepsilon}\left|X_{\EE}^{p,q}\right|^{1/2+\varepsilon}$ would have been enough.}
\item The group $\mathbf{G}_q$ acts on the set
\[
Y^q:=\begin{cases}
\mathbb{P}^2(\F_q) & q\equiv1\Mod3\\
\left\{ v\in\mathbb{P}^2(\F_q[\omega])\,\middle|\,v^{*}\!\cdot\!v=0\right\}  & q\equiv2\Mod3,
\end{cases}
\]
and the Schreier bigraphs $Y_{\EE}^{p,q}=SchB\left(Y^{q},\left\{ S_{p,q}^i\right\} _i\right)$ are (fully) Ramanujan.
\item The girth of $X_{\EE}^{p,q}$ if larger than $2\log_pq$ (in the language of Section \ref{sec:Combinatorics}, $X_{\EE}^{p,q}$ has $6$-logarithmic girth).
\item \label{enu:X_E-Bounded-cutoff}The family $\left\{ X_{\EE}^{p,q}\right\} _q$ satisfies bounded cutoff: for any $0<\varepsilon<\frac{p^{3/2}}{6}$, the total-variation mixing time $t_{\star}$ of non-backtracking random walk on the edges of $X_{\EE}^{p,q}$ satisfies 
\[
\frac{1}{2}\log_{p}N-\frac{1}{2}\log_{p}\left(\tfrac1{\varepsilon}\right)<t_{1-\varepsilon}<t_{\varepsilon}<\frac{1}{2}\log_{p}N+2\log_{p}\left(\tfrac1{\varepsilon}\right)+3,
\]
where $N=N_{X_{\EE}^{p,q}}$ is the number of edges in $X_{\EE}^{p,q}$.
\item (Diameter) For $\varepsilon<\min\left(1,\frac{p^{3/2}}{6}\right)$ and $\ell\geq\frac{1}{2}\log_{p}N+2\log_{p}\left(\tfrac1{\varepsilon}\right)+3$, for any $e\in E_{X_{\EE}^{p,q}}$ we have
\[
\left|\left\{ e'\in E_{X_{\EE}^{p,q}}\,\middle|\,{\text{there is a non-backtracking path}\atop \text{of length \ensuremath{\ell} from \ensuremath{e} to \ensuremath{e'}}}\right\} \right|\geq\left(1-\varepsilon\right)N.
\]
Furthermore, for any two directed edges $e_1,e_2$ in $X_{\EE}^{p,q}$ there is a non-backtracking path from $e_1$ to $e_2$ of length at most $\log_{p}N+10$.
\end{enumerate}
\end{thm}

We remark that the case of $p=2$ has some special additional features - see Section \ref{subsec:eisenstein-lattice-explicit} for an explicit description of $S_2$ (called there $S$) and its action on the Bruhat-Tits tree of $G_2$.

\begin{proof}
\begin{enumerate}
\item For $p$ which is inert in $E=\Q\left(\sqrt{-3}\right)$, let $\Lambda_{\mE}^p$ be the lattice defined in Theorem \ref{thm-simply-transitive}. 
Let $\Lambda_{\mE}^p\left(q\right)=\Lambda_{\mE}^p\cap K_q\left(q\right)$, where $K_q(q)=\ker\left(G\left(\Z_q\right)\rightarrow G\left(\F_q\right)\right)$, which is a congruence subgroup in $PSU_3\left(\Q_p\right)$.
By Corollary \ref{cor:SA-classical} 
\[
\Lambda_{\EE}^p(q)\backslash\Lambda_{\EE}^p \cong \begin{cases} PSU_3\left(\F_q\right) & q\equiv2\Mod3\\ PSL_3\left(\F_q\right) & q\equiv1\Mod3 \end{cases}\quad = \; \mathbf{G}_q.
\]
In Theorem \ref{thm-simply-transitive} we show that $\Lambda_{\EE}^p$
acts simply-transitively on the left side of the $(p^3+1,p+1)$-biregular tree $\B_p$. 
We note that $s\mapsto p^{-1}s$ maps $S_p$ bijectively to the elements of level $2$ in $\Lambda_{\EE}^p$, and that by Proposition \ref{prop:nbs-crit}, for each $1\leq i\leq p^3+1$ the vertices $\left\{ sv_0\,\middle|\,s\in S_p^i\right\}$ share a common neighbor which we call $v_i$. 
Therefore, by Theorem \ref{thm:biCay-tree} we have $CayB\left(\Lambda_{\EE}^p,\left\{ S_p^i\right\} _i\right)\cong\B_p$, and furthermore, that
\begin{align*}
X_{\EE}^{p,q} & =CayB\left(\mathbf{G}_q,\left\{ S_{p,q}^i\right\} _i\right)\cong CayB\left(\Lambda_{\EE}^p\left(q\right)\backslash\Lambda_{\EE}^p,\left\{ S_{p,q}^i\right\} _i\right)\\ & =\Lambda_{\EE}^p\left(q\right)\backslash CayB\left(\Lambda_{\EE}^p,\left\{ S_p^i\right\} _i\right)\cong\Lambda_{\EE}^p\left(q\right)\backslash\B_p.
\end{align*}
Let $K'(q)=G(\R\widehat{\Z}^{3,q})K_3(C)K_q(q)$ where $K_3(C)=\left\{ g\in G\left(\Z_3\right)\,\middle|\,g\equiv\left(\begin{smallmatrix}1 & * & *\\
* & 1 & *\\
* & * & 1
\end{smallmatrix}\right)\Mod3\right\} $ (as in (\ref{eq:K3C})), and note that $\Lambda_{\EE}^p\left(q\right)=K'^p\left(q\right)\cap G\left(\Q\right)$ and $X_{\EE}^{p,q}=X_{K'(q)}^p$.
By Theorem \ref{thm:A-Ram} we get that $\mA_{G}\left(K'\left(q\right)\right)$
is A-Ramanujan, and by Theorem \ref{thm:ram-global-crit}(i) it follows
that $X_{\EE}^{p,q}$ is adj-Ramanujan. 
\item We prove this separately for $q\leq5$ and $q\geq5$. Let $q=2$, and $K'=G(\R\widehat{\Z}^{3,q})K_3(C)$, which satisfies $G(\Q)K'=G(\A)$ by Corollary \ref{cor:adelic-simply-transitive}. If $X_{\EE}^{p,q}$ was Ramanujan, then by Theorem \ref{thm:ram-global-crit}(ii) with $K''=K'\left(2\right)\trianglelefteq K'$,
$\mA_{G}\left(K'\left(2\right)\right)$
was Ramanujan at $p$, so that every automorphic $\pi$ of level $K'\left(2\right)$
was tempered or one-dimensional at $p$. By Corollary
\ref{cor:A-type}, this implies that $\pi$ is not in a global $A$-packet, hence by
Theorem \ref{thm:A-Ram}, $\pi$ is one-dimensional or tempered at
all unramified places, namely when $\ell\neq2,3$. In particular,
we obtain that every $\pi\in\mA_{G}\left(K'\left(2\right)\right)$
is tempered or one-dimensional at $5$, and thus $X_{\EE}^{5,2}$
is Ramanujan by Theorem \ref{thm:ram-global-crit}(i). However, this
is false by an explicit computation of $X_{\EE}^{5,2}$, whose
non-backtracking spectrum is shown in Figure \ref{fig:Eis-graphs}.
Thus $X_{\EE}^{p,2}$ is non-Ramanujan for any $p\neq2,3$.
For $q=5$, the same proof holds using the explicit computation
of $X_{\EE}^{2,5}$ (Figure \ref{fig:Eis-graphs}) as well
as by the general case below.

For any $q\geq5$, Theorem \ref{thm:NonRam-Eis}(2)
shows that $\mA_{G}\left(K'\left(q\right)\right)$ is non-Ramanujan,
and Theorem \ref{thm:ram-global-crit}(ii) with $K''=K'(q)$ shows that $X_{\EE}^{p,q}$
is non-Ramanujan for any $p\neq3,q$. More precisely, there exists
$\pi\in\mA_{G}\left(K'\left(q\right)\right)$ which sits in
a global $A$-packet, hence by Corollary \ref{cor:A-type}, $\pi_p=\pi^{n}(\rho_p)$ has Satake parameter $-p$
for any $p\neq3,q$, so by Proposition \ref{prop:Bspec_reps} it contributes the eigenvalues
$\pm ip^{3/2}$ to the non-backtracking spectrum.
\item By Equation (\ref{eq:exces-autom}) and Corollary \ref{cor:A-type}
we get 
\[
\mE\left(X_{K'\left(q\right)}^p\right)\leq\sum_{{\pi\in\mA_{G}\left(K'\left(q\right)\right)\atop \mathrm{Sat}(\pi_p)=-p}}\dim\pi^{K'\left(q\right)}=\sum_{\pi\in\mA_{G}^{A}\left(K'\left(q\right)\right)}\dim\pi^{K'\left(q\right)}=\dim V_{A}\left(q\right),
\]
where $\mA_{G}^{A}\left(K'\left(q\right)\right)$ is identified
with $\mA_{U_3^{E,\Phi},\mathbf1}^{A}\left(K'\left(q\right)\right)$
the set of $A$-type, level $K'\left(q\right)$, automorphic representations
of $U\left(E,\Phi\right)$, with trivial central character. By Theorem
\ref{thm:SXDH}, for any $\varepsilon>0$ there exists $C_{\varepsilon}$,
such that 
\[
\dim V_{A}\left(q\right)\leq C_{\varepsilon}\dim V\left(q\right)^{3/8+\varepsilon}.
\]
The claim now follows from the fact that 
\[
\dim V\left(q\right)=\left|G\left(\Q\right)\backslash G\left(\A\right)/K'\left(q\right)\right|=\left|G\left(\Q\right)\backslash G\left(\Q\right)K'/K'\left(q\right)\right|
\]
\[
=\left|\Lambda_{\EE}^p\left(q\right)\backslash G_p/K_p\right|=\left|L_{X_{\EE}^{p,q}}\right|\leq\left|X_{\EE}^{p,q}\right|
\]
\item The group $\mathbf{G}_q$ acts transitively on $Y_q$; fix $v_0\in Y_q$
and take $H_q=\mathrm{Stab}_{\mathbf{G}_q}\left(v_0\right)$,
so that $Y^{q}\cong H_q\backslash\mathbf{G}_q$ as a $\mathbf{G}_q$-set.
Let $K_q\left[q\right]=\left\{ k\in K_q\,\middle|\,k\Mod{q}\in H_q\right\} $,
and $K'[q]=G(\widehat{\Z}^{3,q}\R)K_3\left(C\right)K_q\left[q\right]$.
Observe that 
\[
\Lambda_{\EE}^p\left[q\right]:=K'^p\left[q\right]\cap G\left(\Q\right)=\left\{ g\in\Lambda_{\EE}^p\,\middle|\,g\Mod{q}\in H_q\right\} ,
\]
so that $Y^{q}\cong\Lambda_{\EE}^p\left[q\right]\backslash\Lambda_{\EE}^p$
as $\Lambda_{\EE}^p$-sets, and all together 
\begin{align*}
Y_{\EE}^{p,q} & =SchB\left(Y^{q},\left\{ S_{p,q}^i\right\} _i\right)\cong SchB\left(H_q\backslash\mathbf{G}_q,\left\{ S_{p,q}^i\right\} _i\right)\\&\cong SchB\left(\Lambda_{\EE}^p\left[q\right]\backslash\Lambda_{\EE}^p,\left\{ S_{p,q}^i\right\} _i\right)\\
 & =\Lambda_{\EE}^p\left[q\right]\backslash CayB\left(\Lambda_{\EE}^p,\left\{ S_p^i\right\} _i\right)\cong\Lambda_{\EE}^p\left[q\right]\backslash\B_p.
\end{align*}
By Theorem \ref{thm:ram-global-crit}(i), $Y_{E}^{p,q}$ is Ramanujan
if $\mA_{G}\left(K'\left[q\right]\right)$ is Ramanujan, and
the latter follows from the fact that $H_q$ contains a Borel subgroup
of $\mathbf{G}_q$, so that $K_q\left[q\right]$ contains an Iwahori
subgroup of $K_q$, and we can apply Corollary \ref{cor:Ram-Eis}
to $\mA_{G}\left(K'\left[q\right]\right)$.
\item If $\gamma=\mathrm{girth}(X^{p,q}_\EE)$ then there exists a non-backtracking path of length $\gamma$ in $\B_p$ which descends to a closed cycle modulo $\Lambda_{\EE}^p\left(q\right)$. As $\Lambda_{\EE}^p$ acts transitively on $\B_p^{hs}$, we can assume this path starts at $v_0$, so that it ends in $gv_0$ for $g$ satisfying both $g\in\Lambda_{\EE}^p\left(q\right)$ and $g^{*}g=p^{\gamma}I$. This implies that $p^{\gamma}=\sum_{j=1}^3N_{\Q[\omega]/\Q}(a_{1j})$
with $a_{12},a_{23}\in q\Z[\omega]$, and $a_{11}\in1+3q\Z[\omega]$.
If $a_{11}=1$ then either $a_{12}\neq0$ or $a_{13}\neq0$, so we
must have $p^{\gamma}\geq q^2+1$, and if $a_{11}\neq1$ then $p^{\gamma}\geq9q^2-6q+1$,
which is the smallest norm in $1+\left(3q\Z[\omega]\backslash\{0\}\right)$,
belonging to $-3q+1$. Thus we have
\[
\gamma>\log_pq^2=2\log_{\sqrt{Kk}}q^2=\tfrac2{4}\log_{\sqrt{Kk}}q^{8}>\tfrac2{5}\log_{\sqrt{Kk}}N
\]
for $q$ large enough, since $N=\left(p^3+1\right)\left|\mathbf{G}_q\right|$,
and $\left|\mathbf{G}_q\right|<q^{8}$.
\item The graphs $X_{\EE}^{p,q}$ are left-transitive (being Cayley
bigraphs), and by (1,3,5) they are adj-Ramanujan, have 6-logarithmic
girth, and for $\varepsilon>0$ satisfy $\mE\left(X_{\EE}^{p,q}\right)\leq C_{\varepsilon}\left|L_{X_{\EE}^{p,q}}\right|^{3/8+\epsilon}<N^{\delta}$
where $\delta=\frac2{1+\log_p(p^3)}=\frac12$ and $N$ is large enough. Theorem \ref{thm:cutoff-density}(2)
yields that for $\varepsilon\leq m^{\frac{\delta}{\delta-1}}\sqrt{K}=\frac{p^{3/2}}{6}$
we have
\[
t_{\varepsilon}-\log_{\sqrt{Kk}}N<\frac1{\delta}\log_{\sqrt{Kk}}\frac{K}{\varepsilon^2}=2\log_{p^2}\frac{p^3}{\varepsilon^2}=4\log_{\sqrt{Kk}}\left(\frac1{\varepsilon}\right)+3,
\]
and the lower bound on $t_{1-\varepsilon}$
holds for any $(K\!+\!1,k\!+\!1)$-bigraph by Theorem \ref{thm:cutoff-NBRW}.
\item Let $\ell\geq\frac{1}{2}\log_{p}N+2\log_{p}\left(\tfrac1{\varepsilon}\right)+3$ and \[Z_{e}=\left\{ e'\in E_{X_{\EE}^{p,q}}\,\middle|\,{\text{there is a non-backtracking path}\atop \text{of length \ensuremath{\ell} from \ensuremath{e} to \ensuremath{e'}}}\right\}. \]
Recall the notation of Subsection \ref{subsec:Cutoff} and assume
w.l.o.g.\ that $e\in\overrightarrow{LR}$. Since $\mathrm{supp}\left(\mathbf{p}_{e}^{\ell}\right)\subseteq Z_{e}$
we have by (6) and the definition of $\left\Vert \cdot\right\Vert _{TV}$
that
\[
1-\tfrac{\left|Z_{e}\right|}{N_{X_{\EE}^{p,q}}}=\mathbf{u}\left(\overrightarrow{LR}\backslash Z_{e}\right)=\left|\mathbf{p}_{e}^{\ell}\left(\overrightarrow{LR}\backslash Z_{e}\right)-\mathbf{u}\left(\overrightarrow{LR}\backslash Z_{e}\right)\right|\leq\left\Vert \mathbf{p}_{e}^{\ell}-\mathbf{u}\right\Vert _{TV}<\varepsilon,
\]
so that $|Z_e|\geq(1-\varepsilon)N$. In particular, at $\ell=\log_{\sqrt{Kk}}N_{X_{\EE}^{p,q}}+5$
we have $\left\Vert \mathbf{p}_{e}^{\ell}-\mathbf{u}\right\Vert _{TV}<\frac12$,
which implies that $Z_{e_1}$ must intersect with $\left\{ \overleftarrow{e}\,\middle|\,e\in Z_{\overleftarrow{e_2}}\right\} $
(where $\overleftarrow{e}$ indicates the opposite directed edge),
from which it follows that there is a non-backtracking path of length
$2\ell$ from $e_1$ to $e_2$.
\end{enumerate}
\end{proof}

\subsection{Gauss, Mumford and CMSZ cases}\label{subsec:Mum-CMSZ}

In this subsection we address the graphs arising from the Mumford
and CMSZ lattices, which give fully Ramanujan graphs. The Gauss lattice
which was studied in \cite{Evra2018RamanujancomplexesGolden} gives
adj-Ramanujan bigraphs as in the case of the Eisenstein lattice, and if Conjecture
\ref{conj:level-packet} holds for $E=\Q[i]$, $i=\sqrt{-1}$,
$p=2$ and $K'_2=K_2\left(C\right)=\left\{ g\in G\left(\Z_2\right)\,\middle|\,g\equiv\left(\begin{smallmatrix}1 & * & *\\
* & 1 & *\\
* & * & 1
\end{smallmatrix}\right)\Mod{2+2i}\right\} $, then they are fully Ramanujan.
\begin{thm}
\label{thm:main-Mum} Let $p,q,E,\Phi$ be either as in the Mumford
case or the CMSZ case:

\begin{tabular}{|c||c|c|}
\hline 
 & Mumford ($\MM$)  & CMSZ ($\CC$) \tabularnewline 
\hline 
 $\mO_{E}$ & $\Z[\lambda],\lambda=\frac{-1+\sqrt{-7}}2$ & $\Z[\eta]$,$\eta=\frac{1-\sqrt{-15}}2$
\tabularnewline 
\hline
 $p$ & $p\equiv3,5,6\Mod7$ & $p\equiv7,11,13,14\Mod{15}$ 
\tabularnewline
\hline
 $q$ & $q\notin\left\{ 2,7,p\right\} $ & $q\notin\left\{ 3,5,p\right\} $ 
\tabularnewline
\hline
  $\Phi$ &
$\begin{pmatrix}3 & \overline{\lambda} & \overline{\lambda}\\
\lambda & 3 & \overline{\lambda}\\
\lambda & \lambda & 3
\end{pmatrix}$ & $\begin{pmatrix}10 & -2(\eta+2) & \eta+2\\ -2(\bar{\eta}+2) & 10 & -2(\eta+2)\\ \bar{\eta}+2 & -2(\bar{\eta}+2) & 10
\end{pmatrix}$
\tabularnewline
\hline 
\end{tabular}

Let 
\[
S_p:=\Bigg\{ g\in M_3\left(\mO_{E}\right)\,\Bigg|\,{g^{*}\Phi g=p^2\Phi,\;g\text{ is not scalar},\atop g\Mod{M}\in H}\Bigg\},
\]
where $M=\lambda$ and $H$ is the group of upper-triangular matrices
in the Mumford case, and $M=\left(3,1+\eta\right)$ and $H$ is as
in (\ref{eq:CMSZ_H}) in the CMSZ case. Let $S_p=\bigsqcup_iS_p^i$
be the partition induced by $g\sim h\text{ if and only if }p\mid g^{\ast}h$. Let $\mathbf{G}_q$
be $PSL_3(\F_q)$ when $q$ splits in $E$ and $PSU_3\left(\F_q\right)$
otherwise, and let $S_{p,q}^i$ denote the image of $S_p^i$
in $\mathbf{G}_q$. Then the Cayley bigraph
\[
X^{p,q}=CayB\left(\mathbf{G}_q,\left\{ S_{p,q}^i\right\} _i\right)
\]
is a (fully) Ramanujan $(p^3\!+\!1,p\!+\!1)$-regular bigraph with
$10$-logarithmic girth, and hence satisfies the Bounded cutoff and
Diameter properties from Theorem \ref{thm:main-Eis}.
\end{thm}

\begin{proof}
Let $\Lambda^p$ be the lattice spanned by $S_p$ (which is either
$\Lambda_{\MM}^p$ or $\Lambda_{\CC}^p$). As
in the proof of Theorem \ref{thm:main-Eis} claim 1, we have a natural
identification $X^{p,q}\cong\Lambda^p\left(q\right)\backslash\B_p$.
To show it is fully Ramanujan we use again Theorem \ref{thm:ram-global-crit}(i)
applied to $K'\left(q\right)=G(\R\widehat{\Z}^{2,q})\boldsymbol{I}_2K_q\left(q\right)$
where $\boldsymbol{I}_2$ is the standard Iwahori in $G\left(\Z_2\right)$
in the Mumford case, and $K'\left(q\right)=G(\R\widehat{\Z}^{3,q})K_3\left(H\right)K_q\left(q\right)$
where $K_3\left(H\right)\leq G\left(\Z_3\right)$ corresponds
to (\ref{eq:CMSZ_H}) in the CMSZ case. In both cases, there is a
prime $r$ ramified in $E$ ($7$ for Mumford and $5$ for CMSZ) such
that $G\left(\Z_{r}\right)\subseteq K'\left(q\right)$, so
that by Theorem \ref{thm:Ram-gen}, $\mA_{G}\left(K'\left(q\right)\right)$
is Ramanujan. The bound on the girth follows from the fact that $g^{*}\Phi g=p^{\gamma}\Phi$,
$g\equiv I\Mod{q}$ and $\Phi\not\equiv0\Mod{q}$ implies $1\equiv p^{\gamma}\Mod{q}$,
which gives $q<p^{\gamma}$, hence $\gamma>\log_pq=2\log_{\sqrt{Kk}}q=\frac2{8}\log_{\sqrt{Kk}}q^{8}>\frac2{9}\log_{\sqrt{Kk}}N_{X_{M}^{p,q}}$. As $X^{p,q}$ are Ramanujan with logarithmic girth, we can now infer bouned cutoff directly from Theorem \ref{thm:bounded-cutoff}, and the diameter consequence follows exactly as for the Eisentein case.
\end{proof}

\begin{rem} 
Even if Conjecture \ref{conj:level-packet} does not hold for the Gauss case, that is, for $E=\Q[i]$, $p=2$ and $K'_2=K_2(C)$, the Cayley bigraphs $X_{\GG}^{p,q}=CayB\left(\mathbf{G}_q,\left\{ S_{p,q}^i\right\} _i\right)$ are still adj-Ramanujan, satisfy the SXDH, $10$-logarithmic girth,  Bounded cutoff and the Diameter properties from Theorem \ref{thm:main-Eis}.
\end{rem}

\subsection{Ramanujan and non-Ramanujan complexes} \label{subsec:complexes}

The results of this paper also give new explicit constructions
of Ramanujan complexes, as was done in \cite{Evra2018RamanujancomplexesGolden} using the Gauss lattice.
A new feature that we get here is the first explicit \emph{non-Ramanujan
}$\widetilde{A}_2$-complexes, obtained from the principal congruence subgroups
of the Eisenstein lattice. We briefly recall the relevant background,
and refer to \cite{Lubotzky2005a,Lubotzky2013,Evra2018RamanujancomplexesGolden}
for more details.

When the prime $p$ splits over $\Q[\omega]$, $\omega=\frac{-1+\sqrt{-3}}2$,
namely when $p\equiv1\Mod3$, the group $G_p=G\left(\Q_p\right)$
is isomorphic to $PGL_3\left(\Q_p\right)$, whose Bruhat-Tits
building $\B=\B_{3,p}$ is two dimensional and was described in Section \ref{subsec:biexp-example}. For a prime $q\neq p$ the quotient $X_{\EE}^{p,q}:=\Lambda_{\EE}^p\left(q\right)\backslash\B_p$
can be described as a Cayley complex of a finite group $\mathbf{G}_q$
(either $PU_3\left(\F_q\right)$, $PSU_3\left(\F_q\right)$, $PGL_3\left(\F_q\right)$ or $PSL_3\left(\F_q\right)$),
with respect to the generating set 
\[
S_p=\Bigg\{ g\in M_3\left(\Z\left[\omega\right]\right)\,\Bigg|\,{gg^{*}=pI,\;g\text{ is not scalar},\atop g\equiv\left(\begin{smallmatrix}1 & * & *\\
* & 1 & *\\
* & * & 1
\end{smallmatrix}\right)\Mod3}\Bigg\}
\]
reduced modulo $q$ (compare with (\ref{eq:set-Sp}), where $p$ is inert). 

Recall the Hecke operators $A_1$ and $A_2=A_1^*$ from Section \ref{subsec:biexp-example}. If $X=\Gamma\backslash\B$ is a Ramanujan complex, then we have
\[
\mathrm{Spec_0}\left(A_1|_{X}\right)\subseteq\mathrm{Spec}\left(A_1|_{\B_p}\right)=\left\{ p\left(\alpha+\beta+\overline{\alpha\beta}\right)\,\middle|\,\alpha,\beta\in\C,\ |\alpha|=|\beta|=1\right\} .
\]
Surprisingly, it was shown in \cite{kang2010zeta} that the other direction holds
as well: if $\mathrm{Spec}_0\left(A_1|_{X}\right)\subseteq\mathrm{Spec}\left(A_1|_{\B_p}\right)$
then $X$ is Ramanujan with respect to all geometric operators\footnote{The underlying reason is that every infinite-dimensional irreducible
unitary representation of $PGL_3(\Q_p)$ which is Iwahori-spherical
but not spherical happens to be tempered. This does not hold for $PGL_{4}$
and above \cite{kang2016riemann}.}. 

Unlike the case of rank-one $PU_3\left(\Q_p\right)$,
here there is a continuum of representations of $PU_3\left(\Q_p\right)\cong PGL_3\left(\Q_p\right)$
which can appear as the local factor at $p$ of an endoscopic lift
from $U_2\times U_1$. In general, $A_1$ acts on the unique
$K$-fixed vector in a principal series representation of $PGL_3\left(\Q_p\right)$
with Satake parameters $\left(z_1,z_2,z_3\right)$ by the scalar
$p\left(z_1+z_2+z_3\right)$. Denoting by $W_{\vec{z}}$ the
irreducible subrepresentation which contains this vector, the representations
which appear as local factors in A-packets are precisely $V_{\left(z\sqrt{p},z/\sqrt{p},z^{-2}\right)}$
for $z\in\C$ of norm $1$ (see e.g.\ \cite{kang2010zeta}).
The ``endoscopic spectrum'' of $A_1$ is thus the closed curve
\[
\mathfrak{E}_p:=\left\{ \Spec A_1\big|_{V_{\left(z\sqrt{p},z/\sqrt{p},z^{-2}\right)}^{K}}\,\middle|\,z\in S^1\right\} =\left\{ zp^{3/2}+\tfrac{p}{z^2}+z\sqrt{p}\,\middle|\,z\in S^1\right\} .
\]
Figure \ref{fig:A2-complexes} shows the nontrivial $A_1$-spectrum of some of the Eisenstein Cayley-complexes $X_{\EE}^{p,q}$, with the underlying group $\bold{G}_q$ indicated, and with the corresponding spectrum of $A_1$ on $\B$ in blue, and on the endoscopic curve $\mathfrak{E}_p$ in red.

\subsection{Golden gates} \label{subsec:Golden-gates}

The lattices constructed in this paper can be also used to give new constructions of \emph{Golden Gates}, which are optimal topological generators for unitary groups \cite{sarnak2015letter,Parzanchevski2018SuperGoldenGates,Evra2018RamanujancomplexesGolden}.
This was explored in details for the Gauss lattice $\Lambda_{\GG}^p$ in \cite{Evra2018RamanujancomplexesGolden}, so we only briefly explain this application here. 
Let $\Lambda_{*}^p$ be one of our lattices, and $\Gamma_{*}^p=PU_3(E,\Phi)\left(\Z[1/p]\right)$ as in Section \ref{sec:lattices}. 
We can observe the elements 
\[
S_p=\left\{ 1\neq g\in\Lambda_{*}^p\,\middle|\,g^{*}\Phi g= p'\Phi\right\},\qquad p'=\begin{cases}
p & p\text{ splits}\\
p^2 & p\text{ inert}
\end{cases} 
\]
as matrices in $PU_3(E,\Phi)(\R)$, which is isomorphic to the standard projective unitary group $PU(3)$, since $\Phi$ is definite. 
Any lattice $\Lambda_{*}^p\leq\Delta\leq\Gamma_{*}^p$ is of the form $G\left(\Q\right)\cap K$ for some $K=K^p\times G(\Z_p)\leq G(\widehat{\Z})$, and we say that $\Delta$ satisfies the Ramanujan property if every automorphic representation of $G/\Q$ of level $K$ is either one-dimensional or tempered. 
When $\Delta$ satisfies the Ramanujan property, words in $S_p$ cover $PU(3)$ in an almost optimal rate up to an element in the finite group $\Delta\cap G\left(\Z\right)$ \cite[Sec.\ 4]{Evra2018RamanujancomplexesGolden}. 
In addition, there is an algorithmic way to approximate elements in $PU(3)$ by elements in $\Delta$, and the word problem in $\Delta$ w.r.t.\ $S_p$ (up to $\Delta\cap G(\Z)$) is efficiently solvable: the Cayley graph (resp.\ bigraph) of $\left(\Lambda_{*}^p,S_p\right)$ is precisely the $1$-skeleton of the building $\B_p$, and each two vertices in the building are contained in a common apartment, where a shortest path can be quickly found. 
In the field of quantum computation, an element in $PU(3)$ represents a gate acting on a single qutrit, and the properties we have described imply that any gate can be approximated efficiently using the fundamental gates $S_p$ as building blocks.
When $S_p$ has both the almost optimal covering and an efficient solution to the word problem of $\Delta$, up to $D:=\Delta\cap G(\Z)$, we say that $S_p$ is a golden gate set of $PU(3)/D$.

\begin{prop} \label{prop:GG}
If $* \in \{\MM,\CC\}$, then $S_p$ is a golden gate set of $PU(3)$, and if $* \in \{ \EE, \GG \}$, then $S_p$ is a golden gate set of $PU(3)/G(\Z)$ (where $G(\Z)$ is computed in Lemma \ref{lem-(G,E,M,C)-(I)}).
\end{prop}

\begin{proof}
The efficient solution to the word problem follows from the fact that $\Lambda_*^p$ acts simply-transitively on $\B^{hs}$, as described above.
The almost optimal covering property follows from the Ramanujan property of $\Lambda_{*}^p$, when $* \in \{\MM,\CC\}$, and $\Gamma_{*}^p$, when $* \in \{ \GG,\EE\}$, which follows from \cite[Thm. 1.4]{Evra2018RamanujancomplexesGolden}.
\footnote{Conjecturally $\Lambda_{\GG}^p$ should also posses the Ramanujan property, see \cite{Evra2018RamanujancomplexesGolden}.}
\end{proof}

Considerations of error correction have prompted the search for Golden Gates which are elements of $PU(3)$ of finite order, and these were termed \emph{Super-Golden-Gates } in \cite{Parzanchevski2018SuperGoldenGates}.
Such gate sets for $PU(2)$ were constructed there by finding lattices which act simply-transitively on the edges of Bruhat-Tits trees of $PGL_2(\Q_2)$.
Several examples of super golden gates for $PU(3)$ arise from our lattices:

\begin{thm} \label{thm:super-GG-Eis}
Let $\sigma, \tau, A \in \Gamma_{\EE}^2$ be as in \eqref{eq:Eisentein-2-mats}, denote $\Sigma := \{\sigma, \tau, A \}$, $C := \langle \sigma, \tau \rangle$, $\Delta_{\EE}^2:=\left\langle \Sigma \right\rangle \leq \Gamma_{\EE}^2$ and let $S_2 \subset \Gamma_{\EE}^2$ as in Theorem \ref{thm:main-Eis}.
Then $\Delta_{\EE}^2 = C \ltimes \Lambda_{\EE}^2 \cong (\nicefrac{\Z}3)^2*\nicefrac{\Z}3$, $S^2 = \{c A^{\pm1} c^{-1} \,:\, c\in C \}$ and $\Delta_{\EE}^2$ acts simply-transitively on the edges of the Bruhat-Tits tree of $G(\Q_2)$. 
Furthermore, $\Sigma$ is a super golden gate for $PU(3)/D$, where $D = G(\Z) \cap \mathbf{I}_2$ for $\mathbf{I}_2 \leq G(\Z_2)$ an Iwahori subgroup and  $G(\Z) \cong S_3 \ltimes C_6^2$ by Lemma \ref{lem-(G,E,M,C)-(I)}.
\end{thm}

We note that if $\Delta_{\EE}^2$ had the Ramanujan property, then $\Sigma$ was a super-golden gate set for $PU(3)$.\footnote{Alternatively, one can proceed as in Section \ref{sec:Combinatorics} and replace the Ramanujan property by a density hypothesis to obtain optimal covering for $PU(3)$ itself.}

\begin{proof}
It is easy to check that  $\sigma, \tau, A$ are all of order $3$, that $C \cong (\nicefrac{\Z}3)^2$, that $A \in \Lambda_{\EE}^2$, that $C$ normalizes $\Lambda_{\EE}^2$ and that $\ell(c \gamma c^-1) = \ell(\gamma)$ for any $\gamma \in \Lambda_{\EE}^2$ and $c\in C$.
This implies $\{ c A^{\pm1} c^{-1} \,:\,c \in C \} = S_2$ and $\Delta_{\EE}^2 = C \ltimes \Lambda_{\EE}^2$, so in particular $\Delta_{\EE}^2$ is a congruence subgroup.
The group $C$ acts simply-transitively on the neighbors of $v_0$, whereas $A$ rotates the three edges leaving a neighbor of $v_0$, and since  $\Lambda_{\EE}^2$ acts transitively on hyperspecial vertices in $\B_2$, we get that $\Delta_{\EE}^2$ acts simply-transitively on all the edges, and by Bass-Serre theory we get $\Delta_{\EE}^2\cong(\nicefrac{\Z}3)^2*\nicefrac{\Z}3$.
To show that $\Sigma$ is a super golden gate for $PU(3)/D$ note that
\[
L^2(PU(3)/D) \cong  L^2(\Gamma_{\EE}^2 \backslash (G(\R) \times G(\Q_2)))^{\mathbf{I}_2} \cong L^2 (G(\Q) \backslash G(\A) )^{G(\widehat{\Z}^2)\mathbf{I}_2},
\]
and the elements of $\Sigma$ act on this space by elements in the Iwahori algebra of $G(\Q_2)$. 
The Ramanujan property holds by Theorem \ref{thm:Ram-gen}(1) which implies the almost optimal covering property. 
Navigation is achieved by the simply-transitive action on the edges.
Hence $\Sigma$ is a super golden gate for $PU(3)/D$.
\end{proof}

We give two more examples of super golden gates (the arguments are similar to previous construction and are left as an exercise for the reader): 

\begin{itemize}

\item Consider the Eisenstein case $G = PU_3(\Q[\omega],I)$ at the ramified prime $p=3$. 
Then $\Gamma_{\EE}^3 = G(\Z[1/3])$ acts transitively on the edges of the Bruhat-Tits building $\B_3$ of the ramified group $G(\Q_3)$, which is a $4$-regular tree, but it has no subgroup which acts simply-transitively on the edges of $\B_3$.
Consider the elements $\sigma = \bsmx & & 1\\ & 1 & \\ -1 & & \esmx$, $\tau_1 = \bsmx \overline{\omega} & -\omega & -\omega\\ -\omega & \overline{\omega} & \omega \\ -\overline{\omega} & \overline{\omega} & 1 \esmx$, $\tau_2 = \bsmx -\overline{\omega} & \omega & 1\\ \omega & -\overline{\omega} & -1\\
1 & -1 & -1 \esmx$ in $\Gamma_{\EE}^3 = G(\Z[1/3])$, and denote $\Sigma = \{\sigma, \tau_1, \tau_2\}$, $C = \langle \tau_1, \tau_2 \rangle$ and $\Delta = \langle \Sigma \rangle$.
The element $\sigma$ is of order $4$ and rotates the four edges containing a vertex $v_0$, the group $C$ acts transitively on the edges containing a neighbor of $v_0$, and $C \cong Q_8$, the quaternion group. 
It follows that $\Delta$ acts transitively (but not freely) on the edges of $\B_3$, and $\Sigma$ is a super-golden gate set for $PU(3)/D$, where $D = \Delta \cap \mathbf{I}_3$ for $\mathbf{I}_3 \leq G(\Z_3)$ an Iwahori subgroup. 

\item Consider the Mumford case $G = PU_3(\Q[\lambda],\Phi)$ at the split prime $p=2$. 
In \cite{mumford1979algebraic} it is shown that $\Gamma_{\MM}^2 = G(\Z[1/2])$ acts simply-transitively on the pointed triangles of the Bruhat-Tits building $\B_2$ of $G(\Q_2) \cong PGL_3(\Q_2)$.
Consider the elements $\sigma = \bsmx 1 &  & \overline{\lambda}\\ & & -1\\ & 1 & -1 \esmx$ and $\varsigma = \bsmx & \overline{\lambda} \\ & & \overline{\lambda} \\ \lambda \esmx$ in $\Gamma_{\MM}^2$, and denote $\Sigma := \{ \sigma, \varsigma\}$.
Both $\sigma$ and $\varsigma$ are of order $3$, and their action on the building $\B_2$ is such that $\sigma$ fixes an edge $e_0$ and
rotates the three triangles which contain $e_0$, whereas $\varsigma$
rotates one of these triangles around itself. 
This implies that $\langle \Sigma \rangle$ acts transitively on the pointed triangles in $\B_2$, so by Mumford's result we have $\langle \Sigma \rangle = \Gamma_{\MM}^2$.
We get that $\Sigma$ is a super-golden-gates set of $PU(3)$: 
The argument is similar to the above proof where
\[
L^2(PU(3)) \cong L^2(\Gamma_{\MM}^2\backslash(G(\R)\times G(\Q_2)))^{\mathbf{I}_2} \cong L^2(G(\Q)\backslash G(\A))^{G(\widehat{\Z}^2)\mathbf{I}_2},
\]
and the Ramanujan property holds by \cite[Thm. 1.4]{Evra2018RamanujancomplexesGolden} since $K_7 \subseteq G(\widehat{\Z}^2)\mathbf{I}_2$.
\end{itemize}

\subsection{Optimal strong approximation}\label{subsec:optimal-SA}

We note that the diameter property of Theorems \ref{thm:main-Eis} and \ref{thm:main-Mum} can be formulated as an optimal strong approximation property, or optimal lifting property, for our $p$-arithmetic lattices $\Lambda\leq G(\Z[1/p])\leq PSU_3(\Q_p)$, which act simply transitive on $\B_p^{hs}$. 
Let us first recall the optimal strong approximation property, which was first discovered in \cite{Sarnak2015Appendixto2015} for the arithmetic group $SL_2(\Z)$, in our context of $p$-arithmetic groups (see also \cite{Golubev2023SarnaksDensityConjecture}).

The $p$-arithmetic group $\Lambda$ is equipped for any unramified prime $q\ne p$ with a natural homomorphism, $\Lambda\rightarrow\mathbf{G}_q$, $g\mapsto g\Mod{q}$, where $\mathbf{G}_q=PSL_3(\F_q)$ when $q$ splits and $\mathbf{G}_q=PSU_3(\F_q)$ when $q$ is inert. 
By strong approximation (Corollary \ref{cor:SA-classical}), these modulo maps are onto, i.e. $\Lambda\Mod{q}=\mathbf{G}_q$.

Recall the level function, $\ell\,:\,\Lambda \rightarrow \N$, $\ell(g) = -2\min_{i,j}\ord_p g_{ij}$, and for $r \in 2\N$, denote by $B(r) \subset \Lambda$ the ball of radius $r$ according to this metric. 
Since $\Lambda$ acts simply transitively on the hyperspecial vertices of a $(p^3+1,p+1)$-regular tree $\B_p$ and since $\ell(g) = \dist_{\B_p}(v_0,gv_0)$, then $|B(r)| = p^{2r} + p^{2r-3}$.
In particular, by a simple union bound we get that if $B(r)\Mod{q}=\mathbf{G}_q$, or even $|B(r)\Mod{q}|= (1-\epsilon)|\mathbf{G}_q|$ for a fixed $\epsilon>0$, then $r$ is of size at least $\log_{p^2}|\mathbf{G}_q|$.

In recent years a strengthening of the strong approximation property for the arithmetic group $\Lambda$, called super strong approximation (which applies also for non arithmetic groups), states that the Cayley graphs $\left\{ \mbox{Cay}\left(\mathbf{G}_q,S\Mod{q}\right)\right\} _q$ form a family of expanders, for any fixed generating set $S\subset\Lambda$.
Take $S=\left\{ s\in\Lambda\,\middle|\,\ell\left(s\right)=2\right\}$ and note that the diameter of the Cayley graph $\mbox{Cay}\left(\mathbf{G}_q,S\Mod{q}\right)$ is the minimal $\ell$, such that $B\left(2\ell\right)\Mod{q}=\mathbf{G}_q$.
Since expander graphs have logarithmic diameter, we get that $B\left(C\cdot\log_{p^2}|\mathbf{G}_q|\right)\Mod{q}=\mathbf{G}_q$, where $C$ is a constant which depends only on the expansion parameter
of these graphs.

Following \cite{Sarnak2015Appendixto2015, Ghosh2014MetricDiophantineapproximation}, we now define the notions
of optimal covering and almost covering exponents and the property
of optimal strong approximation, in the form suited for our $p$-arithmetic groups.

\begin{defn}
Define the covering exponent of $\Lambda$ to be
\[
\kappa\left(\Lambda\right)=\liminf_{q\rightarrow\infty}\left\{ \kappa\,:\,B\left(\kappa\cdot\log_{p^2}|\mathbf{G}_q|\right)\Mod{q}=\mathbf{G}_q\right\} .
\]
Define the almost covering exponent of $\Lambda$ to be
\[
\kappa_{\mu}(\Lambda)=\liminf_{\varepsilon\rightarrow 0}\liminf_{q\rightarrow\infty}\left\{ \kappa\,:\,|B\left(\kappa\cdot\log_{p^2}|\mathbf{G}_q|\right)\Mod{q}|=\left(1-\varepsilon\right)|\mathbf{G}_q|\right\} .
\]
Say that $\Lambda$ has the \emph{optimal strong approximation }property if $\kappa_{\mu}\left(\Lambda\right)=1$.
\end{defn}

Note that by the above simple union bound and the super strong approximation property we get that
\[
1\leq\kappa_{\mu}\left(\Lambda\right)\leq\kappa\left(\Lambda\right)\leq C.
\] 
We note that for $q$ large enough it is easy to show that $\kappa\left(\Lambda\right)\leq2\cdot\kappa_{\mu}\left(\Lambda\right)$.
Hence if $\Lambda$ has the optimal strong approximation property then $1\leq\kappa\left(\Lambda\right)\leq2$. 
It natural to hope that $\kappa\left(\Lambda\right)=1$, however this is too strong of a property to expect for, as was shown in \cite{Sarnak2015Appendixto2015} in the case of $SL_2\left(\Z\right)$. 
Here we shall focus on almost covering exponent exclusively, and we leave it as an open question to determine or give lower and upper bounds on $\kappa\left(\Lambda\right)$. 

We now show how the diameter property of Theorems \ref{thm:main-Eis} and \ref{thm:main-Mum} implies the optimal strong approximation property.

\begin{prop} \label{prop:OSA}
Let $\Lambda\leq PSU_3(\Q_p)$ be one of the $p$-arithmetic lattices constructed in Theorems \ref{thm:main-Eis} or \ref{thm:main-Mum}, where $p$ is an inert prime. 
Then for any small enough $\varepsilon > 0$, if $r=\log_{p^2}|\mathbf{G}_q|+4\log_{p^2}\left(\tfrac1{\varepsilon}\right)+3$, then $|B(r)\Mod{q}| \geq (1-\epsilon)|\mathbf{G}_q|$.
In particular, $\Lambda$ has the optimal strong approximation property.
\end{prop}

\begin{proof}
The diameter property of Theorems \ref{thm:main-Eis} and \ref{thm:main-Mum} states that for $\left(1-\epsilon\right)N$ of the $N$ directed edges of the Cayley bigraph $X^{p,q}=\mbox{CayB}(\mathbf{G}_q,\left\{ S_{p,q}^i\right\} _i)$, there is a non backtracking path of length at most $r$ from the fundamental directed edge. 
This in particular implies that for $\left(1-\varepsilon\right)|\mathbf{G}_q|$ of the $|\mathbf{G}_q|$ left vertices of $X^{p,q}$, there is a path of length at most $r$ from the identity vertex. 
Since $B\left(r\right)\Mod{q}$ is precisely the set of elements in $\mathbf{G}_q$ which have a path of length at most $r$ from the identity, we get the claim.
\end{proof}

\subsection{Picard modular surfaces}\label{subsec:Picard}

Recall that the arithmetic bigraphs considered in this paper are of the following form $X_p(q) = \Lambda^p(q) \backslash \B_p$, where $\B_p$ is the Bruhat-Tits tree of the $p$-adic group $G(\Q_p) = PU_3(\Q_p)$ and $\Lambda^p(q) = \{ g\in\Lambda\,:\,g\equiv I\mod q\}$ is the level $q$ congruence subgroup of a $p$-arithmetic subgroup $\Lambda\leq G(\Z[1/p])$, for some projective unitary group scheme $G$ over $\Z$, w.r.t. to a matrix algebra and a definite
Hermitian form (called type (I) in the introduction). 
These graphs are $p$-adic analogues of Picard modular surfaces, which are $2$-dimensional complex orbifolds (and for $q$ large actually manifolds) of the form $X(q)=\Lambda(q) \backslash \mathbb{B}$, where $\mathbb{B}$ is the open unit ball in $\C^2$ which is the symmetric space of the Lie group $G(\R) = PU(2,1)$ and $\Lambda(q) = \{ g\in\Lambda\,:\,g\equiv I\mod q\}$ is the level $q$ congruence subgroup of an arithmetic subgroup $\Lambda\leq G(\Z)$, for some projective unitary group scheme $G$ over $\Z$, w.r.t. to a matrix algebra and an indefinite Hermitian form. 

It is interesting to ask whether the results we proved in this paper translate in any meaningful way from the $p$-adic world to the real world, namely, what can we say about the Picard modular surfaces.
We note that two of our main results in the $p$-adic world were: 
\begin{description}
\item[(ST)] Constructing $p$-arithmetic groups $\Lambda^p$ which act simply-transitively on the hyperspecial vertices of the Bruhat-Tits building $\B_p$, for any unramified prime $p$.
\item[(R)] Proving the Ramanujan property for the bigraphs $X_p(q)=\Lambda^p(q) \backslash \B_p$, for $\Lambda^p$ as above, any inert prime $p$ and $q$ running over all integers coprime to a single ramified prime.
\end{description}

On the one hand, property (ST) shows that the origin vertex forms a fundamental domain for the action of $\Lambda^p$ on $\B_p^{hs}$, in particular showing that $\mbox{Vol}(\Lambda^p \backslash \B_p) = 1$, which in turn enables us to give a Cayley bigraph structure on the finite bigraphs $X_p(q)=\Lambda^p(q)\backslash\B_p$.
We expect that the analogue of property (ST) in the setting of the Picard modular surfaces would be that  the action of $\Lambda$ on $\B$ has a simple fundamental domain, and in particular $\mbox{Vol}(\Lambda \backslash \mathbb{B})$ is small. 
See for instance \cite{Stover}, where it is shown that for $\Lambda = PU_{2,1}(\Z[\omega])$, the principal arithmetic subgroup in the Eisenstein case w.r.t. to the indefinite Hermitian from $\Phi = \mbox{diag}\left(1,1-1\right)$, the Picard modular surface $X(1) = \Lambda \backslash \mathbb{B}$ has minimal volume among all possible non-compact quotients of $\mathbb{B}$.

On the other hand, following Rogawski, a possible analogue of property (R) for the Picard modular surfaces is the vanishing of their $1$-dimensional cohomology. 
In \cite[Thm.15.3.1]{Rogawski1990Automorphicrepresentationsunitary}, Rogawski proved such vanishing for Picard modular surfaces of type (II), i.e. $G=PU(D,\sigma)$ where $D$ is a division algebra and $\sigma$ an involution of the second type. 
The $p$-adic analogue of this result is that the finite quotients $X_p(q) = \Lambda^p(q)\backslash\B_p$ are Ramanujan for unitary groups of type (II), i.e. $G=PU(D,\sigma)$ as before and $\sigma$ a definite involution (see \cite{ballantine2000ramanujan,Ballantine2011Ramanujanbigraphsassociated,ballantine2015explicit}).
Below we prove such a result for Picard modular surfaces of type (I), using our previous analysis of Rogawski's work in Section \ref{section:automorphic}.

\begin{thm} \label{thm:Picard}
Let $G=PU_3(E,\Phi)$ be a projective unitary group scheme over $\Z$ of type (I), where $E$ is a quadratic imaginary field and $\Phi$ an indefinite Hermitian form. 
Let $\Lambda=G(\Z)$, the principal arithmetic subgroup of $G$, $\Lambda(q)$ the principal congruence subgroup of level $q$, call $\Gamma\leq G(\Z)$ a level $q$ congruence subgroup if it contains $\Lambda(q)$ and denote its corresponding Picard modular surface by $X(\Gamma) = \Gamma \backslash \mathbb{B}$.
Then for any $q$ coprime to the discriminant of $E$ and any level $q$ congruence subgroup $\Gamma$,
\[
H^1\left( X(\Gamma),\C \right) = 0.
\]
\end{thm}

\begin{proof}
By Matsushima's formula we have the following decomposition
\[
H^1\left(X(\Gamma),\C \right) \cong \bigoplus_{\pi}H^1\left(\mathfrak{g},K_{\infty};\pi_{\infty}\right)\otimes\pi_{f}^{K_{f}(\Gamma)},
\]
where $\pi=\pi_{\infty}\pi_{f}$ runs over all discrete automorphic
representations of $G$, $\mathfrak{g}$ is the Lie algebra of $G_{\infty} = G(\R) = PU(2,1)$, $K_{\infty} = P\left(U(2)\times U(1)\right)$ is a maximal compact subgroup of $G_{\infty}$, $H^1\left(\mathfrak{g},K_{\infty};\pi_{\infty}\right)$ is the $1$-dimensional $\left(\mathfrak{g},K_{\infty}\right)$-cohomology of $\pi_{\infty}$, $K_{f}(\Gamma)\leq K_{f}=G(\hat{\Z})$ an adelic congruence subgroup such that $\Gamma=G(\Q)\cap K_{f}(\Gamma)$ and $\pi_{f}^{K_{f}(\Gamma)}$ the space of $K_{f}(\Gamma)$-fixed vectors. 
By \cite[Prop. 15.2.1, Sec. 14.4, Thm. 13.3.6(c)]{Rogawski1990Automorphicrepresentationsunitary}, $\pi$ and all of its local factors sit in $A$-packets. 
Let $p$ be a ramified prime, and by assumption $p\nmid q$, hence $K_p=G(\Z_p)$ is the $p$-factor of $K_{f}(\Gamma)=\prod_{\ell}K_{\ell}(\Gamma)$.
By  Theorem \ref{thm:ram-global-crit}(i), $\sigma^{K_p}=0$ for any member $\sigma$ in an $A$-packet at a ramified prime, in particular $\pi_p^{K_p(\Gamma)}=0$, hence $\pi_{f}^{K_{f}(\Gamma)}=0$, and therefore
\[
\bigoplus_{\pi}H^1\left(\mathfrak{g},K_{\infty};\pi_{\infty}\right)\otimes\pi_{f}^{K_{f}(\Gamma)}=0.\qedhere
\]
\end{proof}

It is interesting to note that we get from the above Theorem a property that is in a certain sense opposite to the famous virtual Betti conjecture, in the special case of arithmetic non-uniform lattices of $PU(2,1)$.
Recall that the virtual Betti conjecture asks whether for any lattice $\Gamma$ of $PU(2,1)$ there is a cover of $X(\Gamma)=\Gamma\backslash\mathbb{B}$ with non-trivial first cohomology. 
By the above theorem, for a non-uniform arithmetic lattice $\Gamma$ of $PU(2,1)$, we can find an infinite tower of lattices $\Gamma_i$ commensurable to $\Gamma$ with trivial first cohomology. 
This is done by changing $K_{f}(\Gamma)=\prod_{\ell}K_{\ell}(\Gamma)$, such that $K_p(\Gamma)$ is replaced with $K_p$ for some prime $p$ which ramifies at $E$, and then one is free to consider smaller subgroups of $K_{\ell}(\Gamma)$ at any other place. 
If $K_p(\Gamma)=K_p$ for some ramified prime to begin with, then the resulting lattices $\Gamma_i$ will actually be subgroups of $\Gamma$, in which case we get that the virtual Betti conjecture cannot hold for such covers $X(\Gamma_i)$ of $X(\Gamma)$.

\bibliographystyle{amsalpha}
\bibliography{mybib}

@article{Stover,
	author = {Stover, M.},
	journal = {Proceedings of the American Mathematical Society},
	number = {9},
	pages = {3045-3056},
	title = {Volumes of {P}icard modular surfaces},
	volume = {139},
	year = {2011}}

@Article{Ghosh2014MetricDiophantineapproximation,
  author       = {Ghosh, A. and Gorodnik, A. and Nevo, A.},
  title        = {Metric {D}iophantine approximation on homogeneous varieties},
  number       = {08},
  pages        = {1435--1456},
  volume       = {150},
  creationdate = {2017-03-30T00:00:00},
  journal      = {Compositio Mathematica},
  publisher    = {Cambridge Univ Press},
  year         = {2014},
}

@article{winnie2020ramanujan,
  title={The {R}amanujan conjecture and its applications},
  author={Li, W.C.W.},
  journal={Philosophical Transactions of the Royal Society A},
  volume={378},
  number={2163},
  pages={20180441},
  year={2020},
  publisher={The Royal Society Publishing}
}

@MastersThesis{Morovits2023DirectedExpanderGraphs,
  author       = {Morovits, Aviv},
  title        = {On Directed Expander Graphs},
  year         = {2024},
  creationdate = {2018-06-20T00:00:00},
  institution  = {Hebrew University of Jerusalem},
}

@article{evra2023cohomological,
  title     = {The Cohomological {S}arnak-{X}ue Density Hypothesis for ${SO}_5$},
  author    = {Evra, Shai and Gerbelli-Gauthier, Mathilde and Gustafsson, Henrik},
  journal   = {arXiv preprint arXiv:2309.12413},
  year      = {2023}
}

@article{oi2023depth,
  title={Depth-preserving property of the local Langlands correspondence for quasi-split classical groups in large residual characteristic},
  author={Oi, Masao},
  journal={manuscripta mathematica},
  volume={171},
  number={3-4},
  pages={529--562},
  year={2023},
  publisher={Springer}
}

@article {bellaiche:2009,
    AUTHOR = {Bella\"{\i}che, Jo\"{e}l and Chenevier, Ga\"{e}tan},
     TITLE = {Families of {G}alois representations and {S}elmer groups},
   JOURNAL = {Ast\'{e}risque},
  FJOURNAL = {Ast\'{e}risque},
    NUMBER = {324},
      YEAR = {2009},
     PAGES = {xii+314},
      ISSN = {0303-1179,2492-5926},
      ISBN = {978-2-85629-264-8},
   MRCLASS = {11F80 (11F70 11F85 11G40 11S25 14F30 14G22)},
  MRNUMBER = {2656025},
MRREVIEWER = {Adolfo\ Quir\'{o}s},
}

@book {flicker:2006,
    AUTHOR = {Flicker, Yuval Z.},
     TITLE = {Automorphic representations of low rank groups},
 PUBLISHER = {World Scientific Publishing Co. Pte. Ltd., Hackensack, NJ},
      YEAR = {2006},
     PAGES = {xii+485},
      ISBN = {981-256-803-4},
   MRCLASS = {11F70 (11G18 22E50)},
  MRNUMBER = {2251430},
MRREVIEWER = {David\ Whitehouse},
       DOI = {10.1142/9789812773623},
       URL = {https://doi.org/10.1142/9789812773623},
}

@article{gan2001exact,
	author = {Gan, Wee Teck and Hanke, Jonathan P and Yu, Jiu-Kang},
	journal = {Duke Mathematical Journal},
	number = {1},
	pages = {103--133},
	publisher = {Duke University Press},
	title = {On an exact mass formula of Shimura},
	volume = {107},
	year = {2001}}

@article{gan2023automorphic,
      title={Automorphic Forms and the Theta Correspondence}, 
      author={Wee Teck Gan},
journal   = {arXiv preprint arXiv:2303.14918},
  year      = {2023}
}

@article{cartwright1993groups,
	author = {Cartwright, Donald I and Mantero, Anna Maria and Steger, Tim and Zappa, Anna},
	journal = {Geometriae Dedicata},
	number = {2},
	pages = {143--166},
	publisher = {Springer},
	title = {Groups acting simply transitively on the vertices of a building of type $\tilde{A}_{2}$, {I}},
	volume = {47},
	year = {1993}}

@article{pan2002depth,
	author = {Pan, Shu-Yen},
	journal = {Duke Mathematical Journal},
	number = {3},
	pages = {531--592},
	publisher = {Duke University Press},
	title = {Depth preservation in local theta correspondence},
	volume = {113},
	year = {2002}}

@article{pan2001splittings,
	author = {Pan, Shu-Yen},
	doi = {10.2140/pjm.2001.199.163},
	fjournal = {Pacific Journal of Mathematics},
	issn = {0030-8730},
	journal = {Pacific J. Math.},
	mrclass = {22E46},
	mrnumber = {1847153},
	mrreviewer = {Wee Teck Gan},
	number = {1},
	pages = {163--226},
	title = {Splittings of the metaplectic covers of some reductive dual pairs},
	url = {https://doi.org/10.2140/pjm.2001.199.163},
	volume = {199},
	year = {2001}}

@article{konno2003note,
	author = {Konno, Takuya},
	journal = {Kyushu Journal of Mathematics},
	number = {2},
	pages = {383--409},
	publisher = {Faculty of Mathematics, Kyushu University},
	title = {A note on the {L}anglands classification and irreducibility of induced representations of p-adic groups},
	volume = {57},
	year = {2003}}

@article{Landherr1935quivalenzHF,
	author = {Walther Landherr},
	journal = {Abhandlungen aus dem Mathematischen Seminar der Universit{\"a}t Hamburg},
	pages = {245-248},
	title = {{\"A}quivalenz Hermitescher Formen {\"u}ber einem beliebigen algebraischen Zahlk{\"o}rper},
	volume = {11},
	year = {1935}}

@article{prasad1994unrefined,
	author = {Moy, Allen and Prasad, Gopal},
	fjournal = {Inventiones Mathematicae},
	issn = {0020-9910},
	journal = {Invent. Math.},
	mrclass = {22E50 (20G05)},
	mrnumber = {1253198},
	number = {1-3},
	pages = {393--408},
	title = {Unrefined minimal {$K$}-types for {$p$}-adic groups},
	volume = {116},
	year = {1994}}

@incollection{kudla2004tate,
	author = {Kudla, Stephen S},
	booktitle = {An introduction to the Langlands program},
	pages = {109--131},
	publisher = {Springer},
	title = {Tate's thesis},
	year = {2004}}

@article{ballantine2000ramanujan,
	author = {Ballantine, C.M.},
	creationdate = {2012-10-10T00:00:00},
	journal = {Can. J. Math.},
	number = {6},
	pages = {1121--1148},
	publisher = {Toronto, Published for the Canadian Mathematical Society by the University of Toronto Press.},
	title = {Ramanujan type buildings},
	volume = {52},
	year = {2000}}

@article{Ballantine2011Ramanujanbigraphsassociated,
	author = {Ballantine, Cristina and Ciubotaru, Dan},
	creationdate = {2018-05-20T00:00:00},
	doi = {10.1090/S0002-9939-2011-10856-6},
	fjournal = {Proceedings of the American Mathematical Society},
	issn = {0002-9939},
	journal = {Proc. Amer. Math. Soc.},
	number = {6},
	pages = {1939--1953},
	title = {Ramanujan bigraphs associated with {${\rm SU}(3)$} over a {$p$}-adic field},
	url = {https://doi.org/10.1090/S0002-9939-2011-10856-6},
	volume = {139},
	year = {2011}}

@incollection{ballantine2015explicit,
	author = {Ballantine, C. and Feigon, B. and Ganapathy, R. and Kool, J. and Maurischat, K. and Wooding, A.},
	booktitle = {Women in Numbers Europe},
	creationdate = {2016-11-02T00:00:00},
	pages = {1--16},
	publisher = {Springer},
	series = {Assoc. Women Math. Ser.},
	title = {Explicit construction of {R}amanujan bigraphs},
	year = {2015}}

@article{BiluLinial2006,
	author = {Bilu, Y. and Linial, N.},
	journal = {Combinatorica},
	number = {5},
	pages = {495--519},
	publisher = {Springer},
	title = {Lifts, discrepancy and nearly optimal spectral gap},
	volume = {26},
	year = {2006}}

@article{borel1976admissible,
	author = {Borel, A.},
	creationdate = {2013-08-31T00:00:00},
	journal = {Inventiones mathematicae},
	number = {1},
	pages = {233--259},
	publisher = {Springer},
	title = {Admissible representations of a semi-simple group over a local field with vectors fixed under an {I}wahori subgroup},
	volume = {35},
	year = {1976}}

@book{Brown1989,
	author = {Brown, K. S.},
	creationdate = {2017-01-24T00:00:00},
	doi = {10.1007/978-1-4612-1019-1},
	isbn = {0-387-96876-8},
	mrclass = {20-02 (20E32 20G15 22E99 51B25)},
	mrnumber = {969123},
	mrreviewer = {W. M. Kantor},
	pages = {viii+215},
	publisher = {Springer-Verlag, New York},
	title = {Buildings},
	url = {http://dx.doi.org/10.1007/978-1-4612-1019-1},
	year = {1989}}

@article{cartwright2003ramanujan,
	author = {Cartwright, D.I. and Sol{\'e}, P. and {\.Z}uk, A.},
	creationdate = {2012-08-31T00:00:00},
	journal = {Discrete mathematics},
	number = {1},
	pages = {35--43},
	publisher = {Elsevier},
	title = {Ramanujan geometries of type $\tilde{A}_n$},
	volume = {269},
	year = {2003}}

@article{casselman1980unramified,
	author = {Casselman, W.},
	creationdate = {2013-08-31T00:00:00},
	journal = {Compositio Mathematica},
	number = {3},
	pages = {387--406},
	publisher = {Noordhoff International},
	title = {The unramified principal series of $ p $-adic groups. {I}. {T}he spherical function},
	volume = {40},
	year = {1980}}

@article{Clozel2002Automorphicformsand,
	author = {Clozel, L.},
	creationdate = {2018-05-20T00:00:00},
	doi = {10.1007/BF02784510},
	fjournal = {Israel Journal of Mathematics},
	issn = {0021-2172},
	journal = {Israel J. Math.},
	mrclass = {11F70 (11G18 14G35)},
	pages = {175--187},
	title = {Automorphic forms and the distribution of points on odd-dimensional spheres},
	url = {https://doi.org/10.1007/BF02784510},
	volume = {132},
	year = {2002}}

@article{Haagerup1988,
	author = {Cowling, M. and Haagerup, U. and Howe, R.},
	creationdate = {2016-12-09T00:00:00},
	journal = {J. Reine Angew. Math.},
	pages = {97-110},
	title = {Almost {$L^2$} matrix coefficients.},
	url = {http://eudml.org/doc/153027},
	volume = {387},
	year = {1988}}

@book{davidoff2003elementary,
	author = {Davidoff, Giuliana and Sarnak, Peter and Valette, Alain},
	creationdate = {2016-11-03T00:00:00},
	publisher = {Cambridge University Press},
	series = {London Mathematical Society Student Texts},
	title = {Elementary number theory, group theory and {R}amanujan graphs},
	volume = {55},
	year = {2003}}

@article{deligne1974conjecture,
	author = {Deligne, Pierre},
	creationdate = {2017-02-26T00:00:00},
	journal = {Publications Math{\'e}matiques de l'Institut des Hautes {\'E}tudes Scientifiques},
	number = {1},
	pages = {273--307},
	publisher = {Springer},
	title = {La conjecture de {W}eil. {I}},
	volume = {43},
	year = {1974}}

@article{first2016ramanujan,
	author = {First, U.A.},
	creationdate = {2016-12-20T00:00:00},
	journal = {arXiv:1605.02664},
	title = {The {R}amanujan Property for Simplicial Complexes},
	year = {2016}}

@book{Garrett1997,
	author = {Garrett, Paul},
	creationdate = {2017-01-24T00:00:00},
	doi = {10.1007/978-94-011-5340-9},
	isbn = {0-412-06331-X},
	mrclass = {20G25 (20E42 51E24)},
	mrreviewer = {Herbert Abels},
	pages = {xii+373},
	publisher = {Chapman \& Hall, London},
	title = {Buildings and classical groups},
	url = {http://dx.doi.org/10.1007/978-94-011-5340-9},
	year = {1997}}

@article{Godsil1988Walkgeneratingfunctions,
	author = {Godsil, C. D. and Mohar, B.},
	booktitle = {Proceedings of the {V}ictoria {C}onference on {C}ombinatorial {M}atrix {A}nalysis ({V}ictoria, {BC}, 1987)},
	creationdate = {2018-05-20T00:00:00},
	doi = {10.1016/0024-3795(88)90245-5},
	fjournal = {Linear Algebra and its Applications},
	issn = {0024-3795},
	journal = {Linear Algebra Appl.},
	mrreviewer = {M. Doob},
	pages = {191--206},
	title = {Walk generating functions and spectral measures of infinite graphs},
	url = {https://doi.org/10.1016/0024-3795(88)90245-5},
	volume = {107},
	year = {1988}}

@phdthesis{greenberg1995spectrum,
	author = {Greenberg, Yoseph},
	creationdate = {2018-02-26T00:00:00},
	institution = {The Hebrew University},
	school = {Hebrew University},
	title = {On the spectrum of graphs and their universal covering},
	year = {1995}}

@article{grigorchuk1999asymptotic,
	author = {Grigorchuk, R.I. and {\.Z}uk, A.},
	creationdate = {2012-09-12T00:00:00},
	journal = {Random walks and discrete potential theory, Sympos. Math},
	pages = {188--204},
	title = {On the asymptotic spectrum of random walks on infinite families of graphs},
	volume = {39},
	year = {1999}}

@book{harris2001geometry,
	author = {Harris, M. and Taylor, R.},
	creationdate = {2017-03-30T00:00:00},
	isbn = {0-691-09090-4},
	mrclass = {11G18 (11F70 11S37 14G35 22E45)},
	pages = {viii+276},
	publisher = {Princeton University Press},
	series = {Annals of Mathematics Studies},
	title = {The geometry and cohomology of some simple {S}himura varieties},
	volume = {151},
	year = {2001}}

@inproceedings{hashimoto1989zeta,
	author = {Hashimoto, K.},
	booktitle = {Automorphic Forms and Geometry of Arithmetic Varieties},
	creationdate = {2012-10-23T00:00:00},
	journal = {Automorphic forms and geometry of arithmetic varieties.},
	pages = {211--280},
	series = {Advanced Studies in Pure Mathematics},
	title = {Zeta functions of finite graphs and representations of $p$-adic groups},
	volume = {15},
	year = {1989}}

@article{HLW06,
	author = {Hoory, S. and Linial, N. and Wigderson, A.},
	journal = {Bulletin of the American Mathematical Society},
	number = {4},
	pages = {439--562},
	publisher = {Providence, RI: The Society, 1979-},
	title = {Expander graphs and their applications},
	volume = {43},
	year = {2006}}

@inproceedings{Howe1979,
	author = {Roger {Howe} and Ilya I. {Piatetski-Shapiro}},
	booktitle = {Automorphic forms, representations and {L}-functions},
	location = {Corvallis},
	series = {Proc. Sym. Pure Math.},
	title = {A counterexample to the ''generalized {R}amanujan conjecture'' for (quasi-)split groups},
	volume = {33(1)},
	year = {1979}}

@article{ihara1966discrete,
	author = {Ihara, Y.},
	creationdate = {2012-10-22T00:00:00},
	journal = {Journal of the Mathematical Society of Japan},
	number = {3},
	pages = {219--235},
	publisher = {Mathematical Society of Japan},
	title = {On discrete subgroups of the two by two projective linear group over $p$-adic fields},
	volume = {18},
	year = {1966}}

@article{kamber2016lp,
	author = {Kamber, Amitay},
	creationdate = {2017-02-03T00:00:00},
	journal = {arXiv:1701.00154},
	title = {${L}_p$-expander Complexes},
	year = {2016}}

@article{kang2016riemann,
	author = {Kang, M.H.},
	creationdate = {2016-12-03T00:00:00},
	journal = {Journal of Number Theory},
	pages = {281--297},
	publisher = {Elsevier},
	title = {Riemann {H}ypothesis and strongly {R}amanujan complexes from ${GL}_n$},
	volume = {161},
	year = {2016}}

@article{kang2010zeta,
	author = {Kang, M.H. and Li, W.C.W. and Wang, C.J.},
	creationdate = {2012-10-23T00:00:00},
	journal = {Israel J. Math.},
	number = {1},
	pages = {335--348},
	publisher = {Springer},
	title = {The zeta functions of complexes from $\mathrm{PGL}(3)$: a representation-theoretic approach},
	volume = {177},
	year = {2010}}

@article{Kesten1959,
	author = {Kesten, H.},
	creationdate = {2013-08-22T00:00:00},
	fjournal = {Transactions of the American Mathematical Society},
	issn = {0002-9947},
	journal = {Trans. Amer. Math. Soc.},
	pages = {336--354},
	title = {Symmetric random walks on groups},
	volume = {92},
	year = {1959}}

@article{kotani2000zeta,
	author = {Kotani, M. and Sunada, T.},
	creationdate = {2016-12-03T00:00:00},
	journal = {J. Math. Sci. Univ. Tokyo},
	pages = {7--25},
	publisher = {Citeseer},
	title = {Zeta Functions of Finite Graphs},
	volume = {7},
	year = {2000}}

@article{li2004ramanujan,
	author = {Li, W.C.W.},
	creationdate = {2012-08-14T00:00:00},
	journal = {Geometric and Functional Analysis},
	number = {2},
	pages = {380--399},
	publisher = {Springer},
	title = {Ramanujan hypergraphs},
	volume = {14},
	year = {2004}}

@article{Lubetzky2017RandomWalks,
	author = {Lubetzky, E. and Lubotzky, A. and Parzanchevski, O.},
	creationdate = {2017-03-25T00:00:00},
	doi = {10.4171/JEMS/990},
	issue = {11},
	journal = {J. Eur. Math. Soc.},
	pages = {3441--3466},
	title = {Random walks on {R}amanujan complexes and digraphs},
	volume = {22},
	year = {2020}}

@article{lubetzky2016cutoff,
	author = {Lubetzky, Eyal and Peres, Yuval},
	creationdate = {2017-02-08T00:00:00},
	journal = {Geometric and Functional Analysis},
	number = {4},
	pages = {1190--1216},
	publisher = {Springer},
	title = {Cutoff on all {R}amanujan graphs},
	volume = {26},
	year = {2016}}

@article{Lubotzky2013,
	author = {Lubotzky, A.},
	creationdate = {2014-09-29T00:00:00},
	doi = {10.1007/s11537-014-1265-z},
	issn = {0289-2316},
	journal = {Japanese Journal of Mathematics},
	number = {2},
	pages = {137-169},
	publisher = {Springer Japan},
	title = {Ramanujan complexes and high dimensional expanders},
	url = {http://dx.doi.org/10.1007/s11537-014-1265-z},
	volume = {9},
	year = {2014}}

@article{Lub12,
	author = {Lubotzky, A.},
	creationdate = {2012-06-03T00:00:00},
	journal = {Bull. Amer. Math. Soc},
	pages = {113--162},
	title = {Expander graphs in pure and applied mathematics},
	volume = {49},
	year = {2012}}

@article{LPS88,
	author = {Lubotzky, A. and Phillips, R. and Sarnak, P.},
	creationdate = {2012-06-07T00:00:00},
	journal = {Combinatorica},
	number = {3},
	pages = {261--277},
	publisher = {Springer},
	title = {Ramanujan graphs},
	volume = {8},
	year = {1988}}

@article{Lubotzky2005a,
	author = {Lubotzky, A. and Samuels, B. and Vishne, U.},
	creationdate = {2012-06-07T00:00:00},
	journal = {Israel J. Math.},
	number = {1},
	pages = {267--299},
	publisher = {Springer},
	title = {{ \!}{R}amanujan complexes of type $\tilde{A}_{d}$},
	volume = {149},
	year = {2005}}

@article{marcus2013interlacing,
	author = {Marcus, A. and Spielman, D.A. and Srivastava, N.},
	creationdate = {2013-08-12T00:00:00},
	journal = {Annals of Mathematics},
	pages = {307-325},
	title = {Interlacing families {I}: Bipartite {R}amanujan graphs of all degrees},
	volume = {182},
	year = {2015}}

@article{margulis1988explicit,
	author = {Margulis, G.A.},
	creationdate = {2012-08-31T00:00:00},
	journal = {Problemy Peredachi Informatsii},
	number = {1},
	pages = {51--60},
	publisher = {Russian Academy of Sciences, Branch of Informatics, Computer Equipment and Automatization},
	title = {Explicit group-theoretical constructions of combinatorial schemes and their application to the design of expanders and concentrators},
	volume = {24},
	year = {1988}}

@article{mumford1979algebraic,
	author = {Mumford, David},
	creationdate = {2016-11-03T00:00:00},
	journal = {American Journal of Mathematics},
	number = {1},
	pages = {233--244},
	publisher = {JSTOR},
	title = {An algebraic surface with {K} ample,$({K}^2)= 9$, $p_g= q= 0$},
	volume = {101},
	year = {1979}}

@article{Nil91,
	author = {Nilli, A.},
	creationdate = {2012-06-07T00:00:00},
	journal = {Discrete Mathematics},
	number = {2},
	pages = {207--210},
	publisher = {Elsevier},
	title = {On the second eigenvalue of a graph},
	volume = {91},
	year = {1991}}

@article{Parzanchevski2018SuperGoldenGates,
	author = {O. Parzanchevski and P. Sarnak},
	creationdate = {2018-03-03T00:00:00},
	doi = {https://doi.org/10.1016/j.aim.2017.06.022},
	issn = {0001-8708},
	journal = {Advances in Mathematics},
	note = {Special volume honoring David Kazhdan},
	pages = {869--901},
	title = {{S}uper-{G}olden-{G}ates for {$PU(2)$}},
	url = {http://www.sciencedirect.com/science/article/pii/S0001870817301640},
	volume = {327},
	year = {2018}}

@book{Platonov1994Algebraicgroupsand,
	author = {Platonov, Vladimir and Rapinchuk, Andrei},
	creationdate = {2017-05-23T00:00:00},
	isbn = {0-12-558180-7},
	mrclass = {11E57 (11-02 20Gxx)},
	mrnumber = {1278263},
	pages = {xii+614},
	publisher = {Academic Press, Boston},
	series = {Pure and Applied Mathematics},
	title = {Algebraic groups and number theory},
	volume = {139},
	year = {1994}}

@incollection{Rogawski1992Analyticexpressionnumber,
	author = {Rogawski, J. D.},
	booktitle = {The zeta functions of {P}icard modular surfaces},
	mrclass = {11S37 (11F70 11S40 14G10 22E55)},
	mrreviewer = {Dipendra Prasad},
	pages = {65--109},
	publisher = {Univ. Montr\'eal, Montreal, QC},
	title = {Analytic expression for the number of points mod {$p$}},
	year = {1992}}

@book{Rogawski1990Automorphicrepresentationsunitary,
	author = {Rogawski, J. D.},
	creationdate = {2018-05-20T00:00:00},
	doi = {10.1515/9781400882441},
	isbn = {0-691-08586-2; 0-691-08587-0},
	mrclass = {22E55 (11F70 11R39 22-02)},
	pages = {xii+259},
	publisher = {Princeton University Press, Princeton, NJ},
	series = {Annals of Mathematics Studies},
	title = {Automorphic representations of unitary groups in three variables},
	url = {https://doi.org/10.1515/9781400882441},
	volume = {123},
	year = {1990}}

@misc{Sarnak2015Appendixto2015,
	author = {Sarnak, P.},
	creationdate = {2016-11-03T00:00:00},
	note = {Appendix to \cite{sarnak2015letter}},
	title = {Optimal lifting of integral points},
	year = {2015}}

@misc{sarnak2015letter,
	author = {Sarnak, P.},
	creationdate = {2016-11-03T00:00:00},
	note = {https://publications.ias.edu/sarnak/paper/2637},
	title = {Letter to {A}aronson and {P}ollington on the {S}olvay-{K}itaev {T}heorem and {G}olden {G}ates},
	year = {2015}}

@article{sarveniazi2007explicit,
	author = {Sarveniazi, A.},
	creationdate = {2012-08-31T00:00:00},
	journal = {Duke Mathematical Journal},
	number = {1},
	pages = {141--171},
	publisher = {Duke University Press},
	title = {Explicit construction of a {R}amanujan $\left(n_{1},n_{2},\ldots,n_{d-1}\right)$-regular hypergraph},
	volume = {139},
	year = {2007}}

@book{serre1980trees,
	author = {Serre, Jean-Pierre},
	creationdate = {2012-08-09T00:00:00},
	isbn = {3-540-10103-9},
	mrclass = {20H10 (05C05 22E50)},
	mrnumber = {607504},
	note = {Translated by John Stillwell},
	pages = {ix+142},
	publisher = {Springer-Verlag, Berlin-New York},
	title = {Trees},
	year = {1980}}

@article{Shin2011Galoisrepresentationsarising,
	author = {Shin, Sug Woo},
	creationdate = {2018-05-20T00:00:00},
	doi = {10.4007/annals.2011.173.3.9},
	fjournal = {Annals of Mathematics. Second Series},
	issn = {0003-486X},
	journal = {Annals of Mathematics},
	mrclass = {11F80 (11G18)},
	mrnumber = {2800722},
	mrreviewer = {Peter Scholze},
	number = {3},
	pages = {1645--1741},
	title = {Galois representations arising from some compact {S}himura varieties},
	url = {https://doi.org/10.4007/annals.2011.173.3.9},
	volume = {173},
	year = {2011}}

@article{spielman2011spectral,
	author = {Spielman, D.A. and Teng, S.H.},
	creationdate = {2013-08-13T00:00:00},
	journal = {SIAM J. Comput.},
	number = {4},
	pages = {981--1025},
	publisher = {SIAM},
	title = {Spectral sparsification of graphs},
	volume = {40},
	year = {2011}}

@incollection{sunada1986functions,
	author = {Sunada, Toshikazu},
	booktitle = {Curvature and topology of Riemannian manifolds},
	creationdate = {2018-02-26T00:00:00},
	pages = {266--284},
	publisher = {Springer},
	title = {L-functions in geometry and some applications},
	year = {1986}}

@article{brito2018spectral,
	author = {Brito, Gerandy and Dumitriu, Ioana and Harris, Kameron Decker},
	doi = {10.1017/S0963548321000249},
	journal = {Combinatorics, Probability and Computing},
	publisher = {Cambridge University Press},
	title = {Spectral gap in random bipartite biregular graphs and applications},
	year = {2021}}

@inproceedings{Tits1979Reductivegroupsover,
	author = {Tits, J.},
	booktitle = {Automorphic forms, representations and {L}-functions},
	location = {Corvallis},
	pages = {29-69},
	series = {Proc. Sym. Pure Math.},
	title = {Reductive groups over local fields},
	volume = {33(1)},
	year = {1979}}

@incollection{Sarnak2005NotesgeneralizedRamanujan,
	author = {Sarnak, P.},
	booktitle = {Harmonic analysis, the trace formula, and {S}himura varieties},
	creationdate = {2018-08-06T00:00:00},
	pages = {659--685},
	series = {Clay Math. Proc.},
	title = {Notes on the generalized {R}amanujan conjectures},
	volume = {4},
	year = {2005}}

@article{Gelbart1991LfunctionsFourier,
	author = {Gelbart, S. S. and Rogawski, J. D.},
	creationdate = {2018-08-06T00:00:00},
	journal = {Inventiones mathematicae},
	number = {1},
	pages = {445--472},
	publisher = {Springer},
	title = {L-functions and {F}ourier-{J}acobi coefficients for the unitary group ${U}(3)$},
	volume = {105},
	year = {1991}}

@article{Jacobowitz1962Hermitianformsover,
	author = {Jacobowitz, Ronald},
	journal = {Amer. J. Math.},
	number = {3},
	pages = {441--465},
	publisher = {JSTOR},
	title = {Hermitian forms over local fields},
	volume = {84},
	year = {1962}}

@inproceedings{flath1979decomposition,
	author = {Flath, D.},
	booktitle = {Automorphic forms, representations and {L}-functions},
	location = {Corvallis},
	pages = {179--183},
	series = {Proc. Sym. Pure Math.},
	title = {Decomposition of representations into tensor products},
	volume = {33(1)},
	year = {1979}}

@article{Ngo2010Lelemmefondamental,
	author = {Ng{\^o}, Bao Ch{\^a}u},
	journal = {Publications math{\'e}matiques de l'IH{\'E}S},
	number = {1},
	pages = {1--169},
	publisher = {Springer},
	title = {Le lemme fondamental pour les algebres de {L}ie},
	volume = {111},
	year = {2010}}

@article{Evra2018RamanujancomplexesGolden,
	author = {Evra, Shai and Parzanchevski, Ori},
	creationdate = {2018-11-07T00:00:00},
	journal = {Geometric and Functional Analysis},
	pages = {193-235},
	title = {Ramanujan complexes and {G}olden {G}ates in ${PU} (3)$},
	volume = {32},
	year = {2022}}

@article{hong2021spectrum,
  title={Spectrum of weighted adjacency operator on a non-uniform arithmetic quotient of ${PGL}_3$},
  author={Hong, Soonki and Kwon, Sanghoon},
  journal={arXiv preprint arXiv:2108.01275},
  year={2021}
}

@article{Chapman2019CutoffRamanujancomplexes,
	author = {Chapman, Michael and Parzanchevski, Ori},
	doi = {10.4171/cmh/537},
	fjournal = {Commentarii Mathematici Helvetici. A Journal of the Swiss Mathematical Society},
	issn = {0010-2571},
	journal = {Comment. Math. Helv.},
	mrclass = {05C81 (05C48 11E95 20E42 60J05)},
	mrnumber = {4468991},
	number = {3},
	pages = {431--456},
	title = {Cutoff on {R}amanujan complexes and classical groups},
	url = {https://doi.org/10.4171/cmh/537},
	volume = {97},
	year = {2022}}

@article{sarnak1991bounds,
	author = {Sarnak, Peter and Xue, Xiaoxi},
	journal = {Duke Mathematical Journal},
	number = {1},
	pages = {207--227},
	publisher = {Duke University Press},
	title = {Bounds for multiplicities of automorphic representations},
	volume = {64},
	year = {1991}}

@article{shin2020construction,
	author = {Shin, Sug Woo},
	journal = {Shimura Varieties},
	pages = {209},
	publisher = {Cambridge University Press},
	title = {Construction of automorphic {G}alois representations: The self-dual case},
	volume = {457},
	year = {2020}}

@article{Haines2008Parahoric,
	author = {T. Haines and M. Rapoport},
	doi = {https://doi.org/10.1016/j.aim.2008.04.020},
	issn = {0001-8708},
	journal = {Advances in Mathematics},
	number = {1},
	pages = {188-198},
	title = {Appendix: On parahoric subgroups},
	url = {https://www.sciencedirect.com/science/article/pii/S0001870808001333},
	volume = {219},
	year = {2008}}

@article{Kempton2016NonBacktrackingRandom,
	author = {Kempton, Mark Condie},
	date = {2016},
	journal = {Open Journal of Discrete Mathematics},
	publisher = {Scientific Research Publishing, Inc.},
	title = {Non-Backtracking Random Walks and a Weighted {I}hara's Theorem}}

@article{Golubev2022CutoffgraphsSarnakXue,
	author = {Konstantin Golubev and Amitay Kamber},
	date = {2022},
	doi = {https://doi.org/10.1016/j.ejc.2022.103530},
	issn = {0195-6698},
	journal = {European Journal of Combinatorics},
	pages = {103530},
	title = {Cutoff on graphs and the {S}arnak--{X}ue density of eigenvalues},
	url = {https://www.sciencedirect.com/science/article/pii/S0195669822000269},
	volume = {104}}

@article{Bordenave2019Eigenvaluesrandomlifts,
	author = {Bordenave, Charles and Collins, Beno\^{\i}t},
	doi = {10.4007/annals.2019.190.3.3},
	fjournal = {Annals of Mathematics. Second Series},
	issn = {0003-486X},
	journal = {Ann. of Math. (2)},
	mrclass = {60B20 (05C80 46L54)},
	mrnumber = {4024563},
	number = {3},
	pages = {811--875},
	title = {Eigenvalues of random lifts and polynomials of random permutation matrices},
	url = {https://doi.org/10.4007/annals.2019.190.3.3},
	volume = {190},
	year = {2019}}

@article{McKay1981expectedeigenvaluedistribution,
	author = {McKay, Brendan D.},
	doi = {10.1016/0024-3795(81)90150-6},
	fjournal = {Linear Algebra and its Applications},
	issn = {0024-3795},
	journal = {Linear Algebra Appl.},
	mrclass = {05C50 (05C80)},
	mrnumber = {629617},
	mrreviewer = {M. Doob},
	pages = {203--216},
	title = {The expected eigenvalue distribution of a large regular graph},
	url = {https://doi.org/10.1016/0024-3795(81)90150-6},
	volume = {40},
	year = {1981}}

@inproceedings{Carbone2001classificationrank1,
	author = {Carbone, Lisa},
	booktitle = {Lecture notes from a $p$-adic groups seminar at {H}arvard {U}niversity},
	date = {2001},
	organization = {Citeseer},
	title = {On the classification of rank 1 groups over non-archimedean local fields}}

@inproceedings{Tits1966Classificationalgebraicsemisimple,
	author = {Tits, J.},
	booktitle = {Algebraic {G}roups and {D}iscontinuous {S}ubgroups ({P}roc. {S}ympos. {P}ure {M}ath., {B}oulder, {C}olo., 1965)},
	date = {1966},
	mrclass = {20.27},
	mrnumber = {0224710},
	mrreviewer = {R. Steinberg},
	pages = {33--62},
	publisher = {Amer. Math. Soc., Providence, R.I.},
	title = {Classification of algebraic semisimple groups}}

@article{Marshall2014Endoscopycohomologygrowth,
	author = {Marshall, Simon},
	date = {2014},
	journal = {Compositio Mathematica},
	number = {6},
	pages = {903--910},
	publisher = {London Mathematical Society},
	title = {Endoscopy and cohomology growth on {$U(3)$}},
	volume = {150}}

@article{Angel2015nonbacktrackingspectrum,
	author = {Angel, Omer and Friedman, Joel and Hoory, Shlomo},
	journal = {Transactions of the American Mathematical Society},
	number = {6},
	pages = {4287--4318},
	title = {The non-backtracking spectrum of the universal cover of a graph},
	volume = {367},
	year = {2015}}

@article{Cartwright1993Groupsactingsimply,
	author = {Cartwright, Donald I and Mantero, Anna Maria and Steger, Tim and Zappa, Anna},
	date = {1993},
	journal = {Geometriae Dedicata},
	number = {2},
	pages = {167--223},
	publisher = {Springer},
	title = {Groups acting simply transitively on the vertices of a building of type $\tilde{A}_{2}$, {II}: The cases q= 2 and q= 3},
	volume = {47}}

@article{Kato2006ArithmeticstructureCMSZ,
	author = {Kato, Fumiharu and Ochiai, Hiroyuki},
	date = {2006},
	journal = {Journal of Algebra},
	number = {2},
	pages = {1166--1185},
	publisher = {Elsevier},
	title = {Arithmetic structure of {CMSZ} fake projective planes},
	volume = {305}}

@inproceedings{Tate1979Numbertheoreticbackground,
	author = {Tate, John},
	booktitle = {Automorphic forms, representations and L-functions (Proc. Sympos. Pure Math., Oregon State Univ., Corvallis, Ore., 1977), Part},
	pages = {3--26},
	title = {Number theoretic background},
	volume = {2},
	year = {1979}}

@article{Gerbelli-Gauthier2019Growthcohomologyarithmetic,
	author = {Gerbelli-Gauthier, Mathilde},
	journal = {arXiv preprint arXiv:1910.06900},
	title = {Growth of cohomology of arithmetic groups and the stable trace formula: the case of $ {U} (2, 1) $},
	year = {2019}}

@article{Ferrari2007Theoremedelindice,
	author = {Ferrari, Axel},
	journal = {Manuscripta mathematica},
	number = {3},
	pages = {363--390},
	publisher = {Springer},
	title = {Th{\'e}oreme de l'indice et formule des traces},
	volume = {124},
	year = {2007}}

@Article{Nestoridi2021Boundedcutoffwindow,
  author   = {Nestoridi, Evita and Sarnak, Peter},
  title    = {Bounded cutoff window for the non-backtracking random walk on {R}amanujan Graphs},
  doi      = {10.1007/s00493-023-00017-8},
  number   = {2},
  pages    = {367--384},
  volume   = {43},
  journal  = {Combinatorica},
  mrclass  = {60J10 (05E99 20C30)},
  mrnumber = {4627314},
  year     = {2023},
}

@Article{Brown1984Presentationsgroupsacting,
  author    = {Brown, Kenneth S},
  date      = {1984},
  title     = {Presentations for groups acting on simply-connected complexes},
  number    = {1},
  pages     = {1--10},
  volume    = {32},
  journal   = {Journal of Pure and Applied Algebra},
  publisher = {Elsevier},
}

@article{Moy1986RepresentationsU21,
	author = {Allen Moy},
	doi = {doi:10.1515/crll.1986.372.178},
	journal = {J. Reine Angew. Math.},
	lastchecked = {2023-03-21},
	number = {372},
	pages = {178--208},
	title = {Representations of {$U (2,1)$} over a p-adic field.},
	url = {https://doi.org/10.1515/crll.1986.372.178},
	year = {1986}}

@Article{Barbasch2013Unitaryequivalencesreductive,
  author    = {Barbasch, Dan and Ciubotaru, Dan},
  date      = {2013},
  title     = {Unitary equivalences for reductive p-adic groups},
  number    = {6},
  pages     = {1633--1674},
  volume    = {135},
  journal   = {American Journal of Mathematics},
  publisher = {Johns Hopkins University Press},
}

@InCollection{Yu2009BruhatTitstheory,
  author     = {Yu, Jiu-Kang},
  booktitle  = {Ottawa lectures on admissible representations of reductive {$p$}-adic groups},
  title      = {Bruhat-{T}its theory and buildings},
  doi        = {10.1090/fim/026/02},
  isbn       = {978-0-8218-4493-9},
  pages      = {53--77},
  publisher  = {Amer. Math. Soc., Providence, RI},
  series     = {Fields Inst. Monogr.},
  url        = {https://doi.org/10.1090/fim/026/02},
  volume     = {26},
  mrclass    = {20E42},
  mrnumber   = {2508720},
  mrreviewer = {Yair\ Glasner},
  year       = {2009},
}

@Article{Lusztig1989AffineHeckealgebras,
  author     = {Lusztig, George},
  title      = {Affine {H}ecke algebras and their graded version},
  doi        = {10.2307/1990945},
  issn       = {0894-0347,1088-6834},
  number     = {3},
  pages      = {599--635},
  url        = {https://doi.org/10.2307/1990945},
  volume     = {2},
  fjournal   = {Journal of the American Mathematical Society},
  journal    = {J. Amer. Math. Soc.},
  mrclass    = {16A64 (20H15 22E50)},
  mrnumber   = {991016},
  mrreviewer = {Fran\c{c}ois\ Digne},
  year       = {1989},
}

@Article{Selberg1956Harmonicanalysisdiscontinuous,
  author       = {Selberg, A.},
  date         = {1956},
  journaltitle = {J. Indian Math. Soc. (N.S.)},
  title        = {Harmonic analysis and discontinuous groups in weakly symmetric {R}iemannian spaces with applications to {D}irichlet series},
  issn         = {0019-5839,2455-6475},
  pages        = {47--87},
  volume       = {20},
  fjournal     = {The Journal of the Indian Mathematical Society. New Series},
  mrclass      = {10.1X},
  mrnumber     = {88511},
  mrreviewer   = {F.\ V.\ Atkinson},
}

@Article{Golubev2023SarnaksDensityConjecture,
  author       = {Golubev, Konstantin and Kamber, Amitay},
  date         = {2023},
  journaltitle = {Forum of Mathematics, Sigma},
  title        = {On {S}arnak's Density Conjecture and Its Applications},
  doi          = {10.1017/fms.2023.40},
  pages        = {e48},
  volume       = {11},
  publisher    = {Cambridge University Press},
}

@Book{Terras2011Zetafunctionsgraphs,
  author     = {Terras, Audrey},
  title      = {Zeta functions of graphs},
  isbn       = {978-0-521-11367-0},
  note       = {A stroll through the garden},
  pages      = {xii+239},
  publisher  = {Cambridge University Press, Cambridge},
  series     = {Cambridge Studies in Advanced Mathematics},
  volume     = {128},
  mrclass    = {05-02 (05C25 11M41)},
  mrnumber   = {2768284},
  mrreviewer = {Christopher\ Knox\ Storm},
  year       = {2011},
}

\end{document}